\documentclass[12pt]{amsart}
\usepackage{etex} 
\usepackage{amssymb}
\usepackage[noadjust]{cite}
\usepackage{booktabs}
\usepackage{url}
\usepackage{hyphenat}
\usepackage{mathtools}
\usepackage{cancel} 
\usepackage[T1]{fontenc}
\usepackage[utf8]{inputenc}
\usepackage{mdwtab}
\usepackage{enumitem}
\usepackage{bbding}
\usepackage{dsfont}
\DeclareMathAlphabet\mathbfcal{OMS}{cmsy}{b}{n}

\usepackage{mathrsfs}
\usepackage{xr-hyper}
\usepackage{xcolor}

\definecolor{darkblue}{rgb}{0,0,0.4}
\usepackage[pdftex,lmargin=1in, rmargin=1in, tmargin=1in, bmargin=1in]{geometry}
  \usepackage[bookmarks=true, bookmarksopen=true,%
    bookmarksdepth=3,bookmarksopenlevel=2,%
    colorlinks=true,%
    linkcolor=darkblue,%
    citecolor=darkblue,%
    filecolor=darkblue,%
    menucolor=darkblue,%
    urlcolor=darkblue]{hyperref}
  \hypersetup{pdftitle={Bordered HF- with torus boundary}}
  \hypersetup{pdfauthor={Robert Lipshitz, Peter S. Ozsváth and Dylan
      P. Thurston}}
\usepackage{color}
\usepackage{tikz}

\usetikzlibrary{matrix,arrows,cd}
\tikzstyle{every picture}=[> = latex']

\tikzset{taa/.style={->, double}}
\tikzset{moda/.style={->, dashed}}
\tikzset{alga/.style={->, thick}}
\newcommand{\st}{^\text{st}}

\newread\testin

\graphicspath{{draws/}{mpdraws/}{}}
\makeatletter
\def\input@path{{}{draws/}}
\makeatother

\def\mathcenter#1{%
  \vcenter{\hbox{$#1$}}%
}

\DeclareRobustCommand{\widebar}[1]{\overline{#1}{}}

\makeatletter
\newcommand\mi@kern[1]{%
  \settowidth\@tempdima{$\mi@obj^{#1}$}
  \kern-\@tempdima
  #1
  \settowidth\@tempdima{$\mi@obj$}
  \kern\@tempdima
}

\newtoks\mi@toksp
\newtoks\mi@toksb
\DeclareRobustCommand{\manyindices}[5]{
  \def\mi@obj{#5}
  \mi@toksp\expandafter{\mi@kern{#2}}
  \mi@toksb\expandafter{\mi@kern{#1}}
  \@mathmeasure4\textstyle{#5_{#1}^{#2}}
  \@mathmeasure6\textstyle{#5_{#3}^{#4}}
  \dimen0-\wd6 \advance\dimen0\wd4
  \@mathmeasure8\textstyle{\hphantom{{}_{#1}^{#2}}#5^{\the\mi@toksp#4}_{\the\mi@toksb#3}}
  \hbox to \dimen0{}{\kern-\dimen0\box8}
}
\makeatother 

\usepackage{ifpdf}
\ifpdf
  \let\textalt\texorpdfstring
\else
  \newcommand{\textalt}[2]{#1}
\fi

\newcommand{\RR}{\mathbb R}
\newcommand{\CC}{\mathbb C}
\newcommand{\ZZ}{\mathbb Z}

\newcommand{\FF}{\mathbb F}

\newcommand{\bD}{\mathbb{D}}

\newcommand{\wt}{\widetilde}

\newcommand\HHH{\mathbb{H}}

\newcommand\Hab{{\mathcal H}_{\alphas,\betas}}
\newcommand\HaHb{{\mathcal H}_{\alphas^H,\betas}}
\newcommand\Hapb{{\HD}_{\alphas',\betas}}

\newcommand{\co}{\nobreak\mskip2mu\mathpunct{}\nonscript
  \mkern-\thinmuskip{:}\penalty300\mskip6muplus1mu\relax}

\newcommand{\OneHalf}{{\textstyle\frac{1}{2}}}

\newcommand{\into}{\hookrightarrow}

\newcommand{\bdy}{\partial}
\newcommand{\lra}{\longrightarrow}

\newcommand{\lbracket}{[}
\newcommand{\rbracket}{]}

\newcommand{\spinc}{\mathfrak s}

\DeclareMathOperator{\divis}{div}

\DeclareMathOperator{\Sym}{Sym}

\DeclareMathOperator{\Hom}{Hom}

\DeclareMathOperator{\spin}{spin}
\newcommand{\SpinC}{\spin^c}
\newcommand{\Spinc}{\SpinC}

\DeclareMathOperator{\ind}{ind}

\DeclareMathOperator{\ev}{ev}
\DeclareMathOperator{\gr}{gr}

\DeclareMathOperator{\br}{br} 

\newcommand{\emb}{{\mathrm{emb}}} 

\DeclareMathOperator{\glueop}{glue}
\DeclareMathOperator{\indsq}{ind^\square}

\theoremstyle{plain}

\numberwithin{equation}{section}
\newtheorem{theorem}[equation]{Theorem}

\newtheorem{proposition}[equation]{Proposition}
\newtheorem{lemma}[equation]{Lemma}
\newtheorem{corollary}[equation]{Corollary}

\newtheorem{definition}[equation]{Definition}

\theoremstyle{definition}

\theoremstyle{remark}
\newtheorem{example}[equation]{Example}
\newtheorem{remark}[equation]{Remark}

\newcommand{\HF}{\mathit{HF}}
\newcommand{\HFa}{\widehat {\HF}}

\newcommand{\HFm}{{\HF}^-}
\newcommand{\HFmm}{{\mathbf{HF}}^-}

\newcommand{\CF}{{\mathit{CF}}}
\newcommand{\CFa}{\widehat {\mathit{CF}}}
\newcommand{\CFm}{\mathit{CF}^-}
\newcommand{\CFmm}{\mathbf{CF}^-}

\newcommand{\x}{\mathbf x}
\newcommand{\y}{\mathbf y}

\newcommand{\w}{\mathbf w}

\newcommand\HH{\mathit{HH}}

\newcommand\Hochschild\HH

\newcommand{\Ainf}{A_\infty}
\newcommand{\Alg}{\mathcal{A}}

\newcommand\Blg{\mathcal{B}}

\newcommand{\alphas}{{\boldsymbol{\alpha}}}
\newcommand{\betas}{{\boldsymbol{\beta}}}

\newcommand{\cM}{\mathcal{M}}

\newcommand{\ModPol}{\mathfrak{M}}
\newcommand{\BModPol}{\mathfrak{BM}}
\newcommand{\HModPol}{\mathfrak{HM}}

\newcommand{\cN}{\mathcal{N}}
\newcommand{\ocN}{\widebar{\mathcal{N}}}
\newcommand{\tcN}{\widetilde{\mathcal{N}}}

\newcommand{\tcM}{\widetilde{\mathcal{M}}}
\newcommand{\DD}{\textit{DD}}

\newcommand{\CFD}{\mathit{CFD}}
\newcommand{\CFDD}{\mathit{CFDD}}
\newcommand{\CFA}{\mathit{CFA}}

\newcommand{\CFDa}{\widehat{\CFD}}

\newcommand{\CFDm}{\CFD^-}
\newcommand{\CFAc}{\mathbf{CFA}^-} 
\newcommand{\MAlgc}{\mathbfcal{A}_-} 
\newcommand{\CFK}{\mathit{CFK}}

\newcommand{\CFKm}{\CFK^-}

\newcommand{\CFDDa}{\widehat{\CFDD}}
\newcommand{\CFDDm}{\CFDD^-}

\newcommand{\CFAa}{\widehat{\CFA}}

\newcommand{\Source}{{S^{\mspace{1mu}\triangleright}}}
\newcommand{\SourceSub}[1]{{S_{#1}^{\mspace{1mu}\triangleright}}}
\newcommand{\biSource}{T^{\mspace{1mu}\rotatebox{90}{$\scriptstyle\lozenge$}}}

\newcommand{\bdSource}{S^{\mspace{1mu}{\Box}}}

\newcommand{\cZ}{\mathcal{Z}}
\newcommand{\PtdMatchCirc}{\cZ}
\newcommand{\PMC}{\PtdMatchCirc}

\newcommand{\CircPts}{{\mathbf{a}}}

\newcommand{\dg}{\textit{dg} }

\newcommand{\fModule}{\mathfrak{M}}
\newcommand{\fNodule}{\mathfrak{N}}
\newcommand\Id{\mathbb{I}}
\newcommand\Ground{\mathds{k}}
\newcommand\Groundl{\mathrlap{\hspace{1.25pt}\mathrm{l}}{\mathrm{l}}} 

\newcommand\DT{\boxtimes}
\newcommand\Gen{\mathfrak{S}}

\newcommand\ModFlow{\tcM}
\newcommand\Tensor{\mathcal T}

\newcommand{\Field}{{\FF_2}}

\newcommand{\Ring}{R}
\DeclareMathOperator{\nbd}{nbd}

\newcommand{\dbar}{\bar{\partial}}
\newcommand{\Heegaard}{\mathcal{H}}
\newcommand{\HD}{\Heegaard}

\renewcommand{\th}{^\text{th}}
\renewcommand{\st}{^\text{st}}
\newcommand{\bigGroup}{G'}
\newcommand{\interGroup}{G}
\newcommand{\smallTGroup}{G(\mathbb{T})}

\newcommand{\grb}{\gr'}

\DeclareMathOperator{\Mor}{Mor}

\newcommand{\op}{\mathrm{op}}

\newcommand{\ol}[1]{\overline{#1}{}}

\makeatletter
\newcommand\honestalg[3]{\bigl\lbracket
\begin{smallmatrix} #1\@ifempty{#3}{}{&#3} \\ #2 \end{smallmatrix}
\bigr\rbracket}

\makeatother
\newcommand{\lab}[1]{$\scriptstyle #1$}

\newcommand{\lsub}[2]{{}_{#1}#2}

\newcommand{\AsUnDefAlg}{\Alg_-^{0,\mathrm{as}}} 
\newcommand{\UnDefAlg}{\Alg_-^{0}}

\newcommand{\MAlg}{\Alg_-} 
\newcommand{\MAlgg}{\Alg_-^g}

\newcommand{\CFAm}{\CFA^-}
\newcommand{\CFAmb}{\CFA^-_{nu}}
\newcommand{\CFAmg}{\CFA^-_{g}}

\newcommand\unit{\mathbf 1}

\newcommand{\corolla}[1]{\Psi_{#1}}
\newcommand{\wcorolla}[2]{\corolla{#1}^{#2}}

\newcommand{\DegenTree}{\mathord{\downarrow}}
\newcommand{\IdTree}{\DegenTree} 

\newcommand{\kotimes}[1]{\otimes}

\newcommand{\wADiag}{\boldsymbol{\Gamma}}
\newcommand{\wADiagCell}{\boldsymbol{\gamma}}
\newcommand{\wMDiag}{\mathbf{M}}
\newcommand{\TrMPrim}{\mathbf{p}}
\newcommand{\TrMDiag}{\mathbf{m}}

\newcommand{\Trees}{\mathcal{T}}

\newcommand{\wDiagCell}[2]{\wADiagCell^{#1,#2}}

\newcommand{\wModDiagCell}[2]{\mathbf{m}^{#1,#2}}
\newcommand{\wTrPMDiag}[2]{\mathbf{p}^{#1,#2}}
\newcommand{\wTrPMDiagNS}{\mathbf{p}}
\newcommand{\stump}{\top}

\DeclareMathOperator{\RootJoin}{RoJ} 
\DeclareMathOperator{\LeftJoin}{LeJ} 
\DeclareMathOperator{\LRjoin}{LR} 
\newcommand\LRjoinW{\LRjoin} 

\DeclareMathOperator{\wRootJoin}{\RootJoin} 

\newcommand{\Filt}{\mathcal{F}}

\newcommand{\wAlg}{\mathscr A}
\newcommand{\wBlg}{\mathscr B}

\newcommand{\One}{\boldsymbol{1}}

\newcommand{\CDisk}{\mathbb{D}}

\newcommand{\wTreesCx}[3][*]{X_{#1}^{#2,#3}}

\newcommand{\wMTreesCx}[3][*]{X\!M_{#1}^{#2,#3}}

\newcommand{\xwTreesCx}[3][*]{\widetilde{X}_{#1}^{#2,#3}}

\newcommand{\xwMTreesCx}[2]{\widetilde{X\!M}_*^{#1,#2}}

\newcommand\goesto\mapsto
\newcommand\Sphere{\mathbb S}
\newcommand\DegSource{S}
\newcommand{\suppo}[1]{\boldsymbol{[}#1\boldsymbol{]}}
\newcommand{\Associaplex}[2]{{\mathbf X}^{#1,#2}}

\newcommand{\Yvar}{\mathcal{Y}}
\newcommand{\uId}{\lsub{U}\!\CFDDm(\Id)}

\definecolor{darkgreen}{rgb}{0,0.4,0} 
\definecolor{darkbrown}{rgb}{.48,0.33,.24}  

\begin{document}
\title[Bordered \texorpdfstring{$\HFm$}{HF minus} with torus boundary]{Bordered \texorpdfstring{$\HFm$}{HF minus} for three-manifolds with torus boundary}
\author[Lipshitz]{Robert Lipshitz}
\thanks{RL was partly supported by NSF grant DMS-2204214, partly
    supported by NSF Grant DMS-1928930 while in residence at SLMath,
    and partly supported by a Simons Fellowship.}
\address{Department of Mathematics, University of Oregon\\
  Eugene, OR 97403}
\email{lipshitz@uoregon.edu}

\author[Ozsv\'ath]{Peter Ozsv\'ath}
\thanks{PSO was supported by NSF grant DMS-2104536.}
\address {Department of Mathematics, Princeton University\\ New
  Jersey, 08544}
\email {petero@math.princeton.edu}

\author[Thurston]{Dylan~P.~Thurston}
\thanks {DPT was supported by NSF grant DMS-2110143.}
\address{Department of Mathematics\\
         Indiana University,
         Bloomington, Indiana 47405\\
         USA}
\email{dpthurst@indiana.edu}

\begin{abstract}
  We define an invariant for bordered 3-manifolds with torus
  boundary, taking the form of a module over a
  weighted $A$-infinity algebra associated to a torus defined in
  previous work. On setting $U=0$, we obtain
  the bordered three-manifold invariants with torus boundary
  constructed earlier.
\end{abstract} 

\maketitle

\tableofcontents


\section{Introduction}
  
Heegaard Floer homology $\HFmm$ is a 3-manifold
invariant~\cite{OS04:HolomorphicDisks}, which is a graded module over
the ring of formal power series $\Field[[U]]$.
Its $U=0$ specialization is
a somewhat simpler 3-manifold invariant, denoted
$\HFa$. In~\cite{LOT1}, we extended $\HFa$ to define \emph{bordered Floer homology}, which associates 
\begin{itemize}
\item a differential graded algebra
  $\Alg(F)$  to a closed, connected surface $F$;
\item an $\Ainf$-module $\CFAa(Y_1)$ over $\Alg(F)$
  to a 3-manifold $Y_1$ equipped with an identification
  $\partial Y_1\cong F$; and
\item a \dg module (or type $D$ structure) $\CFDa(Y_2)$ over $\Alg(F)$
  to a 3-manifold $Y_2$ equipped with an identification
  $\partial Y_2\cong -F$.
\end{itemize}
A \emph{pairing theorem} expresses as a (suitable) tensor product the
invariant $\HFa$ for a $3$-manifold~$Y$ that is divided into two
pieces $Y_1$ and $Y_2$ along a connected surface $F$: if $\CFa(Y)$
denotes the chain complex whose homology is $\HFa(Y)$ then
\[
  \CFa(Y)\simeq \CFAa(Y_1)\DT \CFDa(Y_2).
\] 

Although $\HFa$ is adequate for many topological applications, the
unspecialized version $\HFmm$ carries rich additional
structure, making it, for example, the basic building block for the
invariant of smooth, closed
$4$-manifolds~\cite{OS06:HolDiskFour}.

In this paper, we construct a generalization of bordered Floer
homology for 3-manifolds with torus boundary,
with a view towards expressing
the unspecialized Heegaard Floer homology $\HFmm(Y)$ for a $3$-manifold
that is divided into two along a torus.

The general structure is as follows. Roughly, to $T^2$ we associate an
algebra $\MAlg=\MAlg(T^2)$. More precisely, $\MAlg$ is a curved formal
deformation of an $\Ainf$-algebra or what we call a weighted
algebra. (See Section~\ref{sec:intro:alg} or \cite{LOT:torus-alg}.) To
a 3-manifold~$Y$ with
boundary identified with $T^2$ we associate a (weighted) right
$\Ainf$-module $\CFAm(Y)$ over $\MAlg$, which is well-defined up to
(weighted) homotopy equivalence. There is also a (weighted) type \DD\
bimodule $\CFDDm(\Id)$ over $\MAlg$ which is, at least
philosophically, associated to the identity cobordism of $T^2$. In
future work~\cite{LOT:torus-pairing} we will establish a pairing theorem:
if $\bdy Y_2=T^2=-\bdy Y_1$ then
\[
  \bigl(\CFAm_{U=1}(Y_1)\otimes_{\Field}\CFAc(Y_2)\bigr)\DT_{\MAlg^{U=1}\otimes\MAlgc}\CFDDm(\Id)\simeq \CFmm(Y_1\cup_{\bdy}Y_2),
\]
where $\DT$ is an appropriate tensor product and $\CFAc$ and $\MAlgc$
denote appropriate completions of $\CFAm$ and $\MAlg$ (see
Section~\ref{sec:boundedness}). With a little further
algebra (see Section~\ref{sec:intro-primitives} or~\cite{LOT:abstract}), we
can re-associate this tensor product as
\[
  \CFAc(Y_2)\DT_{\MAlgc}(\CFAm_{U=1}(Y_1)\DT_{\MAlg^{U=1}}\CFDDm(\Id))
\]
and so, if we define
\[
  \CFDm(Y_1)=\CFAm_{U=1}(Y_1)\DT_{\MAlg^{U=1}}\CFDDm(\Id)
\]
then we recover the familiar-looking pairing theorem
\[
  \CFmm(Y_1\cup_{\bdy}Y_2)\simeq \CFAc(Y_2)\DT_{\MAlgc}\CFDm(Y_1),
\]
which is a strict generalization of the pairing theorem for $\HFa$
for $3$-manifolds glued along a torus boundary~\cite{LOT1}.

This sketch hides some algebraic details about weighted algebras and
modules, as well as the definitions of the invariants and proofs of
the main results. We spell out some of these missing details presently
and give references for some of the others.

\begin{remark}
  In~\cite{OS04:HolomorphicDisks}, $\HFm$ denoted a module defined over a
  polynomial ring $\Field[U]$. The invariant $\HFmm$ we consider here is the
  completion of this module over the ring of power series
  $\Field[[U]]$. With additional bookkeeping, one could work out the
  uncompleted theory; but we have not done so, in part because
  the current
  three-dimensional and four-dimensional applications of Heegaard Floer homology
  can be formulated  in terms of the completed theory.
\end{remark}

\subsection{The torus algebra}\label{sec:intro:alg}
The algebra $\MAlg$ is studied in detail in our previous
paper~\cite{LOT:torus-alg}. We review aspects of the construction here.
\subsubsection{Weighted algebras}
A \emph{weighted $\Ainf$-algebra}, or simply \emph{weighted algebra}, over a
commutative ring $\Ground$ of characteristic $2$ is a curved
$\Ainf$-algebra $\Alg$ over $\Ground[[t]]$
together with an isomorphism
of $\Ground[[t]]$-bimodules $\Alg\cong A\otimes_{\Ground}\Ground[[t]]$
so that the curvature is contained in
$tA$~\cite[Definition~\ref{Abs:def:wAinfty}]{LOT:abstract} (see
also~\cite[Remark~\ref{Abs:remark:weighted-is-deformation}]{LOT:abstract} and~\cite[Definition~\ref{TA:def:weighted-alg}]{LOT:torus-alg}). Given a weighted
$\Ainf$-algebra $\Alg$ we can extract operations
\[
\mu_n^w\co A^{\otimes n}\to A
\]
by 
\[
  \mu_n=\sum_{w=0}^\infty \mu_n^wt^w.
\]
(We are suppressing the grading shifts. A more precise version
is given in Section~\ref{sec:gradings-abstract};  see especially
Formula~\eqref{eq:graded-mu}.)
The condition that the curvature lies in $tA$ is equivalent to the
condition that $\mu_0^0=0$. We will typically refer to
$\Alg=(A,\{\mu_n^w\})$, not $(\Alg,\{\mu_n\})$, as the weighted
$\Ainf$-algebra. The condition that $(\Alg,\{\mu_n\})$ is a
curved $\Ainf$-algebra is equivalent to the 
condition that the $\mu_n^w$ satisfy the \emph{weighted $\Ainf$-relations}
\begin{equation}\label{eq:wAlg-rel}
  \sum_{\substack{1\leq i\leq j\leq n+1\\ w_1+w_2=w}}\mu^{w_1}_{n-j+i+1}(a_1,\dots,a_{i-1},\mu_{j-i}^{w_2}(a_i,\dots,a_{j-1}),a_{j},\dots,a_n)=0,
\end{equation}
for each $w\geq 0$ and $n\geq 0$.

A weighted $\Ainf$-algebra has an \emph{undeformed $\Ainf$-algebra}
$(A,\{\mu_n^0\})$, and an
\emph{underlying $\Ground$-bimodule} $A$.

\subsubsection{Building the algebra in three steps}\label{sec:build-alg-combinatorially}
We define the weighted algebra $\MAlg$ associated to the torus in three steps:
\begin{enumerate}
\item There is an associative algebra $\AsUnDefAlg$, which is the path algebra with relations:
  \[
    \mathcenter{
    \begin{tikzpicture}
      \node at (0,0) (iota0) {$\iota_0$};
      \node at (3,0) (iota1) {$\iota_1$};
      \draw[->, bend left=45] (iota0) to node[above]{\lab{\rho_1}} (iota1);
      \draw[->, bend right=20] (iota1) to node[above]{\lab{\rho_2}} (iota0);
      \draw[->, bend right=20] (iota0) to node[above]{\lab{\rho_3}} (iota1);
      \draw[->, bend left=45] (iota1) to node[above]{\lab{\rho_4}} (iota0);
    \end{tikzpicture}}/(\rho_2\rho_1=\rho_3\rho_2=\rho_4\rho_3=\rho_1\rho_4=0).
  \]
\item We define an $\Ainf$-algebra $\UnDefAlg$ which is a
  deformation of $\AsUnDefAlg$. That is, as an $\Field[U]$-module,
  $\UnDefAlg=\Field[U]\otimes_{\Field}\AsUnDefAlg$; $\UnDefAlg$ has
  trivial differential; and $\mu_2$ on $\UnDefAlg$ is induced by
  multiplication on $\AsUnDefAlg$. 
\item Finally, we define a weighted algebra $\MAlg$ whose undeformed
  $\Ainf$-algebra is $\UnDefAlg$.
\end{enumerate}

It is convenient to formulate $\UnDefAlg$ and $\MAlg$ using its gradings.
Gradings are specified in terms of a group
law on $(\OneHalf\ZZ)\times\ZZ^4$, with multiplication specified by
\begin{multline*}
(m;a,b,c,d)\cdot (m';a',b',c',d')
\\=
\left(m+m'
+\OneHalf\left|
  \begin{smallmatrix}
    a & b\\
    a' & b'
  \end{smallmatrix}
  \right|
+\OneHalf\left|
  \begin{smallmatrix}
    b & c\\
    b' & c'
  \end{smallmatrix}
  \right|
+\OneHalf\left|
  \begin{smallmatrix}
    c & d\\
    c' & d'
  \end{smallmatrix}
  \right|
+\OneHalf\left|
  \begin{smallmatrix}
    d & a\\
    d' & a'
  \end{smallmatrix}
  \right|,a+a',b+b',c+c',d+d'\right).
\end{multline*}
Define
\begin{align*}
  \grb(\rho_1)&=(-\OneHalf;1,0,0,0) & 
  \grb(\rho_2)&=(-\OneHalf;0,1,0,0) \\
  \grb(\rho_3)&=(-\OneHalf;0,0,1,0) & 
  \grb(\rho_4)&=(-\OneHalf;0,0,0,1)
\end{align*}
and let $\bigGroup$ be the subgroup of $(\OneHalf\ZZ)\times\ZZ^4$
generated by $\grb(\rho_1)$, $\grb(\rho_2)$, $\grb(\rho_3)$,
$\grb(\rho_4)$, and $\lambda=(1;0,0,0,0)$. These formulas determine a
grading on $\AsUnDefAlg$ by $\bigGroup$.

The $\Ainf$-algebra $\UnDefAlg$ is uniquely characterized (up to
isomorphism) by the following properties \cite[Theorem~\ref{TA:thm:UnDefAlg-unique}]{LOT:torus-alg}:
\begin{itemize}
\item $\UnDefAlg$ is an $\Ainf$-deformation of $\AsUnDefAlg$,
\item $\UnDefAlg$ is graded by $(\bigGroup,\lambda)$, and
\item
  $\mu_4(\rho_4,\rho_3,\rho_2,\rho_1)=U\iota_1$.
\end{itemize}

The weighted algebra $\MAlg$ is uniquely characterized (up to
isomorphism) by the following properties \cite[Theorem~\ref{TA:thm:MAlg-unique}]{LOT:torus-alg}:
\begin{itemize}
\item The undeformed $\Ainf$-algebra of $\MAlg$ is $\UnDefAlg$.
\item $\MAlg$ is $(\bigGroup,\lambda)$-graded,
  with respect to the weight
  grading $\lambda_w=(1;1,1,1,1)$, and
\item $\mu^1_0=\rho_1\rho_2\rho_3\rho_4+\rho_2\rho_3\rho_4\rho_1+\rho_3\rho_4\rho_1\rho_2+\rho_4\rho_1\rho_2\rho_3$.
\end{itemize}

The ground ring for $\MAlg$ is $\Ground=\Field[U]\langle \iota_0,\iota_1\rangle$.

\begin{example}
  The only term which can cancel 
  $
    \mu_2^0(\mu_4^0(\rho_4,\rho_3,\rho_2,\rho_1),\rho_4)=U\rho_4
  $
  in the $\Ainf$-relation is
  $
    \mu_2^0(\rho_4,\mu_4^0(\rho_3,\rho_2,\rho_1,\rho_4))=U\rho_4.
  $
  Thus, 
  $
    \mu_4^0(\rho_3,\rho_2,\rho_1,\rho_4)=U\iota_0.
  $
  By a similar argument, $\mu_4^0(\rho_2,\rho_1,\rho_4,\rho_3)=U\iota_1$ and $\mu_4^0(\rho_1,\rho_4,\rho_3,\rho_2)=U\iota_0$.

  Similarly, the $5$-input, weight $0$ $\Ainf$-relation with inputs
  $(\rho_3,\rho_4,\rho_3,\rho_2,\rho_1)$ implies that
  $
  \mu_4^0(\rho_3\rho_4,\rho_3,\rho_2,\rho_1)=U\rho_3.
  $
  Thus,
  \[
    \mu_4^0(\mu_4^0(\rho_3\rho_4,\rho_3,\rho_2,\rho_1),\rho_2,\rho_1,\rho_4)=U^2\iota_0,
  \]
  and grading considerations imply that the only term that can cancel
  this one in the $7$-input, weight $0$ $\Ainf$-relation is
  \[
    \mu_6^0(\rho_3\rho_4,\rho_3,\rho_2,\mu_2(\rho_1,\rho_2),\rho_1,\rho_4)=U^2\iota_0.
  \]
\end{example}
Following our previous
paper~\cite[Definition~\ref{TA:def:TilingPattern}]{LOT:torus-alg}, we
will sometimes refer to operations of the form
$\mu_n^w(a_1,\dots,a_n)=U^k\iota_\ell$ as \emph{centered operations}.

\subsubsection{Relationship of the algebra to holomorphic curves}
\label{subsec:RelateHolC}
Let $\PMC$ denote the genus $1$ pointed matched circle.  Given a Reeb
chord $\rho$ in $\PMC$ (of any positive length) there is an associated
algebra element $\rho\in\MAlg$. (In~\cite{LOT1}
this algebra element was denoted $a(\rho)$, but in this paper
denoting it by $\rho$ will not cause confusion.)

\begin{figure}
  \centering
  \includegraphics{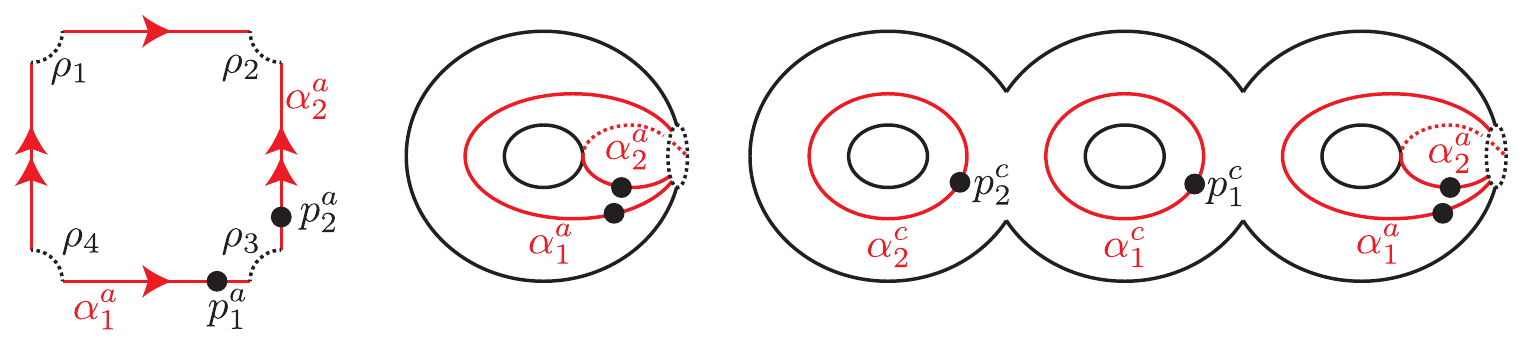}
  \caption[The $\alpha$-curves on a punctured surface]{\textbf{Some curves on punctured Riemann surfaces.} Left and center: the case $g=1$. Right: the case $g=3$.}
  \label{fig:curves-on-surfs}
\end{figure}

Fix $g\geq 1$ and let $\Sigma=\Sigma_g$ be a Riemann surface of genus $g$ with a single
puncture, $\alpha_1^a$, $\alpha_2^a$ arcs in $\Sigma$ and
$\alpha_1^c,\dots,\alpha_{g-1}^c$ circles in $\Sigma$ as shown in
Figure~\ref{fig:curves-on-surfs}. Choose also a marked point $p^a_i$
on $\alpha_i^a$ and $p^c_i$ on $\alpha_i^c$. Note that $\PMC$ is
identified with $(\bdy \Sigma,\bdy(\alpha_1^a\cup\alpha_2^a))$. Let
$\alphas=\alpha_1^a\cup\alpha_2^a\cup\alpha_1^c\cup\dots\cup\alpha_{g-1}^c$.

Fix a 
non-negative integer $w$, a sequence of chords $\rho^1,\dots,\rho^n$ in $\PMC$ so that
$\rho^1$ starts on $\alpha_i^a$, and a split complex structure
$J=j_\Sigma\times j_\bD$ on $\Sigma\times \bD^2$. (We use superscripts in the list of
Reeb chords because $\rho_i$ refers to a specific Reeb chord.) Let
$\cN([\alpha_i^a];\rho^1,\dots,\rho^n;w)$
denote the moduli space of $(j_\Sigma\times j_{\bD})$-holomorphic maps
\[
  u\co (S,\bdy S)\to (\Sigma\times\bD^2, \alphas\times S^1),
\]
where $S$ is a surface with boundary, interior punctures, and boundary
punctures, satisfying the following conditions:
\begin{enumerate}[label=(\arabic*)]
\item The surface $S$ has exactly $n$ boundary punctures
  $q_1,\dots,q_n$. Further, there are distinct points
  $\theta_1,\dots,\theta_n\in(S^1\setminus \{1\})$, ordered
  counterclockwise around $\bdy \bD^2$ (starting at $1\in\bdy \bD^2$), so that at $q_i$, the map
  $u$ is asymptotic to $\rho^i\times\{\theta_i\}$. Moreover, all the
  $q_i$ lie on a single component of $\bdy S$.
\item The surface $S$ has exactly $w$ interior punctures
  $r_1,\dots,r_w$. Further, there are points $z_1,\dots,z_w$ in the
  interior of $\bD^2$ so that at $r_i$, $u$ is asymptotic to
  $S^1\times z_i$ (i.e., a simple orbit).
\item If $\overline{S}$ denotes the result of filling in the punctures
  of~$S$, then the projection $\pi_{\bD}\circ u$ extends to a $g$-fold
  branched cover $\overline{S}\to \bD^2$.
\item\label{item:pa1}
  The map $u$ passes through the points
  $(p^c_1,1),\dots,(p^c_{g-1},1)$ and the point $(p^a_i,1)$.
\end{enumerate}
This space $\cN$ is the moduli space of simple boundary degenerations;
see Definitions~\ref{def:BoundaryDegeneration-simple}
and~\ref{def:CurveCount}. Its expected dimension $\indsq(\vec{\rho},w)$
is computed in Proposition~\ref{prop:emb-ind-bdy-degen};
$\indsq(\vec{\rho},w)=2$ corresponds to the rigid moduli spaces.

We will show that, for appropriate choices of $j_\Sigma$, the moduli
spaces $\cN$ with $\indsq(\vec{\rho},w)\leq 2$ are transversely cut out
and the number of points in them are, in fact, independent of
$j_\Sigma$; this is part of Theorem~\ref{thm:AlgOfSurface} below.

We will often identify $\bD^2\setminus\{1\}$ with a hyperbolic half-plane
$\HHH$ presented as $(-\infty,1]\times \RR\subset \CC$. If
$\overline{S}'=\overline{S}\setminus(\pi_{\bD}\circ u)^{-1}(1)$ then
each component of $\bdy\overline{S}'$ is an arc $A$; and elementary complex
analysis shows that  $\pi_\HHH\circ u|_A\co A\to \{1\}\times\RR$ is
a monotone function. (This is an analogue of \emph{boundary
  monotonicity} from~\cite{LOT1}.)

Fix a point $z\in \Sigma\setminus \alphas$. For any
$u\in \cN([\alpha_i^a];\rho^1,\dots,\rho^n;w)$, let $n_z(u)$ denote the local
multiplicity of $\pi_\bD\circ u$ at $z$. Since
$\Sigma\setminus\alphas$ is connected, $n_z(u)$ is independent of
$z$. In fact, $n_z(u)$ is also independent of $u$: it depends only on
$\rho^1,\dots,\rho^n$ and~$w$. (Specifically, $n_z(u)$ is $w$ plus the
number of times $\rho_4$, or any other $\rho_j$, appears in the $\rho^i$.)

We can use these moduli spaces to define a weighted $\Ainf$-algebra:
\begin{definition}\label{def:holo-def-alg}
  Let $\AsUnDefAlg$ be as in
  Section~\ref{sec:build-alg-combinatorially}. We define operations
  making $\Field[U]\otimes_{\Field}\AsUnDefAlg$ into a weighted $\Ainf$-algebra $\MAlg$ over
  $\Field[U]$ as follows:
  \begin{enumerate}
  \item The operation $\mu_1^0$ vanishes, and the operation $\mu_2^0$
    is inherited from $\UnDefAlg$.
  \item The operation $\mu_0^1$ is $\rho_{1234}+\rho_{2341}+\rho_{3412}+\rho_{4123}$.
  \item The centered operations (operations which output $U^k\iota_i$) are
    given by
    \[
      \mu_n^w(\rho^1,\dots,\rho^n)=\#\cN([\alpha_i^a];\rho^1,\dots,\rho^n;w)U^{n_z(u)}\iota_i
    \]
    where $\iota_i$ is $\iota_0$ if $\rho^1$ starts on $\alpha_1$ and
    $\iota_1$ if $\rho^1$ starts on $\alpha_2$.
  \item For each centered operation
    $\mu_n^w(\rho^1,\dots,\rho^n)=U^m\iota_i$ and any $\rho^0$
    so that $\rho^0\rho^1\neq 0$, we have
    \[
      \mu_n^w(\rho^0\rho^1,\dots,\rho^n)=
      \rho^0\mu_n^w(\rho^1,\dots,\rho^n).
    \]
    Similarly, for any $\rho^{n+1}$ so that
    $\rho^n\rho^{n+1}\neq 0$, we have
    \[
      \mu_n^w(\rho^1,\dots,\rho^n\rho^{n+1})=
      \mu_n^w(\rho^1,\dots,\rho^n)\rho^{n+1}.
    \]
  \end{enumerate}
  The above are the only nonzero operations of the form
  $\mu^w_n(\rho^1,\dots,\rho^n)$.
\end{definition}
We can think of the non-centered operations as corresponding to
holomorphic curves with two components, one of which is at $e\infty$;
see Figure~\ref{fig:non-centered-op}.

\begin{figure}
  \centering
  \includegraphics{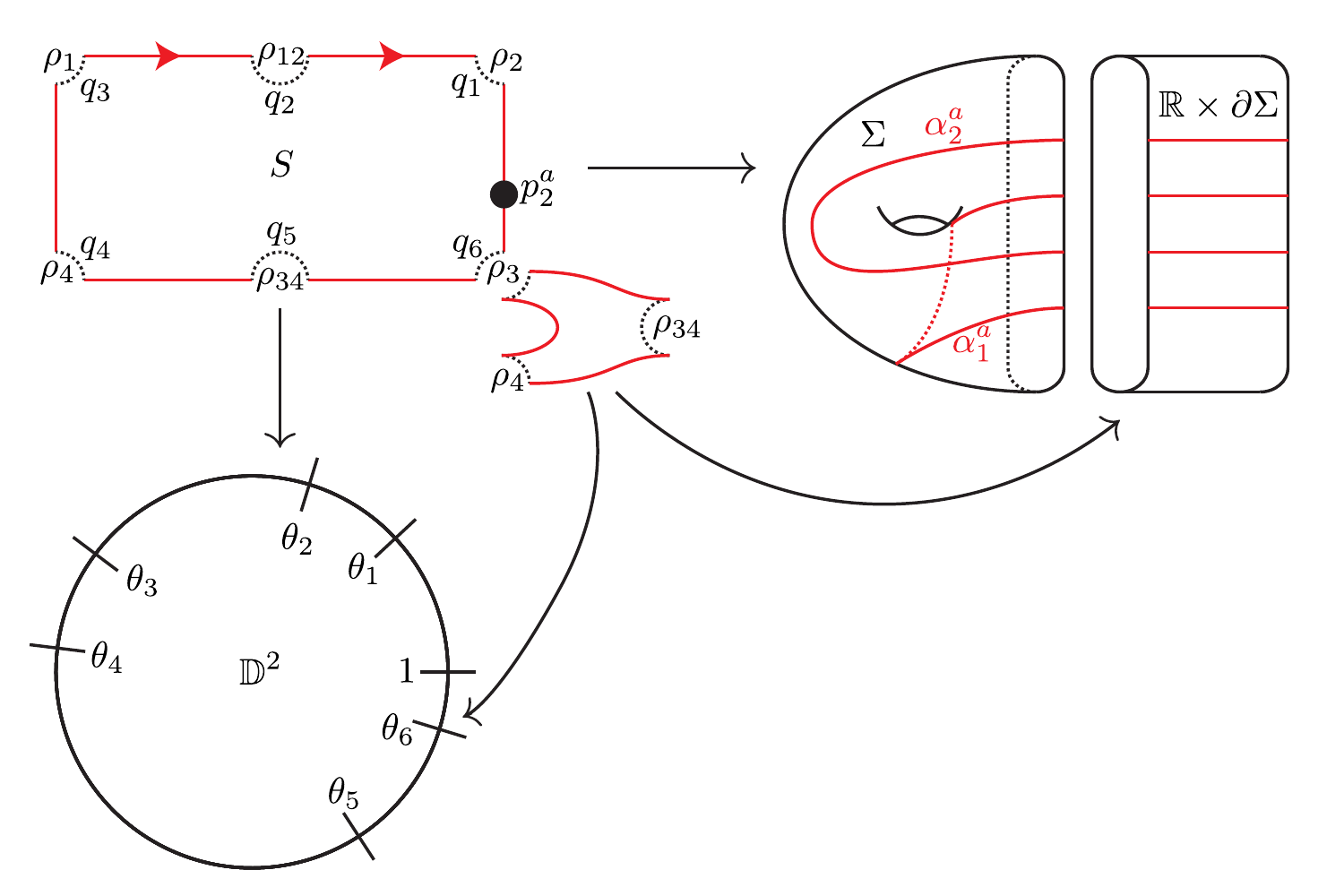}
  \caption[Schematic of holomorphic curves inducing a non-centered operation on $\MAlg$]{\textbf{A schematic for a non-centered operation.} The surface
    $S$ (top-left) maps to $\Sigma$ (top-right) times $\mathbb{D}^2$
    (bottom-left). The component with punctures labeled $\rho_3$,
    $\rho_4$, and $\rho_{34}$ maps to $\RR\times\bdy\Sigma$ (east
    infinity) and the point~$\theta_6$. This schematic corresponds to
    $\mu_6(\rho_2,\rho_{12},\rho_1,\rho_4,\rho_{34},\rho_{34}) = U^2 \rho_4$.}
  \label{fig:non-centered-op}
\end{figure}

If the Heegaard genus $g=1$, it is not hard to describe the spaces
$\cN([\alpha_i^a];\rho^1,\dots,\rho^n;w)$ 
combinatorially; see Section~\ref{sec:GenusOne}.
Using this description, one can show that the operations $\mu_n^w$ define a weighted algebra
in this case, and that algebra agrees with $\MAlg$ from
Section~\ref{sec:build-alg-combinatorially}. (See
Proposition~\ref{prop:IdentifyAlgebraGenusOne}.) Index computations and a
degeneration argument then show that the counts of the spaces
$\cN([\alpha_i^a];\rho^1,\dots,\rho^n;w)$ are, in fact, independent of $g$
(and, as mentioned above, the complex structure). Together, this implies:
\begin{theorem}\label{thm:alg-well-defd}
  Definition~\ref{def:holo-def-alg} defines a weighted $\Ainf$-algebra
  that agrees with the weighted $\Ainf$-algebra defined in
  Section~\ref{sec:build-alg-combinatorially}.
\end{theorem}
(This is proved in Section~\ref{sec:algebra}, as Theorem~\ref{thm:AlgOfSurface}.)

\subsection{The type A module}
\subsubsection{Weighted modules}
Given a weighted $\Ainf$-algebra $\Alg$ over $\Ground$, a weighted
(right) $\Ainf$-module over $\Alg$ is a right $\Ainf$-module $\fModule$ over
$\Alg$ together with an isomorphism
of $\Ground[[t]]$-modules
$\fModule\cong M\otimes_{\Ground}\Ground[[t]]$. Like a weighted algebra,
$M$ inherits operations 
\begin{equation}\label{eq:mod-op-form}
  m_{1+n}^w\co M\otimes_\Ground \otimes
  \overbrace{A\otimes_{\Ground}\dots\otimes_\Ground A}^k\to M,
\end{equation}
for $n,w\geq 0$, defined by
\[
m_{1+k}=\sum m_{1+k}^wt^w.
\]
(We will usually suppress the subscript $\Ground$ from tensor products
over the ground ring $\Ground$ relevant at the time.)
A weighted $\Ainf$-module is specified by the data
$\fModule=(M,\{m_{1+k}^w\})$. The curved $\Ainf$-relation turns into a
compatibility condition for the $m_{1+k}^w$ similar to
Formula~\eqref{eq:wAlg-rel} (see,
e.g.,~\cite[Definition~\ref{Abs:def:wmod}]{LOT:abstract}).

Given weighted $\Ainf$-modules $\fModule$ and $\fNodule$ over $\Alg$,
there is a chain complex $\Mor_\Alg(\fModule,\fNodule)$ of morphisms
between them. The complex $\Mor_\Alg(\fModule,\fNodule)$ is the complex of
morphisms of $\Ainf$-modules over the curved $\Ainf$-algebra~$\Alg$.
Explicitly, an element is a sequence of maps
\[
  f_{1+n}^w\co M\otimes
  \overbrace{A\otimes\dots\otimes A}^n\to N
\]
for $n,w\geq 0$, with differential given by the sum of all ways of
pre-composing with an operation $\mu_m^v$ on $\Alg$ or $m_{1+m}^v$ on
$\fModule$ or post-composing by an operation $m_{1+m}^v$ on
$\fNodule$. This makes the category of weighted $\Ainf$-modules over
$\Alg$ into a \dg category, and so it inherits the usual auxiliary
notions of homomorphisms (cycles in the morphism complex), homotopic
morphisms (morphisms whose difference is a boundary), and homotopy
equivalences.

\begin{remark}
  The definition of a weighted $\Ainf$-module over $\Alg$ depends both
  on $\Alg$ and the ring $\Ground$ that we are viewing $\Alg$ over:
  the operations $m_{1+n}^w$ in Formula~\eqref{eq:mod-op-form} are
  $\Ground$-equivariant. We will later construct two modules, $\CFAm$ and
  $\CFAmb$, which differ mainly in which ground ring they are defined
  over.
\end{remark}
\subsubsection{Unitality conditions}
\label{subsec:Unitality}

A weighted $\Ainf$-algebra $\Alg=(A,\{\mu_n^w\})$ is \emph{unital} if
there is an element $\unit\in A$ so that
$\mu^0_2(\unit,a)=\mu^0_2(a,\unit)=a$ for all $a\in A$ and
$\mu^w_n(a_1,\dots,a_n)=0$ if $(w,n)\neq (0,2)$ and some $a_i=\unit$.

Similarly, an $\Ainf$-module $\fModule$ over a  unital weighted
$\Ainf$-algebra $\Alg$ is called {\em unital} if 
$m^0_2(\x,\unit)=\x$ for all $\x\in M$, and
\[
  m^w_n(\x,a_1,\dots,a_{n-1})=0
\]
if $(w,n)\neq (0,2)$ and some $a_i=\unit$.  Given two  unital
weighted $\Ainf$-modules $\fModule$ and $\fNodule$ over a 
unital weighted $\Ainf$-algebra $\Alg$, there is a complex of 
unital morphisms between them, by requiring that each $f_n^w$
vanishes if some input is the unit $\unit$. This gives rise to the
notion of  unital homotopy equivalences.

There are various other weaker notions of unitality one could work
with
(compare~\cite[Section~\ref{Abs:subsec:UnitsHomPert}]{LOT:abstract});
the above version is sufficient for our purposes. The
notion we are using here is referred to as \emph{strictly unital} in
our previous papers and some other parts of the literature.

\subsubsection{Sketch of the definition of the type A module and statement of invariance}\label{sec:intro-CFA-def}
Fix a bordered Heegaard diagram (in the sense of~\cite{LOT1})
$\HD=(\Sigma_g,\alphas^a,\alphas^c,\betas,z)$ with $\bdy \HD$ the
genus $1$ pointed matched circle, representing a bordered 3-manifold
$(Y,\phi\co T^2\to \bdy Y)$. Write
$\alphas=\alphas^a\cup\alphas^c$. To this data we associate a right
weighted $\MAlg$-module $\CFAm(\HD)$ as follows.
As a left
$\Field\langle\iota_1,\iota_2\rangle[U]$-module, $\CFAm(\HD)$ is the same as
$\Field[U]\otimes_\Field\CFAa(\HD)$, the 
free $\Field[U]$-module generated by $g$-tuples $\x=\{x_i\}\subset
(\alphas\cap \betas)$ with exactly one $x_i$ on each $\alpha$-circle
and $\beta$-circle and at most one $x_i$ on each $\alpha$-arc, with
the same action of the idempotents. The
weighted operations are defined by counting pseudo-holomorphic curves.

\begin{definition}
\label{def:Basic}
A \emph{basic algebra element} is an element of the form
\begin{itemize}
\item  $U^m \rho$ where $m\geq 0$ and $\rho$ is some Reeb chord (the
  \emph{Reeby} elements) or
\item $U^m\iota_i$ where $m>1$ and $\iota_i\in\{\iota_0,\iota_1\}$ is
  one of the two basic idempotents.
\end{itemize}
Given a sequence $\vec{a}=(a_1,\dots,a_m)$ of basic algebra elements,
if we set $U=1$ and drop all of the terms which are idempotents, we
obtain a sequence of chords
$\vec{\rho}=(\rho^{n_1},\dots,\rho^{n_m})$; we call this the
\emph{underlying chord sequence} of the sequence of basic algebra
elements. That is, $\vec{\rho}$ has one term for each Reeby
element in~$\vec{a}$.
\end{definition}

Given generators $\x$ and $\y$ for $\CFAm(\HD)$, an
integer $w$, and a homotopy class $B\in\pi_2(\x,\y)$, let
\[
  \cM^B(\x,\y;\vec{a};w)
\]
denote the moduli space of embedded $J$-holomorphic maps
\[
  u\co (S,\bdy S)\to \bigl(\Sigma\times[0,1]\times\RR,(\alphas\times\{1\}\times\RR)\cup(\betas\times\{0\}\times\RR)\bigr)
\]
asymptotic to $\x\times[0,1]\times\{-\infty\}$,
$\y\times[0,1]\times\{+\infty\}$, and 
at $e\infty$ asymptotic to the chord sequence
$\vec{\rho}$ and $w$ simple Reeb orbits.
Curves in the moduli space are required to satisfy
certain other technical conditions, analogous
to~\cite[Conditions (M-1)--(M-11)$\setminus$(M-9)]{LOT1};
(M-9) is omitted because
curves here are allowed to cross the basepoint~$z$. See
Section~\ref{sec:def-mod-sp}. We are using the notation $J$ to
indicate a family of almost complex structures on
$\Sigma\times[0,1]\times\RR$ depending on the sequence $\vec{a}$ and
the locations in $[0,1]\times\RR$ of the Reeb chords and orbits; this
dependence is made precise in Definition~\ref{def:AdmissibleJs}, of a
coherent family of almost complex structures.  (The reason for using
such families relates to transversality for boundary degenerations. In
particular, the non-Reeby elements of $\vec{a}$ only affect the moduli
space through the choice of $J$, not as asymptotics of the holomorphic
curves $u$.)
The integer $w$ is slightly redundant: it is determined by the data
of $\vec{\rho}$ and~$B$.

Let $\ind(B,\vec{a})$ denote the expected dimension of the moduli
space $\cM^B(\x,\y;\vec{a};w)$. Define a module $\CFAmb(\HD)$ which,
as an $\Field$-vector space, is the same as $\CFAm(\HD)$, i.e.,
\[
  \CFAmb(\HD)=\CFAm(\HD)=\Field[U]\otimes_{\Field}\CFAa(\HD).
\]
Fix a suitably generic coherent family of almost complex structures
$J$; we call such families \emph{tailored}
(Definition~\ref{def:tailored}).  Make $\CFAmb(\HD)$ into a weighted
module over $\MAlg$ with
\[
m_{1+k}^w\co \CFAmb(\HD)\otimes(\MAlg)^{\otimes k}\to \CFAmb(\HD)
\]
given by 
\begin{equation}\label{eq:m-on-CFA}
  m_{1+n}^w(\x,a_1,\dots,a_n)=\sum_\y
  \sum_{\substack{B\in\pi_2(\x,\y)\\\ind(B,a_1,\dots,a_n)=1}}
  U^{m+n_z(B)}\cdot\#\cM^B(\x,\y;a_1,\dots,a_n;w)\cdot\y,
\end{equation}
where $m$ is the total number of factors of $U$ in the
$a_1,\dots,a_n$. The conditions on the almost complex structures
guarantee that these counts are finite (Lemma~\ref{lem:0d-compactness}). 

\begin{remark}
  This use of the term \emph{tailored} for almost complex structures
  is unrelated to the use in the literature on embedded contact
  homology~\cite{CGHH:ECH-OB}.
\end{remark}

\begin{theorem}\label{thm:CFAmb-is}
  Assuming that the bordered Heegaard diagram is provincially
  admissible and the family of almost complex structures $J$ is tailored,
  the operations $m_{1+k}^w$ give $\CFAmb$ the structure
  of a  unital weighted $\Ainf$-module over $\MAlg$
  (with ground ring $\Field\langle \iota_1,\iota_2\rangle$).
\end{theorem}

The key step in the proof of Theorem~\ref{thm:CFAmb-is} is the
following analysis of the codimension-1 boundary of $\cM^B(\x,\y;\vec{\rho};w)$:
\begin{theorem}\label{thm:master}
  Fix a provincially admissible bordered Heegaard diagram $\HD$ and a
  tailored family of almost complex structures $J$ on
  $\Sigma\times[0,1]\times\RR$.  Given a homology class $B$ and 
  sequence
  $\vec{a}=(a_1,\dots,a_k)$ of basic algebra elements so that
  $\ind(B;a_1,\dots,a_k;w)=2$, the sum of the following is
  $0\pmod{2}$:
  \begin{enumerate}[label=(e-\arabic*), ref=(e-\arabic*)]
  \item 
    \label{end:TwoStoryBuilding}
    The number of two-story holomorphic buildings, i.e., 
    \[
    \sum_{\substack{\x'\\B_1+B_2=B\\w_1+w_2=w\\j=0,\dots,n}}\bigl(\#\cM^{B_1}(\x,\x';a_1,\dots,a_j;w_1)\bigr)\bigl(\#\cM^{B_2}(\x',\y;a_{j+1},\dots,a_n;w_2)\bigr).
    \]
  \item 
   \label{end:Collision}
   The number of collisions of levels, i.e.,
    \[
    \sum_{a_ia_{i+1}\neq 0} \#\cM^B(\x,\y;a_1,\dots,a_ia_{i+1},\dots,a_n;w)
    \]
  \item 
    \label{end:EscapingOrbit}
    The number of orbit curve degenerations, which corresponds to
    \[
      \sum_{i=0}^{n+1}\,
      \sum_{\substack{\sigma\in\{\rho_{1234},\rho_{2341},\\\rho_{3412},\rho_{4123}\}}}
      \#\cM(\x,\y;a_1,\dots,a_{i-1},\sigma,a_{i},\dots,a_n; w-1).
    \]
  \item 
    \label{end:SimpleBoundaryDegeneration}
    If $\x=\y$, the number of simple boundary degeneration ends, which
    corresponds to
    \[
      U^k\#\cN([\alpha_i^a];\vec{\rho};w),
    \]
    where $\alpha_i^a$ is the $\alpha$-arc containing an element of $\x$,
    $\vec{\rho}$ is the underlying chord sequence of $\vec{a}$, and $k$ is the sum of all 
    the $U$-powers in the~$a_i$; except that if any $a_i$ is of the form
    $U^m\iota_j$ then this count is defined to be zero.
  \item
    \label{end:CompositeBoundaryDegeneration}
    The number of composite boundary degeneration ends, which corresponds to
    pairs $(u,v)$ where:
    \begin{itemize}
    \item $u$ is an index one embedded curve in
      $\cM^{B_1}(\x,\y,a_1,\dots,a_{i-1},b,a_{j+1},\dots,a_n;w_1)$ for
      some $1\leq i<j\leq n$ and basic algebra element $b$;
    \item $v\in \cN(b;a_i,\dots,a_j;w_2)$ is a composite boundary
      degeneration (Definition~\ref{def:composite-bdy-degen}; see also
      Figure~\ref{fig:non-centered-op});
    \item either $a_i=b\sigma$ or $a_j=\sigma b$ for some Reeb chord
      $\sigma$; and
    \item $w$ is the sum $w=w_1+w_2$ and $m$
      is the energy
      $E(\sigma,a_{i+1},\dots,a_j,w_2)/4$ or
      $E(a_{i},\dots,a_{j-1},\sigma,w_2)/4$ from
      Formula~\eqref{eq:ExtendEnergy} (i.e., $w_2$ plus a quarter the
      total length of the algebra elements).
      \end{itemize}
  \end{enumerate}
\end{theorem}

Since the almost complex structure used above depends on the sequence of basic
algebra elements and not just the underlying chord sequence,
there is no guarantee that the actions $m^w_{1+n}$ are
$U$-equivariant, a property which is useful, for example, when
we construct $\CFDm$ from $\CFAm$.

We refine $\CFAmb$ to construct a module $\CFAm$ whose operations are 
$U$-equivariant as follows.
In Section~\ref{sec:GroundingU} we extend the definition
of $\CFAmb(\HD)$ to a weighted module $\CFAmg(\HD)$ over a larger
algebra, denoted $\MAlgg$. As an algebra, $\MAlgg$ is identified with
$\MAlg[X,e]/(X^2=0)$ and has a differential given by $dX=1+e$. The previous action of
$U$ on $\CFAmb(\HD)$ is now interpreted as the action by $U\cdot
e$, where $U$ is in the ground field and the algebra element $e$
represents a pinching condition on the almost complex structure
(Definition~\ref{def:sufficient-pinching} and
Section~\ref{sec:GroundingU}). In modules, the element~$X$ acts by
counting holomorphic curves in a one-parameter family of
almost complex structures. 

There is a quasi-isomorphism $\MAlg\to \MAlgg$ of weighted algebras,
and the module $\CFAm(\HD)$ over $\MAlg$ is defined by restricting
scalars (as in Lemma~\ref{lem:RestrictionFunctor}) under this
quasi-isomorphism; see Definition~\ref{def:CFAm}.
The actions on $\CFAm(\HD)$ are $U$-equivariant
(i.e., $U$ is in the ground field).
The action on $\CFAm(\HD)$ defined by this restriction of scalars
incorporates counts of curves with respect to variations of almost
complex structures.  The fact that $\CFAmg(\HD)$ is a weighted
$\Ainf$-module over $\MAlgg$ is Theorem~\ref{thm:CFAmg-is} and the
corresponding fact for $\CFAm(\HD)$ is Theorem~\ref{thm:CFAm-is}.
The weighted module $\CFAm$ is homotopy equivalent
to the more directly constructed $\CFAmb$ from
Theorem~\ref{thm:CFAmb-is} (of course not $U$-equivariantly); see
Proposition~\ref{prop:CFA-is-CFA}.

The module $\CFAm$ has some familiar properties.
Each generator $\x$ represents a $\SpinC$-structure
$\spinc(\x)\in\Spinc(Y)$. If $\spinc(\x)\neq\spinc(\y)$ then
$\pi_2(\x,\y)=\emptyset$ (see Section~\ref{sec:background}), so
$\CFAm(\HD)$ decomposes as a direct sum of weighted $\Ainf$-modules
\begin{equation}\label{eq:spinc-decomp}
  \CFAm(\HD)=\bigoplus_{\spinc\in\Spinc(Y)}\CFAm(\HD,\spinc).
\end{equation}
These follow from corresponding splittings of $\CFAmg(\HD)$.

The second key theorem about $\CFAm$ is its invariance:
\begin{theorem}\label{thm:CFAm-invt}
  Up homotopy equivalence (of  unital weighted $\Ainf$-modules),
  $\CFAm(\HD,\spinc)$ depends only on the bordered 3-manifold
  $(Y,\phi)$ specified by $\HD$, and its underlying $\SpinC$ structure~$\spinc$.
\end{theorem}

Invariance is first spelled out for $\CFAmb$
(Theorem~\ref{thm:CFAmb-invt}).
The proof uses an
adaptation of the invariance proof in the ordinary bordered setting,
with a little extra care taken for boundary degenerations.
The invariance result has a straightforward extension to $\CFAmg$
(Theorem~\ref{thm:CFAmg-invt}). Invariance for $\CFAm$ now follows from 
basic properties of the restriction functor (Lemma~\ref{lem:RestrictionFunctor}).

\begin{remark}
  The definitions are set up so that for a Heegaard diagram $\HD$ decomposed as
  a gluing of two bordered Heegaard diagrams $\HD_1\cup\HD_2$, the basepoint
  lies on the type $A$ side~$\HD_1$ and near the boundary.
\end{remark}
\subsection{The dualizing bimodule}
In the $\HFa$ case of bordered Floer homology, we also defined a
module $\CFDa(\HD)$ using pseudo-holomorphic curves.  We
observed that, setting $\CFDDa(\Id)$ to be the type \DD\
bimodule associated to the identity cobordism,
$\CFDa(\HD)\cong \CFAa(\HD)\DT\CFDDa(\Id)$~\cite{LOT4}. In this paper, we define
$\CFDm(\HD)$ to be $\CFAm(\HD)\DT\CFDDm(\Id)$, where $\CFDDm(\Id)$ is
defined algebraically. Making sense of this requires some more algebra.

\subsubsection{Weighted type D structures}
Given a weighted $\Ainf$-algebra $(\Alg,\{\mu_k^w\})$ over a ring
$\Ground$ and an element $X$ in the center of $\Ground$ (e.g., $X=1$
or $X=U$), a \emph{(left) weighted type $D$ module over $\Alg$}
consists of a $\Ground$-module $P$ and a map
$\delta^1\co P\to \Alg\otimes P$ satisfying a structure 
equation, stated in terms of the map
$\delta^n\co P\to \Alg^{\otimes n}\otimes P$ defined by iterating
$\delta^1$ $n$ times.  The structure equation is
\[
  \sum_{n,w\geq 0} (\mu_n^w\otimes\Id)\circ \delta^n=0.
\]
We need some additional hypotheses to  ensure that this sum converges; see
Section~\ref{sec:boundedness}.

\begin{remark}
  We have specialized the definition of weighted type $D$ structure
from our previous
paper~\cite[Definition~\ref{Abs:def:wD}]{LOT:abstract} to the case of
charge $1$.
\end{remark}
\subsubsection{Trees}
Given \dg algebras $\Alg$ and $\Blg$ over $\Ground$ and $\Groundl$,
respectively, a type \DD\ bimodule over $\Alg$
and $\Blg$ is a type $D$ module over the tensor product algebra
$\Alg\otimes_\Field\Blg$. To
extend this definition to weighted $\Ainf$-algebras, we need an
analogue of the tensor product in the weighted $\Ainf$
context. This is an adaptation of a construction of
Saneblidze-Umble~\cite{SU04:Diagonals} to the weighted case. We
explain the background on trees that we need now and the tensor
product of weighted algebras and the definition of weighted type \DD\
bimodules in Section~\ref{sec:intro-DD}.

Following~\cite[Section~\ref{Abs:sec:wAlgs}]{LOT:abstract}, by a \emph{marked
  tree} we mean a planar, rooted tree together with a subset $I$ of the non-root
leaves of $T$. Leaves in $I$ are called \emph{inputs}, and the root is the
\emph{output}. The inputs of $T$ inherit an ordering from the orientation on the
plane. Vertices of $T$ that are not in $I$ and are not the output are
\emph{internal}; in particular, some leaves are considered internal.  A \emph{weighted tree} is a marked tree $T$ together with a
function $w$, the \emph{weight}, from the internal vertices of $T$ to
$\ZZ_{\geq 0}$. A weighted tree is \emph{stable} (or \emph{stably-weighted}) if
any internal vertex with valence less than $3$ has positive
weight. (An example is drawn in Figure~\ref{fig:stableTrees}.) The \emph{total weight} of $T$ is the sum of the
weights of the internal vertices of $T$. Let $\Trees_{n,w}$ be the set of stably-weighted trees with $n$ inputs and total weight $w$.

The \emph{$n$-input, weight $w$ corolla} $\wcorolla{n}{w}$ has a single internal
vertex, which has weight $w$, and $n$ inputs; $\wcorolla{n}{w}\in \Trees_{n,w}$
as long as $w\geq 0$ or $n\geq 2$.

The \emph{dimension} of $T$ is defined by $\dim(T)=n+2w-v-1$, where $v$ is the number of
internal vertices. The \emph{differential} of a stably-weighted tree is the
formal linear combination of all stably-weighted trees obtained by inserting an
edge in $T$, i.e., the sum of $S$ over pairs $(S,e)$ where $S$ is a
stably-weighted tree, $e$ is an edge between two internal vertices of
$S$, and the result of collapsing the
edge $e$ and adding the weights of the two vertices at the endpoints
of $e$ is the original tree~$T$.

Given stably-weighted trees $S$ and $T$, let $T\circ_i S$ be the
result of connecting the output of $S$ to the $i\th$ input of $T$. It will
be convenient to include two degenerate stably-weighted trees: the
\emph{identity tree}~$\IdTree$ with one input and no internal
vertices; and the \emph{stump}~$\stump$ with no inputs and no internal
vertices. We extend the operation $\circ_i$ to these generalized weighted
trees: composition with $\IdTree$ is the identity map, and
$S\circ_i \stump$ vanishes if the $i\th$ input of $S$ does not feed
into a valence $3$, weight $0$ vertex, and deletes the $i\th$ input to
$S$ if the $i\th$ input does feed into a valence $3$, weight $0$
vertex.  Let $\xwTreesCx{n}{w}$ be the $\Field$-vector space spanned
by the stably-weighted trees of total weight $w$ with $n$ marked
inputs, including the two degenerate stably-weighted trees (which span
$\xwTreesCx{1}{0}$ and $\xwTreesCx{0}{0}$, respectively). Let
$\wTreesCx{n}{w}=\xwTreesCx{n}{w}$ if $(n,w)\not\in\{(0,0),(1,0)\}$
and $\wTreesCx{0}{0}=\wTreesCx{1}{0}=\{0\}$.

\begin{figure}
  \centering
  \includegraphics{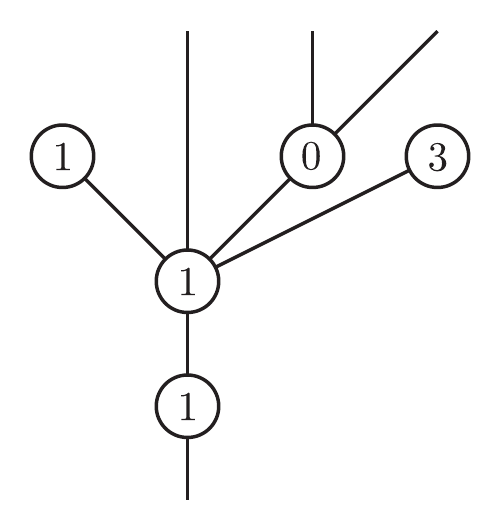}
  \caption[Stably-weighted trees and composition of weighted
  operations] {\textbf{A stably-weighted tree}. The tree has three
    inputs and total weight $6$. The corresponding
    operation is
    $(x_1,x_2,x_3)\mapsto
    \mu^1_{1}(\mu^1_{4}(\mu^1_0,x_1,\mu^0_{2}(x_2,x_3),\mu^3_0))$.}
  \label{fig:stableTrees}
\end{figure}

If $S$ and $T$ are stably-weighted trees, where $S$ has $n$ inputs, let $S\circ T=\sum_{i=1}^n S\circ_iT$. If $S_1,S_2,T_1$ and $T_2$ are stably-weighted trees so that $S_1$ and $S_2$ have the same number of inputs, let
\[
  (S_1,S_2)\circ(T_1,T_2)=\sum_{i=1}^n (S_1\circ_i T_1,S_2\circ_iT_2).
\]
We will sometimes also consider pairs of trees $(S_1,S_2)$ and $(T_1,T_2)$ where
$S_1$ has $n$ inputs and $S_2$ has $n-1$, and define
\[
  (S_1,S_2)\circ'(T_1,T_2)=\sum_{i=1}^{n-1} (S_1\circ_{i+1}T_1,S_2\circ_iT_2).
\]

\subsubsection{Algebra diagonals and weighted type DD bimodules}\label{sec:intro-DD}

Fix a weighted algebra $\Alg=(A,\{\mu_n^w\})$ over $\Field[\Yvar_1,\Yvar_2]$. In
the case of bordered Floer homology, $\Yvar_1$ will act by $1$ and $\Yvar_2$
will act by $U$.

Given a stably-weighted tree $T$ with $n$ inputs there is an induced map
$\mu(T)\co A^{\otimes n}\to A$ gotten by replacing each internal
vertex $v$ of $T$ with valence $n+1$ and weight $w(v)$ by the
operation $\mu_n^{w(v)}$, feeding the elements of $A^{\otimes n}$
into the inputs (in order), and composing according to the
edges of the tree. (An example is shown in
Figure~\ref{fig:stableTrees}.) So, given two weighted algebras $\Alg$, $\Blg$ and a
pair of stably-weighted trees $(S,T)$ with $n$ inputs each,
there is a corresponding map
$\mu_{\Alg}(S)\otimes_{\Field} \mu_{\Blg}(T)\co (A\otimes_{\Field} B)^{\otimes
  n}\cong A^{\otimes n}\otimes_{\Field} B^{\otimes n}\to A\otimes_{\Field} B$.

\begin{definition}\cite[Definition~\ref{Abs:def:wDiagCells}]{LOT:abstract}
  \label{def:algebra-diagonal}
  A \emph{weighted algebra diagonal} consists of an element
  \[
    \wDiagCell{n}{w}\in \bigoplus_{w_1,w_2\leq
      w}\Yvar_1^{w-w_1}\Yvar_2^{w-w_2}\xwTreesCx{n}{w_1}\otimes_{\Field}\xwTreesCx{n}{w_2}
    \subset \bigoplus_{w_1,w_2\leq
      w}\xwTreesCx{n}{w_1}\otimes_{\Field}\xwTreesCx{n}{w_2}\otimes_{\Field}\Field[\Yvar_1,\Yvar_2]
  \]
  of dimension $n+2w-2$,
  for each $n,m\geq 0$ with
  $n+2w\geq 2$, satisfying the following properties:
  \begin{itemize}
  \item \textbf{Compatibility}: 
    \begin{equation}\label{eq:wCell-compat}
      \bdy\wDiagCell{n}{x}=\sum_{\substack{v+w=x\\1\leq i\leq n\\i-1\leq j\leq
          n}}
      \wDiagCell{j-i+1}{v}\circ \wDiagCell{n+i-j}{w}.
    \end{equation}
  \item \textbf{Non-degeneracy}:
    \begin{itemize}
    \item $\wDiagCell{2}{0}=\wcorolla{2}{0}\otimes\wcorolla{2}{0}$.
    \item $\wDiagCell{0}{1}=\wcorolla{0}{1}\otimes\wcorolla{0}{1}+
    \Yvar_1 \stump\otimes\wcorolla{0}{1}+\Yvar_2 \wcorolla{0}{1}\otimes \stump.$
    \item At most one of each pair of trees in the diagonal is a
      degenerate tree.
    \end{itemize}
  \end{itemize}
\end{definition}
(The second non-degeneracy property corresponds to
specializing~\cite[Definition~\ref{Abs:def:wDiagCells}]{LOT:abstract}
to the diagonals with \emph{maximal seed}.)

Given a weighted algebra diagonal $\wADiag$ and weighted algebras
$\Alg$ and $\Blg$ over ground rings $\Ground$ and $\Groundl$
over $\Field[\Yvar_1,\Yvar_2]$ we can form the tensor product
$\Alg\otimes_{\wADiag}\Blg$ with underlying
$(\Ground\otimes_{\Field}\Groundl)$-module
$A\otimes_{\Field[\Yvar_1,\Yvar_2]}B$, differential $\mu_1^0$ the
usual tensor differential on the tensor product, and
operations $\mu_n^w$ given by
\[
  \mu_n^w=\sum_{n_{S,T}(S\otimes T)\in\wDiagCell{n}{w}}n_{S,T}\mu_\Alg(S)\otimes\mu_\Blg(T).
\]
It follows from the definition that this is, indeed, a weighted
$\Ainf$-algebra (cf. \cite[Section~\ref{Abs:sec:w-alg-tens}]{LOT:abstract}).
We now define a weighted type \DD\ bimodule over $\Alg$ and $\Blg$ to be a
weighted type $D$ module over $\Alg\otimes_{\wADiag}\Blg$. (As was the case for
type $D$ modules, we require some additional boundedness hypothesis to ensure
this sum converges; see Section~\ref{sec:boundedness}.)

Note that a weighted type \DD\ bimodule is determined by a
$(\Ground,\Groundl)$-bimodule $P$ and a map of
$(\Ground,\Groundl)$-bimodules
$\delta^1\co P\to (\Alg\otimes_{\Field}\Blg)\otimes P$; the structure
equation that $\delta^1$ must satisfy depends on the diagonal
$\wADiag$.

In order to make use of this construction, we need to know that:
\begin{theorem}\cite[Theorem~\ref{Abs:thm:wDiag-exists}]{LOT:abstract}
  A weighted algebra diagonal exists.
\end{theorem}

\subsubsection{The dualizing bimodule, and a graded miracle}\label{sec:intro-DD-Id}
Let $\MAlg^{U=1}$ be the result of setting $U=1$ in $\MAlg$. Fix a
weighted algebra diagonal. View $\MAlg$ as an algebra over
$\Field[\Yvar_2]$ by letting $\Yvar_2$ act by $U$, and $\MAlg^{U=1}$ as an
algebra over $\Field[\Yvar_1]$ by letting $\Yvar_1$ act by $1$.

The \emph{dualizing bimodule} $\CFDDm(\Id)$ is the type \DD\ bimodule
over $\MAlg^{U=1}$ and $\MAlg$ generated by two elements,
$\iota_0\otimes\iota_0$ and $\iota_1\otimes\iota_1$, with differential
$\delta^1$ given by
\begin{equation}\label{eq:DD-id-diff}
\begin{split}
  \delta^1(\iota_0\otimes\iota_0)&=(\rho_1\otimes\rho_3+\rho_3\otimes\rho_1+\rho_{123}\otimes\rho_{123}+\rho_{341}\otimes\rho_{341})\otimes(\iota_1\otimes\iota_1)\\
  \delta^1(\iota_1\otimes\iota_1)&=(\rho_2\otimes\rho_2+\rho_4\otimes\rho_4+\rho_{234}\otimes\rho_{412}+\rho_{412}\otimes\rho_{234})\otimes(\iota_0\otimes\iota_0)
\end{split}
\end{equation}

Even though we do not give a holomorphic curve interpretation of
$\CFDDm(\Id)$ in this paper, the Heegaard diagram it corresponds to is
helpful for understanding the terms that appear: see Figure~\ref{fig:DD-id-HD}.

By taking advantage of the gradings on $\MAlg$, it is not hard to show
that $\CFDDm(\Id)$ is, in fact, a type \DD\ bimodule (again of charge
1 in the terminology of \cite{LOT:abstract}); see
Section~\ref{sec:DD-Id}. In particular, although the bimodule
structure relation depends on a weighted algebra diagonal, it turns
out that the relation holds for any choice of weighted algebra
diagonal.
 
\begin{figure}
  \centering
  \includegraphics{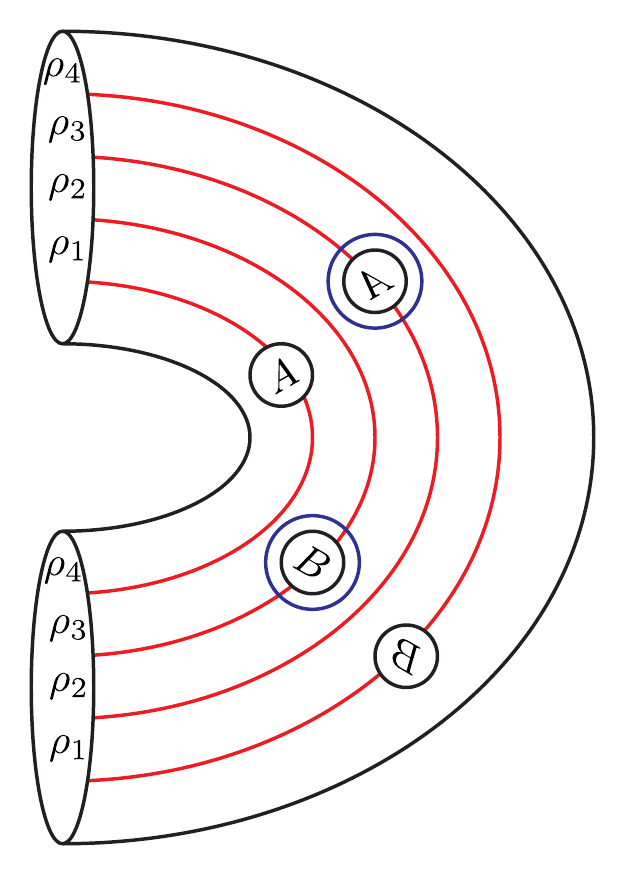}
  \caption[Heegaard diagram corresponding to the dualizing bimodule $\CFDD^-(\Id)$]{\textbf{The type bimodule $\CFDDm(\Id)$.} Although we do
    not associate bimodules to bordered Heegaard diagrams with two
    boundary components in this paper, we expect that the bimodule
    $\CFDD^-(\Id)$ corresponds to this bordered Heegaard diagram, as
    $\CFDDa(\Id)$ does. The positive domains with index $1$ correspond
    to the terms in
    Formula~\eqref{eq:DD-id-diff}. (Compare~\cite[Proposition
    10.1]{LOT2} or~\cite[Theorem 1]{LOT4}.)}
  \label{fig:DD-id-HD}
\end{figure}

\subsection{First pairing theorem}
\subsubsection{More abstract algebra}
Just as a weighted algebra diagonal allows us to take the external tensor
product of weighted $\Ainf$-algebras, there is a notion of weighted
module diagonals, which allows us to take the external tensor product
of weighted modules. To state it, we need a variant of
weighted trees.  Given a stably-weighted tree $T$, the \emph{leftmost
  strand} of $T$ is the minimal path from the leftmost input of
$T$ to the root of $T$.  A \emph{weighted module tree} is a
stably-weighted tree in which the leftmost leaf is an input; equivalently, a
weighted module tree has no leaves to the left of its
leftmost strand. The weighted module trees span a quotient complex
$\wMTreesCx{1+n}{w}$ of the weighted trees complex
$\wTreesCx{1+n}{w}$.
\begin{definition}\cite[Definition~\ref{Abs:def:w-mod-diag}]{LOT:abstract}
  Given a weighted algebra diagonal $\wADiag$, a \emph{weighted module
    diagonal} compatible with $\wADiag$ is a formal linear combination of pairs of
  trees
  \begin{multline*}
    \wModDiagCell{n}{w}\in \bigoplus_{w_1,w_2\leq w}
    \Yvar_1^{w-w_1}\Yvar_2^{w-w_2}\xwMTreesCx{n}{w_1}\otimes_{\Field}\xwMTreesCx{n}{w_2}\otimes_{\Field}\Field[\Yvar_1,\Yvar_2]\\
    \subset \bigoplus_{w_1,w_2\leq w}
    \xwMTreesCx{n}{w_1}\otimes_{\Field}\xwMTreesCx{n}{w_2}\otimes_{\Field}\Field[\Yvar_1,\Yvar_2]
  \end{multline*}
  of dimension $n+2w-2$, satisfying the following conditions:
  \begin{itemize}
  \item \textbf{Compatibility}: 
    \begin{equation}\label{eq:wModDiagCell-compat}
    \bdy\wModDiagCell{n}{x}=\sum_{v+w=x}\left(
      \sum_{j=1}^n\wModDiagCell{n-j+1}{v}\circ_1\wModDiagCell{j}{w}+\sum_{2\leq
      i\leq j\leq n}\wModDiagCell{n+1+i-j}{v}\circ_i\wDiagCell{j-i}{w}
    \right).
    \end{equation}
  \item \textbf{Non-degeneracy}:
    $\wModDiagCell{2}{0}=\wcorolla{2}{0}\otimes\wcorolla{2}{0}$.
  \end{itemize}
\end{definition}

\begin{theorem}\cite[Theorem~\ref{Abs:thm:wModDiag-exists}]{LOT:abstract}
  Given any weighted algebra diagonal $\wADiag$ there is a weighted
  module diagonal compatible with $\wADiag$.
\end{theorem}

Given a weighted module diagonal $\TrMDiag$ and weighted (right)
modules $\fModule$, $\fNodule$ over weighted algebras $\Alg$, $\Blg$ (over $\Ground$
and $\Groundl$) the tensor product $\fModule\otimes_{\wMDiag}\fNodule$ is the
weighted (right) module over $\Alg\otimes_{\wMDiag}\Blg$ defined as
follows. The underlying $(\Ground\otimes_{\Field}\Groundl)$-module is
$M\otimes_{\Field[\Yvar_1,\Yvar_2]}N$. To define the operations on
$\fModule\otimes_{\wMDiag}\fNodule$, note that given a weighted $\Ainf$-module
$(M,\{m_{1+n}^w\})$ over a weighted $\Ainf$-algebra
$(A,\{\mu_{n}^w\})$ and a weighted module tree $T$ with $n+1$
inputs there is an induced map
\[
  m^M(T)\co M\otimes A^{\otimes n}\to M
\]
defined by replacing each $(1+k)$-input, weight $w$ vertex on the
leftmost strand of $T$ with the operation $m_{1+k}^w$ on $M$, every
other $k$-input, weight $w$ vertex with the operation $\mu_{k}^w$ on
$A$, and composing according to the tree. Then the operation $m_{1+n}^w$
on $M\otimes_{\TrMDiag}N$ is defined by
\[
  m_{1+n}^w=\sum_{n_{S,T}(S\otimes T)\in \wModDiagCell{1+n}{w}}n_{S,T}\cdot m^M(S)\otimes m^N(T).
\]
(The operation $m_{1}^0$ is the usual differential on the tensor product.)

Next, given a weighted module $M$ and weighted type $D$ structure $P$
over a weighted algebra $\Alg$, there is an associated chain complex
$M\DT_{\Alg} P$ with underlying module $M\otimes_{\Field[\Yvar_1,\Yvar_2]}P$ and differential
\begin{equation}\label{eq:d-on-box-prod}
  \bdy=\sum_{n,w\geq 0}(m_{1+n}^w\otimes \Id_P)\circ(\Id_M\otimes \delta^n).
\end{equation}
(Again, convergence requires some boundedness for $M$ and/or $P$; see
Section~\ref{sec:boundedness}.)

\begin{proposition}\cite[Lemmas~\ref{Abs:lem:w-DT-sq-0} and~\ref{Abs:lem:w-DT-bifunc}]{LOT:abstract}
  The endomorphism $\bdy$ on $M\DT P$ is, indeed, a 
  differential. Moreover, changing $M$ or $P$ by a homotopy equivalence
  changes $M\DT P$ by a homotopy equivalence.
\end{proposition}

Let $\MAlg^{U=1}$ and $\CFAm_{U=1}(Y)$ be the result of setting
$U=1$ on $\MAlg$ and $\CFAm(Y)$, respectively.  We are now able to
state the first version of the pairing theorem, to be proved in a future paper:
\begin{theorem}\label{thm:pairing1}\cite{LOT:torus-pairing}
  Given 3-manifolds $Y_1$ and $Y_2$ with $\bdy Y_1=T^2=-\bdy Y_2$, for
  any choice of weighted algebra and module diagonals there is a
  homotopy equivalence
  \[
    \CFmm(Y_1\cup_{T^2}Y_2)\simeq \bigl(\CFAm_{U=1}(Y_1)\otimes_{\wMDiag}\CFAc(Y_2)\bigr)\DT_{\MAlg^{U=1}\otimes_{\wADiag}\MAlgc}\CFDDm(\Id).
  \]
\end{theorem}
(Here, $\CFAc$ and $\MAlgc$ are completions of $\CFAm$ and $\MAlg$;
see Section~\ref{sec:boundedness}.)

\begin{remark}
  Recall that for a general pointed matched circle, $\CFDDa(\Id)$ is a
  type \DD\ bimodule over $\Alg(\PMC)$ and
  $\Alg(-\PMC)=\Alg(\PMC)^\op$~\cite[Section 8.1]{LOT2}. For the genus
  $1$ pointed matched circle, there is an isomorphism
  $\Alg(\PMC)\cong \Alg(-\PMC)$ exchanging $\rho_1$ and $\rho_3$ and
  fixing $\rho_2$ (and $\rho_4$). To obtain the type \DD\ identity
  module described here, we have composed the action by $\Alg(-\PMC)$
  with this isomorphism. (We did the same thing in~\cite[Section
  10.1]{LOT2}; in the notation of~\cite[Section 3.1]{LOT4}, $\rho_3$ is
  identified with $a_o(\rho_1)$.)
\end{remark}

\subsection{Primitives and the type D module}\label{sec:intro-primitives}
In this subsection we sketch the definition of
$\CFDm(\HD)$ and state a version of the pairing theorem closer to the
pairing theorem for the hat case.

\subsubsection{Module diagonal primitives and the definition of CFD}
We will define $\CFDm(\HD)=\CFAm_{U=1}(\HD)\DT\CFDDm(\Id)$, but
making sense of this one-sided box product requires some additional
abstract algebra and, before that, some additional forestry.
\begin{definition}\cite[Section~\ref{Abs:sec:w-prim}]{LOT:abstract}
  We describe here two ways to join a sequence of weighted trees
  together; the first is appropriate for trees in $\wTreesCx{*}{*}$
  and the second for trees in $\wMTreesCx{*}{*}$. Specifically, given
  weighted trees $S_1,\dots,S_n$ in $\wTreesCx{*}{*}$ their \emph{weight
    $w$ root joining} $\wRootJoin^w(S_1,\dots,S_n)$ is the result of adding a
  new root vertex with weight $w$ and feeding the outputs of
  $S_1,\dots,S_n$ into the new vertex. The \emph{root joining} is the
  sum over $w$ of $\Yvar_1^w$ times the weight $w$ root joining. In the
  special case $n=0$ (i.e., we are joining no trees) the root joining
  is the sum over $w>0$ of $\Yvar_1^w$ times the $0$-input, weight $w$ corolla.
    
  Given weighted trees $T_1,\dots,T_n$ in $\wMTreesCx{*}{*}$ their
  \emph{left joining} $\LeftJoin(T_1,\dots,T_n)$ is obtained by feeding the
  output of each $T_i$ into the leftmost input of $T_{i+1}$. (One does
  not introduce a new $2$-valent vertex in this process.) In the
  special case that $n=0$ (i.e., we are joining no trees) we declare
  the left joining to be the generalized tree
  $\IdTree$.

  So, the number of internal vertices of
  $\wRootJoin^w(S_1,\dots,S_n)$ is one more than the sum of the numbers of internal
  vertices of $S_1,\dots,S_n$, while the number of internal vertices
  of $\LeftJoin(T_1,\dots,T_n)$ is equal to the sum of the
  number of internal vertices of $T_1,\dots,T_n$.

  Combining these operations, given a sequence of pairs of trees
  $(T_1,S_1),\dots,(T_n,S_n)$ in
  $\wMTreesCx{*}{*}\otimes_{\Field} \wTreesCx{*}{*}$ so that for each $i$,
  $T_i$ has one more input than $S_i$, the \emph{left-root joining} is
  \[
  \LRjoinW((T_1,S_1),\dots,(T_n,S_n))=\LeftJoin(T_1,\dots,T_n)\otimes\wRootJoin^{*}(S_1,\dots,S_n).
  \]
  The left-root joining is an element of (a completion of)
  $\wMTreesCx{*}{*}\otimes_{\Field} \wTreesCx{*}{*}$.

  There is a related operation which gives an element of 
  $\wMTreesCx{*}{*}\otimes_{\Field} \wMTreesCx{*}{*}$, defined by
  \[
    \LRjoinW'((T_1,S_1),\dots,(T_n,S_n)) =
    \LeftJoin(T_1,\dots,T_n) \otimes
    \wRootJoin^{*}(\DegenTree,S_1,\dots,S_n).
  \]
\end{definition}

\begin{definition}\cite[Definition~\ref{Abs:def:wM-prim}]{LOT:abstract}
  Given a weighted algebra diagonal $\wADiag$, a \emph{weighted module
    diagonal primitive} compatible with $\wADiag$ consist of a linear
  combination of trees for $n\geq 1$ and $w\geq 0$, $(n,w)\neq (1,0)$,
  of the form
  \begin{equation}\label{eq:wM-prim-1}
    \wTrPMDiag{n}{w}=\sum_{(S,T)}(S,T)\in \bigoplus_{w_1,w_2\leq w}\Yvar_1^{w-w_1} \Yvar_2^{w-w_2}\bigl(\wMTreesCx{n}{w_1}\otimes_{\Field} \xwTreesCx{n-1}{w_2}\bigr),
  \end{equation}
  of dimension $n+2w-2$, satisfying the following conditions:
  \begin{itemize}
  \item \textbf{Compatibility}:
    \begin{equation}
      \label{eq:wM-prim-compat}
      \partial \wTrPMDiag{*}{*} = 
      \LRjoinW((\wTrPMDiag{*}{*})^{\otimes\bullet})+\wTrPMDiag{*}{*}\circ'\wDiagCell{*}{*}.
    \end{equation}
    Here, $\wTrPMDiag{*}{*}=\sum_{n,w}\wTrPMDiag{n}{w}$ and
    $(\wTrPMDiag{*}{*})^{\otimes\bullet} = \sum_{n \ge
      0}(\wTrPMDiag{*}{*})^{\otimes n}$.
  \item \textbf{Non-degeneracy}: $\wTrPMDiag{2}{0}=\wcorolla{2}{0}\otimes \DegenTree$.
  \end{itemize}
\end{definition}

\begin{proposition}\cite[Proposition~\ref{Abs:prop:wprim-exist}]{LOT:abstract}
  Given any weighted algebra diagonal, there is a compatible weighted
  module diagonal primitive.
\end{proposition}

From now on, fix a weighted algebra diagonal $\wADiag$. For example, the proof
of existence of weighted diagonals~\cite[Lemma 5.3 and Theorem
6.11]{LOT:abstract} induces a specific choice of weighted algebra
diagonal, and we may fix that one throughout.

\begin{definition}\cite[Definition~\ref{Abs:def:wM-prim}]{LOT:abstract}
  Given a weighted module diagonal primitive, there is an associated
  weighted module diagonal defined by
  \[
    \LRjoinW'((\wTrPMDiag{*}{*})^{\otimes\bullet}).
  \]
\end{definition}
It is not hard to show that, for any weighted module diagonal
primitive, this defines a weighted module
diagonal~\cite[Lemma~\ref{Abs:lem:wprim-gives-diag}]{LOT:abstract}.

In the following definition, for notational reasons, we view a
left-left type \DD\ bimodules
$P$ over $\Alg$ and $\Blg$ as a left-right type \DD\ bimodule over
$\Alg$ and $\Blg^\op$.
\begin{definition}\cite[Definition~\ref{Abs:def:w-one-sided-DT}]{LOT:abstract}
  Given a weighted module diagonal primitive $\wTrPMDiagNS$ compatible with the
  weighted algebra diagonal $\wADiag$, weighted algebras $\Alg$ and
  $\Blg$, a weighted module $\fModule$ over $\Alg$, and a weighted type \DD\
  bimodule $P$ over $\Alg$ and $\Blg$ (compatible with $\wADiag$), we
  define the \emph{one-sided box tensor product of $\fModule$ and $P$},
  \[
    \fModule\DT_{\Alg}P,
  \]
  to be the weighted type $D$ structure $(M\otimes_{\Ground}P,\delta^1)$
  where $\delta^1$ is given by
  \[
    \delta^1(x\otimes y)=m_1^0(x)\otimes y \otimes 1+
    \sum_{n,w\geq 0}\sum_{(S,T)\in \wTrPMDiag{n}{w}}(m^M(S) \otimes \Id_{P}\otimes
    \mu^{\Blg^{\op}}(T^\op))\circ (\Id_M \otimes\delta_P).
  \]
\end{definition}
The fact that this defines a weighted type $D$ structure
is~\cite[Proposition~\ref{Abs:prop:one-sided-DT-works}]{LOT:abstract}. (Again,
we have specialized to the charge-1 case. Also, we again need some boundedness
to ensure the sums converge; see Section~\ref{sec:boundedness}.)

Recall that we have fixed a weighted algebra diagonal 
$\wADiag$. Given a weighted module diagonal primitive $\wTrPMDiagNS$
compatible with $\wADiag$, define
\[
  \CFDm(\HD)=\CFAm_{U=1}(\HD)\DT_{\MAlg^{U=1}}\CFDDm(\Id).
\]
\begin{theorem}\label{thm:intro-CFD}
  The object $\CFDm(\HD)$ is a weighted type $D$ structure. Further,
  up to homotopy equivalence, $\CFDm(\HD)$ depends on $\HD$ only
  through the bordered 3-manifold it specifies, and is independent of
  the choice of weighted module diagonal primitive $\wTrPMDiagNS$.
\end{theorem}
This invariance theorem follows from the fact that $\CFAm$ is
a weighted $\Ainf$-module
(Theorem~\ref{thm:CFAm-is}), invariance of $\CFAm$
(Theorem~\ref{thm:CFAm-invt}), and abstract 
algebra
(specifically,~\cite[Lemma~\ref{Abs:lem:w-one-side-DT-id-chain-map}
and Corollary~\ref{Abs:cor:w-DT-preserve-hequiv}]{LOT:abstract}).

\subsubsection{Second pairing theorem}
\begin{theorem}\label{thm:pairing2}\cite{LOT:torus-pairing}
  Given 3-manifolds $Y_1$ and $Y_2$ with $-\bdy Y_1=T^2=\bdy Y_2$,
  there is a homotopy equivalence
  \[
    \CFmm(Y_2\cup_{T^2}Y_1)\simeq \CFAc(Y_2)\DT\CFDm(Y_1).
  \]
\end{theorem}
This follows from the definitions, Theorem~\ref{thm:pairing1} (whose
proof will be given in~\cite{LOT:torus-pairing}), and abstract algebra
(specifically,~\cite[Proposition~\ref{Abs:prop:one-sided-DT-works}]{LOT:abstract}).

\subsection*{Related work}
There are several other approaches to giving a bordered extension of
$\HFm$. Zemke has given a construction of an algebra associated to the
torus and type $D$ and $A$ modules associated to 3-manifolds with
torus boundary, satisfying a pairing theorem~\cite{Zemke:bordered},
with both computational and conceptual
applications~\cite{Zemke:bordered,Zemke:lattice}. His construction
begins from the link surgery theorem~\cite{ManolescuOzsvath:surgery},
and both repackages and extends that result. The relationship of his
construction to the construction given here, or indeed to the $\HFa$
case of bordered Floer homology, is currently unknown.

Hanselman-Rasmussen-Watson have given a reformulation of the $\HFa$ case of
bordered Floer homology with torus boundary in terms of immersed
curves in the torus~\cite{HRW}. Their construction starts from $\Alg$,
$\CFDa$, and $\CFAa$, and the invariance and pairing results for
classical bordered Floer homology are ingredients in their
reformulation. Recently, Hanselman has used the rational surgery formula~\cite{OS11:RatSurg} to
refine this immersed curve invariant to compute $\HFm$, by equipping
it with an appropriate local system over $\Field[U]$~\cite{Hanselman:HFm}.

\subsection*{Acknowledgments}
We thank Rumen Zarev for helpful conversations. The proofs of
stabilization invariance in
Sections~\ref{subsec:StabilizationInvarianceBDeg}
and~\ref{sec:CFA-stab-invariance} are adapted from a book in progress
by Andr{\'a}s Stipsicz, Zolt{\'a}n Szab\'o, and the second author, and
we thank them for helpful conversations. We also thank Wenzhao Chen, Jonathan Hanselman, and 
Shikhin Sethi for comments on an early draft of this paper.


\section{Background}\label{sec:background}
\subsection{Topological preliminaries}
The topological context for the $\HFmm$ version of bordered Floer
homology---the notions of Heegaard diagrams, $\SpinC$-structures, and
so on---is the same as the setting for the original, $\HFa$
version~\cite{LOT1}. For the reader's convenience, we recall briefly much of
that terminology. The differences from our previous
monograph are summarized at the end of the section.

The unique genus $1$ pointed matched circle, shown in
Figure~\ref{fig:pmc}, consists of an oriented circle $Z$, $4$ points
$\mathbf{a}\subset Z$, and one additional basepoint
$z\in Z\setminus\mathbf{a}$. Order the points in $\mathbf{a}$ as
$a_1,a_2,a_3,a_4$ using the orientation of $Z$ and so that $z$ lies
between $a_4$ and $a_1$. We view the points as matched in pairs,
$a_1\leftrightarrow a_3$ and $a_2\leftrightarrow a_4$. When referring
to $a_i$ we view the index $i$ as an element of $\ZZ/4\ZZ$. A
\emph{Reeb chord} (or simply \emph{chord}) in $Z$ is a homotopy class of orientation-preserving
maps $\rho\co ([0,1],\{0,1\})\to (Z,\mathbf{a})$. There are four
\emph{short Reeb chords} $\rho_1,\dots,\rho_4$, where $\rho_i$ has
endpoints $a_i$ and $a_{i+1}$ and is otherwise disjoint from
$\mathbf{a}$; we will view the indices $i$ of the $\rho_i$ as elements
of $\ZZ/4\ZZ$ as well.

\begin{figure}
  \centering
  \begin{tikzpicture}
    \draw (0,0) circle (1);
    \draw[->] (1,0) arc (0:30:1);
    \fill[color=red] (1,0) circle (.1);
    \fill[color=red] (0,1) circle (.1);
    \fill[color=red] (-1,0) circle (.1);
    \fill[color=red] (0,-1) circle (.1);
    \fill (.707,-.707) circle (.1);
    \node at (1.3,0) (a1) {$\textcolor{red}{a_1}$};
    \node at (0,1.3) (a2) {$\textcolor{red}{a_2}$};
    \node at (-1.3,0) (a3) {$\textcolor{red}{a_3}$};
    \node at (0,-1.3) (a4) {$\textcolor{red}{a_4}$};
    \node at (.9,-.9) (z) {$z$};
  \end{tikzpicture}
  \caption[The genus $1$ pointed matched circle]{\textbf{The genus $1$
      pointed matched circle.}}
  \label{fig:pmc}
\end{figure}
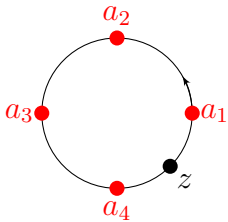

A Reeb chord $\rho$ has an \emph{initial endpoint} $\rho^-=\rho(0)$
and a \emph{terminal endpoint} $\rho^+=\rho(1)$.  Reeb chords $\rho$
and $\sigma$ are \emph{consecutive} if $\rho^+=\sigma^-$. Consecutive
Reeb chords can be concatenated; we write the concatenation of $\rho$
and $\sigma$ as $\rho\sigma$. (In~\cite{LOT1} we wrote the
concatenation of $\rho$ and $\sigma$ as $\rho\uplus\sigma$ because we
also considered sets of Reeb chords, but in the present setting
$\rho\sigma$ will not cause confusion.) Any Reeb chord can be obtained
by successively concatenating the four short Reeb chords, and we often write
$\rho_{i\dots j}$ to denote the concatenation
$\rho_i\rho_{i+1}\cdots\rho_j$. A Reeb chord $\rho$ has a
\emph{support} $\suppo{\rho}\in H_1(Z,\mathbf{a})$. We will also often
consider sequences $\vec{\rho}=(\rho^1,\dots,\rho^n)$ of Reeb chords;
the support of such a sequence is
$\suppo{\vec{\rho}}=\suppo{\rho^1}+\cdots+\suppo{\rho^n}$.

A \emph{bordered Heegaard diagram} consists of a compact, oriented
surface-with-boundary $\Sigma$, whose genus is denoted $g$; two arcs
$\alpha_1^a,\alpha_2^a\subset \Sigma$ with boundary on $\bdy \Sigma$;
$g-1$ circles $\alpha_1^c,\dots,\alpha_{g-1}^c$ and $g$ circles
$\beta_1,\dots,\beta_g$ in the interior of $\Sigma$; and a basepoint
$z\in \bdy \Sigma$ that is not an endpoint of either
$\alpha_i^a$. All of the $\alpha$-curves are required to be disjoint
from each other,
as are all of the $\beta$-circles. If we let $\alphas$ (respectively
$\betas$) be the union of the $\alpha$-curves (respectively $\beta$-circles), then
$\Sigma\setminus \alphas$ and $\Sigma\setminus\betas$ are required to
be connected. We will also typically assume that
$\alphas\pitchfork\betas$. (Unlike~\cite{LOT1}, we 
consider here only bordered Heegaard diagrams for 3-manifolds with genus $1$
boundary.)

The (oriented) boundary of a bordered Heegaard diagram
$\HD=(\Sigma,\alphas,\betas,z)$ is the genus $1$ pointed matched
circle. In particular, each Reeb chord in the genus $1$ pointed
matched circle corresponds to an arc in
$(\bdy\Sigma,\alphas\cap\bdy\Sigma)$.
Given a chord $\rho$, we let
$[\rho^-]$ (respectively $[\rho^+]$) be the $\alpha$-arc containing
the initial (respectively terminal) endpoint of $\rho$.

From a bordered Heegaard diagram one can construct a
3-manifold whose boundary is identified with a standard torus---i.e.,
a \emph{bordered 3-manifold} (with torus boundary); we refer to our
monograph~\cite[Construction 4.6]{LOT1} for the construction. Up to
diffeomorphism (rel boundary), every bordered 3-manifold with torus
boundary arises this way~\cite[Lemma 4.9]{LOT1}. Two bordered Heegaard
diagrams represent diffeomorphic bordered 3-manifolds if and only if
they become diffeomorphic as diagrams after a sequence of
\emph{isotopies}, \emph{handleslides} among the $\beta$-circles or of
$\alpha$-arcs or circles over $\alpha$-circles, and
\emph{stabilizations}~\cite[Proposition 4.10]{LOT1}. (These moves are
collectively called \emph{Heegaard moves}.)

Given a bordered Heegaard diagram $\HD$, a \emph{generator} (for the
bordered Floer modules) is a $g$-tuple of points
$\x=\{x_i\}_{i=1}^g\subset \alphas\cap\betas$ so that $\x$ has exactly
one point on each $\alpha$-circle and each $\beta$-circle and,
consequently, one point on one of the two $\alpha$-arcs. Let
$\Gen(\HD)$ denote the set of generators for $\HD$.

Given generators $\x$ and $\y$, let $\pi_2(\x,\y)$ denote the set of
homology classes, or domains, connecting $\x$ to $\y$. That is, an
element of $\pi_2(\x,\y)$ is a linear combination $B$ of connected
components of $\Sigma\setminus(\alphas\cup\betas)$ (\emph{regions})
satisfying the following condition. Let $\bdy^\alpha(B)$ (respectively
$\bdy^\beta(B)$, $\bdy^\bdy(B)$) denote the part of $\bdy(B)$ lying in
$\alphas$ (respectively $\betas$, $\bdy\Sigma$). Then,
$B\in \pi_2(\x,\y)$ if and only if $\bdy(\bdy^\beta(B))=\x-\y$.
Equivalently, $B$ is an element of
\begin{multline*}
  H_2(\Sigma\times[0,1]\times[-1,1], \x\times[0,1]\times\{-1\}\cup\y\times[0,1]\times\{1\}\\ \cup\betas\times\{0\}\times[-1,1]\cup(\alphas\cup\bdy\Sigma)\times\{1\}\times[-1,1])
\end{multline*}
so that the projection of $\bdy B$ to
$H_1(\x\times[0,1]\times\{-1\},\x\times\{0,1\}\times\{-1\})=\ZZ^g$ is
$(1,\dots,1)$, and similarly for $\y$. Note that unlike in our
monograph~\cite{LOT1}, we do not require that $B$ have multiplicity
$0$ at the region adjacent to $z$.

A domain $B$ has a multiplicity $n_z(B)$ at the region adjacent to $z$.
A \emph{periodic domain} is an element of $\pi_2(\x,\x)$ (for some
generator $\x$) with $n_z(B)=0$. A domain $B$ is \emph{provincial} if
$\bdy^\bdy (B)=0$. The set of periodic domains is in canonical
bijection with $H_2(Y,\bdy Y)$ (where $Y$ is the 3-manifold specified
by the Heegaard diagram); and the set of provincial periodic domains
is in canonical bijection with $H_2(Y)$. The boundary of a
domain $B$, $\bdy^\bdy B$, is an element of $\ZZ^4$, via the
multiplicities at $\rho_1$, $\rho_2$, $\rho_3$, and $\rho_4$,
respectively. Let $n_p(B)$ be the average multiplicity of $B$ at the
boundary (or the puncture $p$). That is, if $\bdy^\bdy B=(a,b,c,d)$
then $n_p(B)=(a+b+c+d)/4$. (In this notation, $n_z(B)=d$.)

There is a concatenation map
$\pi_2(\w,\x)\times\pi_2(\x,\y)\to\pi_2(\w,\y)$ which we will either
write as $(B_1,B_2)\mapsto B_1*B_2$ or $(B_1,B_2)\mapsto B_1+B_2$,
depending on the context.

A bordered Heegaard diagram is called \emph{provincially admissible}
if every non-zero provincial periodic domain has both positive and negative
coefficients, and \emph{admissible} if every periodic domain has both
positive and negative coefficients. Every bordered Heegaard diagram is
isotopic to an admissible one, and any two admissible (respectively
provincially admissible) Heegaard diagrams can be connected by a
sequence of Heegaard moves through admissible (respectively
provincially admissible) diagrams~\cite[Proposition 4.25]{LOT1} (whose
proof is cited to~\cite[Section 5]{OS04:HolomorphicDisks}).

Every generator $\x$ represents a $\SpinC$-structure
$\spinc(\x)\in\Spinc(Y)$~\cite[Section 4.3]{LOT1}. Further,
$\spinc(\x)=\spinc(\y)$ if and only if
$\pi_2(\x,\y)\neq\emptyset$~\cite[Lemma 4.21]{LOT1}. (In~\cite{LOT1},
we consider only domains with multiplicity $0$ at $z$, but since we
can add copies of $[\Sigma]$ to the domain,
$\pi_2(\x,\y)\neq\emptyset$ if and only if there is a
$B\in\pi_2(\x,\y)$ with $n_z(B)=0$.) If $\spinc(x)=\spinc(y)$, the set
$\pi_2(\x,\y)$ is an affine copy of $\ZZ\oplus H_2(Y,\bdy Y)$.

When considering holomorphic curves, we will attach a cylindrical end
to $\bdy\Sigma$, giving a non-compact surface with a single end. We
will abuse notation and denote that surface by~$\Sigma$, as well; it
will be clear from context which of the two versions of $\Sigma$ we
are discussing. We can think of this non-compact $\Sigma$ as a closed
surface $\overline{\Sigma}$ minus a point $p$, and will often refer to
the puncture (point at infinity) of $\Sigma$ as $p$. The non-compact
version of $\Sigma$ still has a circle at infinity, and we can talk
about Reeb chords in that circle (which are the same as the Reeb
chords for the compact version).

We recapitulate the conventions that differ in this paper from our
monograph~\cite{LOT1}:
\begin{itemize}
\item In this paper, the only pointed matched circle we consider is
  the (unique) pointed matched circle representing the torus, and we
  consider only the middle strands grading. So, the algebra $\Alg$ here is
  sometimes denoted $\Alg(\PMC_1,0)$ or $\Alg(T^2,0)$ in the
  literature.
\item We have dropped the notation $\uplus$ from concatenation of Reeb
  chords.
\item The notation $\pi_2(x,y)$ denotes all domains connecting $x$ and
  $y$, not just those with $n_z=0$.
\item We are not distinguishing in the notation between the compact
  surface-with-boundary $\Sigma$ (denoted $\overline{\Sigma}$
  in~\cite{LOT1}) and the result of attaching a cylindrical end to it
  (denoted $\Sigma$ in~\cite{LOT1}). We are using $\overline{\Sigma}$
  to denote the one-point compactification of the non-compact surface
  $\Sigma$. (This was denoted $\Sigma_{\widebar{e}}$ in~\cite{LOT1}.)
\end{itemize}

\subsection{\texorpdfstring{$G$}{G}-set gradings on weighted algebras, modules and type \texorpdfstring{$D$}{D} structures}
\label{sec:gradings-abstract}
Like the $\HFa$ case, the algebra $\MAlg$ is graded by a
non-commutative group with a distinguished central element, and the
modules $\CFAm$ and $\CFDm$ are graded by sets with right- and
left-actions by this group, respectively. Also like the $\HFa$ case,
there are several natural choices for the non-commutative group. In
this section we describe what a group-valued grading on a weighted
module or type $D$ structure means abstractly; the gradings on $\CFAm$
and $\CFDm$ themselves are constructed in
Section~\ref{sec:gradings}. In Section~\ref{sec:big-gr-group} we
recall one of the options for a grading group. (This material is
used in Sections~\ref{sec:DD-Id} and~\ref{sec:gradings}.)

Fix a group $G$ and central elements $\lambda_d,\lambda_w\in G$. We
call $\lambda_w$ the \emph{weight grading}; when
$\lambda_w$ is the
identity element of $G$ then we will often write $\lambda_d$ simply
as~$\lambda$.
Recall that a \emph{$(G,\lambda_d,\lambda_w)$-graded weighted algebra} is a
weighted algebra $\Alg=(A,\{\mu_n^w\})$ together with subspaces
$A_g\subset A$,
$g\in G$, of \emph{homogeneous elements of grading $g$}, so that
$A_g\cap A_h=\{0\}$ if $g\neq h$ and 
\begin{equation}\label{eq:graded-mu}
  \mu_n^k\co A_{g_1}\otimes\cdots\otimes A_{g_n}\to
  A_{\lambda_d^{n-2}\lambda_w^kg_1\cdots g_n}
\end{equation}
\cite[Section~\ref{TA:sec:W-gradings}]{LOT:torus-alg}.
That is, if we write $\gr(a)=g$ whenever $a\in A_g$ then
\[
\gr(\mu_n^k(a_1,\dots,a_n))=\lambda_d^{n-2}\lambda_w^{k}\gr(a_1)\cdots\gr(a_n).
\]
We will assume that either $A=\bigoplus_{g\in G}A_g$ (i.e., is generated
by the homogeneous elements), or is the completion of $\bigoplus_{g\in
  G}A_g$ with respect to some filtration. (The former will be the case
for $\MAlg$ and the latter for its completion $\MAlgc$ used in Section~\ref{sec:CFD}.)

Given a set $S$ with a right action of $G$, a grading of a weighted right
module $\fModule=(M,\{m_{1+n}^k\})$ by $S$ is a collection of
subspaces $M_s\subset M$, $s\in S$, so that $M_s\cap M_t=\{0\}$
if $s\neq t$ and
\[
  m_{1+n}^k\co M_s\otimes A_{g_1}\otimes\cdots\otimes A_{g_n}\to
  M_{s\cdot \lambda_d^{n-1}\lambda_w^kg_1\cdots g_n}.
\]
Writing $\gr(M_s)=s$, this is equivalent to 
\[
\gr(m_{1+n}^k(\x,a_1,\dots,a_n))=\gr(\x)\cdot \lambda_d^{n-1}\lambda_w^{k}\gr(a_1)\cdots\gr(a_n).
\]
Again, we assume that either $M=\bigoplus_{s\in S}M_s$ or that $M$ is
the completion of this direct sum with respect to a filtration.

Given a set $S$ with a left action by $G$, a grading of a weighted
type $D$ structure $(P,\delta^1\co P\to A\otimes P)$ by $G$ consists
of subspaces $P_s$, $s\in S$, satisfying $P_s\cap P_t=\{0\}$ if $s\neq
t$ and
\[
  \delta^1\co P_s\to \bigoplus_{g\cdot t=s} \lambda_d^{-1}A_g\otimes P_t.
\]
That is, writing $\gr(P_s)=s$,
\[
\gr(\delta^1(x))=\lambda_d^{-1}\gr(x).
\]
We again assume that either $P=\bigoplus_{s\in S}P_s$ or that $P$ is
the completion of this direct sum with respect to a filtration.

Given weighted algebras $\Alg$ and $\Alg'$ graded by
$(G,\lambda_d,\lambda_w)$ and $(G',\lambda_d',\lambda_w')$, and a
weighted algebra diagonal, the tensor product $\Alg\otimes\Alg'$
inherits a grading by the group
$G\times_{\ZZ}G'=G\times G'/\langle \lambda_d^{-1}\lambda_d'\rangle$
with central elements $\lambda_d=\lambda_d'$ and
$\lambda_w\lambda_w'\lambda_d^{-2}$
(compare~\cite[Section~\ref*{Abs:sec:w-alg-tens}]{LOT:abstract}).
Here, we assume that the variables $\Yvar_1$ and
$\Yvar_2$ in the definition of a weighted diagonal---see
Section~\ref{sec:intro-DD}---have gradings $\lambda_d^{-2}\lambda_w$
and $(\lambda_d')^{-2}\lambda_w'$, respectively. This gives rise to
the notion of a graded, weighted type \DD\ structure $P$ over $\Alg$ and
$\Alg'$, graded by a left $(G\times_\ZZ G')$-set $T$.

Continuing in this vein, given a weighted \DD\ structure~$P$ as above
and a weighted module~$\fModule$ graded by
the $G$-set $S$, the one-sided box product $\fModule\DT P$ is graded
by the $G'$-set
$S\times_GT$~\cite[Proposition~\ref*{Abs:prop:one-sided-DT-works}]{LOT:abstract}.
Given another weighted module $\fModule'$ graded by the $G'$-set $S'$,
the triple box product $\fModule\DT P\DT\fModule'$ is graded by the
$\ZZ$-set $(S\times_{G}T\times_{G'}(S')^\op,\lambda_d=\lambda_d')$.

\subsection{The big grading group}\label{sec:big-gr-group}
Because we will need it in Section~\ref{sec:DD-Id}, we recall the
definition of the big grading group
$\bigGroup$~\cite[Section~\ref*{TA:sec:big-group}]{LOT:torus-alg}
(also recalled in the introduction). Other options for
gradings are discussed in Section~\ref{sec:gradings}.

Consider the group $\OneHalf\ZZ\times\ZZ^4$ with multiplication
\begin{multline*}
(m;a,b,c,d)\cdot (m';a',b',c',d')
\\=
\left(m+m'
+\OneHalf\left|
  \begin{smallmatrix}
    a & b\\
    a' & b'
  \end{smallmatrix}
  \right|
+\OneHalf\left|
  \begin{smallmatrix}
    b & c\\
    b' & c'
  \end{smallmatrix}
  \right|
+\OneHalf\left|
  \begin{smallmatrix}
    c & d\\
    c' & d'
  \end{smallmatrix}
  \right|
+\OneHalf\left|
  \begin{smallmatrix}
    d & a\\
    d' & a'
  \end{smallmatrix}
  \right|,a+a',b+b',c+c',d+d'\right).
\end{multline*} 
The
elements
\begin{align*}
  \grb(\rho_1)&=(-1/2;1,0,0,0) & \grb(\rho_2)&=(-1/2;0,1,0,0) \\
  \grb(\rho_3)&=(-1/2;0,0,1,0) & \grb(\rho_4)&=(-1/2;0,0,0,1)\\
  \lambda&=(1;0,0,0,0)
\end{align*}
generate an index-2 subgroup of this group, which we
denote~$\bigGroup$. A grading $\grb$ on the algebra $\MAlg$ by
$\bigGroup$ is
determined by the gradings of the $\rho_i$ and $\grb(U)=(-1;1,1,1,1)$. We will
call the first entry in the grading the \emph{Maslov component}, and
the remaining four the \emph{$\SpinC$ component}. For a Reeb
chord~$\rho$, the $\SpinC$
component of $\grb(\rho)$ is just the support of~$\rho$, while the Maslov
component is $-|\rho|/4$ if $|\rho|$ is divisible by $4$, and
$-1/2-\lfloor |\rho|/4\rfloor$ otherwise.
(This is also $\iota(\rho)$ from Eq.~\eqref{eq:iota}.)
For this grading, the
central element is $\lambda_d=\lambda=(1;0,0,0,0)$ and the weight grading
is $\lambda_w=(1;1,1,1,1)$. (Note that $\grb(U)=\grb(\rho_{1234})$ and
$\lambda_w=\lambda^2\grb(U)$, so both lie in $\bigGroup$.)

\subsection{The associaplex}
\label{sec:associaplex}

In the course of studying the algebraic background behind weighted
$\Ainf$-algebras, we introduced a CW complex called the (weighted)
\emph{associaplex} \cite[Sec.\ 8]{LOT:abstract}.

Recall that the \emph{associahedron}, relevant to $\Ainf$-algebras, is
a polytope that is a compactification of the space of $n+1$ points on
$\bdy D^2$, one of which is distinguished  (or
equivalently of $n$ points on a line), modulo symmetries. The \emph{associaplex}
$\Associaplex{n}{w}$, relevant to weighted $\Ainf$-algebras, is
instead a compactification of the space of $w$ distinct, unordered
points in the interior of $D^2$ and $n+1$ points in $\bdy D^2$, one of
which is distinguished, again modulo symmetries (M\"obius
transformations). The most familiar compactification is the
Deligne-Mumford compactification, which adds strata corresponding to
decompositions into disks and spheres connected at nodes, each with
enough special points to be stable. The Deligne-Mumford
compactification recovers the associahedron when all the points are on
the boundary, but records more information than we need about interior
points, so  $\Associaplex{n}{w}$ is obtained from the Deligne-Mumford compactification by collapsing all the spheres.  So, interior marked
points collide as in a symmetric product, recording just the
multiplicities. On the boundary, more information is remembered,
through disks bubbling off.  We refer the reader to our previous
paper~\cite[Section 8]{LOT:abstract} for further details, and a direct
construction not using the Deligne-Mumford compactification.


\section{Moduli spaces of holomorphic curves}\label{sec:moduli}
The main goal of this section is to prove Theorem~\ref{thm:master},
the workhorse technical result guaranteeing that $\CFAm$ and $\CFDm$
are well-defined. We start with some examples, in
Section~\ref{sec:moduli-examples}, to illustrate the kinds of
codimension-1 degenerations which occur; this section is not needed
for the rest of the paper.  The work starts in
Section~\ref{sec:def-mod-sp}, where we define the various moduli
spaces. Key to that is the formulation of the families of almost
complex structures we will use, in Section~\ref{sec:J}. Next, we
compute the expected dimensions of the moduli spaces, in
Section~\ref{sec:ind}. Section~\ref{sec:pinched} formulates a
condition, being sufficiently pinched, that is required to guarantee
that boundary degenerations correspond to operations on the algebra.
(Part of the proof that condition suffices is in
Section~\ref{sec:algebra}.)  In Section~\ref{sec:transversality}, we
show that for appropriate families of almost complex structures, most of the
moduli spaces are transversely cut out by the
$\dbar$-equations. (Again, a few remaining cases are deferred to
Section~\ref{sec:algebra}.)  Section~\ref{sec:gluing} collects the
gluing results we will need, showing that near various kinds of broken
holomorphic curves the moduli space behaves like a
manifold-with-boundary. Finally, Section~\ref{sec:compactness}
combines these ingredients to prove Theorem~\ref{thm:master}.

Throughout this section we fix a bordered Heegaard diagram
$\HD=(\Sigma,\alpha_1^a,\alpha_2^a,\alphas^c,\betas,z)$ with genus
$g$.

\subsection{Some examples}\label{sec:moduli-examples}
Before diving into the section's work, we describe examples of the
different degenerations in Theorem~\ref{thm:master}. We will refer to
these examples later in the section to elucidate some of the more
complicated definitions. Definitions of the moduli spaces discussed
here are given in Section~\ref{sec:def-mod-sp}, but a reader familiar
with bordered Floer homology can likely read this section without
having read those definitions first.

The first two kinds of degenerations in Theorem~\ref{thm:master} are familiar
from the $\HFa$ case. For example, consider the genus-1 bordered Heegaard
diagram and domain shown in Figure~\ref{fig:split-degen-eg}.
The moduli space shown is denoted
$\cM^B(x,y;\rho_3,\rho_4,\rho_3,\rho_2,\rho_{12};0)$, where $B$ is the
domain. (The $0$ denotes that there
are no orbits on this curve, only chords.) The moduli space is 1-dimensional,
corresponding to the location of the boundary branch point (end of the cut). As the
branch point approaches the chord $\rho_{34}$ (cut shrinks to zero) the chords
$\rho_3$ and $\rho_4$ come together, degenerating a split curve, just as would
happen for bordered $\HFa$; this is a collision end. The main component of the result is an element of
$\cM^B(x,y;\rho_{34},\rho_3,\rho_2,\rho_{12};0)$. (In general, the chords involved
in the split curve degeneration can be arbitrarily long.)
Another collision end (and split curve) occurs in Figure~\ref{fig:orbit-curve-eg} when the branch
point approaches~$\rho_{34}$; we discuss the other ends of these
examples after giving another example.

Consider the Heegaard diagram and domain shown in
Figure~\ref{fig:simple-bdy-degen-eg}. Here, the domain $B$ has multiplicity $1$
everywhere in the diagram; the curve is an element of
$\cM^B(x,x;\rho_4,\rho_3,\rho_2,\rho_1;0)$. As the branch point approaches the
left edge of the rectangle, the curve breaks as a two-story building, with the
bottom story in $\cM^{B_1}(x,x;\rho_4,\rho_3;0)$ and the top story in
$\cM^{B_2}(x,x;\rho_2,\rho_1;0)$. Again, this kind of degeneration (but not this
particular domain) occurs in the $\HFa$ case. Indeed, this is the kind of
degeneration that corresponds most closely to broken flows in Morse theory.

The remaining ends are new. In Figure~\ref{fig:simple-bdy-degen-eg}, as the
branch point approaches the right boundary of the rectangle (cut shrinks to
zero), the curve degenerates into a component with boundary entirely in the
$\alpha$-curves (a boundary degeneration) and a bigon mapped to $x$ by a
constant map (a trivial strip). Let $\HHH$ denote the half-plane
$\HHH=(-\infty,1]\times \RR=\{x+iy\mid x\leq 1\}$. By rescaling, the component
with boundary entirely in the $\alpha$-curves inherits a map to
$(\Sigma\times\HHH,\alphas\times\{1\}\times\RR)$, asymptotic at $\infty$ in
$\HHH$ to the point $x$. This is an element in $\cN(x;\rho_4,\rho_3,\rho_2,\rho_1;0)$, and is called a \emph{simple boundary degeneration}.

In Figure~\ref{fig:split-degen-eg}, when the branch point approaches $\rho_{12}$
(i.e., the cut becomes as long as possible), the curve splits into a main
component, in $\cM^{B_1}(x,y;\rho_3,\rho_2;0)$ and a boundary degeneration in
$\cN(\rho_2;\rho_4,\rho_3,\rho_2,\rho_{12})$. The two are connected by a join curve at
$e\infty$, with asymptotics $\rho_1$, $\rho_2$, and $\rho_{12}$. This is called a
\emph{composite boundary degeneration}.

In Figure~\ref{fig:orbit-curve-eg} we see one more kind of degeneration,
involving an orbit. As the branch point approaches the puncture in the middle,
the curve splits off a component at $e\infty$ with an orbit on one end and the
length-4 chord $\rho_{2341}$ on the other. The moduli space before the degeneration is
$\cM^B(x,y;\rho_4,\rho_3,\rho_4,\rho_3,\rho_{23},\rho_2,\rho_{12};1)$. After
the degeneration, the main component is a curve in 
$\cM^B(x,y;\rho_4,\rho_3,\rho_{2341},\rho_4,\rho_3,\rho_{23},\rho_2,\rho_{12};0)$
and a curve at $e\infty$ called an \emph{orbit curve}.

\begin{figure}
  \centering
  \includegraphics{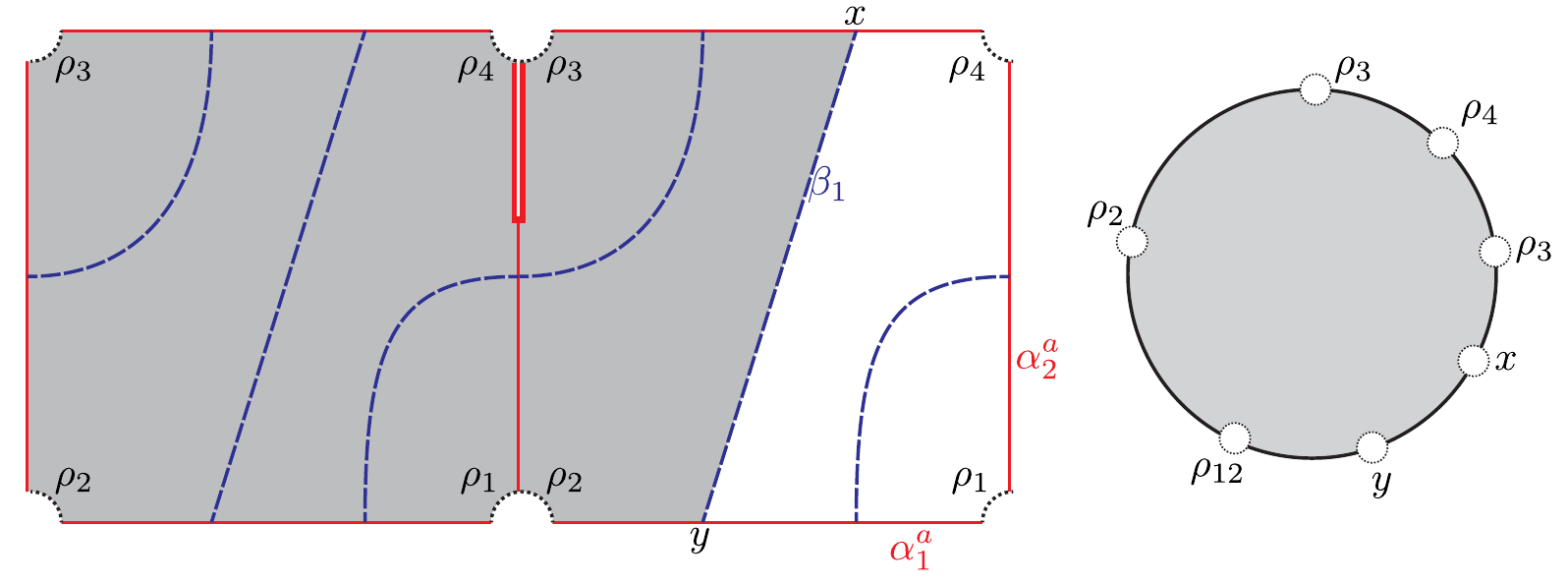}
  \caption[Example of a split curve end and a composite boundary
  degeneration]{\textbf{A split curve and a composite boundary
      degeneration.} Right: the decorated source of a holomorphic
    curve. Left: its image. The end of the slit is a boundary branch
    point. As the slit shrinks to zero, the curve degenerates a split
    curve with $e\infty$ asymptotics $\rho_3$ and $\rho_4$ and
    $w\infty$ asymptotic $\rho_{34}$; this is a collision end. As the slit approaches the chord
    $\rho_{12}$ the curve degenerates a composite boundary
    degeneration (a boundary degeneration and a join curve).}
  \label{fig:split-degen-eg}
\end{figure}

\begin{figure}
  \centering
  \includegraphics{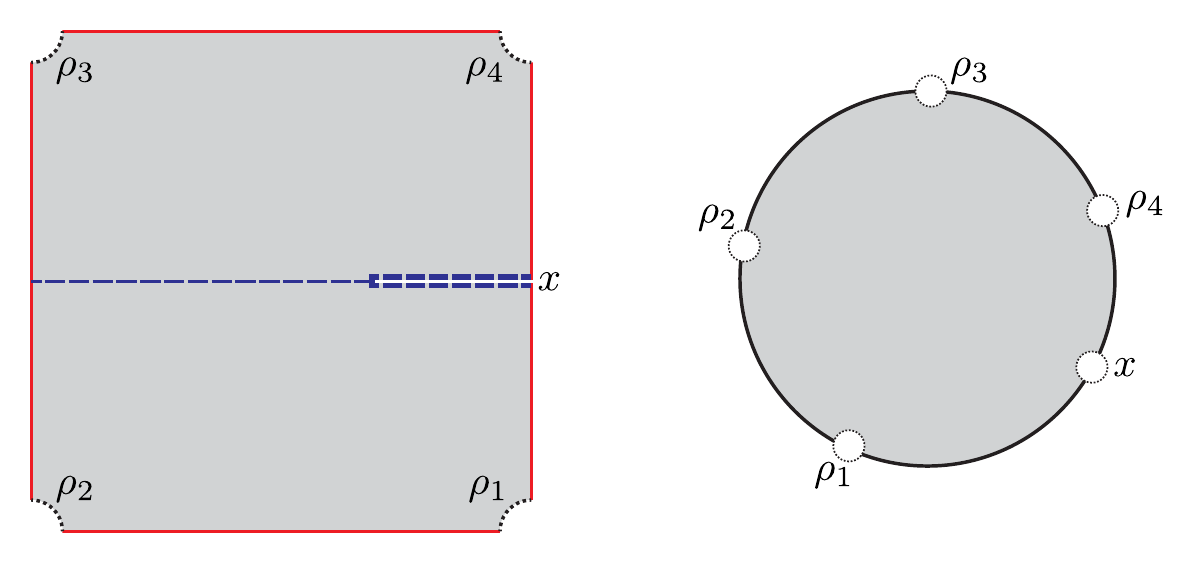}
  \caption[Example of a 2-story end and a simple boundary
  degeneration]{\textbf{A two-story end and a simple boundary
      degeneration.} When the slit gets long, approaching the left
    edge, the curve degenerates to a 2-story holomorphic
    building. When the slit shrinks to zero, the curve becomes a
    simple boundary degeneration.}
  \label{fig:simple-bdy-degen-eg}
\end{figure}

\begin{figure}
  \centering
  \includegraphics{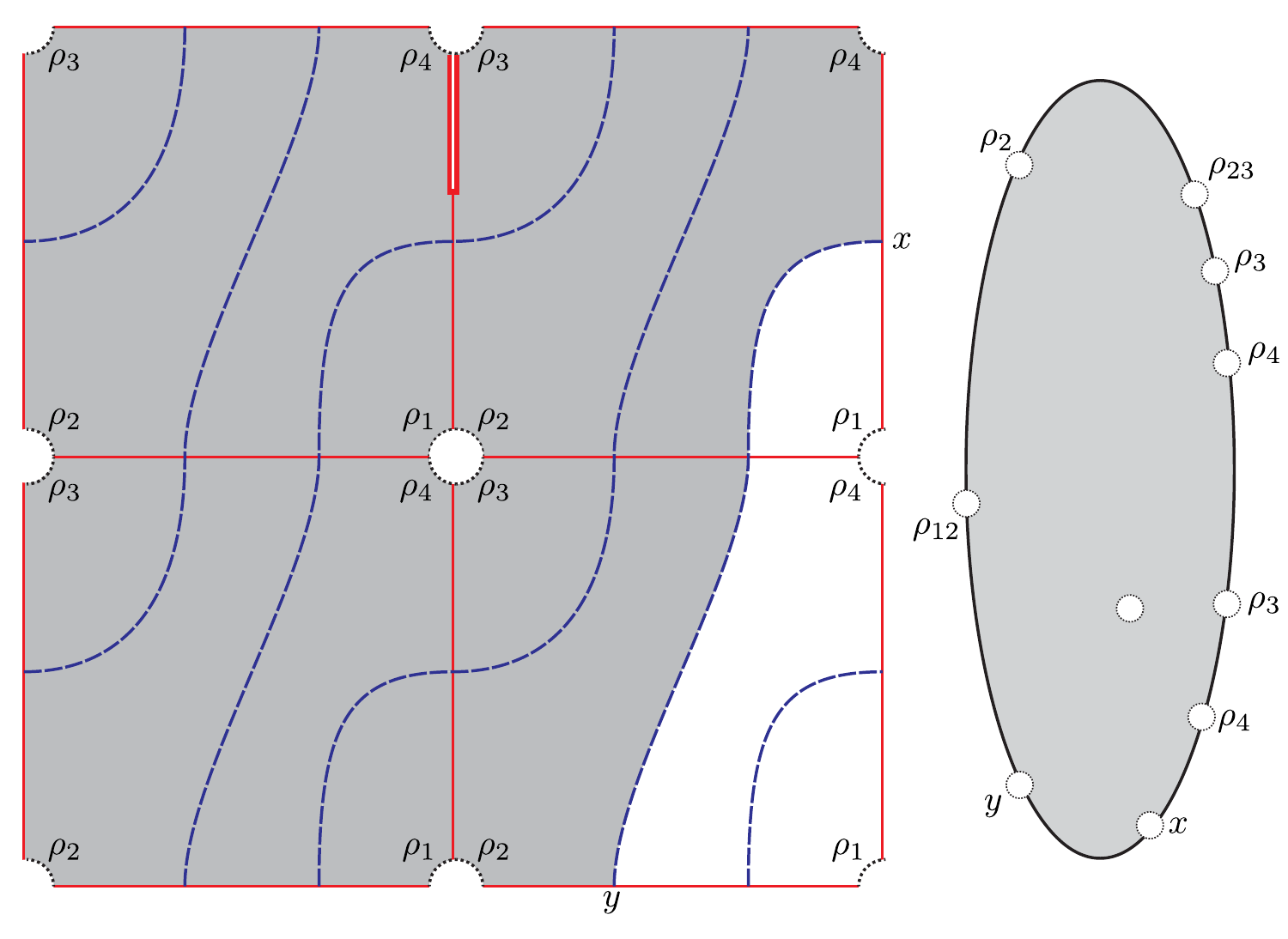}
  \caption[Example of an orbit curve end and a collision end.]{\textbf{An orbit curve end and a collision end.} When the
    slit approaches the puncture, an orbit curve degenerates. When the
    slit shrinks to zero, a split curve degenerates.}
  \label{fig:orbit-curve-eg}
\end{figure}

\subsection{Definitions of the moduli spaces}\label{sec:def-mod-sp}
In this subsection, we will define the moduli spaces of holomorphic
curves used later in the paper. Briefly, these are:
\begin{enumerate}[label=(\arabic*)]
\item\label{item:m-fixed-source} The moduli spaces of curves in
  $\bigl(\Sigma\times[0,1]\times\RR,(\alphas\times\{1\}\times\RR)\cup(\betas\times\{0\}\times\RR)\bigr)$
  with a fixed source. These are denoted $\cM^B(\x,\y;\Source)$ where
  $\x$ and $\y$ are generators, $B$ is a homology class, and $\Source$
  is a decorated Riemann surface.
\item\label{item:m-embedded} The moduli spaces of embedded curves in
  $\bigl(\Sigma\times[0,1]\times\RR,(\alphas\times\{1\}\times\RR)\cup(\betas\times\{0\}\times\RR)\bigr)$.
  These are denoted $\cM^B(\x,\y;\vec{a};w)$ where $\vec{a}$ is
  a sequence of basic algebra elements and $w\in\ZZ_{\geq 0}$ is the number of orbits.
\item\label{item:m-east-infty} The moduli spaces of holomorphic curves
  in
  $\bigl(\RR\times \partial
  \Sigma\times[0,1]\times\RR,\alphas\times\{1\}\times\RR\bigr)$ with a
  fixed source, i.e., curves at $e\infty$. These are denoted
  $\cN(\biSource)$.
\item\label{item:m-degen-fixed} The moduli spaces of curves in
  $\bigl(\Sigma\times\HHH,\alphas\times\{1\}\times\RR\bigr)$ with a
  fixed source (where
  $\HHH=(-\infty,1]\times \RR$). These are denoted $\cN(\x;\bdSource)$ or
  $\cN(*;\bdSource)$ depending on whether the asymptotics at $\infty$
  in~$\HHH$
  are fixed or allowed to vary.
\item\label{item:m-degen-embedded} The moduli spaces of embedded curves in
  $\bigl(\Sigma\times\HHH,\alphas\times\{1\}\times\RR\bigr)$. These
  are denoted $\cN(\x;\vec{a};w)$ or $\cN(*;\vec{a};w)$, again
  depending on whether the asymptotics at $\infty$ are fixed or
  not. (A variant $\cN([\alpha_i^a];\vec{a};w)$ is introduced in
  Section~\ref{sec:transversality}.) Again, $\vec{a}$ is a sequence of
  basic algebra elements.
\end{enumerate}

Properties of these moduli spaces are developed in later
subsections. For the construction of $\CFAa$ and $\CFDa$, we do not
use directly the moduli spaces with a fixed source; rather, those are
used as auxiliary steps in understanding the moduli spaces of embedded
curves. Types~\ref{item:m-fixed-source},~\ref{item:m-embedded},
and~\ref{item:m-east-infty} are generalizations of moduli spaces
considered for bordered $\HFa$~\cite[Chapter 5]{LOT1};
Types~\ref{item:m-degen-fixed} and~\ref{item:m-degen-embedded} are
analogues of boundary degenerations or disk bubbles (see,
e.g., \cite{FOOO1}, and also~\cite{OhZhu11:scale}).

Before turning to the construction of the moduli spaces themselves, we
describe the kinds of almost complex structures we will use.
\subsubsection{Families of almost complex structures}
\label{sec:J}

Difficulties with transversality force us to use somewhat more
intricate families of almost complex structures than for the case of
$\HFa$. In this section, we specify the compatibility conditions those
families are required to satisfy. We will define two notions:
\begin{itemize}
\item An ``$\eta$-admissible almost complex structure''
  (Definition~\ref{def:ac}). This spells out the conditions a single
  almost complex structure is required to satisfy.
\item A ``coherent family of $\eta$-admissible almost complex
  structures'' (Definition~\ref{def:AdmissibleJs}). This spells out
  the compatibility conditions for families of almost complex
  structures needed to define points in our moduli spaces which, in
  turn, are used to define the operations $m_{1+n}^w$.
\end{itemize}
Transversality conditions for those
families are deferred to Section~\ref{sec:transversality}; a coherent
family of almost complex structures which is also sufficiently generic
for counting curves will be called a ``tailored family''
(Definition~\ref{def:tailored}). For technical reasons, we will not be
able to work with a single almost complex structure for all moduli spaces.

A key ingredient in these definitions is a particular moduli space of
marked polygons:
\begin{definition}\label{def:polygons}
  A \emph{bimodule component} is a conformal disk with $n\geq 2$
  boundary marked points and some number of interior marked points, a
  choice of two distinguished boundary marked points called
  $\pm\infty$, and a function from the other marked points to
  $\ZZ_{>0}$, called the \emph{energy}, so that the energy of each
  interior marked point is a multiple of $4$. We require that there be
  at least one marked point other than $\pm\infty$.

  Forgetting all of the marked points except $\pm\infty$ induces an
  identification of the disk minus $\{\pm\infty\}$ with
  $[0,1]\times\RR$, well-defined up to translation.  Under this
  identification, the
  boundary of the disk is divided into two parts, which we refer to as
  $\{0\}\times\RR$ and $\{1\}\times\RR$. The energies of the boundary
  marked points on $\{1\}\times\RR$ (respectively $\{0\}\times\RR$)
  defines a sequence $\vec{E}$ (respectively $\vec{E}'$). Call the
  marked points labeled by $\pm\infty$ the \emph{module vertices}, the
  marked points on $\{0\}\times\RR$ the \emph{left algebra vertices},
  and the marked points on $\{1\}\times\RR$ the \emph{right algebra
    vertices}.

  Given a pair of non-negative integers $(k,\ell)$, sequences of
  positive integers $\vec{E}'=(E'_1,\dots,E'_k)$ and
  $\vec{E}=(E_1,\dots,E_\ell)$, and a non-negative integer $w$,
let $\ModPol_{k,\ell,\vec{E'},\vec{E},w}$ be a copy of the
  associaplex $\Associaplex{k+1+\ell}{w}$.
  We think of $\ModPol_{k,\ell,\vec{E'},\vec{E},w}$ as a 
  compactification of the space of module polygons with $k$ vertices on
  $\{0\}\times\RR$, $\ell$ vertices on $\{1\}\times\RR$, total energy
  of interior punctures $4w$, and energies of the boundary punctures
  specified by $\vec{E}'$ and $\vec{E}$. 
  
  \begin{figure}
    \centering
    \includegraphics{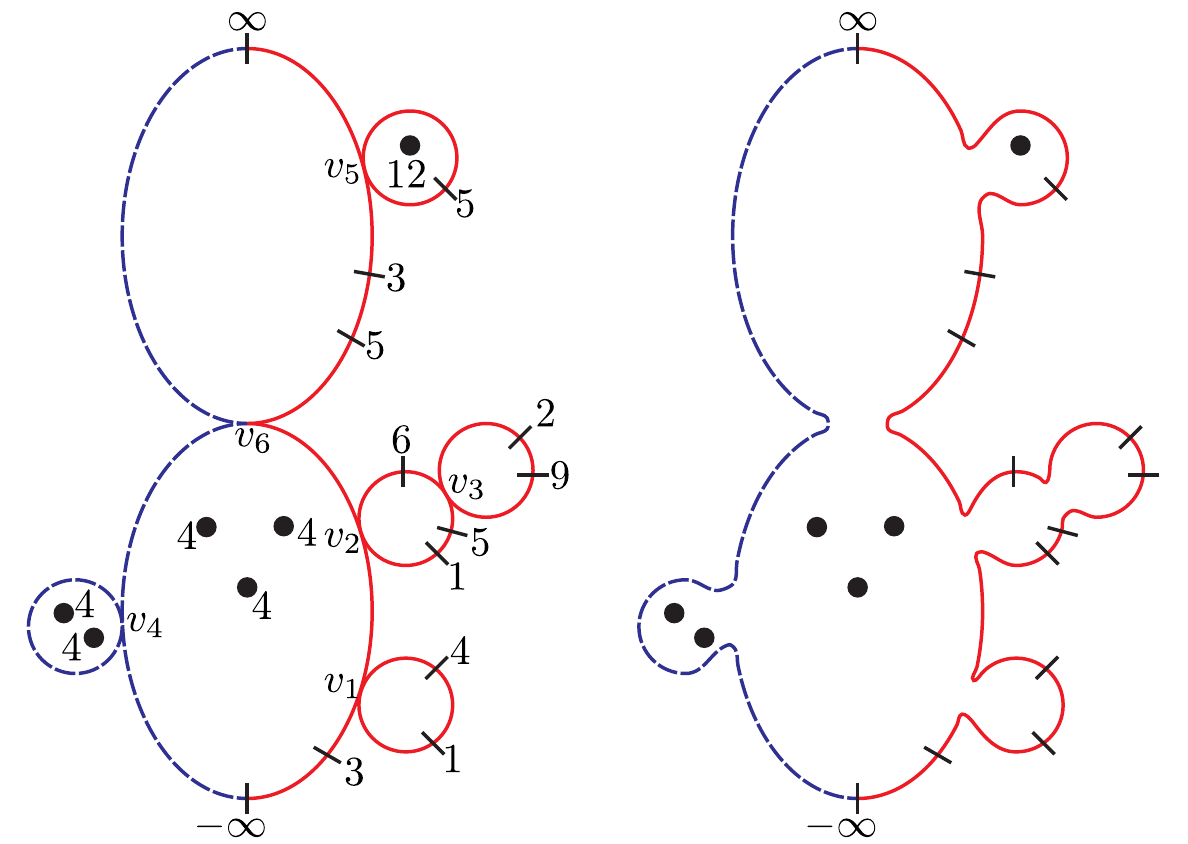}
    \caption[A bimodule polygon and its gluing]{\textbf{A bimodule
        polygon $x$ (left) and its gluing $\widetilde{x}$
        (right)}. The left boundary is \textcolor{blue}{dashed} and the right boundary is
      \textcolor{red}{solid}. Disks in~$x$ whose boundary is entirely \textcolor{red}{solid} are left algebra-type
      components, disks whose boundary is entirely \textcolor{blue}{dashed} are right
      algebra-type components, and the two disks with both \textcolor{blue}{dashed} and
      \textcolor{red}{solid} boundaries are bimodule components. The energies of the
      nodes $v_1,\dots,v_5$ are $5$, $23$, $11$, $8$, and $17$,
      respectively. The node $v_6$ is not assigned an energy.}
    \label{fig:bimod-poly-glue}
  \end{figure}
  
  By definition, a point $x\in \ModPol_{k,\ell,\vec{E'},\vec{E},w}$ is
  a tree of disks meeting at nodes. We can glue the disks in $x$
  together at the nodes to obtain a single disk $\widetilde{x}$. (See Figure~\ref{fig:bimod-poly-glue}.)
  The identification between the boundary of $x$ and $\widetilde{x}$
  identifies the arcs in the boundary of each disk in $x$ with a
  subset of either $\{0\}\times\RR$ or $\{1\}\times\RR$. Call a disk
  in $x$ a \emph{left algebra-type component} if its whole boundary
  corresponds to a subset of $\{0\}\times\RR$, a \emph{right algebra-type
    component} if its whole boundary corresponds to a subset of
  $\{1\}\times\RR$, and a \emph{bimodule-type component}
  otherwise. Suppose
  $v$ is a node of $x$ so that at least one of the disks meeting at
  $v$ is a (left or right) algebra-type component. Cutting $x$ at $v$ gives
  two trees of disks, one of which consists entirely of algebra-type
  components. Define the energy of $v$ to be sum of the energies of the
  marked points (not nodes) of that tree of algebra-type components.
  With this definition, each bimodule-type component in $x$ is a bimodule component.

  Let $\ModPol$ be the disjoint union of all the
  $\ModPol_{k,\ell,\vec{E}',\vec{E},w}$ and, given a positive integer
  $E$, let $\ModPol_E$ be the disjoint union of all
  $\ModPol_{k,\ell,\vec{E'},\vec{E},w}$ so that
  $E'_1+\cdots+E'_k+E_1+\cdots+E_\ell+4w=E$. We call a point in
  $\ModPol$ a \emph{bimodule polygon}. In particular, a bimodule
  component is a special case of a bimodule polygon.
\end{definition}

We can also define a notion of module polygons: a \emph{module
  polygon} is an element of $\ModPol_{0,n,,\vec{E},w}$ for some $n$,
$\vec{E}$, and $w$. (A \emph{module component} is a module polygon
with a single component.) However, some components of a module polygon
may be bimodule components, because algebra-type components can bubble
off along $\{0\}\times\RR$.
 
\begin{figure}
  \centering
  \includegraphics{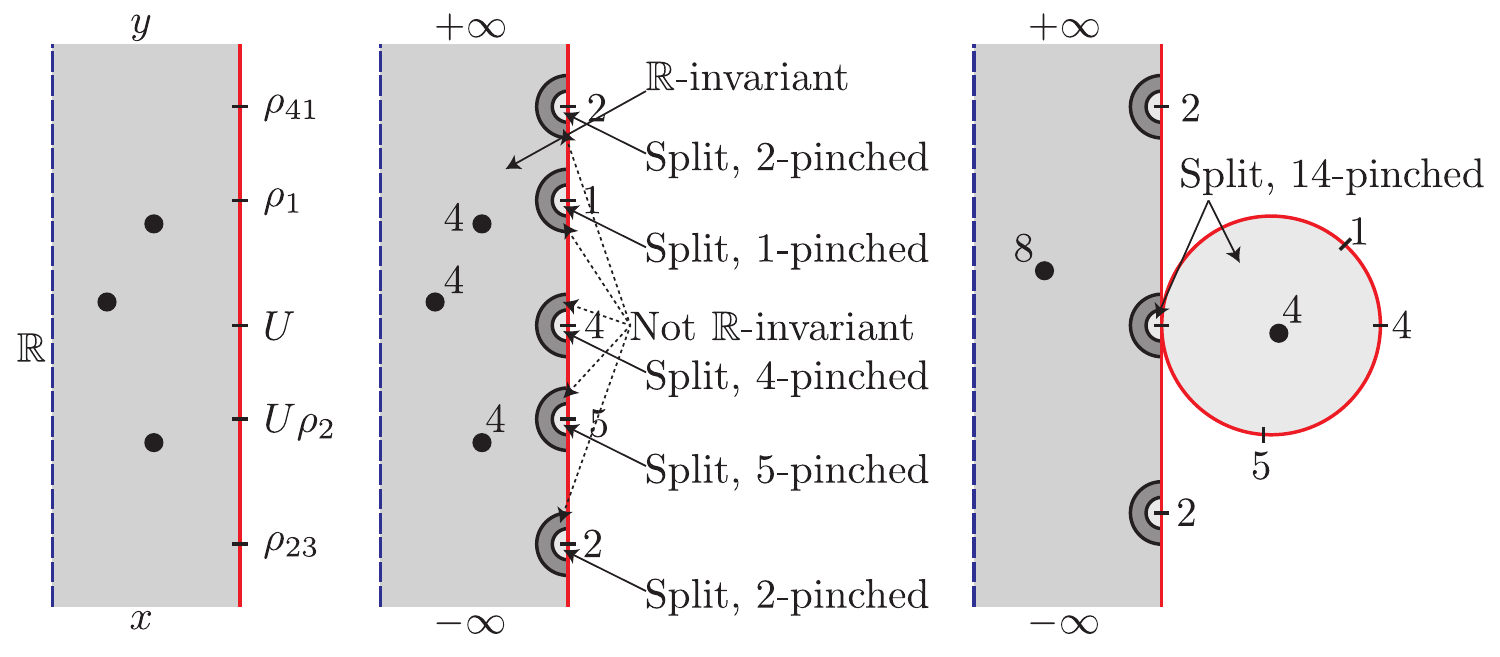}
  \caption[Polygons and almost complex structures]{\textbf{Polygons and almost complex structures.} Left: a moduli space of curves we will consider, corresponding to a term $m_{6}^3(x,\rho_{23},U\rho_{2},U,\rho_1,\rho_{41})=y$. Center: a corresponding module polygon and the conditions on an $\eta$-admissible almost complex structure for that module polygon. In the figure, $n$-pinched means $\eta(n)$-pinched. Right: a degeneration of such a module polygon, where two interior marked points come together and a disk bubbles off the boundary, together with the conditions imposed on a corresponding almost complex structure.}
  \label{fig:poly-cx-str}
\end{figure}

\begin{figure}
  \centering
  \includegraphics{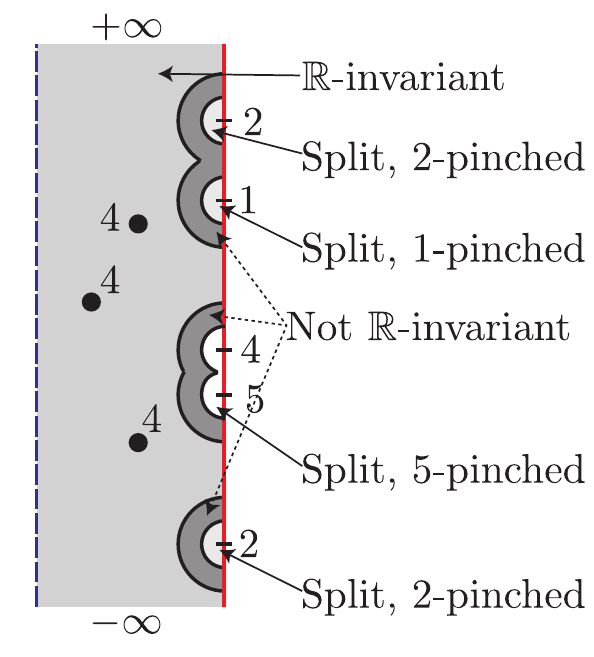}
  \caption[Polygons and almost complex structures, continued]{\textbf{Polygons
      and almost complex structures, continued.} If $\delta$ is larger
    than a quarter the distance between two of the algebra marked points,
    the various disks in Condition~\ref{NearlyConstant} are not disjoint.}
  \label{fig:poly-cx-str-overlap}
\end{figure}

A second ingredient to specify the complex structures we will consider
is a pinching function and a pinched almost complex structure. (The
two notions come together in Definition~\ref{def:ac}.)
\begin{definition}\label{def:pinching-fn}
  A \emph{pinching function} is a monotone-decreasing function
  $\eta\co \ZZ_{>0}\to \RR_{>0}$.
\end{definition}

\begin{definition}
  \label{def:ePinched}
  Fix, once and for all, a small positive number $\delta_0$.  Given another
  real number $\epsilon>0$, call a complex structure $j$ on $\Sigma$
  \emph{$\epsilon$-pinched} if:
  \begin{enumerate}
  \item There is a circle $C\subset \Sigma\setminus\alphas$ separating
    $\alpha_1^a\cup \alpha_2^a$ from
    $\alpha_1^c\cup\cdots\cup\alpha_{g-1}^c$ which has a neighborhood
    disjoint from the $\alpha$-curves biholomorphic to
    $\{z\in\CC\mid 1<|z|<1/\epsilon\}$.
  \item There is no essential circle in the component of $\Sigma\setminus C$
    containing $\alpha_1^a\cup \alpha_2^a$ which has a neighborhood
    biholomorphic to $\{z\in\CC\mid 1<|z|<1/\delta_0\}$.
  \end{enumerate}
\end{definition}
(Later, we will sometimes consider limits as $\epsilon\to 0$. The
second bullet point is to ensure that no other curves in this
component of $\Sigma\setminus C$ collapse in the result.)

Since we are pinching a specified isotopy class of curves $C$, the
space of $\epsilon$-pinched complex structures on $\Sigma$ is
contractible.

We are now ready to turn to the conditions on our almost complex
structures.  

\begin{definition}\label{def:ac}
  Fix a pinching function $\eta$ and a bimodule component $P$
  with $n$ right algebra
  vertices. Recall that there is an identification of $P$ with
  $[0,1]\times\RR$, well-defined up to translation. An almost complex
  structure $J$ on $\Sigma\times P$ is \emph{$\eta$-admissible} if it
  satisfies the following conditions.
  \begin{enumerate}[label=(J-\arabic*)]
  \item\label{item:J-piD} The projection map
    $\pi_P\co \Sigma\times P \to P$ is
    $J$-holomorphic.
  \item\label{item:J-s-t} Thinking of $P$ as $[0,1]\times\RR$, if
    $\partial/\partial s$ and $\partial/\partial t$ are the vector
    fields generating translation in $[0,1]$ and $\RR$, respectively,
    $J(\partial/\partial s)=\partial/\partial t$.
  \item\label{item:J-split} There is a neighborhood $D$ of the
    puncture in $\Sigma$ so that the almost complex structure $J$ splits as
    $j_{\Sigma}\times j_{P}$ over $D\times P\subset \Sigma\times
    P$. Also, $J$ splits as $j_{\Sigma}\times j_{P}$ over some
    neighborhood of $\{0\}\times\RR$.
  \item \label{NearlyConstant}Let $(1,t_1),\dots,(1,t_n)$ be the
    images of the right algebra vertices under some holomorphic identification
    of $P$ with $[0,1]\times\RR$. (The $t_i$ are well-defined up to an
    overall translation.) Then, there is some $0<\delta<1/2$ so that:
    \begin{enumerate}[label=(\alph*), ref=\ref{NearlyConstant}(\alph*)]
    \item\label{item:J-const-away-punct} outside the balls of radius
      $2\delta$ around the $(1,t_i)$, $J$ is
      invariant under (small) translations in $\RR$, and
    \item\label{item:J-const-at-punct} inside the ball of radius $\delta$ around $(1,t_i)$, $J$ is
      split as $j_i\times j_P$, for some complex structure $j_i$ on $\Sigma$.
    \end{enumerate}
  \item \label{PinchedOverVertex} If $E_i$ is the energy of the $i\th$
    marked point of $P$ then $j_i$ is $\eta(E_i)$-pinched.
  \end{enumerate}

  We extend this definition to the boundary of $\ModPol$ as
  follows. Consider a boundary point $P$ of $\ModPol$, consisting of
  some bimodule components and some algebra components. By an
  $\eta$-admissible almost complex structure $J$ over such a
  configuration we mean an $\eta$-admissible almost complex structure
  in the sense above over each of the bimodule components, so
  that these agree near $\pm\infty$ in the obvious sense (using
  Condition~\ref{item:J-const-away-punct}).
  The almost complex structures on the bimodule components determine an
  almost complex structure on $\Sigma\times P_0$ for each algebra-type
  component
   $P_0$ of $P$ by condition~\ref{item:J-const-at-punct}: we
  simply require the complex structure to be split as
  $j_i\times j_{P_0}$, if $P_0$ is attached at the $i\th$ marked point
  on the right, or as $j_\Sigma\times j_{P_0}$ if $P_0$ is attached at a
  marked point on the left.
\end{definition}

In particular, Conditions~\ref{item:J-piD} and~\ref{item:J-s-t} imply
that we can (and often will) view $J(P)$ as a family of complex
structures on $\Sigma$, parameterized by
$P$. Condition~\ref{item:J-const-away-punct} implies that $J$ has a
well-defined asymptotic behavior at $\pm\infty$: near $\pm\infty$, $J$
is specified by a 1-parameter family of almost complex structures
$j_s$, $s\in[0,1]$, on $\Sigma$. (By condition~\ref{item:J-split},
$j_s = j_\Sigma$ near $s=0$.)

Figures~\ref{fig:poly-cx-str} and~\ref{fig:poly-cx-str-overlap} give schematics of what
Conditions~\ref{NearlyConstant} and~\ref{PinchedOverVertex} require.

We make the space of $\eta$-admissible almost complex structures into
a bundle over $\ModPol$ as follows. Over the top stratum of
$\ModPol$, the definition of this bundle is clear, using the
$C^\infty$ topology on the space of endomorphisms of
$T(\Sigma\times [0,1]\times\RR)$, say. Next, suppose $P$ is a boundary
point of $\ModPol$, and that $\glueop(P)$ is obtained by gluing together
the components of $P$ (with small gluing parameter $\epsilon$, say).
Because the
almost complex structures are split near the marked points, the
$\eta$-admissible almost complex structure over $P$ induces an
$\eta$-admissible almost complex structure $\glueop(J)$ over
$\glueop(P)$. Then, for
each triple of positive real numbers $\delta$, $\delta'$, and
$\epsilon$ and open
neighborhood $V$ of $J$, we declare
there to be a basic open neighborhood $V$ of $J$ consisting of those $J''$
which agree with $\glueop(J')$ on a disk of radius $\delta$ around the
algebra
marked points and are within distance $\delta'$ of $\glueop(J')$ in the
$C^\infty$ topology everywhere, for some $J'\in V$ and gluing
parameter $\epsilon'<\epsilon$. (Here, as in Definition~\ref{def:ac},
distances are measured with respect to the identification of $\glueop(P)$
and the bimodule components of $P$ with $[0,1]\times\RR$.)

\begin{lemma}\label{lem:admis-contract}
  The fibers of the projection from the space of $\eta$-admissible
  almost complex structures projects to $\ModPol$ are contractible.
\end{lemma}
\begin{proof}
  Fix a point $P$ in $\ModPol$; without loss of generality, we may
  assume that $P$ is in the interior of $\ModPol$, i.e., consists of a
  single bimodule component. Let $\mathcal{J}$ denote the space of
  almost complex structures as in Definition~\ref{def:ac} for $P$.
  
  Fix a single curve $C$ representing the isotopy class where the
  pinching occurs. There is a subspace of $\mathcal{J}_0$ of
  $\mathcal{J}$ where the pinching occurs along $C$. The inclusion
  $\mathcal{J}_0\into \mathcal{J}$ is a homotopy equivalence, so it
  suffices to show that $\mathcal{J}_0$ is contractible.

  A key notion we will use is that of a collar. Given a number
  $\delta\in(0,1/2)$, let $F_\delta$ be the union of the closed balls
  of radius $\delta$ around the $t_i$ and $G_\delta$ the complement of
  the union of the open balls of radius $2\delta$ around the $t_i$. By
  construction, $F_\delta$ and $G_\delta$ are closed and at distance
  $d(F_\delta,G_\delta)=\delta$ from each other. By a
  \emph{$(\delta_1,\delta_2)$-collar} we mean a union of strips in
  $([0,1]\times\RR)\setminus(F_{\delta_1}\cup G_{\delta_2})$
  separating $F_{\delta_1}$ from $G_{\delta_2}$. Such a collar exists
  as long as $F_{\delta_1}\cap G_{\delta_2}=\emptyset$.

  Let $E=\max\{E_1,\dots,E_n\}$. 
  
  There is a subspace $\mathcal{J}_1\subset \mathcal{J}_0$ so that on
  each disk around $(1,t_i)$, $J$ is $\eta(E)$-pinched rather than
  just $\eta(E_i)$-pinched as in
  Condition~\ref{PinchedOverVertex}. That is, $J$ is $\eta(E)$-pinched
  on all of $F_\delta$. The inclusion
  $\mathcal{J}_1\into \mathcal{J}_0$ is a homotopy equivalence. (One
  can see this by deforming $\mathcal{J}_0$ into $\mathcal{J}_1$ by
  pinching more along $C$, using a collar.)

  For each $\delta\in(0,1/2)$, let $\mathcal{J}_{1,\delta}$ be the
  subspace of $\mathcal{J}_{1}$ defined using the parameter
  $\delta$. Contractibility of the space of almost complex structures
  on $\Sigma$ compatible with a given area form implies that each
  $\mathcal{J}_{1,\delta}$ is contractible.

  Finally, to see that $\mathcal{J}_{1}$ is contractible, suppose that
  $\delta_1$ and $\delta_2$ are close enough that there is a
  $(\delta_1,\delta_2)$-collar. Fix such a collar. Then the subspace
  of $\mathcal{J}_1$ of almost complex structures which vary only over
  the collar (i.e., satisfy either
  Condition~\ref{item:J-const-away-punct}
  or~\ref{item:J-const-at-punct} outside the collar) is contractible
  (for the same reason $\mathcal{J}_{1,\delta}$ is) and includes into
  $\mathcal{J}_{1,\delta_1}$, $\mathcal{J}_{1,\delta_2}$, and
  $\mathcal{J}_{1,\delta_1}\cap \mathcal{J}_{1,\delta_2}$ as homotopy
  equivalences. So, it follows by an inductive argument that the
  union $\mathcal{J}_1=\bigcup_\delta\mathcal{J}_{1,\delta}$ is also
  contractible.
\end{proof}
We will often shorthand Lemma~\ref{lem:admis-contract} by saying that
the space of $\eta$-admissible almost complex structures is
contractible.

\begin{definition}
  \label{def:AdmissibleJs}
  Fix a pinching function $\eta$.  A \emph{coherent family of
    $\eta$-admissible almost complex structures} is a continuous
  section of the bundle of $\eta$-admissible almost complex structures
  (over $\ModPol$) satisfying the following compatibility
  conditions:
  \begin{enumerate}[label=(\arabic*)]
  \item\label{item:admis-at-infty} All of the almost complex structures agree near $\pm\infty$,
    i.e., they all correspond to the same 1-parameter family of
    complex structures $j_s$, $s\in[0,1]$, on $\Sigma$. We will
    sometimes write $J_\infty$ for the $\RR$-invariant almost complex
    structure on $\Sigma\times[0,1]\times\RR$ corresponding to this
    path $j_s$.
  \item\label{item:admis-Pprime} If $P$ is in the boundary of $\ModPol$ and $P'$ is a bimodule
    component of $P$, then the restriction of $J(P)$ to $P'$
    agrees with $J(P')$.
  \end{enumerate}
\end{definition}
Note that we do not require any compatibility on the algebra-type
components. (It will turn out that the counts of holomorphic curves
corresponding to these components are independent of the complex
structure, by Theorem~\ref{thm:AlgOfSurface}.)

\begin{lemma}\label{lem:admis-J-exists}
  For any pinching function $\eta$, a coherent family $J$ of
  $\eta$-admissible almost complex structures exists.
\end{lemma}
\begin{proof}
  The proof is an induction on the total energy.
  First, fix any
  translation-invariant almost complex structure $J_\infty$ on
  $\Sigma\times[0,1]\times\RR$. We will require that all of the
  $\eta$-admissible almost complex structures $J(P)$ agree with  $J_\infty$
  outside a neighborhood of their algebra punctures.
  Suppose we have constructed $J(P)$ for all polygons $P$ with energy
  $<E_0$ and consider the space of polygons with energy $E_0$. Fix
  some $\eta(E_0)$-pinched almost complex structure on $\Sigma$; we
  will require that almost complex structures agree with this one near
  punctures with energy $E_0$. (This only occurs when the polygon has
  a single bimodule component and that component has a single algebra
  puncture, with energy $E_0$, and no interior punctures.) In
  particular, if $P$ has no algebra
  punctures and no punctures, $J(P)$ is given by $J_\infty$.
  
  The family $J$ is already defined on the subspace of $\ModPol_{E_0}$
  consisting of trees of disks where at least two of the components
  are bimodule components, because each of these components is in
  $\ModPol_E$ for some $E<E_0$ and on all algebra-type components
  $J$ is required to be constant.  We extend $J$ to the strata where
  there is a single bimodule polygon inductively on the
  number of interior punctures and, for each number of interior
  punctures, inductively on the number of boundary punctures. Suppose
  we have constructed $J$ for all bimodule components with $<m$ interior
  punctures and with $m$ interior punctures and $<n$ boundary
  punctures. We will discuss the induction increasing $n$; increasing
  $m$ is similar. (Though it is not needed here, note that $n$
  and $m$ are bounded in terms of $E_0$.)

  Each point $P$ in the boundary of the subspace of $\ModPol_{E_0}$
  with $m$ interior punctures and $n$ boundary punctures is a tree of
  bimodule and algebra-type components where each bimodule component is earlier
  in the induction. So, we have already defined the almost
  complex structure on $\Sigma\times P$ or, equivalently, the family
  of complex structures on $\Sigma$ parameterized by $P$.
  Further, if there is more than one bimodule
  component in $P$, the chosen almost complex structures for those
  components agree with the ones chosen before for $E<E_0$. A
  neighborhood of $P$ is given by gluing together the components of
  $P$ at the nodes, with various parameters. Given some $\delta>0$ and
  a gluing $\tilde{P}$ of $P$ small enough that all punctures on the
  same algebra-type component of $P$ lie within $\delta$ of each
  other in $\tilde{P}$, define the complex structures in
  $\delta$-neighborhoods of the punctures of $\tilde{P}$ to agree with
  the corresponding complex structure at the puncture of $P$, and
  choose interpolating complex structures for the regions between
  $\delta$-neighborhoods and $2\delta$-neighborhoods. Making these
  choices consistently across the different boundary strata defines
  $J$ in a neighborhood of the boundary of $\ModPol_{E_0}$. (Spelling
  out how to make such consistent choices takes some work, but is
  essentially the same as making consistent choices of strip-like ends
  in Floer theory~\cite[Section (9g)]{SeidelBook}.) Contractibility of
  the space of
  $\eta$-admissible almost complex structures guarantees we
  can then extend this family $J$ to all of $\ModPol_{E_0}$.
\end{proof}

To relate the above spaces of almost complex structures and the sequences of algebra elements used in the definitions of the $A_\infty$-operations, we 
use the notion of the energy of a sequence of algebra elements.
Recall that a \emph{basic algebra element} is an
algebra element of the form $U^i\rho^j$ or $U^i\iota_j$, with
those the first kind called
Reeby elements. Define
the \emph{energy} of a basic algebra element to be $E(U^i\rho^j)=4i+|\rho^j|$
and $E(U^i\iota_j)=4i$, where $|\rho^j|$ denotes the length of
$\rho^j$. More generally, given a sequence of basic algebra elements
$\vec{a}$ and a non-negative integer $w$, define the \emph{energy} of
$(\vec{a},w)$ to be
\begin{equation}
  \label{eq:ExtendEnergy}
  E(\vec{a},w)=4w+\sum E(a_j).
\end{equation}
The definition is chosen so that
\begin{equation}
  \label{eq:EisMuInvariant}
  E(\mu^w_k(a_1,\dots,a_k),0)=E(a_1,\dots,a_k,w),
\end{equation}
in the sense that each basic term in the sum on the left-hand side has the
indicated energy.

Now, given a sequence $(a_1,\dots,a_m)$ of basic algebra elements and
a non-negative integer $w$, there is a corresponding component of
$\ModPol$ where the main stratum has $w$ interior marked points of
weight $4$, no boundary marked points on $\{0\}\times\RR$, and $m$ boundary marked points on $\{1\}\times\RR$ labeled by the energies
$E(a_1),\dots,E(a_m)$ (in that order). Let
$\ModPol(a_1,\dots,a_m;w)$ denote this component. Given an
$\eta$-admissible family $J$ of almost complex structures, define
\[
  J(a_1,\dots,a_m;w)
\]
to be the restriction of $J$ to $\ModPol(a_1,\dots,a_m;w)$. That
is, $J(a_1,\dots,a_m;w)$ consists of an almost complex structure on
$\Sigma\times P$ for each polygon
$P\in \ModPol(a_1,\dots,a_m;w)$. See Figure~\ref{fig:poly-cx-str}.

\subsubsection{Maps to \texorpdfstring{$\Sigma\times[0,1]\times\RR$}{the main part}}\label{sec:main-part}

Fix a sequence of basic algebra elements $\vec{a}=(a_1,\dots,a_m)$ in
the sense of Definition~\ref{def:Basic}.  Our moduli spaces will
depend on this sequence: the underlying chord sequence (in the sense
of Definition~\ref{def:Basic}) specifies the asymptotics of the
holomorphic curves, while the whole sequence of basic algebra
elements will specify a family of almost complex structures as in
Section~\ref{sec:J}.

Fix a pinching function $\eta$ and a coherent family $J$ of
$\eta$-admissible almost complex structures.

As we did for $\HFa$, we will encode the asymptotics of our
holomorphic curves in terms of decorations of the source:

\begin{definition}\label{def:decorated-source}
  A \emph{decorated source} $\Source$ is a smooth surface $S$ with
  boundary, interior punctures, and boundary punctures, together with:
  \begin{enumerate}
  \item A labeling of each interior puncture by a positive integer,
    its \emph{multiplicity}, and
  \item A labeling of each boundary puncture by either $-\infty$,
    $+\infty$, or a basic algebra element $a$.
  \end{enumerate}
\end{definition}
We will refer to the boundary punctures labeled by algebra elements as
\emph{algebra punctures}. Sometimes it will be more convenient to talk
about the \emph{ramification} of an interior puncture, which is one
less than its multiplicity. (The terminology is the same as for zeroes
of holomorphic maps.)

A \emph{multiplicity $r$ Reeb orbit} is a map $S^1\to \partial
\Sigma\times[0,1]\times\RR$ which maps to $[0,1]\times\RR$ by a
constant map and which winds around $\partial\Sigma$ $r$ times.

The following is the analogue of~\cite[Definition 5.2]{LOT1}:
\begin{definition}\label{def:respectful}
  Let $\Source$ be a decorated source. We say a smooth map
  \begin{equation}
  \label{eq:u-source-targ}
    u\co (S,\bdy S) \to (\Sigma\times[0,1]\times\RR,\alphas\times\{1\}\times\RR\cup\betas\times\{0\}\times\RR)
  \end{equation}
  \emph{respects $\Source$} if $u$ is
  proper and:
  \begin{itemize}
  \item at each interior puncture with multiplicity $r$, $u$ is
    asymptotic to a multiplicity $r$ Reeb orbit;
  \item at each boundary puncture labeled $-\infty$, $u$ is asymptotic
    to $x\times[0,1]\times\{-\infty\}$ for some
    $x\in\alphas\cap\betas$;
  \item at each boundary puncture labeled $+\infty$, $u$ is asymptotic
    to $y\times[0,1]\times\{+\infty\}$ for some
    $y\in\alphas\cap\betas$;
  \item at each boundary puncture labeled $U^m\rho$, $u$ is asymptotic
    to $\rho\times\{(1,t)\}$ for some $t\in\RR$; and
  \item at each boundary puncture labeled $U^m\iota_j$, $u$ is
    asymptotic to a point $(x,1,t)$ where $x$ lies in the interior of
    the $\alpha$-arc corresponding to $\iota_j$. (In particular, these last
    punctures are removable singularities.) 
  \end{itemize}
\end{definition}
Given a map $u$ as in Definition~\ref{def:respectful}, where $\Source$
has $n$ boundary punctures and $w$ interior punctures, $n+2w\geq 3$,
projection to $[0,1]\times\RR$ specifies a module polygon
$P(u)\in \ModPol$ with $n$ vertices, the images of the punctures (as
long as those images are distinct). The
family $J$ then specifies a corresponding almost complex structure on
$\Sigma\times[0,1]\times\RR$, $J(P(u))$. In the special case that
$n=2$ and $w=0$, $P(u)$ would be a bigon with no interior punctures,
which we disallowed; in this case, let $J(P(u))$ be $J_\infty$, the
$\RR$-invariant almost complex structure that all the $J(P)$ agree
with near $\pm\infty$.

The following is the analogue of~\cite[Definition 5.3]{LOT1}:
\begin{definition}\label{def:moduli-fixed-source}
  Given generators $\x,\y\in\Gen(\HD)$, a homology class
  $B\in\pi_2(\x,\y)$, and a decorated source $\Source$, let
  $\tcM^B(\x,\y;\Source)$ be the set of maps $u$ as in Formula~\eqref{eq:u-source-targ}
  so that:
  \begin{enumerate}[label=(M-\arabic*),series=moduli]
  \item The map $u$ is $(j,J(P(u))$-holomorphic for some almost complex
    structure $j$ on $S$.
  \item The map $u$ respects $\Source$.
  \item\label{item:bdy-mon} For each $t\in\RR$ and each $i$, $u^{-1}(\beta_i\times\{(0,t)\})$ consists
    of exactly one point, $u^{-1}(\alpha^c_i\times\{(1,t)\})$ consists
    of exactly one point, and  $u^{-1}(\alpha^a_i\times\{(1,t)\})$
    consists of at most one point.
  \end{enumerate}
  Let $\cM^B(\x,\y;\Source)$ be the quotient of
  $\tcM^B(\x,\y;\Source)$ by the action of $\RR$ by translation.
\end{definition}
Condition~\ref{item:bdy-mon} is called \emph{strong boundary
  monotonicity}. In particular, this condition implies that if
$\cM^B(\x,\y;\Source)\neq\emptyset$ then the algebra
punctures, and
hence the basic algebra elements labeling them, are ordered. (In
particular, the images in $\{1\}\times\RR$ of the algebra punctures
are distinct, so $P(u)$ is defined.) Specifically, if we
fill in the punctures of $S$ to obtain a surface $\overline{S}$ then
there must be an arc in $\bdy \overline{S}$ starting at a (filled-in)
$-\infty$ puncture and ending at a (filled-in) $+\infty$-puncture, and
containing all the algebra punctures. The boundary orientation
orients that arc, and hence orders the algebra punctures. Thus,
$\Source$ specifies a list $(a_1,\dots,a_k)$ of basic algebra elements.

\begin{definition}\label{def:moduli-embedded}
  Given generators $\x,\y\in\Gen(\HD)$,
  a sequence of basic algebra elements $\vec{a}$,
  a homology class $B\in\pi_2(\x,\y)$, and a non-negative integer $r$,
  let $\cM^B(\x,\y;\vec{a};w,r)$ be the
  union over all decorated sources $\Source$ of the set of
  $u\in\cM^B(\x,\y;\Source)$ so that:
  \begin{enumerate}[resume*=moduli]
  \item The map $u$ is an embedding.
  \item If we list the algebra punctures of $\Source$ according to
    the $\RR$-coordinates $u$ sends them to, the corresponding
    sequence of algebra elements is $\vec{a}$.
  \item The source $\Source$ has $w$ interior punctures with total
    ramification $r$.
  \end{enumerate}
  Let $\cM^B(\x,\y;\vec{a};w)=\cM^B(\x,\y;\vec{a};w,0)$.
\end{definition}
The spaces $\cM^B(\x,\y;\vec{a};w)$ will be used to define the modules
$\CFAmb(\HD)$ and $\CFAm(\HD)$ and, indirectly, $\CFDm(\HD)$.

\subsubsection{Maps to east infinity}\label{sec:e-infty}

Here we discuss the relevant holomorphic curves in
$\RR\times\bdy\Sigma\times[0,1]\times\RR$. The curves which appear are
essentially the same as in the case of $\HFa$~\cite{LOT1}, except that
chords can be longer, curves can also be asymptotic to orbits, and
punctures labeled by $U^m\iota_j$ can collide with other punctures.

To shorten notation, write $Z=\partial \Sigma$. The manifold
$\RR\times Z\times[0,1]\times\RR$ is non-compact in several
directions; by \emph{far east $\infty$} we mean
$\{\infty\}\times Z\times[0,1]\times\RR$, and by \emph{west infinity}
we mean $\{-\infty\}\times Z\times[0,1]\times\RR$.

\begin{definition}\label{def:bi-decorated-source}
  A \emph{bi-decorated source} $\biSource$ is a smooth surface $T$ with
  boundary, interior punctures, boundary punctures, and boundary
  marked points, together with:
  \begin{enumerate}
  \item A labeling of each puncture by either $w\infty$ or $e\infty$,
  \item A labeling of each interior puncture by a positive integer,
    its \emph{multiplicity},
  \item A labeling of each boundary puncture by some Reeby algebra
    element, and 
  \item A labeling of each boundary marked point by some non-Reeby basic algebra
    element (i.e., of the form $U^m\iota_j$).
  \end{enumerate}
  We require that the sum of the $U$-powers labeling the $w\infty$
  punctures is equal to the sum of the $U$-powers labeling the
  $e\infty$ punctures plus the sum of the $U$-powers labeling the
  boundary marked points.
\end{definition}

\begin{definition}\label{def:e-inf-mod-space}
  Given a bi-decorated source $\biSource$, a smooth map
  \[
    u\co (T,\bdy T)\to (\RR\times Z\times [0,1]\times\RR,\RR\times
    \mathbf{a}\times \{1\}\times\RR)
  \]
  \emph{respects $\biSource$} if $u$ is proper and:
  \begin{itemize}
  \item The map $\pi_\bD\circ u\co T\to [0,1]\times\RR$ is constant, where
    $\pi_\bD$ is projection from $\RR\times Z\times [0,1]\times\RR$ to
    the last two factors,
  \item At each interior puncture with multiplicity $r$, $u$ is
    asymptotic to a multiplicity $r$ Reeb orbit, at $\{-\infty\}\times
    Z$ if the puncture is labeled $w\infty$ and at $\{\infty\}\times Z$ if
    the puncture is labeled $e\infty$, 
  \item At each boundary puncture labeled $(e\infty,U^m\rho)$, $u$ is
    asymptotic to $\{\infty\}\times \rho\times\{(1,t)\}$, while at
    each boundary puncture labeled $(w\infty,U^m\rho)$, $u$ is asymptotic
    to $\{-\infty\}\times \rho\times\{(1,t)\}$ (for some $t\in\RR$), and
  \item Each boundary marked point labeled by $U^m\iota_j$ is mapped
    to a point on an element of $\mathbf{a}$ corresponding to $\iota_j$.
  \end{itemize}

  Let $\tcN(\biSource)$ denote the space of holomorphic maps $u$ which
  respect $\biSource$, and
  $\cN(\biSource)=\tcN(\biSource)/(\RR\times\RR)$ denote the quotient
  of $\tcN(\biSource)$ by translation in the two $\RR$ factors.
\end{definition}
Note that a map $u$ as in Definition~\ref{def:e-inf-mod-space} is
holomorphic if and only if its projection to $\RR\times Z$ is
holomorphic.
Also, the sum of the energies of the
$w\infty$ (boundary and interior) punctures is equal to the
sum of the energies of the
$e\infty$ punctures and the boundary marked points.

We will see later that, in codimension $1$, only four kinds of curves
at $e\infty$ show up:
\begin{enumerate}
\item \emph{Split curves}, as defined in~\cite{LOT1}. This is the case
  that $T$ consists of a single disk with three boundary punctures,
  two labeled $e\infty$ and one labeled $w\infty$, and no interior
  punctures. If the
  $e\infty$ boundary punctures are labeled $a$ and $b$ then
  the $w\infty$ puncture by the product $ab$. Reading
  counterclockwise around the boundary, the cyclic ordering of these
  punctures is $(a,b,ab)$. See Figure~\ref{fig:splits}, left.
\item \emph{Join curves}, also as defined in~\cite{LOT1}. This is the case
  that $T$ consists of a single disk with three boundary punctures,
  two labeled $w\infty$ and one labeled $e\infty$, and no interior punctures. The
  $w\infty$ boundary punctures are labeled $a$ and $b$, and
  the $e\infty$ puncture by the concatenation $ab$. Reading
  counterclockwise around the boundary, the cyclic ordering of these
  punctures is $(b,a,ab)$. Further,
  these occur only as parts of composite boundary degenerations
  (described below).
\item \emph{Orbit curves}, which are new. The surface $T$ is a disk
  with one boundary puncture, labeled by $w\infty$ and some length-4
  Reeb chord, and one interior puncture labeled $e\infty$ and
  multiplicity $1$. The corresponding map to $\RR\times Z$ is a
  degree-1 branched cover, with a single boundary branch point. (The
  whole boundary is mapped to a ray on one component of
  $\RR\times\mathbf{a}$, going out to $-\infty$, i.e., west infinity.)
  See Figure~\ref{fig:OrbitCurve}.
\item \emph{Pseudo split curves}, consisting of a single disk with two
  boundary punctures, one labeled $(e\infty, U^m\rho)$ and the other
  $(w\infty, U^{m+n}\rho)$, and a boundary marked point labeled
  $U^n\iota_j$; see Figure~\ref{fig:splits}, center. (These correspond to a puncture labeled by $U^m\rho$
  colliding with a puncture labeled $U^n\iota_j$.)
\end{enumerate}

\begin{figure}
\input{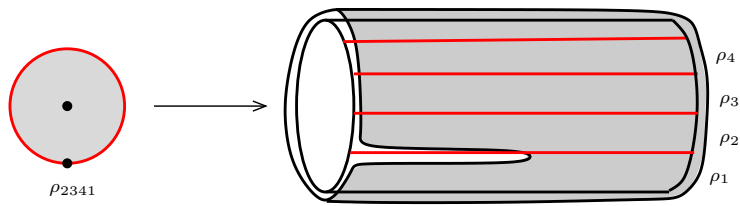}
\caption[An orbit curve]{\label{fig:OrbitCurve} {\bf{An orbit curve}.}  The source is
  drawn on the left; the map to $\RR\times Z$ is indicated on the
  right.}
\end{figure}

\begin{figure}
  \centering
  \includegraphics{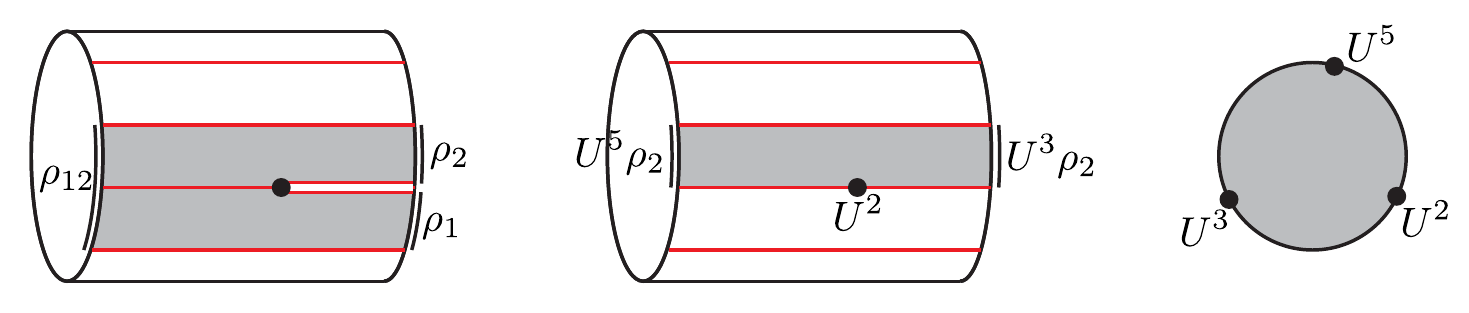}
  \caption[A split curve, a pseudo split curve, and a fake
      split curve]{\textbf{A split curve, a pseudo split curve, and a fake
      split curve.} Left: the image of a split curve in $\RR\times Z$,
    with boundary $(a,b,ab)=(\rho_1,\rho_2,\rho_{12})$. Center: the
    image of a pseudo-split curve in $\RR\times Z$. Right: the source
    of a fake split curve; this curve is mapped to
    $\Sigma\times[0,1]\times\RR$ by a constant map.}
  \label{fig:splits}
\end{figure}

It will be convenient to view one more kind of curve as a kind of
degenerate case of an $e\infty$ curve:
\begin{enumerate}[resume]
\item \emph{Fake split curves} consist of a disk with three
  boundary marked points (or punctures), labeled $U^m\iota_j$,
  $U^n\iota_j$, and $U^{m+n}\iota_j$, mapped to
  $\Sigma\times[0,1]\times\RR$ by a constant map to some point
  $(x,1,t)$ where $x$ lies in the interior of the $\alpha$-arc
  corresponding to $\iota_j$. See Figure~\ref{fig:splits}, right.
\end{enumerate}

\subsubsection{Boundary degenerations}\label{sec:bdy-degen}
Let $\HHH$ denote the half plane
\begin{equation}
  \label{eq:DefH}
  \HHH=(-\infty,1]\times \RR=\{x+iy\mid x\leq 1\}\subset \CC,
\end{equation}
with $\bdy \HHH=\{1\}\times \RR$. Let $\overline{\HHH}$ denote the
one-point compactification of $\HHH$, which is just a disk.

As usual, we fix a complex structure $j$ on $\Sigma$.  Boundary
degenerations are holomorphic maps from a surface $S$ to the space
$\Sigma\times \HHH$, equipped with the product complex structure $J$,
which is $j$ on the $\Sigma$ factor and the usual complex structure
$\HHH$ on the second. (In Section~\ref{sec:alg-bdeg}, we consider
boundary degenerations using a more general class of almost complex
structure on $\Sigma\times\HHH$.)

We start by defining a general boundary degeneration, which is needed
for the compactness statement in Section~\ref{sec:compactness}. We then
specialize to the kinds of boundary degenerations that occur in
codimension $1$: \emph{simple boundary degenerations} and
\emph{composite boundary degenerations}.

\begin{definition}
  \label{def:BoundaryDegeneration-source}
  A \emph{boundary degeneration source} $\bdSource$ consists of:
  \begin{enumerate}[label=(\arabic*)]
  \item A surface-with-boundary $S$ with interior and boundary punctures,
  \item A labeling of each interior puncture by a positive integer
    (which will be the multiplicity of the Reeb orbit), and
  \item A labeling of each boundary puncture by a basic algebra element.
  \end{enumerate}
  We will call the interior punctures and the boundary punctures
  labeled by elements $U^m\rho$ \emph{east punctures} and the boundary
  punctures labeled by elements $U^m\iota_j$ \emph{removable
    punctures}.

  A map
    \[
    u \co (S,\bdy S)\to (\Sigma\times \overline{\HHH}, \alphas\times\bdy \overline{\HHH})
  \]
  \emph{respects $\bdSource$} if it satisfying the following
  conditions:
  \begin{enumerate}[label=(BD-\arabic*),ref=(BD-\arabic*)]
  \item At each interior puncture with multiplicity (label) $r$, $u$
    is asymptotic to a multiplicity $r$ Reeb orbit.
  \item At each boundary puncture $q$ labeled by $U^m\rho$,
    $u$ is asymptotic to $\rho\times\{(1,t_q)\}$ for some
    $t_q\in\RR\cup\{\infty\}$.
  \item At each boundary puncture $q$ labeled by $U^m\iota_j$, $u$ is
    asymptotic to $x\times\{(1,t_q)\}$ for some $x$ in the
    $\alpha$-arc corresponding to $q$ and some
    $t_q\in\RR\cup\{\infty\}$. 
  \item  Exactly one $t_q$ is equal to $\infty$. Denote this $q$ by $q_\infty$.
  \item\label{item:bd-monotone} For each $t\in\RR$ which is not one of
    the $t_q$ and each $i=1,\dots,g-1$,
    $\overline{u}^{-1}(\alpha_i^c\times(1,t))$ consists of exactly one
    point and $u^{-1}((\alpha_1^a\cup\alpha_2^a)\times(1,t))$ consists
    of exactly one point.  (This condition is called \emph{boundary
      monotonicity}.)
  \item\label{item:bd-U-weight} The energy of the label of $q_\infty$
    is the sum of the energies of the labels of the other punctures
    (including the interior punctures). (This is a constraint on the
    power of~$U$ in the label of $q_\infty$.)
  \end{enumerate}  
  Let $\tcN(*;\bdSource)$ be the moduli space of maps $u$ respecting
  $\bdSource$ which are $(j_S,J)$-holomorphic with respect to some complex
  structure $j_S$ on $S$. Let
  \(
    \cN(*;\bdSource)
  \)
  be the quotient of $\tcN(*,\bdSource)$ by translation and scaling
  on~$\HHH$ (i.e., $(x,s,t)\mapsto (x,a(s-1)+1,at+b)$ for
  $a\in\RR_{>0}$
  and $b\in\RR$).

  For each boundary puncture $q$ there is an evaluation map
  \[
    \ev_q\co \cN(*;\bdSource)\to (\alpha_1^a\cup\alpha_2^a\cup\{p\})\times\alpha_1^c\times\cdots\times\alpha_{g-1}^c
  \]
  defined by
  \[
    \ev_q(u) = (\pi_\Sigma\circ u)\bigl((\pi_\HHH\circ u)^{-1}(t_q)\bigr).
  \]
  We have $p\in \ev_q(u)$ if and only if $q$ is labeled by a Reeby
  element. The
  evaluation map corresponding to $q_\infty$ (the puncture mapped to
  $\infty$) will come up most frequently, so let
  \[
    \ev = \ev_{q_\infty}.
  \]
  Given $\x\in
  (\alpha_1^a\cup\alpha_2^a\cup\{p\})\times\alpha_1^c\times\cdots\times\alpha_{g-1}^c$,
  let 
  \[
    \cN(\x;\bdSource)=\ev^{-1}(\x)\subset \cN(*;\bdSource).
  \]
  If $\x$ has the form $\{p,y_1,\dots,y_{g-1}\}$, so $q_\infty$ is
  labeled $U^m\rho$, we will also write $\cN(\x;\bdSource)$ as
  $\cN(\rho\times\y;\bdSource)$ where $\y=\{y_1,\dots,y_{g-1}\}$.

  Given $\bdSource$ and $u\in \tcN(*;\bdSource)\neq\emptyset$, the
  orientation of $\bdy \HHH$ gives an ordering
  $\vec{a}(\bdSource)=(a;a_1,\dots,a_m)$ of the algebra elements
  marking $\bdSource$, where we declare that the algebra element $a$
  labeling $q_\infty$ comes first. By Condition~\ref{item:bd-monotone},
  this ordering is independent of the choice of $u$.
\end{definition}

In codimension-1, two types of configurations involving boundary degenerations
can occur. So, we name them:
\begin{definition}\label{def:BoundaryDegeneration-simple}
  A \emph{simple boundary degeneration source} is a boundary
  degeneration source so that $q_\infty$ is labeled $U^m\iota_j$ and
  no other boundary puncture is labeled $U^m\iota_j$. A
  \emph{(possibly non-embedded) simple boundary degeneration} is a
  holomorphic curve $u$ respecting a simple boundary degeneration
  source.

  Let $\x$ be a $g$-tuple of points in
  $(\alpha_1^a\cup\alpha_2^a)\times \alpha^c_1\times\dots\times
  \alpha^c_{g-1}$, $\vec{a}$ be a sequence of
  Reeby algebra elements, and $w$ and $r$ be non-negative integers.
  Let $\cN(\x;\vec{a};w,r)$ be the set of pairs
  $(\bdSource,u)$ where $u$ is a simple
  boundary degeneration with source $\bdSource$ such that 
 \begin{enumerate}[label=(BDs-\arabic*),ref=(BDs-\arabic*),series=BDs]
 \item The map $u$ is an embedding. 
  \item The source $\bdSource$ has $w$ interior punctures with total
    ramification $r$. (As above, the ramification of a puncture is one
    less than its multiplicity.)
  \item The sequence of algebra element asymptotics of $u$ is
    $(U^n;\vec{a})$ (for some appropriate $n$).
  \end{enumerate}
  We will be particularly interested in the case $r=0$, so let
  \[
    \cN(\x;\vec{a};w)=\cN(\x;\vec{a};w,0)
  \]
  We call an element of $\cN(\x;\vec{a};w)$ a
  \emph{simple boundary degeneration based at $\x$}.

  Sometimes it will be more convenient not to fix the point $\x$; let
  \begin{align*}
    \cN(*;\vec{a};w,r)&=\bigcup_\x\cN(\x;\vec{a};w,r)\\
    \cN(*;\vec{a};w)&=\bigcup_\x\cN(\x;\vec{a};w,0).
  \end{align*}
  After discussing transversality, we will also consider another space
  $\cN([\alpha_i^a];\rho^1,\dots,\rho^n;w)$ which is equal to
  $\cN(\x;\vec{\rho};w)$ for some generic $\x$ intersecting $\alpha_i^a$; see
  Section~\ref{sec:transversality}.
\end{definition}
See Figure~\ref{fig:bdy-degen-source}, left.

\begin{definition}\label{def:composite-bdy-degen}
  A \emph{composite boundary degeneration source} is a boundary
  degeneration source so that none of the punctures are labeled by
  elements of the form $U^m\iota_j$. A \emph{(possibly non-embedded)
    composite boundary degeneration} is a triple $(\bdSource, u,v)$
  where $\bdSource$ is a composite boundary degeneration source, $u$
  is a holomorphic curve respecting $\bdSource$, and $v$ is a join curve, such that
  the Reeb chord $u$ is asymptotic to at~$q_\infty$ is one of the
  $w\infty$ Reeb chords of $v$.

  Let $\x$ be a $(g-1)$-tuple of points in
  $\alpha^c_1\times\dots\times \alpha^c_{g-1}$, $\vec{a}$ be a
  sequence of Reeby basic algebra elements, $b$ be another Reeby basic algebra element, and $w$ and
  $r$ be non-negative integers.  Let
  $\cN(b;a_1,\dots,a_n;w,r)$ denote the moduli space of
  composite boundary degenerations
  $(\bdSource,u,v)$ such that
  \begin{enumerate}[label=(BDc-\arabic*),ref=(BDc-\arabic*)]
 \item The map $u$ is an embedding. 
  \item The source $\bdSource$ has $w$ interior punctures with total
    ramification $r$.
  \item\label{item:l-or-r} Either:
    \begin{enumerate}[label=(\alph*)]
    \item\label{item:right-ext} $u$ is asymptotic to $c$ at
      $\infty$ and to the sequence of Reeby basic algebra elements
      $(a_1,\dots,a_{n-1})$ along $\{1\}\times\RR$,
      $a_n=cb$ for some basic algebra element $c$, and $v$ is a join
      curve with asymptotics $(c,b)$ at $w\infty$ and $a_n$
      at $e\infty$, or
    \item\label{item:left-ext} $u$ is asymptotic to $c$ at $\infty$ and to the
      sequence of Reeby basic algebra elements $(a_2,\dots,a_{n})$ along
      $\{1\}\times\RR$, $a_1=bc$, and $v$ is a join
      curve with asymptotics $(b,c)$ at $w\infty$ and $a_1$
      at $e\infty$.
    \end{enumerate}
  \end{enumerate}
  An element of $\cN(b;a_1,\dots,a_n;w,r)$ is a
  \emph{composite boundary degeneration}. In Point~\ref{item:l-or-r},
  the first case is called \emph{right-extended} and the second case
  is called \emph{left-extended}.  As before, we are mostly interested
  in the case $r=0$ and will drop it from the notation in that case.
\end{definition}
See Figure~\ref{fig:bdy-degen-source}, right.
Since there is a unique join curve asymptotic to chords $\sigma$,
$\tau$, and $\sigma\tau$, the moduli space of composite boundary
degenerations with given Reeb chords is in bijection with the space of maps $u$ as in
Definition~\ref{def:composite-bdy-degen}.

\begin{figure}
  \centering
  \includegraphics{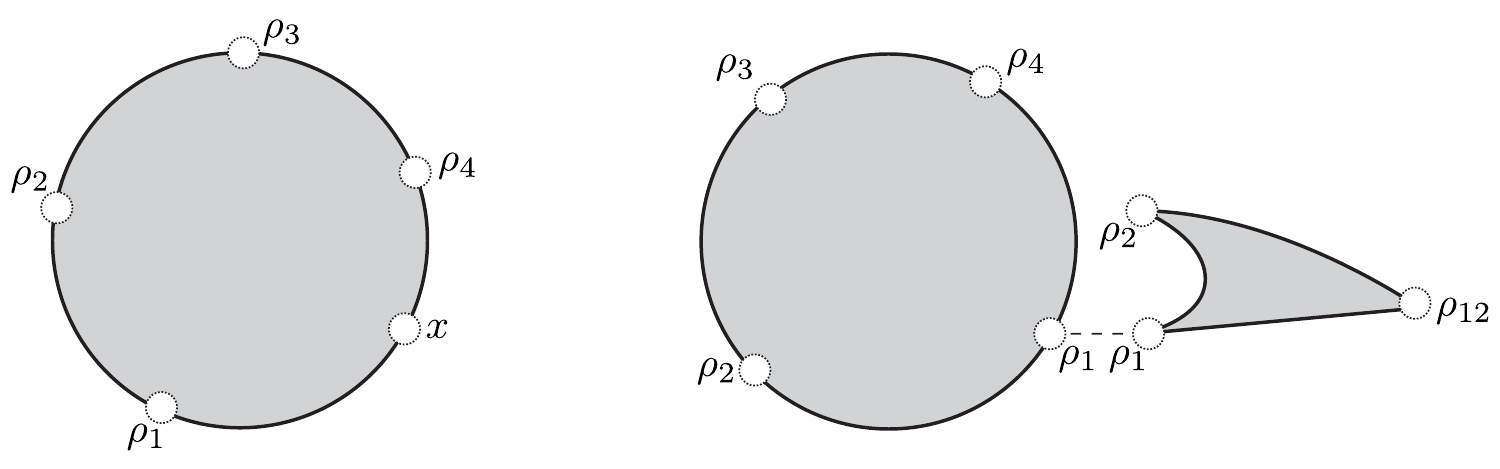}
  \caption[Decorated sources for boundary
  degenerations]{\textbf{Sources for boundary degenerations.} Left:
    the source for the simple boundary degeneration arising in
    Figure~\ref{fig:simple-bdy-degen-eg}. The moduli space of embedded
    curves here is denoted $\cN(\x;\rho_4,\rho_3,\rho_2,\rho_1;0,0)$
    where $\x$ is the point on the $\alpha$-arc labeled $x$ in
    Figure~\ref{fig:simple-bdy-degen-eg}.  Right: the source for the
    right-extended composite boundary degeneration arising in
    Figure~\ref{fig:split-degen-eg}.  This moduli space is denoted
    $\cN(\rho_2;\rho_4,\rho_3,\rho_2,\rho_{12};0,0)$. The join curve
    in Definition~\ref{def:composite-bdy-degen} has asymptotics
    $(\rho_1,\rho_2)$ at $w\infty$ and $\rho_{12}$ at $e\infty$. In
    the notation of Definition~\ref{def:composite-bdy-degen},
    $c=\rho_1$ and $b=\rho_2$.}
  \label{fig:bdy-degen-source}
\end{figure}

\begin{remark}
  We will see in Section~\ref{sec:algebra} that rigid, simple boundary
  degenerations (with inputs Reeb chords) correspond to centered
  operations on $\MAlg$, and hence also to centered tiling patterns
  from~\cite{LOT:torus-alg}. Composite boundary degenerations
  correspond to the left-extended and right-extended tiling patterns
  from~\cite{LOT:torus-alg}.
\end{remark}

\subsubsection{The tautological correspondence}\label{sec:tautological}
For computing the index, it will be convenient to formulate certain of
our holomorphic curves in terms of the symmetric
product. Specifically, fix a coherent family of $\eta$-admissible
almost complex structures, for some $\eta$. Consider a module polygon
$P$ with a single component, and a $J(P)$-holomorphic map
$u\co S\to \Sigma\times[0,1]\times\RR$ as in
Formula~\eqref{eq:u-source-targ} with $\pi_\bD\circ u$ a branched
cover (or, equivalently, non-constant on each component). We can equally well view $u$ as a
map to $\Sigma\times P$.  There is a corresponding map
$\phi\co P\to\Sym^g(\Sigma)$ defined by
\[
  \phi(x)=(\pi_{\Sigma}\circ u)\bigl((\pi_P\circ u)^{-1}(x)\bigr)
\]
where $\pi_\Sigma\co \Sigma\times P\to \Sigma$ and
$\pi_{P}\co \Sigma\times P\to P$ are the obvious projections. (The
set $(\pi_P\circ u)^{-1}(x)$ has $g$ elements, counted with
multiplicity.) By the removable singularities theorem (and
Condition~\ref{item:J-split}), we can also
view $\phi$ as a map from $\overline{P}$ to $\Sym^g(\overline{\Sigma})$, where
$\overline{\Sigma}$ is the result of filling in the algebra and
interior puncture of
$\Sigma$, and $\overline{P}$ is the result of filling in the
punctures of $P$. The map $\phi$
sends arcs on the boundary of the disk to the three Lagrangian tori
\begin{align*}
  T_\betas &= \beta_1\times\cdots\times \beta_g,\\
  T_{\alphas,1}&=\overline{\alpha_1^a}\times\alpha_1^c\times\cdots\times\alpha_{g-1}^c,\\
  T_{\alphas,2}&=\overline{\alpha_2^a}\times\alpha_1^c\times\cdots\times\alpha_{g-1}^c.
\end{align*}

By Condition~\ref{item:J-piD}, for each point $x\in P$, the almost
complex structure $J(P)$ restricts 
to a complex structure $j_x$ on $\Sigma=\pi_P^{-1}(x)$. The complex
structure $j_x\times\cdots\times j_x$ on $\overline{\Sigma}^g$ induces a complex
structure $\Sym^g(j_x)$ (and, in fact, a smooth structure) on
$\Sym^g(\overline{\Sigma})$, characterized by the property that the projection
$\Sigma^g\to \Sym^g(\overline{\Sigma})$ is holomorphic. This construction makes
$P\times \Sym^g(\overline{\Sigma})$ into a smooth fiber bundle over $P$, with an
almost complex structure induced by the $\Sym^g(j_x)$ and the complex
structure on $P$.

There are two ways that the map $\phi$ induced by $u$ might intersect
the diagonal. First, if $x\in P$ is a branch point of $\pi_P\circ u$
of ramification $r$ then $\phi(x)$ lies in the stratum of the diagonal with
$r$ repeated entries. Second, if there are two distinct points
$a,b\in (\pi_P\circ u)^{-1}(x)$ so that
$\pi_\Sigma(u(a))=\pi_\Sigma(u(b))$ then $\phi(x)$ lies in the
diagonal. The latter case, of course, corresponds to
self-intersections of $u$. By unique continuation of $J$-holomorphic
curves (and its asymptotic behavior
near the corners), $u$
has at most finitely many self-intersections, and by the
Riemann-Hurwitz formula, $\pi_P\circ u$ has at most finitely many
branch points. So, there are finitely many points $x\in P$ so that
$\phi(x)$ lies in the diagonal. Away from the diagonal, $u$ being
$J(P)$-holomorphic implies that $\phi$ is $\Sym^g(J(P))$-holomorphic;
since the intersections with the diagonal are discrete, it follows
that $\phi$ is $\Sym^g(J(P))$-holomorphic everywhere. (In particular,
this implies that $\phi$ is smooth.)

It is possible to translate the conditions that $u$
encounters a certain sequence of Reeb chords to the curve $\phi$, but
we will not need this. If all of the Reeb chords are length-1, $u$
corresponds to an ordinary holomorphic polygon $\phi$ in a Heegaard
multi-diagram
$(\overline{\Sigma},\alphas^1,\alphas^2,\alphas^1,\alphas^2,\dots,\betas)$,
where
$\alphas^1=\{\alpha_1^c,\dots,\alpha_{g-1}^c,\overline{\alpha}_1^a\}$
and
$\alphas^2=\{\alpha_1^c,\dots,\alpha_{g-1}^c,\overline{\alpha}_2^a\}$
or vice-versa. Further, if $u$ is an embedding, the index of the
$\overline{\partial}$ operator at $u$ and $\phi$ agrees
(cf.~\cite[Proposition 4.8 and Section 13]{Lipshitz06:CylindricalHF});
we will use this fact in Section~\ref{sec:ind}.

\subsection{Expected dimensions}\label{sec:ind}
In this subsection, we compute the expected dimensions of the various
moduli spaces introduced in Section~\ref{sec:def-mod-sp}. As in the
case of bordered $\HFa$~\cite[Chapter 5]{LOT1} or the cylindrical
formulation of Heegaard Floer homology~\cite[Section
4]{Lipshitz06:CylindricalHF}, there are two versions of the index
formulas: a version where the topological type of the source is fixed
but the maps may not be embeddings and a version where the Euler
characteristic of the source is determined by the holomorphic curve
being embedded (like the adjunction formula for closed
curves). The index formulas with a fixed source are given in
Propositions~\ref{prop:index-source}
and~\ref{prop:index-source-bdy-degen}; we deduce them by a gluing
argument similar to the $\HFa$ case~\cite{LOT1}. The embedded index
formulas (generalizations of the formula $\mu=e+n_\x+n_\y$) are given
in Propositions~\ref{prop:emb-ind}
and~\ref{prop:emb-ind-bdy-degen}. We deduce these from Sarkar's index
formula for holomorphic polygons~\cite{Sarkar11:IndexTriangles}, via
the symmetric product reformulation of the moduli spaces.

\subsubsection{Index formulas for a fixed source}
We start with the index formulas for curves with a given topological
source.

\begin{proposition}\label{prop:index-source}
  The expected dimension of the moduli space $\widetilde{\cM}^B(\x,\y;\Source)$ is
  given by
  \begin{equation}\label{eq:index-source}
    \ind(B,\Source,\vec{a};w,r)\coloneqq  g-\chi(\Source)+2e(B)+m+w-r
  \end{equation}
  where $\vec{a}=\vec{a}(\Source)$;  
  $m=|\vec{a}|$ denotes the number of algebra punctures of
  $S$; $w$ denotes the number of interior punctures of $S$;
  and $r$ is the total ramification at the Reeb orbits.
\end{proposition}
\begin{proof}
  We adapt the proof of~\cite[Proposition 5.8]{LOT1}, the strategy of which is
  to deduce the result from the closed case~\cite[Formula (6)]{Lipshitz06:CylindricalHF}.

  \begin{figure}
    \centering
    \includegraphics{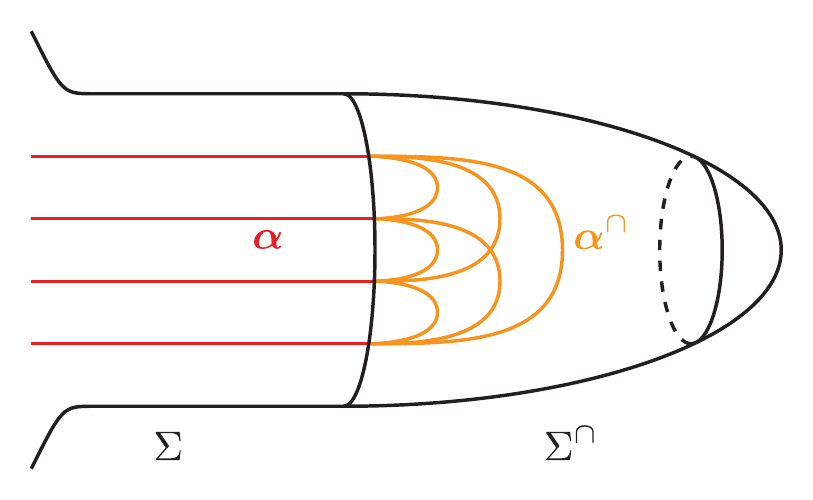}
    \caption[Capping off $\Sigma$, for an index computation]{\textbf{Capping off $\Sigma$.} This figure is adapted from~\cite[Figure 5.2]{LOT1}.}
    \label{fig:capping}
  \end{figure}

  First, we reduce to the case that
  all punctures are labeled by Reeb chords.  Each boundary puncture labeled by
  a non-Reeby basic algebra element (of the form $U^k\iota_j$) increases the expected
  dimension by $1$, because it increases the dimension of
  $\ModPol$ by $1$ and otherwise does not affect the index, and
  also increases Formula~\eqref{eq:index-source} by $1$, because it
  increases $m$ by $1$. So, it suffices to consider the case that
  there are no punctures labeled $U^k\iota_j$. For punctures labeled
  $U^k\rho$, the integer $k$ does not affect the expected dimension or
  Formula~\eqref{eq:index-source}. So, it suffices to assume all
  boundary punctures are labeled simply by Reeb chords.
  
  Fix a curve $u\co S\to \Sigma\times[0,1]\times\RR$ with the
  specified asymptotics.  By
  (pre)gluing split curves to the boundary punctures, we can arrange that at
  each boundary puncture the curve $u$ is asymptotic to a Reeb chord of length
  $\leq 3$. Gluing each split curve increases both the expected dimension and
  $m$ by $1$, so it suffices to verify the result assuming each chord in
  $\vec{a}$ has length $\leq 3$. Similarly, one can replace a Reeb orbit with
  multiplicity $k>1$ with $k$ simple Reeb orbits by pregluing to the map
  $(z-1)\cdots(z-k)\co \CC\setminus\{1,\dots,k\}\to\CC\setminus\{0\}$. This
  increases the expected dimension by $2(k-1)$, decreases $r$ in
  Formula~(\ref{eq:index-source}) by $k-1$, and increases $w$ by $k-1$. So, we may assume $r=0$, that is,
  that every Reeb orbit is simple.

  Now, glue the boundary of $\Sigma$ to a disk $\Sigma^\cap =
  \overline{\HHH}$ and glue the
  $\alpha$-arcs to the six arcs $\alphas^\cap$ shown in
  Figure~\ref{fig:capping}. Each of the Reeb chords in $u$ can be pre-glued to a
  holomorphic $1$-gon in $(\Sigma^\cap,\alphas^\cap)$ and each Reeb orbit in $u$
  can be pre-glued to a copy of $\Sigma^\cap$ itself to obtain a
  map
  $u'\co S'\to (\Sigma\cup\Sigma^\cap)\times[0,1]\times\RR$. From the closed
  case of the index formula~\cite[Formula (6)]{Lipshitz06:CylindricalHF}, the
  expected dimension at $u'$ is given by
  \[
    \ind(u')=g-\chi(S')+2e(B').
  \]
  We have $\chi(S')=\chi(S)+w$ and $e(B')=e(B)+m/2+w$. So, 
  \[
    \ind(u')=g-\chi(S)-w+2e(B)+m+2w=g-\chi(S)+2e(B)+m+w.
  \]
  On the other hand, each of the one-gons has index $1$ but the pre-gluing
  imposes a codimension-1 constraint, and each disk has index $2$ but gluing it
  to a Reeb orbit imposes a codimension-2 constraint. So,
  $\ind(u)=\ind(u')$. This proves the result.
\end{proof}

Here is the analogue for boundary degenerations:
\begin{proposition}\label{prop:index-source-bdy-degen}
  For a source $\bdSource$ for a boundary degeneration, with the
  punctures other than $p_\infty$ labeled by a sequence of algebra
  elements $\vec{a}$, $w$ Reeb orbits with total ramification $r$, and
  covering $\Sigma$ with multiplicity $k$, and for any point
  $\x\in(\alpha_1^a\cup\alpha_2^a\cup\{p\})\times\alpha_1^c\times\cdots\times\alpha_{g-1}^c$,
  the expected dimension of the moduli space of holomorphic maps
  $\tcN(\x,\bdSource)$ is
  \begin{equation}\label{eq:index-source-bdy-degen}
    \ind(k,\bdSource,\vec{a};w,r)\coloneqq  g-\chi(S)+2k(1-2g)+|\vec{a}|+\epsilon+w-r.
  \end{equation}
  where $\epsilon=1$ if $p$ is a coordinate of $\x$ (i.e., the
  puncture $q_\infty$ is labeled by a Reeby element $U^m\rho$) and is $0$ otherwise.
\end{proposition}
(With regards to $\epsilon$, recall that the centered operations on
$\MAlg$ count simple boundary degenerations with index $2$.)
\begin{proof}
We will assume that $p$ is not one of the coordinates of $\x$, though
the other case is similar. As in the proof of
Proposition~\ref{prop:index-source}, it suffices to assume that all of
the punctures except $q_\infty$ are labeled simply by Reeb chords with
$r=0$.
  We use a gluing argument, similar to the proof of
Proposition~\ref{prop:index-source}. Fix a holomorphic curve $u\co
S\to \Sigma\times\HHH$. Let $\overline{\HHH}$ denote the one-point
compactification of $\HHH$. For each Reeb chord $\rho^i$ of
$\Sigma$ preglue $u$ to one of the $1$-gons in
Figure~\ref{fig:capping}, as in the proof of
Proposition~\ref{prop:index-source}, and for each Reeb orbit preglue
$u$ to a $0$-gon. Here, the $1$-gons and $0$-gon are maps to
$\Sigma^\cap\times \overline{\HHH}$ which map to $\overline{\HHH}$ by constant maps. In
particular, the moduli space of $1$-gons has index $1$ and the moduli
space of $0$-gons has index $2$,
and pregluing imposes a $1$-dimensional (respectively $2$-dimensional)
constraint for each $1$-gon (respectively $0$-gon). So, if we let
$u'\co S'\to (\Sigma\cup\Sigma^\cap)\times \overline{\HHH}$ denote the capped-off
curve we have $\ind(u')=\ind(u)+g+2r$: the $g$ comes from having replaced
$\HHH$ by $\overline{\HHH}$ and, in the process, forgotten the codimension-$g$
constraint that $u$ passes through a specified point on each
$\alpha$-curve at $\infty$. The index of the $\dbar$-operator
at $u'$ is given by $\ind(u')=-\chi(S')+2\mu$. These quantities
satisfy
\begin{align*}
  \chi(S')&=\chi(S)+w\\
  \mu&=e(B)+|\vec{\rho}|/2+w+g+r.
\end{align*}
(The $g$ comes from the Maslov index of the projection to $\overline{\HHH}$.)
The Euler measure $e(B)$ is $k(1-2g)$.
Combining these gives
\[
  \ind(u)=-\chi(S)-w+2e(B)+|\vec{\rho}|+2w+2g+r-g-2r=g-\chi(S)+2k(1-2g)+|\vec{\rho}|+w-r
\]
as claimed.
\end{proof}

\subsubsection{Embedded index formulas}
Before stating the embedded index formulas we recall some notation.

Given a single Reeb chord $\rho$ starting at a point $s\in Z$ we define
\begin{equation}\label{eq:iota}
  \iota(\rho)=-m(\suppo{\rho},s)
\end{equation}
to be negative the average multiplicity of the support of $\rho$ near
the point $s$. Explicitly, if
$\rho$ has length $4k$ then $\iota(\rho)=-k$, while if $\rho$ has length
$4k+\ell$, $1\leq \ell\leq 3$, then $\iota(\rho)=-k-1/2$. Given two chords $\rho,\rho'$ we defined
\begin{equation}
  L(\rho,\rho')=m(\suppo{\rho'},\bdy \suppo{\rho}),
\end{equation}
the average multiplicity of the support of $\rho'$ at the boundary of the
support of $\rho$~\cite[Section 3.3.1]{LOT1}. For example,
$L(\rho_1,\rho_2)=1/2=-L(\rho_2,\rho_1)$, $L(\rho_1,\rho_{234})=0$, and
$L(\rho_{1234}\rho',\rho'')=L(\rho',\rho'')$. In general,
$L(\rho,\rho')=-L(\rho',\rho)$ (compare~\cite[Lemma 3.35]{LOT1}).

Given a sequence $\vec{\rho}=(\rho^1,\dots,\rho^m)$ of Reeb chords, let
\begin{equation}
  \iota(\vec{\rho})=\sum_i\iota(\rho^i)+\sum_{i<j}L(\rho^i,\rho^j)
\end{equation}
(compare~\cite[Formula (5.58)]{LOT1}).

\begin{lemma}\label{lem:iota-is-gr}
  If $\vec{\rho}=(\rho^1,\dots,\rho^m)$ is a sequence of Reeb chords so that the
  right idempotent of $\rho^i$ is equal to the left idempotent of $\rho^{i+1}$
  for all $i$, then $\iota(\vec{\rho})$ is the Maslov component of
  \[
    \grb(\rho^1)\cdot\cdots\cdot\grb(\rho^m).
  \]
\end{lemma}
\begin{proof}
  Perform the product $\grb(\rho^1)\cdot\cdots\cdot\grb(\rho^m)$ so that it is
  right-associated (i.e., so that $\grb(\rho^{m-1})\grb(\rho^m)$ is the first
  product, followed by $\grb(\rho^{m-2})(\grb(\rho^{m-1})\grb(\rho^m))$, and so on). The Maslov
  component of $\grb(\rho^i)$ is, by definition, $\iota(\rho^i)$.
  Let $\suppo{\rho^i}=(a,b,c,d)$ and
  $\suppo{\rho^{i+1}}+\cdots+\suppo{\rho^m}=(a',b',c',d')$. Then
  \begin{align*}
    \sum_{j\mid i<j}L(\rho^i,\rho^j)&=m\bigl((a',b',c',d'),\bdy(a,b,c,d)\bigr)\\
                              &=\frac{1}{2}\bigl((a-b)(a'+b')+(b-c)(b'+c')+(c-d)(c'+d')+(d-a)(d'+a')\bigr)\\
                              &=\frac{1}{2}\left(\left|\begin{smallmatrix}a & b\\a'&b'\end{smallmatrix}\right|+\left|\begin{smallmatrix}b & c\\b'&c'\end{smallmatrix}\right|
+\left|\begin{smallmatrix}c & d\\c'&d'\end{smallmatrix}\right|+\left|\begin{smallmatrix}d & a\\d'&a'\end{smallmatrix}\right|\right).
  \end{align*}
  Comparing this to the formula for multiplication in $\bigGroup$ gives the result.
\end{proof}

Next we state the embedded index formulas for holomorphic curves and
boundary degenerations (the analogues of~\cite[Proposition
5.62]{LOT1}), and a basic additivity property of these formulas. We
turn to the proofs after we have given all the statements. 

Recall that, given a domain $B$ and a tuple of points $\x$ in
$\Sigma$, $n_\x(B)$ is the sum over the points $x_i$ in $\x$ of the
average local multiplicity at $x_i$, and $e(B)$ is the Euler measure
of $B$. 
\begin{proposition}\label{prop:emb-ind}
  For an embedded holomorphic curve
  $u\co S\to \Sigma\times[0,1]\times\RR$ in a homology class
  $B\in\pi_2(\x,\y)$ with sequence of basic algebra elements $\vec{a}$,
  underlying chord sequence $\vec{\rho}$, $w$ Reeb orbits, and
  total ramification $r$ at the Reeb orbits, we have
  \begin{align}
    \chi(S)&=\chi_{\emb}(B,\vec{a},w,r)\coloneqq g+e(B)-n_\x(B)-n_\y(B)-\iota(\vec{\rho})+r\\
    \ind(u)&=\ind(B,\vec{a},w,r)\coloneqq e(B)+n_\x(B)+n_\y(B)+|\vec{a}|+\iota(\vec{\rho})+w-2r.
             \label{eq:emb-ind}
  \end{align}
  Further, if $u$ has $d$ transverse double points (and is otherwise embedded)
  then $\chi(S)=\chi_{\emb}(B,\vec{\rho},w,r)+2d$ and
  $\ind(u)=\ind(B,\vec{\rho},w,r)-2d$.
\end{proposition}
That is, Formula~\eqref{eq:emb-ind} in the case $r=0$ gives the expected dimension of
$\cM^B(\x,\y;\vec{a};w)$.
We will often write $\chi_{\emb}(B,\vec{a},w)$ and $\ind(B,\vec{a},w)$
for $\chi_{\emb}(B,\vec{a},w,0)$ and $\ind(B,\vec{a},w,0)$ (i.e., for
curves where each Reeb orbit is simple).

Similar formulas to Proposition~\ref{prop:emb-ind} compute the expected
dimension of the space of embedded boundary degenerations. That is:
\begin{proposition}\label{prop:emb-ind-bdy-degen}
  For a boundary degeneration as in
  Proposition~\ref{prop:index-source-bdy-degen}, if the map to
  $\Sigma\times\HHH$ is an embedding and the projection to $\Sigma$
  has degree $k$ then the Euler characteristic of $\bdSource$ and the
  expected dimension of $\cN(\x;\vec{a};w,r)$ are given by
  \begin{align}
    \chi(S)&=\chi_{\emb}(\vec{a},w,r)\coloneqq g+k(1-2g)-2kg-\iota(\vec{\rho})+r=g+k-4kg-\iota(\vec{\rho})+r\\
    \ind(u)&=\indsq(\vec{a},w,r)\coloneqq
             k(1-2g)+2kg+|\vec{a}|+\epsilon+\iota(\vec{\rho})+w-2r\label{eq:emb-ind-bdy-degen}\\
    &=k+|\vec{a}|+\epsilon+\iota(\vec{\rho})+w-2r\nonumber
  \end{align}
  where $\vec{\rho}$ is the chord sequence underlying $\vec{a}$ and,
  as in Proposition~\ref{prop:index-source-bdy-degen}, $\epsilon$ is
  $1$ if the asymptotic at $\infty$ is $U^m\rho$ and $0$ if the
  asymptotic at $\infty$ is $U^m\iota_j$.
\end{proposition}

Since we will most often be interested in the case $r=0$, let
\begin{align*}
  \chi_{\emb}(\vec{a},w)&=\chi_{\emb}(\vec{a},w,0)& \indsq(\vec{a},w)&=\indsq(\vec{a},w,0).
\end{align*}

To know that the index formulas induce a well-behaved grading on the Floer
complexes, we need to know they are additive:
\begin{proposition}\label{prop:emb-ind-add}
  Suppose $B\in\pi_2(\w,\x)$, $B'\in\pi_2(\x,\y)$, $\vec{a}$ and
  $\vec{a}'$ are sequences of Reeb chords compatible with $B$ and $B'$, and
  $w$ and $w'$ are non-negative integers. Let $(\vec{a},\vec{a}')$ be the
  concatenation of the sequences $\vec{a}$ and $\vec{a}'$. Then
  \[
    \ind(B+B',(\vec{a},\vec{a}'),w+w')=\ind(B,\vec{a},w)+\ind(B',\vec{a}',w').
  \]
\end{proposition}

To state the analogue for boundary degenerations, we need one more concept.
Given a sequence of chords $\vec{\rho}$ and integers $w,r$, the
\emph{domain multiplicity} of $(\vec{\rho},w,r)$ is
\begin{equation}
  \label{eq:domain-mult}
  k=w+r+\frac{1}{4}\sum_{\rho^i\in\vec{\rho}}|\rho^i|.
\end{equation}
Given a sequence of basic algebra elements $\vec{a}$, the domain
multiplicity of $(\vec{a},w,r)$ is the domain multiplicity of
$(\vec{\rho},w,r)$ where $\vec{\rho}$ is the underlying sequence of chords.
If there is a boundary degeneration $u$ with Reeb chords
$\vec{\rho}$, $w$ Reeb orbits and ramification $r$ at the Reeb orbits
then the multiplicity of $u$ is $[\pi_\Sigma\circ u]=k[\Sigma]$, where
$k$ is the domain multiplicity. In particular, the domain multiplicity
is an integer. Additionally, the support
$\suppo{\vec{\rho}}\in H_1(Z,\mathbf{a})$ of $\vec{\rho}$ is of the form
$(\ell,\ell,\ell,\ell)$ (where $\ell=k-w-r$), that is, it is equal at
$\rho_1$, $\rho_2$, $\rho_3$, and $\rho_4$. We call a sequence of
chords $\vec{\rho}$
\emph{equidistributed} if this latter condition holds, and call a
sequence
of basic algebra elements equidistributed if its underlying chord
sequence is.

\begin{lemma}\label{lem:bdy-degen-ind-add}
  Let $B\in\pi_2(\x,\y)$, $\vec{a}=(a_1,\dots,a_n)$ and
  $\vec{a}'=(a'_1,\dots,a'_m)$ be sequences of Reeb chords,
  and $w,w',r,r'\in\ZZ_{\geq 0}$. Assume that $\vec{a}'$ is
  equidistributed and $k$ is the domain multiplicity of
  $(\vec{a}',w',r')$. Fix $i$ and let
  \[
    \vec{b}=(a_1,\dots,a_i,a'_1,\dots,a'_m,a_{i+1},\dots,a_n)
  \]
  Then
  \[
    \ind(B+k[\Sigma],\vec{b},w+w',r+r')=\ind(B,\vec{a},w,r)+\indsq(\vec{a}',w',r').
  \]
  (This corresponds to the case of a simple boundary degeneration.)

  Similarly, if we let
  \[
    \vec{b}'=(a_1,\dots,a_ia'_1,a'_2,\dots,a'_m,a_{i+1},\dots,a_n)
  \]
  or
  \[
    \vec{b}'=(a_1,\dots,a_i,a'_1,\dots,a'_{m-1},a'_ma_{i+1},\dots,a_n)
  \]
  then 
  \[
    \ind(B+k[\Sigma],\vec{b}',w+w',r+r')=\ind(B,\vec{a},w,r)+\indsq(\vec{a}',w',r')-1.
  \]
  (This corresponds to the case of a composite boundary degeneration.)
\end{lemma}

The proofs of the embedded index formulas will be based on Sarkar's
index formula for polygons~\cite{Sarkar11:IndexTriangles}, so we
recall that next. Fix a (classical, not bordered) Heegaard
multi-diagram $(\Sigma_g,\alphas^1,\dots,\alphas^n)$, points
$\x^{i,i+1}\in T_{\alphas^i}\cap T_{\alphas^{i+1}}\subset
\Sym^g(\Sigma)$, and a point
$\x^{n,1}\in T_{\alphas^n}\cap T_{\alphas^1}$. Assume that
$\alphas^i\pitchfork \alphas^j$ for $i\neq j$. Given oriented arcs $A\subset
\alphas^i$ and $B\subset \alphas^j$, $i\neq j$, define the
\emph{jittered intersection number} $A\cdot B$ of $A$ and $B$ as
follows. If $A$ and $B$ intersect only at interior points then $A\cdot
B$ is their intersection number in the usual sense. Otherwise, if,
say, an endpoint of $A$ intersects $B$ then there are four directions
one can push $A$ off $\alphas^i$ slightly so that the pushoff is
transverse to $\alphas^j$ (and hence $B$); let $A\cdot B$ be the
average of the four intersection numbers of these pushoffs with
$B$. So, $A\cdot B\in\frac{1}{4}\ZZ$.

Given a homotopy class of $n$-gons $D\in
\pi_2(\x^{1,2},\x^{2,3},\dots,\x^{n,1})$, Sarkar proves that the
Maslov index of $D$ (expected dimension of the moduli space with a
fixed complex structure on the source) is
\begin{equation}\label{eq:Sarkar-formula}
  \mu(D)=e(D)+n_{\x^{1,2}}(D)+n_{\x^{n,1}}(D)-\left(\sum_{1<i<j\leq n}\bdy^{\alphas^i}(D)\cdot\bdy^{\alphas^j}(D)\right)-\frac{g(n-2)}{2}.
\end{equation}
Here, $\bdy^{\alphas^i}(D)$ is the part of the cellular $1$-chain
$\bdy D$ lying in $\alphas^i$ and $\cdot$ denotes the jittered
intersection number.  The $n$-gon maps on Heegaard Floer homology
count holomorphic $n$-gons with
Maslov index $3-n$. In the case of a bigon in
$\pi_2(\x,\y)$, Formula~\eqref{eq:Sarkar-formula} reduces to the
familiar formula $e(D)+n_\x(D)+n_\y(D)$.

\begin{proof}[Proof of Propositions~\ref{prop:emb-ind} and~\ref{prop:emb-ind-bdy-degen}]
  We will prove the formulas for the index at embedded holomorphic
  curves; the formulas for the Euler characteristic of the source then
  follow from Propositions~\ref{prop:index-source}
  and~\ref{prop:index-source-bdy-degen}.

  It suffices to prove the result when $\vec{a}$ consists entirely of
  chords. Specifically, the proof we will give does not depend on the
  particular choice of almost complex structure. The $U$-powers
  labeling punctures only affect the choice of almost complex
  structure, as do punctures labeled $U^n\iota_j$ (except for adding
  $1$ to both the expected dimension and Formula~\eqref{eq:emb-ind}).
  
  By additivity of the index under gluing, if $\rho^i=\rho'\rho''$
  then the expected dimension of the moduli
  space of curves asymptotic to $(\rho^1,\dots,\rho^i,\dots,\rho^n)$
  is one less than the expected dimension of the moduli space of
  curves asymptotic to $(\rho^1,\dots,\rho',\rho'',\dots,\rho^n)$: the
  two differ by gluing on a split curve. Replacing
  $(\rho^1,\dots,\rho^i,\dots,\rho^n)$ by
  $(\rho^1,\dots,\rho',\rho'',\dots,\rho^n)$ also increases the right
  side of Equation~\eqref{eq:emb-ind}: the only term that changes is
  $|\vec{\rho}|$. (The fact that $\iota$ does not change follows from
  Lemma~\ref{lem:iota-is-gr}.) So, it suffices to prove the result
  when all of the chords $\rho^i$ have length $1$. Similarly, we can
  assume that all of the Reeb orbits are simple.

  After these simplifications, the index $\ind$ is essentially the
  same as the Maslov index in the symmetric product, under the
  tautological correspondence (Section~\ref{sec:tautological}).
  More precisely, because $\mu$ gives the
  expected dimension for polygons with a fixed conformal structure and
  $\ind$ allows the conformal structure to vary,
  \[
    \ind(u)=\mu(B)+|\vec{\rho}|.
  \]

  The next step is to apply Sarkar's formula to compute $\mu(B)$. For
  definiteness, assume that a point of $\x$ lies on $\alpha_1^a$. Then
  (since all the $\rho^i$ have length $1$, which is odd), the polygon
  has boundary
  \[
    (T_\betas,T_{\alphas,1},T_{\alphas,2},T_{\alphas,1},\dots,T_{\alphas,(1+|\vec{\rho}|)\bmod 2}).
  \]
  In particular, the same $\alpha$-curves are repeated many times.  To
  apply Sarkar's formula, we must first choose perturbations of these $\alpha$-curves
  to make them transverse. How we do this perturbation affects $n_\x$,
  $n_\y$, and the $\bdy^{\alphas^i}(D)\cdot\bdy^{\alphas^j}(D)$ terms.

  \begin{figure}
    \centering
    \includegraphics{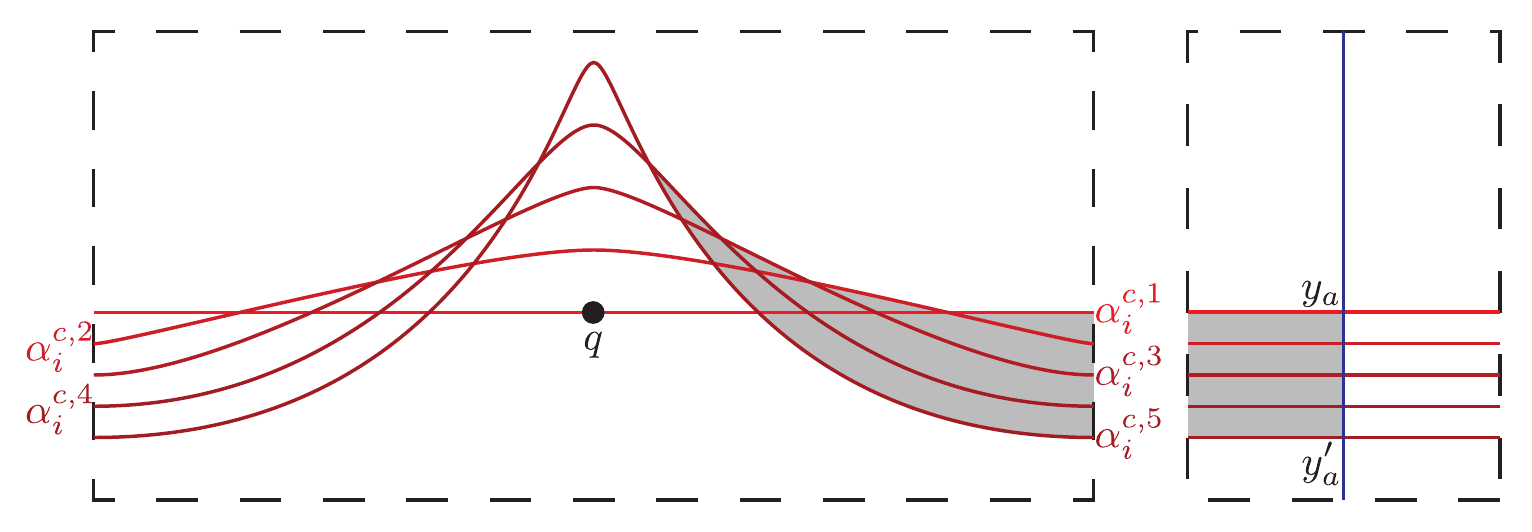}
    \caption[Perturbing $\alpha$-circles]{\textbf{Perturbing
        $\alpha$-circles.} The figure corresponds to the case
      $|\vec{\rho}|=4$. The original intersection point
      $y_a\in \alpha_a^c\cap \beta_b$, the corresponding point
      $y'_a\in \alpha_a^{c,5}\cap \beta_b$, and the point $q$ are also indicated.}
    \label{fig:ind-pert-circ}
  \end{figure}

  \begin{figure}
    \centering
    \includegraphics{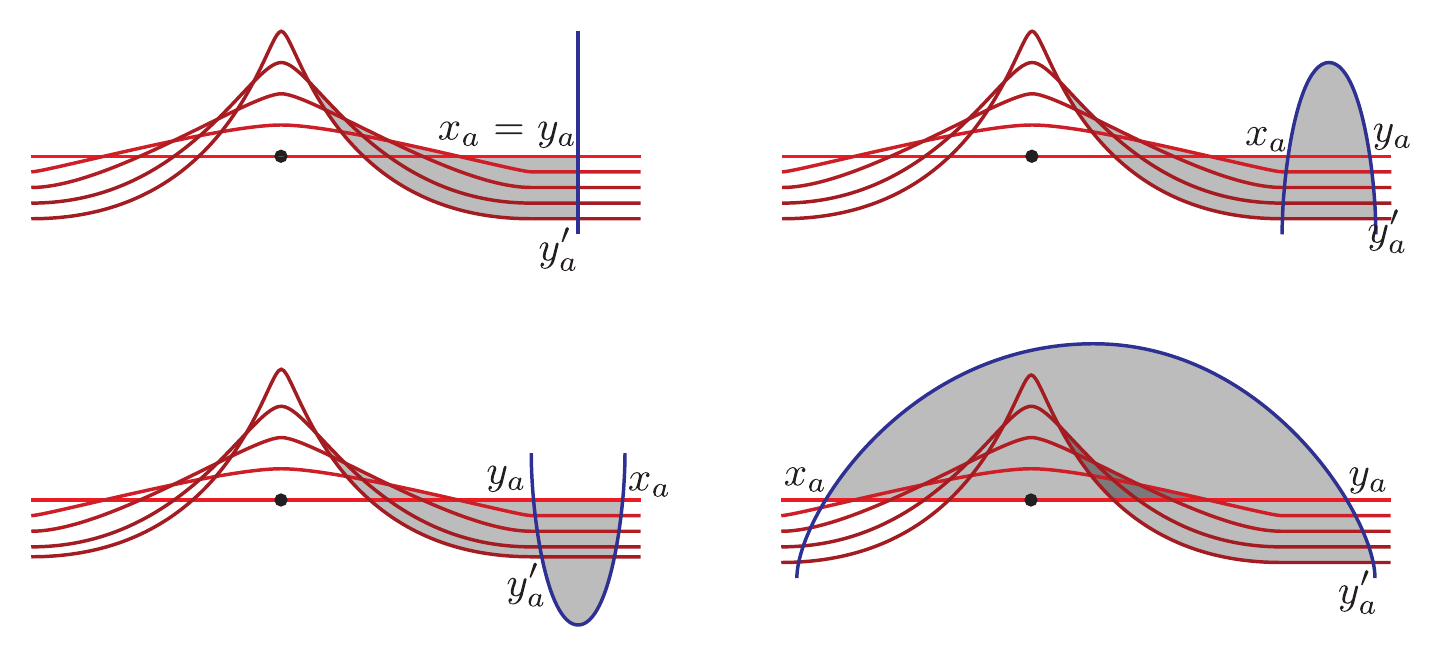}
    \caption{\textbf{Multiplicities in the perturbed diagram.} In the
      first case, $n_\x$ and $n_\y$ both increase by $1/4$. In the
      second, $n_\x$ increases by $1/2$ while $n_\y$ is unchanged. In
      the third, $n_\x$ is unchanged while $n_\y$ increases by
      $1/2$. In the fourth, $n_\x$ and $n_\y$ are both unchanged. In
      the first three cases, $\ell=0$ and the jittered intersection
      number is $-(4-2)/4+0=-1/2$, while in the last $\ell=1$ and the
      jittered intersection number is $-(4-2)/2+1/2=0$.}
    \label{fig:ind-pert-circ-2}
  \end{figure}
  
  For each $\alpha$-circle $\alpha_m^c$, choose the perturbations
  $\alpha_{m}^{c,i}$ as shown in Figure~\ref{fig:ind-pert-circ}, so
  that $\alpha_m^{c,i}$ and $\alpha_m^{c,j}$ are disjoint outside the
  figure for $i\neq j$ and the points $\x$ and $\y$ are outside the
  region where the perturbed circles intersect each other. Locally,
  the perturbed domain is obtained from the original
  domain of bigons by gluing on (adding) the shaded region in the
  figure. The Euler measure of the shaded region is $|\vec{\rho}|/4$
  and the jittered intersection number is
  $-(|\vec{\rho}|-2)/4+\ell/2$, where $\ell$ is the multiplicity with
  which a point $q$ on $\alpha_m^{c}$ in the perturbation region
  occurs in $\bdy^\alpha B$. The
  perturbation increases $n_\x+n_\y$ by $1/2-\ell/2$. (See Figure~\ref{fig:ind-pert-circ-2} for
  some examples.) So, the contribution
  to Formula~\eqref{eq:Sarkar-formula} is
  $|\vec{\rho}|/2-1/2+\ell/2-\ell/2+1/2=|\vec{\rho}|/2$. There
  are $g-1$ such $\alpha$-circles, for a total contribution of
  \[
    \frac{(g-1)|\vec{\rho}|}{2}.
  \]
  
  We claim that for the $\alpha$-arcs we can apply Sarkar's formula
  without perturbing. We prove this by induction on the number of
  $\alpha$-arcs not perturbed. Suppose the formula holds if the first
  $\ell$ copies of the $\alpha_i^a$ are not perturbed. There are two
  natural directions to push off the $(\ell+1)\st$ copy of an
  $\alpha$-arc. The Euler measure of the domain is the same whether
  the $\alpha$-arc is perturbed or not. By definition, the
  $n_{\x^{1,2}}$ and $n_{\x^{n,1}}$ terms and the
  $\bdy^{\alphas^i}\cdot\bdy^{\alphas^j}$-terms for the un-perturbed
  curve are the average of the terms for the two directions of
  perturbing (since both can be viewed as averages over pushing off
  the $\alpha$-curve in four directions). So, since the formula holds
  for either direction of perturbation, it also holds for the average
  of the two directions, i.e., for not perturbing at all.
  
  For the un-perturbed $\alpha$-arcs, the big sum in
  Formula~\eqref{eq:Sarkar-formula} contributes
  \[
    -\frac{1}{4}|\vec{\rho}| +\sum_{i<j}L(\rho^i,\rho^j)=\frac{1}{2}\sum_i\iota(\rho^i)+\sum_{i<j}L(\rho^i,\rho^j);
  \]
  see Figure~\ref{fig:IntNum}.

  \begin{figure}
    \centering
    \includegraphics[scale=.83333]{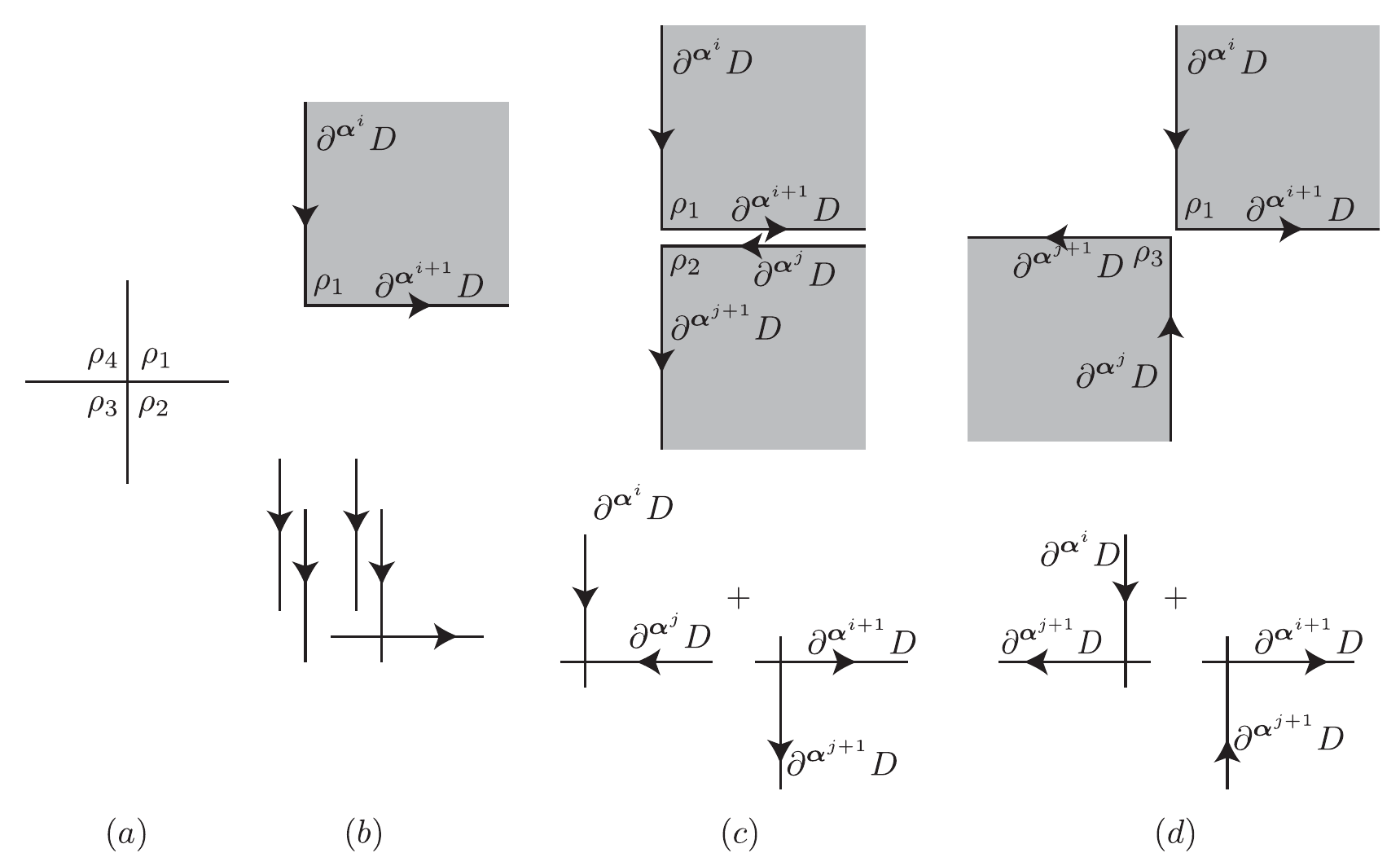}
    \caption[Jittered intersection numbers and
    chords]{\textbf{Jittered intersection numbers and chords.} (a) The
      chords around the point $p=\alpha_1^a\cap \alpha_2^a$. (b) The
      contribution to the jittered intersection number
      $\bdy^{\alphas^i}D\cdot \bdy^{\alphas^{i+1}}D$ coming from the
      arcs into and out of $\rho_1$ is $1/4$. (c) The contribution
      from a chord $\rho_1$ followed later by a chord $\rho_2$ is
      $-1/2$, while the contribution from $\rho_2$ followed later by
      $\rho_1$ is $+1/2$. (d) The contribution from $\rho_1$ followed
      later by $\rho_3$ is $0$, since in this case the two nonzero
      terms cancel. The other combinations of length-1 chords are
      similar.}
    \label{fig:IntNum}
  \end{figure}
  
  Finally, for Formula~\eqref{eq:Sarkar-formula}, each Reeb chord
  corresponds to a $\pi/2$-corner, so contributes $-1/4$ to the Euler
  measure. For Formula~\eqref{eq:emb-ind}, each Reeb chord is viewed
  as having two $\pi/2$-corners, so contributes $-1/2$ to the Euler
  measure. Also, each orbit contributes $0$ to
  Formula~\eqref{eq:Sarkar-formula} but  $-1$ to Formula~\eqref{eq:emb-ind}. So, the Euler measure term in
  Formula~\eqref{eq:Sarkar-formula} is actually
  \[
    e(B)+\frac{1}{4}|\vec{\rho}|+w.
  \]

  Adding up these contributions, Formula~\eqref{eq:Sarkar-formula}
  gives
  \begin{align*}
    \ind(B,\vec{\rho},w,0)&=e(B)+\frac{1}{4}|\vec{\rho}|+w+n_\x(B)+n_\y(B)-\frac{1}{4}|\vec{\rho}|
                            +\sum_{i<j}L(\rho^i,\rho^j)\\
    &\qquad\qquad\qquad\qquad+\frac{(g-1)|\vec{\rho}|}{2}-\frac{g|\vec{\rho}|}{2}+|\vec{\rho}|\\
                          &=e(B)+w+n_\x(B)+n_\y(B)+\sum_{i<j}L(\rho^i,\rho^j)+\frac{|\rho|}{2}\\
                          &=e(B)+2+n_\x(B)+n_\y(B)+|\vec{\rho}|+\sum_i\iota(\rho^i)+\sum_{i<j}L(\rho^i,\rho^j)\\
                          &=e(B)+n_\x(B)+n_\y(B)+|\vec{\rho}|+\iota(\vec{\rho})+w,
  \end{align*}
  as desired.
\end{proof}

\begin{proof}[Proof of Proposition~\ref{prop:emb-ind-add} and Lemma~\ref{lem:bdy-degen-ind-add}]
  This is immediate from the fact that these index formulas agree with
  the Maslov index in the symmetric product and additivity of the
  Maslov index.
\end{proof}

\begin{remark}
  In the case of $\HFa$, instead of reducing to Sarkar's formula we
  proved the embedded index formula (the analogue of
  Proposition~\ref{prop:emb-ind}) by studying how a holomorphic curve
  intersects its translates (similar to the computation
  in the closed case~\cite{Lipshitz06:CylindricalHF}). That argument
  requires the
  almost complex structure be a product near the boundary of the
  strip, to avoid having intersection points disappear on the
  boundary, so does not apply immediately to the kinds of almost
  complex structures considered here. On the other hand, the argument
  relying on Sarkar's formula would not apply directly to the higher
  genus boundary case of bordered Floer theory, since there is not a
  direct identification with Heegaard multi-diagrams in that case.
\end{remark}

\subsection{Pinched almost complex structures}\label{sec:pinched}
So far, we have not explained the role of the pinching function
(Definition~\ref{def:pinching-fn}). In this section,
we will see one reason it is required. (The other is in Section~\ref{sec:algebra}.)

\begin{lemma}\label{lem:bdy-deg-ind-2}
  Given a sequence of basic algebra elements $\vec{a}$ and
  non-negative integers $w$ and $r$, for any $\epsilon$ sufficiently
  small (depending on $\vec{a}$, $w$, and the bordered Heegaard diagram), if $j$
  is an $\epsilon$-pinched complex structure for which
  $\cN(\x,\vec{a};w,r)$ contains a non-constant curve, then
  \[
    \indsq(\vec{a},w,r)\geq 2.
  \]
  Further, if $\indsq(\vec{a},w,r)=2$ then no pair of consecutive
  elements of $\vec{a}$ are multipliable, i.e., have $a_ia_{i+1}\neq
  0$. In particular, no element of $\vec{a}$ has the form $U^n\iota_j$.
\end{lemma}
\begin{proof}
  We start by reducing to the genus $1$ case, using the
  $\epsilon$-pinched condition. Fix a sequence $\epsilon_i\to 0$ and
  suppose the moduli spaces $\cN(\x;\vec{a};w,r)$ are non-empty
  with respect to a sequence of $\epsilon_i$-pinched almost complex
  structures. Choose a convergent subsequence of these complex
  structures.
  In the limit, we obtain a complex structure on the genus-1 surface
  containing $\alpha_1^a\cup\alpha_2^a$ so that the compactified
  moduli space $\ocN(\x;\vec{a};w,r)$ is non-empty for this complex
  structure. This moduli space consists of trees of holomorphic disks
  mapped to a genus-1 surface (because the map to $\HHH$ is a 1-fold
  branched cover); such holomorphic maps are determined by their
  combinatorics and branched points. In particular, it is easy to see
  that, in the genus $1$ case, whenever the compactified moduli space
  $\ocN(\x;\vec{a};w,r)$ is non-empty, the uncompactified space
  $\cN(\x;\vec{a};w,r)$ is also non-empty.
  
  So, it suffices to prove the result in the case that $g=1$. Fix a
  boundary degeneration $u\co S\to \Sigma\times\HHH$. 
  As observed above, since $\pi_\HHH\circ u$ is a 1-fold branched
  cover of $\HHH$, $S$ has genus $0$. So, $\chi(S)=1-w+r$. On the
  other hand, $\pi_\Sigma\circ u$ is a branched cover of degree
  $k>0$. So, if
  $\br$ denotes the ramification number of $\pi_\Sigma\circ u$ and the
  homology class of $\pi_\Sigma\circ u$ is $k[\Sigma]$ then by the
  Riemann-Hurwitz formula,
  $\chi(S)=-k+|\vec{\rho}|/2-\br$, where $\rho$ is the associated
  chord sequence. (Specifically, the Euler measures
  satisfy $e(S)=ke(\Sigma)-\br=-k-\br$ and the Euler characteristic of
  $S$ is $\chi(S)=e(S)+|\vec{\rho}|/2$: since when computing $e$ each
  puncture corresponds to an arc with two corners at its end, each
  puncture decreases $e(S)$ by $2/4$.) So,
  \[
    w+|\vec{\rho}|/2-r=1+k+\br.
  \]
  Hence, by Proposition~\ref{prop:index-source-bdy-degen}, the index is
  given by
  \[
    \indsq(\vec{a},w,r)\geq \indsq(\vec{\rho},w,r)=1-1+w-r-2k+|\vec{\rho}|+w-r=-2k+|\vec{\rho}|+2w-2r=2+2\br.
  \]
  Thus, the index is at least $2$. Further, if the index is equal to
  $2$ then the first inequality is an equality so
  $|\vec{a}|=|\vec{\rho}|$, hence no element of
  $\vec{a}$ has the form $U^n\iota_j$.  Finally, to see that no pair
  of consecutive chords in $\vec{\rho}$ are multipliable, note that if
  $\rho^i$ and $\rho^{i+1}$ are multipliable then there is a boundary
  branch point of $\pi_\Sigma\circ u$ between the corresponding
  punctures, so $\br\geq 1/2$ and hence the index is at least $3$.
\end{proof}

Since for any $C$ there are finitely many pairs of a sequence of basic
algebra elements $\vec{a}$ and non-negative integer $w$ so that the
energy satisfies $E(\vec{a},w)\leq C$, by Lemma~\ref{lem:bdy-deg-ind-2},
for each energy $C$ there is an $\eta(C)>0$ so that if $J$ is
$\eta(C)$-pinched then the conclusion of Lemma~\ref{lem:bdy-deg-ind-2}
hold for any $(\vec{a},w,r)$ with energy $E(\vec{a},w)\leq C$. This is
the condition we require for our pinching function:

\begin{definition}\label{def:sufficient-pinching}
  A pinching function $\eta$ is \emph{sufficient} if for each positive
  integer $C$, for every $(\vec{a},w)$ with $E(\vec{a},w)\leq C$,
  Lemma~\ref{lem:bdy-deg-ind-2} holds for $(\vec{a},w)$ with $\epsilon=\eta(C)$.
\end{definition}

\begin{corollary}\label{cor:pinch-exist}
  There exist sufficient pinching functions.
\end{corollary}
\begin{proof}
  This is immediate from Lemma~\ref{lem:bdy-deg-ind-2}.
\end{proof}

\subsection{Transversality}\label{sec:transversality}
We recall two results of Ozsv\'ath-Szab\'o about the non-existence of
sphere and disk bubbles. For both spheres in $\Sym^g(\Sigma)$ and
disks in $\Sym^g(\Sigma)$ with boundary on some $T_{\alpha,i}$ or
$T_\beta$, the homology class is a multiple of $\Sigma$. Positivity of
domains implies that the moduli spaces in homology class $n[\Sigma]$
are empty unless $n>0$. The relevance of this to the moduli spaces
under consideration is via the tautological correspondence from
Section~\ref{sec:tautological}.

\begin{lemma}\cite[Lemma 3.13]{OS04:HolomorphicDisks}\label{lem:OS-no-spheres}
  Given any complex structure $i$ on $\Sigma$, the set of points in
  $\Sym^g(\Sigma)$ lying on holomorphic spheres has codimension $2$.
\end{lemma}

\begin{corollary}\label{cor:no-spheres}
  For a generic choice of complex structure $i$ on $\Sigma$ and any
  $\RR$-invariant admissible almost complex structure $J$ sufficiently close to
  $i\times j_\bD$, there are no holomorphic spheres through points in
  $(T_{\alpha,1}\cup T_{\alpha,2})\cap T_\beta$.
\end{corollary}
\begin{proof}
  This is immediate from the previous lemma and Gromov compactness.
\end{proof}

\begin{lemma}\cite[Theorem 3.15]{OS04:HolomorphicDisks}\label{lem:OS-no-disks}
  If the moduli space of disks with boundary on $T_{\beta}$
  (respectively $T_{\alpha,1}$, $T_{\alpha,2}$) through a given point
  $p\in T_\beta$ (respectively $p\in T_{\alpha,1}$,
  $p\in T_{\alpha,2}$) is transversely cut out, then the count of
  Maslov index $2$ disks through $p$ vanishes.
\end{lemma}

Next, we turn to the transversality we will need for boundary
degenerations, continuing with the discussion from Section~\ref{sec:bdy-degen}:
\begin{definition}
  \label{def:RegularJ}
  The complex structure $j$ is called \emph{regular for energy $E_0$ boundary
  degenerations} if for all $(\vec{a};w,r)$ with
  $\indsq(\vec{a};w,r)\leq 2$ and $E(\vec{a},w)\leq E_0$, the moduli space
  $\cN(*;\vec{a};w,r)$ is smoothly cut out, and every generator $\x$
  is a regular value of the evaluation map $\ev\co \cN(*;\vec{a};w,r)\to (\alpha_1^a\cup\alpha_2^a)\times
    \alpha^c_1\times\dots\times \alpha^c_{g-1}$.
\end{definition}
In particular, if $j$ is regular for energy $E_0$ boundary
degenerations then for generic $\x$ and $E(\vec{a},w)\leq E_0$, the space $\cN(\x;\vec{a};w)$ is a
smooth manifold of dimension $\indsq(\vec{a};w)-2$.

The proof of existence of regular $j$ is given in Theorem~\ref{thm:AlgOfSurface} below.

Next we discuss the dependence of the moduli space of boundary
degenerations based at $\x$ on the point $\x$:

\begin{lemma}\label{lem:bd-ev-proper}
  Suppose that $\indsq(\vec{a},w,r)=2$ and $j$ is regular for energy
  $E(\vec{a},w)$ boundary degenerations and is
  $\eta(E(\vec{a},w))$-pinched for some sufficient pinching
  function $\eta$. Then, the map
  \[
    \ev\co \cN(*;\vec{a};w,r)\to (\alpha_1^a\cup\alpha_2^a)\times
    \alpha^c_1\times\dots\times \alpha^c_{g-1}
  \]
  is proper.
\end{lemma}
\begin{proof}
  The proof is similar to, but easier than, the proof of
  Theorem~\ref{thm:master} in Section~\ref{sec:compactness}, but we
  will also give fewer details than
  in that proof, so the reader may find it easier to read that proof
  first.

  Fix a sequence in $\cN(*;\vec{a};w,r)$. Considering just the images
  in $\HHH$ of the
  punctures gives a sequence of algebra-type polygons. By compactness of
  $\ModPol$, by taking a subsequence we may assume this sequence of polygons converges in $\ModPol$, to some
  tree of disks
  (perhaps union some sphere bubbles). Then, applying Gromov
  compactness (or the version from symplectic field
  theory~\cite{BEHWZ03:CompactnessInSFT}, if we think of the marked
  points of the disks as punctures), we may assume that the sequence
  of holomorphic curves itself converges. The limit is either an
  element of $\cN(*;\vec{a};w,r)$ or a broken polygon (perhaps
  union some sphere bubbles). In the former case, the result is
  immediate from continuity of the evaluation map, so it remains to
  rule out the latter case. It follows from the second half of
  Lemma~\ref{lem:bdy-deg-ind-2} (the statement about no consecutive
  pair of chords being multipliable) that any constant component
  mapped to $p$ (i.e., curve at $e\infty$) connects to at least two
  other components. So, there are at least two non-constant
  components. But this violates Lemma~\ref{lem:bdy-deg-ind-2} and
  additivity of the index (Lemma~\ref{lem:bdy-degen-ind-add}).
\end{proof}

Lemmas~\ref{lem:bdy-deg-ind-2} and~\ref{lem:bd-ev-proper} imply that the spaces
$\cN(\x,\vec{a};w,r)$ for different regular values $\x$ of $\ev$
(intersecting the same $\alpha$-arc) are cobordant. So, we will often
want to fix $\x$ but not include it in the notation. Let
\begin{align*}
  \cN([\alpha_i^a];\vec{a};w)&=\cN(\x;\vec{a};w),
\end{align*}
where $\alpha_i^a$ is the $\alpha$-arc containing a coordinate in
$\x$. By Lemma~\ref{lem:bdy-deg-ind-2}, if $\indsq(\vec{a},w)=2$ and
$\cN([\alpha_i^a];\vec{a};w)\neq\emptyset$ then all of the $a_i$ are
Reeby elements. Also, if $a_1=U^\ell\rho^1$ and $a_n=U^m\rho^n$,
then necessarily $\alpha_i^a=[(\rho^1)^-]=[(\rho^n)^+]$, the
$\alpha$-arc containing the initial endpoint of $\rho^1$ and the
terminal endpoint of $\rho^n$.

\begin{definition}
  \label{def:CurveCount}
  For $i=1,2$, let $\#\cN([\alpha_i^a];\vec{a};w)$ be
  the number of points in $\cN(\x;\vec{a};w))$, for a
  generic choice of point $\x$ with a component in $\alpha_i^a$. (It
  follows from Lemma~\ref{lem:bd-ev-proper} that this number is
  independent of the choice of $\x$.)
\end{definition}

In Section~\ref{sec:alg-bdeg}, we prove that the degrees appearing in
Definition~\ref{def:CurveCount} are independent of the choice of $j$;
indeed, according to Theorem~\ref{thm:alg-well-defd} (also proved in
that section), these are identified with terms in the operations on
$\MAlg$.

Returning to the main component, we have:

\begin{proposition}\label{prop:transv-exists}
  For any pinching function $\eta$, the set of coherent families of
  $\eta$-admissible almost complex structures so that the moduli
  spaces $\cM^B(\x,\y;\Source)$ are all transversely cut out is
  co-meager. Similarly, the set of coherent families of
  $\eta$-admissible almost complex structures so that the moduli
  spaces of disks appearing in Lemma~\ref{lem:OS-no-disks} are
  transversely cut out is also co-meager.
\end{proposition}
\begin{proof}
  This is an easy adaptation of the case for bordered
  $\HFa$~\cite{LOT1}, which in turn is an adaptation of the
  cylindrical case for Heegaard Floer homology~\cite[Section
  3]{Lipshitz06:CylindricalHF} (which in turn is adapted from
  McDuff-Salamon~\cite{MS04:HolomorphicCurvesSymplecticTopology}). (For
  the second statement in the proposition note that the domain of a
  Maslov index $2$ disk is a copy of $\Sigma$, so these disks are
  somewhere injective---in fact, injective at every interior point.)
\end{proof}

Finally, we collect all of the transversality conditions we will
require for the modules $\CFAm$ and $\CFDm$ to be defined:
\begin{definition}\label{def:tailored}
  Call an $\eta$-admissible family of almost complex structures $J$
  \emph{tailored} if the following conditions hold:
  \begin{enumerate}[label=(TJ-\arabic*)]
  \item\label{item:tail-pinched} The pinching function $\eta$ is
    sufficient in the sense of
    Definition~\ref{def:sufficient-pinching}.
  \item\label{item:tail-main-comp} The moduli spaces
    $\cM^B(\x,\y;\Source)$ are transversely cut out for all data $(B,\Source,\vec{a},w,r)$ such that
    $\ind(B,\Source,\vec{a},w,r)\leq 2$.
  \item\label{item:tail-bdy} For each sequence of basic algebra
    elements $\vec{a}$, integer $w$, decorated source $\Source$
    compatible with $(\vec{a},w)$,  and homology class $B$ so that
    $\ind(B,\Source,\vec{a};w,r)\leq 1$ and each $a_i\in\vec{a}$, the following
    holds. Let $q_i$ be the puncture of $\Source$ corresponding to $a_i$.
    \begin{itemize}
    \item If $a_i=U^n\iota_j$ then for each
      $u\in \cM^B(\x,\y;\Source)$, the complex structure on $\Sigma$
      corresponding to $q_i$ is regular for energy $E(a_i)$ boundary degenerations,
      and $\w=\pi_\Sigma(u(q_i))$ is a regular value of each
      evaluation map
      $\ev\co\cN(*;\vec{a}';w',r')\to (\alpha_1^a\cup\alpha_2^a)\times
      \alpha_1^c\times\cdots\times\alpha_{g-1}^c$, for each sequence
      of basic algebra elements $a'$ and integers $w'$ and $r'$ so
      that the moduli space has index $\leq 2$. So, the index $\leq2$
      moduli spaces $\cN(\w;\bdSource)$ are transversely cut out.
    \item If $a_i=U^n\rho$ then for each $u\in \cM^B(\x,\y;\Source)$,
      the complex structure on $\Sigma$ corresponding to $q_i$ is
      regular for energy $E(a_i)$ boundary degenerations, and if we let
      $p\times \w=\pi_\Sigma(u(q_i))$, where
      $\w\in \alpha_1^c\times\cdots\times\alpha_{g-1}^c$, then the
      moduli spaces $\cN(\rho^i\times \w;\bdSource)$ with index $\leq 2$
      are transversely cut out.
    \end{itemize}
    (By Condition~\ref{item:tail-main-comp}, there are finitely many
    holomorphic curves $u$ for each tuple of data
    $(\vec{a},w,\Source,B)$, so this amounts to countably many
    conditions overall. Also, the moduli spaces $\cN$ here are
    computed with respect to the complex structure induced by $J$ at
    $q_i$.)
  \item\label{item:tail-spheres} There is no $\Sym^g(j)$-holomorphic
    sphere in $\Sym^g(\Sigma)$ passing through any point in
    $(T_{\alpha,1}\cup T_{\alpha,2})\cap T_\beta$, for any complex
    structure $j$ appearing in $J_\infty$.
  \item\label{item:tail-disk} The moduli spaces of $J_\infty$-holomorphic disks
    with boundary on $T_\beta$ and index $\leq 2$ are transversely cut
    out, as are the moduli spaces of holomorphic disks with boundary
    on $T_{\alpha,1}$ and the moduli space of disks with boundary on
    $T_{\alpha,2}$.
  \end{enumerate}
\end{definition}

\begin{corollary}
  \label{cor:TailoredExist}
  Tailored families almost complex structures exist.
\end{corollary}
\begin{proof}
  Fix a sufficient pinching function
  (Condition~\ref{item:tail-pinched}), which exists by
  Corollary~\ref{cor:pinch-exist}.
  Conditions~\ref{item:tail-spheres},~\ref{item:tail-disk}, and the
  special case of~\ref{item:tail-main-comp} with no algebra punctures
  give restrictions on the $\RR$-invariant almost complex structure
  $J_\infty$ from Condition~\ref{item:admis-at-infty}. By
  Corollary~\ref{cor:no-spheres} and
  Proposition~\ref{prop:transv-exists}, a generic almost complex
  structure close to a split one satisfies these conditions. So, start by fixing such a
  generic, $\RR$-invariant almost complex structure $J_\infty$. Then,
  construct the admissible family of almost complex structures as in
  the proof of Lemma~\ref{lem:admis-J-exists}, agreeing with
  $J_\infty$ near $\pm\infty$, using
  Proposition~\ref{prop:transv-exists} to guarantee the moduli spaces
  $\cM$ are transversely cut out
  (Condition~\ref{item:tail-main-comp}). To ensure
  Condition~\ref{item:tail-bdy}, by Sard's theorem, for generic points
  $\w\in \alpha_1^c\times\cdots\times\alpha_{g-1}^c$ and
  $\w'\in
  (\alpha_1^a\cup\alpha_2^a)\times\alpha_1^c\times\cdots\times\alpha_{g-1}^c$,
  the moduli spaces $\cN(\rho\times \w;\bdSource)$ and
  $\cN(\w';\bdSource)$ are smoothly cut out.  For a generic choice of
  $J$, the moduli spaces $u\in\cM^B(\x,\y;\Source)$ with
  $\ind(B,\Source,\vec{a};w,r)=1$ will only be asymptotic to such generic
  $(\rho,\w)$, giving Condition~\ref{item:tail-bdy}.
\end{proof}

\subsection{Gluing}\label{sec:gluing}
In this section, we prove gluing theorems corresponding to the
different cases in Theorem~\ref{thm:master}. In each case, 
to state the gluing theorem we need to define what it means for a
family of unbroken holomorphic curves (as in
Section~\ref{sec:main-part}) to converge to a broken holomorphic
curve, since the gluing theorems assert that near a broken curve, the
space of unbroken curves has a particular form (or, more precisely, a
single end).

A \emph{two-story building} is a pair
\[
  (u,u')\in \cM^{B}(\x,\x';a_1,\dots,a_j;w_1)\times \cM^{B'}(\x',\y;a_{j+1},\dots,a_n;w_2).
\]
Convergence to a two-story building is the usual notion from
symplectic field theory~\cite{BEHWZ03:CompactnessInSFT}: a sequence of
holomorphic curves
$u_i\in \cM^{B_1+B_2}(\x,\y;a_1,\dots,a_n;w_1+w_2)$ as in
Definition~\ref{def:moduli-embedded} converges to a two-story building
$(u^1,u^2)$ if,
as in~\cite[Definition 5.21]{LOT1}, the maps $\pi_\Sigma\circ u_i$ and
$\pi_\bD\circ u_i$ converge to $(\pi_\Sigma\circ u^1,\pi_\Sigma\circ u^2)$ and $(\pi_\bD\circ u^1,\pi_\bD\circ u^2)$ in the sense
of~\cite{BEHWZ03:CompactnessInSFT}. (Implicitly, we also are assuming
that the polygons $P_i\in \ModPol$ underlying the $u_i$ converge to
the broken polygon $(P,P')$ corresponding to $(u,u')$, so by the
definition of a coherent family of almost complex structures,
$J(P_i)\to J((P_1,P_2))$.)

\begin{proposition}\label{prop:glue-2-story}
  Fix a tailored family of almost complex structures $J$, generators $\x,\x',\y$, homology classes
  $B_1\in\pi_2(\x,\x')$, $B_2\in\pi_2(\x',\y)$, non-negative integers
  $w_1,w_2$, a sequence of basic algebra elements $\vec{a}=(a_1,\dots,a_n)$, and an
  integer $0\leq j\leq n$. Assume that
  $\ind(B_1+B_2,\vec{a},w_1+w_2)=2$. Near each two-story building $
  (u^1,u^2)\in \cM^{B_1}(\x,\x';a_1,\dots,a_j;w_1)\times
  \cM^{B_2}(\x',\y;a_{j+1},\dots,a_n;w_2)$ the 1-manifold
  $\cM^{B_1+B_2}(\x,\y;a_1,\dots,a_n;w_1+w_2)$ is homeomorphic
  to $(0,\infty)$.
\end{proposition}
\begin{proof}
  This follows from the same argument as for bordered
  $\HFa$~\cite{LOT1} (which in turn used the gluing lemma
  from~\cite[Appendix A]{Lipshitz06:CylindricalHF}).
\end{proof}

The definition of a sequence $u_i\in \cM^{B}(\x,\y;a_1,\dots,a_n;w)$
\emph{approaching a collision end} is also the same as in the case of
bordered $\HFa$~\cite[Section 5.4]{LOT1}, where the $u_i$ converge to a curve
consisting of a main component and a split curve at $e\infty$. That
is, we require that the sources of the $u_i$ converge to a nodal
Riemann surface with two components, one of which has $w$ interior
punctures and has boundary punctures corresponding to $\x$, $\y$, and
the algebra elements $a_1,\dots,a_ia_{i+1},\dots,a_n$, and the other
has three boundary punctures corresponding to $a_i$, $a_{i+1}$, and
$a_ia_{i+1}$ and no interior punctures. If the elements $a_i$ and
$a_{i+1}$ are Reeby then this is the case of \emph{degenerating a
  split curve}; if exactly one of $a_i$ is Reeby this is
\emph{degenerating a pseudo split curve}, and if neither is Reeby this
is \emph{degenerating a fake split curve}. For convergence to a split
or pseudo split
curve, we further require that the maps $\pi_\Sigma\circ u_i$ and
$\pi_\bD\circ u_i$ converge in the usual SFT
sense~\cite{BEHWZ03:CompactnessInSFT} to the corresponding projections
of a pair of curves $(u,u')$, where $\pi_\Sigma\circ u'$ is a split or
pseudo split
curve at $e\infty$. (Again, convergence of the sources implies that
the underlying polygons in $\ModPol$ also converge, and hence so does
the coherent family of almost complex structures.) The resulting curve
$u\in \cM^B(\x,\y;a_1,\dots,a_ia_{i+1},\dots,a_n;w)$ is the \emph{main
  component} of the limit. For convergence to a fake split curve,
the limit $\pi_\Sigma\circ u'$ is a constant map to some point in the
interior of $\Sigma$.

We give the gluing result for split curves here; gluing for pseudo
split and fake split curves is Proposition~\ref{prop:glue-pseudo-split}.

\begin{proposition}\label{prop:glue-split}
  Fix a tailored family of almost complex structures $J$, generators
  $\x,\y$, a homology class $B\in\pi_2(\x,\y)$, algebra elements
  $\vec{a}$, and a non-negative integer $w$ so that
  $\ind(B,\vec{a},w)=2$. Near each split curve end, the moduli space
  $\cM^{B}(\x,\y;\vec{a};w)$ has an odd number of ends.
\end{proposition}
\begin{proof}
  This follows from the same argument as the case of
  $\HFa$~\cite{LOT1}. Note that it follows from the way the energy is
  defined at points in the boundary of $\ModPol$
  (Definition~\ref{def:polygons}) and the definition of a coherent
  family of almost complex structures
  (Definition~\ref{def:AdmissibleJs}) that the almost complex
  structure on the main component of the limit curve agrees with the
  almost complex structure used to construct
  $\cM^B(\x,\y;a_1,\dots,a_ia_{i+1},\dots,a_n;w)$.
\end{proof}

The definition of orbit curve degenerations is similar to collision ends.
A sequence $u_i\in \cM^{B}(\x,\y;\vec{a};w)$,
\emph{degenerates an orbit curve} if the $u_i$ converge to a broken
curve with two components, one
an element of $\cM^B(\x,\y;a^1,\dots,\sigma,\dots,a^n;w-1)$,
where $\sigma$ is a Reeb chord of length $4$, and the other an orbit curve at $e\infty$.

\begin{proposition}\label{prop:glue-orbit}
  Fix a tailored family of almost complex structures $J$, generators $\x,\y$, a
  homology class $B\in\pi_2(\x,\y)$, basic algebra elements
  $a_1,\dots,a_n$, and a positive integer $w$ so that
  $\ind(B,a_1,\dots,a_n,w)=2$. Near each orbit curve
  degeneration, the moduli space $\cM(\x,\y;a_1,\dots,a_n;w)$ is
  homeomorphic to $(0,\infty)$.
\end{proposition}
\begin{proof}
  This follows from the same proof as, say, for join curves in the
  $\HFa$ case. Specifically, as long as both the moduli space of main
  components and the moduli space of orbit curves are transversely cut
  out, and their evaluation maps to $\RR^j/\RR$ are transverse, we can
  glue them~\cite[Proposition 5.25]{LOT1}. In this case, $j=1$ (it is
  the number of chords where the gluing occurs), so the condition that
  the evaluation maps be transverse is vacuous. The moduli space of
  main components is transversely cut out by assumption. The moduli
  space of orbit curves is transversely cut out because the moduli
  space of constant maps from a disk to $[0,1]\times\RR$ is
  transversely cut out, as is the moduli space of maps from the orbit
  curve component to $\RR\times \bdy\Sigma$.
\end{proof}

We have the following gluing result for pseudo split curves and fake
split curves:

\begin{proposition}\label{prop:glue-pseudo-split}
  With notation as in Proposition~\ref{prop:glue-orbit}, near each
  pseudo split curve degeneration or fake split curve degeneration,
  the moduli space $\cM(\x,\y;\vec{a};w)$ is homeomorphic
  to $(0,\infty)$.
\end{proposition}
\begin{proof}
  The proof is essentially the same as the proof of
  Proposition~\ref{prop:glue-orbit}.
\end{proof}

Next, we turn to gluing boundary degenerations. Simple boundary
degenerations are straightforward. A sequence $u_i$ \emph{converges to a
simple boundary degeneration end} if the corresponding family of polygons
$P_i\in\ModPol$ converges to a broken polygon where one component is a
module component and the other is an algebra-type component, and over the
algebra-type component the sequence $u_i$ converge to a simple boundary degeneration
(Definition~\ref{def:BoundaryDegeneration-simple}). (Here, convergence
of the $u_i$ has a similar meaning to the previous cases, but note
that the targets $\Sigma\times P_i$ are themselves splitting in the
process; compare~\cite[Section 9.1]{BEHWZ03:CompactnessInSFT}.)
\begin{proposition}\label{prop:glue-degen}
  Fix a tailored family of almost complex structures $J$, generators $\x,\y$, a
  homology class $B\in\pi_2(\x,\y)$, a sequence of basic algebra
  elements $\vec{a}$, and a positive integer $w$ so that
  $\ind(B,\vec{a},w)=2$.  Near each simple boundary
  degeneration (with the specified total homology class and
  asymptotics), the moduli space $\cM(\x,\y;\vec{a};w)$
  has an odd number of ends.
\end{proposition}
\begin{proof}
  Let $P_1$ and $P_2$ be the two components of the limit polygon (in
  $\ModPol$) and let $u_i\co S_i\to \Sigma\times P_i$ be the
  holomorphic curve over $P_i$. View $\Sigma\times P_i$ as having a
  cylindrical end at the node where $P_1$ and $P_2$ meet. Then, the
  proposition is the usual kind of gluing in symplectic field theory
  (e.g.,~\cite[Proposition A.2]{Lipshitz06:CylindricalHF}), except
  that, as usual for bordered Floer theory, this is a Morse-Bott
  problem, because of the $\RR$-coordinate in $[0,1]\times \RR\cong
  P_i$. So, we must verify that the evaluation maps to $\RR$ from the
  moduli space near $u_1$ is transverse to evaluation map to from the
  moduli space near $u_2$; but since the almost complex structure on $P_2$ is
  $\RR$-invariant, this is trivially satisfied.
\end{proof}

Convergence to composite boundary degeneration ends is similar, except
that for a composite boundary degeneration there are three components:
a module component $P_1$ connected to an algebra-type component $P_2$ with
three boundary marked points (corresponding to the join curve)
connected to another algebra-type component $P_3$ with $\geq 5$ boundary
marked points. The gluing result for composite boundary degenerations,
however, is not as straightforward: if we view the node between $P_1$
and $P_2$ (or between $P_2$ and $P_3$) as a cylindrical end, then the
holomorphic curve approaches that end at the same time that it approaches
the puncture of $\Sigma$. A product of two cylindrical ends is not a
cylindrical end, so this does not fit into the usual framework for SFT
gluing.

So, the proof of the gluing result for composite boundary degenerations
will be somewhat indirect. Before giving it, we recall some convenient
results from our monograph~\cite[Section 5.5]{LOT1} that let us use basic
properties of maps of stratified spaces to compute degrees.

We take a very weak notion of a stratified space: an $n$-dimensional
stratified space is a topological space $X$ written as a union of
disjoint subspaces $\{X_i\}_{i=0}^n$ so that each $X_i$ is a
topological manifold of dimension $i$ and the closure of $X_k$ is
contained in
$\bigcup_{i\leq k}X_i$, together with a choice of smooth structure on
each $X_i$. Given stratified spaces $X$ and $Y$, a \emph{stratified
  map} $f\co X\to Y$ is a continuous map so that for each stratum
$Y_i$, $f^{-1}(Y_i)$ is a union of connected components of strata of
$X$ and the restriction of $f$ to each $X_j\subset f^{-1}(Y_i)$ is a
smooth map $X_j\to Y_i$.

A map $f\co X\to Y$ of topological spaces is \emph{proper near
  $q\in Y$} if there is an open neighborhood $U\ni q$ so that
$f|_{f^{-1}(U)}\co f^{-1}(U)\to U$ is proper. If $f\co X\to Y$ is a
stratified map of $n$-dimensional stratified spaces which is proper
near $q\in Y$, the \emph{local degree} of $f$ near $q$ is the degree
at any regular value of $f|_{X_n}\co X_n\to Y_n$ near $q$.

Let $\RR_+^m=[0,\infty)^m$, viewed as an $m$-dimensional stratified
space in the obvious way.

\begin{lemma}
  \label{lem:StratifiedSpaces-moduli}
  If $f\co X\to \RR_+^m$ is a stratified map which is proper near $0$,
  and there is an $(m-1)$-dimensional face $\phi\subset \RR_+^m$ whose
  preimage under $f$, denoted $X^\phi$, forms the boundary of the
  manifold with boundary $X_m\cup X^\phi$, then the local degree of $f$
  near $0$ coincides with the local degree of $f$ restricted to
  $X^{\phi}$, mapping to a copy of $\RR_+^{m-1}$.
\end{lemma}
\begin{proof}
  This is essentially~\cite[Lemma~5.38]{LOT1}.
\end{proof}

\begin{proposition}\label{prop:glue-join}
  Fix a tailored family of almost complex structures $J$, generators $\x,\y$, a
  homology class $B\in\pi_2(\x,\y)$, a sequence of basic algebra elements
  $\vec{a}$, and a positive integer $w$ so that
  $\ind(B,\vec{a},w)=2$. Near each composite boundary
  degeneration (with the specified total homology class and
  asymptotics), the moduli space $\cM(\x,\y;\vec{a};w)$
  has an odd number of ends.
\end{proposition}
\begin{proof}
  Fix a composite boundary degeneration end
  \[
    \bigl(u_1\co S_1\to \Sigma\times[0,1]\times\RR,\ u_2\co S_2\to \Sigma\times\HHH,\ 
    v\co T\to \RR\times \bdy \Sigma\times[0,1]\times\RR\bigr),
  \]
  where $v$ is the join component. (Here, $(u_2,v)$ is a composite
  boundary degeneration from
  Definition~\ref{def:composite-bdy-degen}.)  For definiteness, assume
  we are considering Case~\ref{item:right-ext} of
  Definition~\ref{def:composite-bdy-degen} (the right extended
  case). Let $(a_1,\dots,a_{i-1},b,a_{j+1},\dots,a_n)$
  be the algebra elements on $u_1$ and $(a_i,\dots,a_{j-1},c)$
  the algebra elements on $u_2$. So, $v$ is asymptotic to $(c,b)$
  at $w\infty$ and $a_j$ at $e\infty$, and $a_j=cb$.

  \begin{figure}
    \centering
    \includegraphics{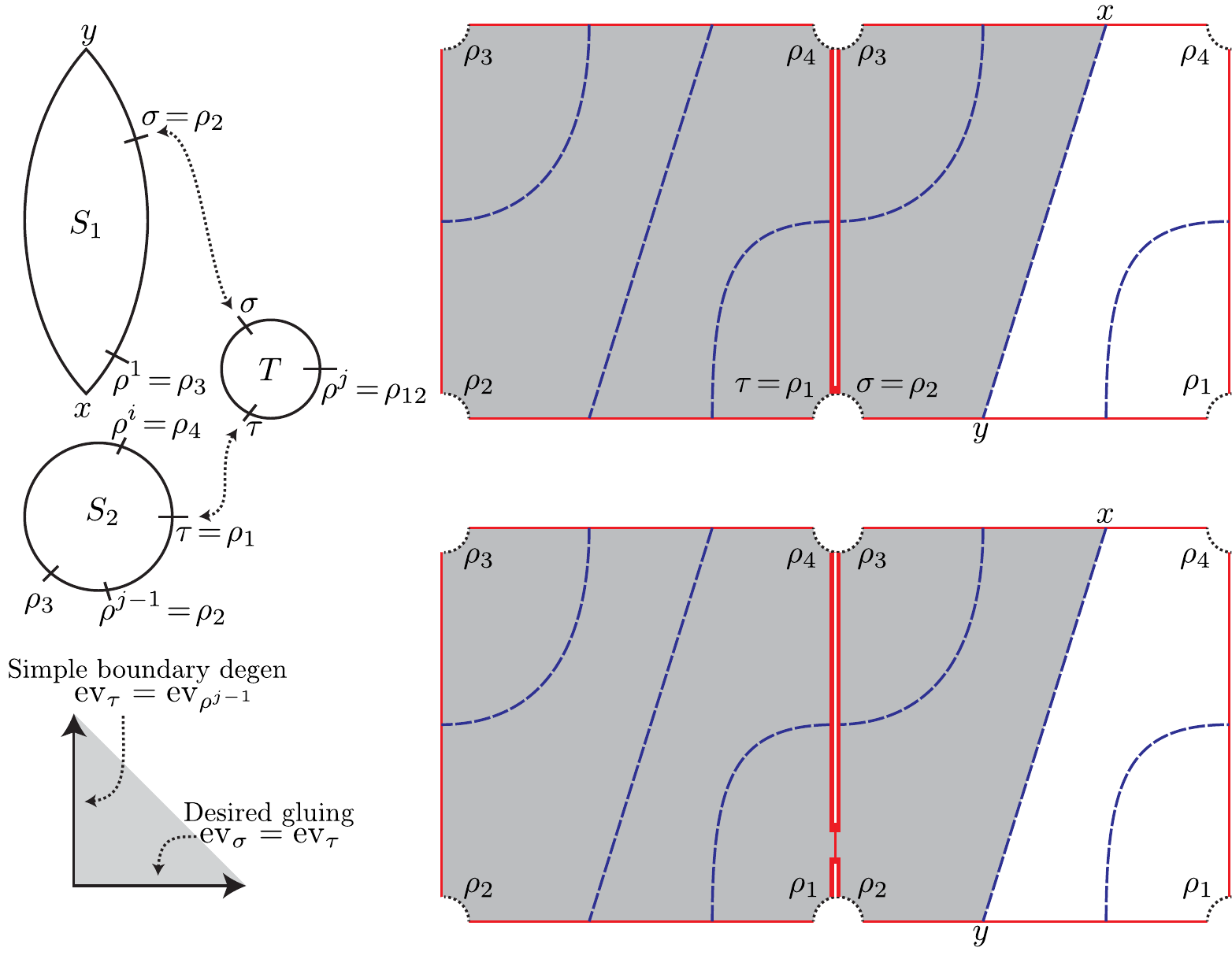}
    \caption[Gluing composite boundary degenerations]{\textbf{Gluing
        composite boundary degenerations.} This is an example of the
      argument in Proposition~\ref{prop:glue-join}. Near the composite
      boundary degeneration in the top row is a family of simple
      boundary degenerations. Gluing them gives a two-parameter family
      of curves, parameterized by the locations of the two boundary
      branch points in the lower-right figure. The desired gluings of
      the composite boundary degeneration are the boundary of this
      space where one cut has shrunk, i.e., $\ev_\sigma=\ev_\tau$.}
    \label{fig:GlueComposite}
  \end{figure}
  
  Instead of gluing $u_1$, $u_2$, and $v$ together directly, we will
  consider gluing $u_1$ to a family of simple boundary degenerations
  near $T$ and see $(u_1,u_2,v)$ as an end of a family arising that
  way, via Lemma~\ref{lem:StratifiedSpaces-moduli}. The argument is
  illustrated in Figure~\ref{fig:GlueComposite}.

  After translating in $\RR$, we can assume that $u_1$ is asymptotic
  to $\sigma$ at the point $(1,0)\in[0,1]\times\RR$. Then, for small
  enough $\epsilon$, for each $t\in(0,\epsilon)$, by
  Lemma~\ref{lem:bd-ev-proper} and
  Proposition~\ref{prop:extended-is-composite} there is a simple
  boundary degeneration $u_t$ through the points
  $(\pi_\bD\circ u_1)^{-1}(1,t)$ homotopic to $u_2$. By
  Proposition~\ref{prop:glue-degen} we can glue $u_1$ to $u_t$. This
  gives a two-parameter family of curves $u_{s,t}$ where $s$ is the
  gluing parameter. The algebra asymptotics of $u_{s,t}$ are
  $(a_1,\dots,a_{j-1},c,b,a_{j+1},\dots,a_n)$.

  Consider the compactification of the moduli space containing the
  $u_{s,t}$.  We have two functions from this compactified moduli
  space to $\RR_+$: the difference in $\RR$-coordinates
  $\ev_{c}-\ev_{a_{j-1}}$ and the difference in
  $\RR$-coordinates $\ev_{b}-\ev_{c}$. The product of these
  maps is a map to the first quadrant $\RR_+^2$. By considering the
  possible limits, this map to $\RR_+^2$ is proper near the
  origin. The subspace with $\ev_{c}-\ev_{a_{j-1}}=0$
  corresponds to the simple boundary degenerations $(u_1,u_t)$, so
  maps by a degree $1$ map to the positive $y$-axis. So, by
  Lemma~\ref{lem:StratifiedSpaces-moduli} the map to $\RR_+^2$ is
  locally degree $1$ near the origin. Hence, the restriction to the
  subspace where $\ev_b-\ev_c=0$ is also degree $1$ near the
  origin; but points in this subspace are the desired gluing of the
  composite boundary degeneration.
\end{proof}

\subsection{Compactness and the codimension-1 boundary}\label{sec:compactness}

We now have all the tools needed to prove Theorem~\ref{thm:master}.
Before doing so, we note that the counts used to define $\CFAa$ are finite:
\begin{lemma}\label{lem:0d-compactness}
  Fix generators $\x$ and $\y$, $B\in\pi_2(\x,\y)$, a sequence of
  algebra elements $\vec{a}$, and an integer $w$, so that
  $\ind(B,\vec{a},w)=1$. Then, for any tailored family of almost complex
  structures $J$, the moduli space $\cM^B(\x,\y;\vec{a};w)$ has
  finitely many elements.
\end{lemma}

\begin{proof}[Proof of Lemma~\ref{lem:0d-compactness} and Theorem~\ref{thm:master}]
  Fix $\x$, $\y$, $B$, $\vec{a}$, and $w$ as in the statement of
  Lemma~\ref{lem:0d-compactness} so that $\ind(B,\vec{a},w)\leq 1$ for
  Lemma~\ref{lem:0d-compactness} and $2$ for Theorem~\ref{thm:master}
  Fix also a decorated source $\Source$ with corresponding sequence of
  algebra elements $\vec{a}$ and $w$ Reeb orbits, and such that
  $\chi(S)=\chi_{\emb}(B,\vec{a},w)$.

  \emph{Step 1.} The general structure of the limit. Consider a sequence of
  holomorphic polygons $u_n$, with respect to almost complex
  structures $J(P_i)$ for some sequence of $P_i$. By compactness of
  $\ModPol$ (which is a quotient of the Deligne-Mumford moduli
  space), there is a subsequence so that the $P_i$ converge to some
  (possibly broken) polygon $P_\infty\in\ModPol$. Write the components
  of $P_\infty$ as $P_{\infty,1},\dots,P_{\infty,m}$. From the
  definition of a coherent family of admissible almost complex
  structures, the $J(P_i)$ for that subsequence converge to
  $J(P_\infty)$. The argument from ordinary bordered Floer
  homology~\cite[Proposition 5.23]{LOT1}, which boils down to applying
  symplectic field theory compactness first to the $\pi_{P}\circ u_n$,
  then to $\pi_\Sigma\circ u_n$ in a neighborhood of the puncture, and
  then to the $u_n$ on the complement of a smaller neighborhood, gives
  a convergent subsequence of the $u_n$. The limit has one component
  for each $P_{\infty,i}$ ($i=1,\dots,m$); these come in six kinds:
  
  \begin{itemize}
  \item \emph{Main components}, which are maps to
    $\Sigma\times P_{\infty,i}$ where $P_{\infty,i}$ is a module
    component. In particular, part of the boundary of these components
    is mapped to $\betas\times \{0\}\times\RR$.
  \item \emph{Boundary degeneration components}, which are maps to
    $\Sigma\times P_{\infty,i}$ where $P_{\infty,i}$ is an algebra
    polygon. In particular, all of the boundary of these components is
    mapped to $\alphas\times \bdy P_{\infty,i}$.
  \item \emph{$E\infty$ components}, which are mapped to $(\RR\times \bdy
    \Sigma)\times P_{\infty,i}$.
  \item Fake \emph{$e\infty$ components}, which are mapped to
    $\Sigma\times P_{\infty,i}$ where the map to $\Sigma$ is constant
    (which is necessarily to a point on
    $\alpha_1^a\cup\alpha_2^a$). (This happens, for instance, if a
    pair of boundary marked points labeled $U^a$ and $U^b$ collide;
    see the discussion of fake split curves in
    Section~\ref{sec:e-infty}.)  Below, we will abuse terminology and
    also refer to these as components at $e\infty$, because they play
    the same role in the argument.
  \item \emph{Disk bubbles}, which are maps to $\Sigma\times
    P_{\infty,i}$ so that the map to $P_{\infty,i}$ is a constant map
    to the boundary while in $\Sigma$ all of
    the boundary of the component is mapped to one of $\alpha_1^a\cup
    \alphas^c$, $\alpha_2^a\cup\alphas^c$, or $\betas$.
  \item \emph{Sphere bubbles}, which have no boundary and are mapped
    to $P_{\infty,i}$ by a constant map.
  \end{itemize}
  These components are glued together by a \emph{gluing graph}, whose
  vertices correspond to the components and whose edges indicate how
  the components intersect or, more accurately, are glued together in
  $u_{n_i}$ for $i\gg 0$. It follows from boundary monotonicity (or,
  essentially equivalently, from the symmetric product formulation in
  Section~\ref{sec:tautological}) that this gluing graph is, in fact,
  a tree (compare~\cite[Lemma 5.57]{LOT1}). There are also some other degenerations which we can
  immediately see do not occur. Pairs of simple Reeb orbits colliding
  has codimension $2$ (see Formula~\ref{prop:index-source}
  or~\ref{prop:emb-ind}). Formation of interior nodes (collapsing
  circles in the source~$S$) also has codimension $2$: this increases $\chi$ by
  $2$, so decreases Formula~\ref{prop:index-source} by $2$. Arcs in
  $S$ collapsing to form nodes without splitting at $\pm\infty$ or
  $e\infty$ or creating a boundary degeneration violates strong
  boundary monotonicity, as in the $\HFa$ case~\cite[Lemma 5.50]{LOT1}.

  \emph{Step 2.} Applying the index formula.  
  Consider a broken holomorphic curve with a main component $u_0$
  (which may be broken at $\pm\infty$),
  boundary degenerations
  $u_1,\dots,u_b$, and $e\infty$ components
  $v_1,\dots,v_c$, where $v_1,\dots,v_c$ have a total of $m$ $w\infty$
  algebra elements and $w'$ $w\infty$ Reeb orbits, with total ramification
  $r$ (of the Reeb orbits). Recall that there are
  $k$ algebra elements and $w$ simple Reeb orbits at (far) east
  $\infty$. After gluing together these components, we have
  \begin{align*}
    \ind(B,\Source,\vec{a};w,0)&=g-\chi(S)+2e(B)+k+w\leq 1\text{ or }2\\
    \chi(S)&=\chi(S_0)+\cdots+\chi(S_b)-bg\\
    e(B)&=e(B_0)+e(B_1)+\cdots+e(B_b).
  \end{align*}
  Thus,
  \begin{multline*}
    \ind(B_0,\SourceSub{0},\vec{a}_0;w_0,0)+\ind(B_1,\bdSource_1,\vec{a}_1;w_1,0)+\cdots+\ind(B_b,\bdSource_b,\vec{a}_b;w_b,0)\\
    +k-m+w-w'+r
    \leq 1\text{ or }2.
  \end{multline*}
  By Lemma~\ref{lem:bdy-deg-ind-2},
  each $\ind(B_i,\bdSource_i,\vec{a}_i;w_i,0)$, $i>0$, is at least $2$, and since
  the gluing graph is a tree, $k-m\geq -b$. Since Reeb orbits are
  never created (though they can collide or hit the boundary and turn
  into chords via orbit curves), $w-w'\geq 0$ and, in fact, $w-w'\geq r\geq 0$. Also,
  $\ind(B_0,\SourceSub{0},\vec{a}_0;w_0,0)\geq 0$. Thus, the left side is
  bounded below by $2b-b$, so $b\leq 1$ or $2$.

  \emph{Step 3.} Classifying the possibilities in the index $1$ case.
  In the index $1$ case, $b\leq 1$. If $b=1$ then we must
  have $\ind(B_0,\SourceSub{0},\vec{a}_0;w_0,0)=0$, so in fact $u_0$ is
  constant. Then, the curve consists of a single boundary degeneration,
  which has index at least $2$, a contradiction. If $b=0$ then
  $\ind(B_0,\SourceSub{0},\vec{a}_0;w_0,0)=1$ so $k-m=0$ and
  $w-w'+r=0$. Thus, there are also no curves at $e\infty$ so the limit
  is an element of $\cM^B(\x,\y;\Source)$, not a broken holomorphic
  curve. This proves Lemma~\ref{lem:0d-compactness}.

  \emph{Step 4.} Classifying the possibilities in the index $2$ case.
  In the index $2$ case, the remaining possibilities are:
  \begin{itemize}
  \item $\ind(B_0,\SourceSub{0},\vec{a}_0;w_0,0)=2$, $b=0$, $k=m$, and $w=w'$. In
    this case, there are no boundary degenerations or curves at
    $e\infty$. So, there are the following subcases:
    \begin{itemize}
    \item A 2-story holomorphic building.
    \item A constant map and a sphere bubble.
    \item A constant map and a disk bubble.
    \end{itemize}
    The latter two kinds of ends do not contribute to the count by
    properties~\ref{item:tail-spheres} and~\ref{item:tail-disk} of the
    definition of a tailored family of almost complex structures, and
    Lemma~\ref{lem:OS-no-disks}.
  \item $\ind(B_0,\SourceSub{0},\vec{a}_0;w_0,0)=1$, $b=0$, $k=m+1$, and $w=w'$. This
    corresponds to degenerating a split curve or a pseudo split curve at
    $e\infty$, or to a fake split curve; all of these cases are
    collision ends. (This uses
    the fact that if $b=0$ no component at $e\infty$ can have more
    than one $w\infty$ chord.)
  \item $\ind(B_0,\SourceSub{0},\vec{a}_0;w_0,0)=1$, $b=0$, $k=m$, and
    $w=w'+1$. This corresponds to degenerating an orbit curve. (In
    particular, we must have $r=0$.)
  \item $\ind(B_0,\SourceSub{0},\vec{a}_0;w_0,0)=1$, $b=1$, $k=m-1$, and $w=w'$. This
    corresponds to a composite boundary degeneration. (Again, this
    uses the fact that the gluing graph is a tree to rule out more
    complicated components at $e\infty$.)
  \item $\ind(B_0,\SourceSub{0},\vec{a}_0;w_0,0)=0$, $b=1$, $k=m$, and $w=w'$. In
    this case, the main component is constant, so in particular has no
    algebra elements on it. Thus, since the gluing graph is a tree, there
    are no components at $e\infty$, either. Thus, this case
    corresponds to a simple boundary degeneration, connected to a
    constant main component.
  \item $\ind(B_0,\SourceSub{0},\vec{a}_0;w_0,0)=0$, $b=0$, $k=m$, and
    $w=w'+2$ or $w=w'+1$ and $r=1$. These cases are excluded by the
    same reasoning as the previous case.
  \item $\ind(B_0,\SourceSub{0},\vec{a}_0;w_0,0)=0$, $b=2$, $k=m-2$, and $w=w'$. This
    is prohibited by the fact that the gluing graph is a tree.
  \end{itemize}
  
  So, we have shown that the degenerations which can contribute are of
  the types specified in the theorem. Conversely,
  Propositions~\ref{prop:glue-2-story},~\ref{prop:glue-split},~\ref{prop:glue-orbit},~\ref{prop:glue-pseudo-split},~\ref{prop:glue-degen},
  and~\ref{prop:glue-join} show that each of these degenerations does
  occur (with multiplicity one).
\end{proof}


\section{The algebra of boundary degenerations}\label{sec:algebra}
\label{sec:alg-bdeg}

Let $\Sigma$ be a surface of genus $g$.  For each energy $E$, fix a
complex structure $j_E$ on $\Sigma$ which is regular for energy $E$
boundary degenerations, in the sense of Definition~\ref{def:RegularJ}.
The counts from Definition~\ref{def:CurveCount} can be organized into
a weighted $\Ainf$-algebra, as follows.

\begin{definition}
  \label{def:AinfAlgebraOfHeegaardSurface}
  Let $\Alg=\AsUnDefAlg$ be the associative algebra from
  Section~\ref{sec:build-alg-combinatorially}. The
  multiplication is an operation $\mu_2^0\co \Alg\otimes\Alg\to\Alg$. Define
  \[
    \mu_0^1=\rho_{1234}+\rho_{2341}+\rho_{3412}+\rho_{4123}
  \]
  and for $w>1$ let $\mu_0^w=0$. For non-negative integers $(w,n)$
  with $n+2w>2$ and $n\geq 1$, let
  $\mu_n^w\co \Alg[U]^{\otimes n}\to \Alg[U]$ be the map characterized by
  \begin{equation}\label{eq:AlgSurf}
    \begin{split}
    \mu_n^w(a_1,\dots, a_n)&=
    \#\cN([\alpha_i^a];a_1,\dots,a_n;w)
    \iota_i U^{n_z+m} \\
    &+ 
    \sum_{\substack{a_1=\rho b\\\indsq(b,a_2,\dots,a_n,w)=2}} \#\cN([\alpha_i^a];b,a_2,\dots,a_n;w) 
    \rho U^{n_z+m}  \\
    &+ 
    \sum_{\substack{a_n=b\rho\\\indsq(a_1,\dots,a_{n-1},b,w)=2}} \#\cN([\alpha_i^a];a_1,a_2,\dots,a_{n-1},b;w)
    \rho U^{n_z+m}
    \end{split}
  \end{equation}
  where the first term only occurs if $\indsq(\vec{a},w)=2$, and the
  moduli spaces are defined by using $J_{E(a_1,\dots,a_n,w)}$ (cf.\
  Formula~\eqref{eq:ExtendEnergy}). In each
  case, the $\alpha$-arc $\alpha_i^a$ is determined by the sequence of
  algebra elements (which must be Reeby for the moduli space to be
  non-empty, by Lemma~\ref{lem:bdy-deg-ind-2}).
  The number $n_z$ is the degree with which the curve covers $z$,
  i.e., the degree of the map to $\Sigma$, and $m$ is
  the total $U$-power in $a_1,\dots,a_n$.

  We call $(\Alg[U],\{\mu_n^w\})$
  the \emph{algebra associated to $\Sigma$ and $\{j_E\}$}.
\end{definition}

Our aim in this section is to prove the following precise form of
Theorem~\ref{thm:alg-well-defd}:

\begin{theorem}
  \label{thm:AlgOfSurface}
  For any energy $E$, there exists a choice of regular $j_E$ for
  energy $E$, in the sense of
  Definition~\ref{def:RegularJ}. Moreover, for any such
  choices $j_E$, the operations $\mu^w_n$ in
  Definition~\ref{def:AinfAlgebraOfHeegaardSurface} coincide with the
  weighted $\Ainf$-algebra associated to the torus, as defined
  in~\cite{LOT:torus-alg}.
\end{theorem}
In particular, Theorem~\ref{thm:AlgOfSurface} implies that the
operations are $U$-equivariant, or equivalently the counts of the
moduli spaces satisfy
\begin{equation}
  \label{eq:Noduli-U-equivariant}
  \#\cN([\alpha_i^a];U^{i_1}\rho^1,\dots,U^{i_n}\rho^n;w)=\#\cN([\alpha_i^a];\rho^1,\dots,\rho^n;w).
\end{equation}

In Equation~\eqref{eq:AlgSurf} of
Definition~\ref{def:AinfAlgebraOfHeegaardSurface}, the last two sums
on the right correspond to counting composite boundary
degenerations, according to the following:
\begin{proposition}\label{prop:extended-is-composite}
  Fix a regular $j$, a non-negative integer
  $w$, a sequence basic algebra elements $\vec{a}=(a_1,\dots,a_n)$ so
  that $\indsq(\vec{a},w)=2$, and a factorization
  $a_1=b c$. For a generic tuple
  $\w=\{w_i\in \alpha_i^c\cap \beta_{\sigma(i)}\}$, the space of
  composite boundary degenerations
  $\#\cN(b;a_1,\dots,a_n;w)$ asymptotic to $b\times\w$
  is transversely cut out, and the count of simple boundary
  degenerations $\#\cN([\alpha_i^a];c,a_2,\dots,a_n;w)$ agrees
  modulo 2 with the count of composite boundary degenerations
  $\#\cN(b;a_1,\dots,a_n;w)$. (As usual, $\alpha_i^a$ is
  determined by the other data.)

  A similar statement holds for boundary degenerations of the form
  $\cN([\alpha_i^a];a_1,\dots,a_{n-1},b;w)$ for any factorization
  $a_n=bc$.
\end{proposition}

Theorem~\ref{thm:AlgOfSurface} follows from the following assertions, verified below:
\begin{enumerate}[label=(a-\arabic*),ref=(a-\arabic*)]
\item \label{a:IndepOfJ} The holomorphic curve counts
  $\#\cN([\alpha_i^a];\vec{a};w)$ are independent of the choice of
  regular $j$.
\item \label{a:Regular} There are choices of $j$ that are regular.
\item \label{a:IndepOfGenus} The holomorphic curve counts are
  independent of the choice of genus $g$ or, equivalently, the choice
  of Heegaard surface $\Sigma$.
\item \label{a:IdentifyWithTorusAlg} If $g=1$, the holomorphic curve
  counts agree with the coefficients in the weighted $\Ainf$
  operations for the torus algebra $\MAlg$ as defined combinatorially
  in our previous paper~\cite{LOT:torus-alg}.
\end{enumerate}

\begin{remark}
  It follows that the operations $\mu^w_n$ endow $\Alg$ with the
  structure of a weighted $\Ainf$-algebra. This could
  be proved directly using the compactification of the spaces of
  pseudo-holomorphic curves, in the spirit of Section~\ref{sec:CFA};
  but this is not necessary, in view of the identification from
  Step~\ref{a:IdentifyWithTorusAlg}, since a combinatorial proof of
  the $\Ainf$-relations for the torus algebra is given
  in~\cite[Theorem~\ref{TA:thm:AinftyAlgebra}]{LOT:torus-alg}.
\end{remark}

As a step towards establishing~\ref{a:IndepOfJ}, we set up a suitable
generalization of the complex structures used on $\Sigma\times\HHH$
and study boundary degenerations for this generalization. This is done
in Section~\ref{sec:GenBDeg}. We then
verify~\ref{a:IdentifyWithTorusAlg} in
Section~\ref{sec:GenusOne}. (This is almost immediate from the
definition of $\MAlg$.)
The theorem then follows quickly from the stabilization invariance of
boundary degeneration counts, which is proved in
Section~\ref{subsec:StabilizationInvarianceBDeg}.

The $\beta$-circles on the bordered Heegaard diagram are irrelevant to
the moduli spaces considered in this section. So, we will typically
fix a \emph{partial Heegaard diagram} $(\Sigma,\alphas)$ consisting of
a compact, connected, oriented surface $\Sigma$ with some genus $g$
and one boundary component, $g-1$ circles
$\alpha_1^c,\dots,\alpha_{g-1}^c$, two arcs $\alpha_1^a,\alpha_2^a$
with boundary on $\bdy\Sigma$, and a basepoint $z\in \bdy\Sigma$, so
that the $\alpha$ arcs and circles are pairwise disjoint and linearly
independent in $H_1(\Sigma,\bdy\Sigma)$ and $z$ is just before one of
the endpoints of $\alpha_1^a$ with respect to the orientation of
$\bdy\Sigma$. As usual, we will also use $\Sigma$ to denote
$\Sigma\setminus\bdy\Sigma$, which we view as a punctured Riemann
surface, and $\alpha_i^c$ to denote the corresponding non-compact arc
in this punctured surface. In particular, moduli spaces of holomorphic
curves mean curves in this punctured Riemann surface, not the surface
with boundary.

\subsection{Generalized boundary degenerations}
\label{sec:GenBDeg}

To establish~\ref{a:IndepOfJ}, we need to generalize somewhat the
almost complex structures we allow on $\Sigma\times \HHH$, beyond
those considered in Section~\ref{sec:bdy-degen}.

\begin{definition}\label{def:ac-HHH}
  An almost complex structure $J$ on $\Sigma\times\HHH$ is
  \emph{admissible to energy $E_0$} if it satisfies the following conditions:
  \begin{enumerate}[label=(JH-\arabic*)]
  \item The projection map
    $\pi_\HHH\co \Sigma\times \HHH \to \HHH$ is $J$-holomorphic.
  \item For each $(p,\zeta)\in \Sigma\times\HHH$, $J$ preserves the
    tangent space $T_\zeta\HHH\subset T_{(p,\zeta)}(\Sigma\times\HHH)$.
  \item The complex structure is split, as $j_\Sigma\times j_\bD$,
    over a neighborhood of the puncture of $\Sigma$.
  \item Each of the complex structures on $\Sigma\times\{\zeta\}$ induced
    by $J$ is $\eta(E_0)$-pinched in the sense of
    Definition~\ref{def:ePinched}, for a sufficient pinching function
    $\eta$ (Definition~\ref{def:sufficient-pinching}).
  \end{enumerate}
\end{definition}  

Sufficiently pinched product complex structures on $\Sigma\times\HHH$
are admissible. For an admissible complex structure $J$,
we can formulate the notion of boundary degenerations analogous to
Definition~\ref{def:BoundaryDegeneration-simple}, with the additional condition:
\begin{enumerate}[label=(BDs-\arabic*),ref=(BDs-\arabic*),resume, series=BDs]
\item If $\vec{a}=(a_1,\dots,a_n)$ is the sequence of basic algebra elements
  $u$ is asymptotic to, then $u$ is asymptotic to $a_1$ at
  $(1,0)\in\bdy\HHH$ and to $a_2$ at $(1,1)\in\bdy\HHH$.
\end{enumerate}
(If $J$ splits as a product $j_\Sigma\times j_\HHH$,
this condition is equivalent to dividing out by
the action of the (two-dimensional) group of conformal automorphisms of $\HHH$.
For non-split $J$, the conformal group no longer acts holomorphically.)
These more general boundary degenerations can be
collected into moduli spaces, which we denote
$\cN(J;\x;\vec{a};w)$, and then used to
define operations $\mu_n^w$ associated to $\Sigma$ and $J$, immediately
generalizing Definition~\ref{def:AinfAlgebraOfHeegaardSurface}. As
in the split case, let
\[
  \cN(J;*;\vec{a};w)=\bigcup_\x\cN(J;\x;\vec{a};w).
\]
Also as in the split case, let 
\[
  \cN(J;[\alpha_i^a];\vec{a};w)
\]
be the preimage of a regular value of the evaluation map from
$\cN(J;*;\vec{a};w)$ to $\alpha_i^a$ (where, as usual, $\alpha_i^a$ is
the $\alpha$-arc containing the initial endpoint of $a_1$ and the
terminal endpoint of $a_n$).

For admissible almost complex structures, we have the following
transversality package, which is also the key step in proving that
the count of $\cN(J;[\alpha_i^a];\vec{a};w)$ is well-defined:

\begin{theorem}
  \label{thm:RegularityOfBoundaryDegenerations}
  For a generic $J$ which is admissible to energy $E_0$, if
  $E(\vec{a},w)\leq E_0$ then the moduli spaces
  $\cN(J;*;\vec{a};w)$ with 
  \[
    \indsq(\vec{a},w)\leq 2
  \]
  are manifolds of dimension $\indsq(\vec{a},w)+g-2$.
  Further, any two such generic $J_1$ and $J_2$ can be connected by a
  path $\{J_t\}_{t\in [1,2]}$ so that 
  the moduli spaces
  \[
    \cN(\{J_t\};*;\vec{a};w) \coloneqq \bigcup_{t\in[1,2]}
    \cN(J_t;*;\vec{a};w)
  \]
  is a manifold-with-boundary, of dimension $\indsq(\vec{a},w)+g-1$.
\end{theorem}
We call almost complex structures as in the first part of
Theorem~\ref{thm:RegularityOfBoundaryDegenerations} \emph{regular}.
\begin{proof}
  The proof is similar to the proof of
  Proposition~\ref{prop:transv-exists} and is left to the reader.
\end{proof}

\begin{remark}
  Although Theorem~\ref{thm:RegularityOfBoundaryDegenerations}
  establishes the existence of admissible almost
  complex structures $J$ that are regular, in a sense analogous to
  Definition~\ref{def:RegularJ}, these $J$ need not be split.  The
  fact that we can take $J$ to be split will be established in
  Section~\ref{subsec:StabilizationInvarianceBDeg}.
\end{remark}

From Theorem~\ref{thm:RegularityOfBoundaryDegenerations}, we conclude
the following:

\begin{proposition}
  \label{prop:BoundaryDegenerationsIndepOfJ}
  Consider a sequence of basic algebra elements $\vec{a}$ and an integer
  $w$ so that $\indsq(\vec{a},w)=2$. Then, for any regular
  $J$ the moduli space $\cN(J;[\alpha_i^a];\vec{a};k)$
  consists of finitely many points, and the modulo-2 count
  $\#(\cN(J;[\alpha_i^a];\vec{a};w))$ is independent of the
  choice of regular, admissible (to energy $E(\vec{a},w)$) $J$.
\end{proposition}
\begin{proof}
  Both statements reduce to compactness results. For finiteness, by a
  similar argument to the proof of Lemma~\ref{lem:bd-ev-proper}, the
  evaluation map $\ev\co \cN(J;*;\vec{a};k)\to
  \alpha_i^a$ is proper, and hence the preimage of any point, i.e., 
  $\cN(J;[\alpha_i^a];\vec{a};k)$, is compact. For the
  second statement, 
  if $J_1$ and $J_2$ are both regular, we connect them by a regular
  path $\{J_t\}_{t\in[1,2]}$ in the sense of  
  Theorem~\ref{thm:RegularityOfBoundaryDegenerations}, so the
  resulting moduli space $\cN(\{J_t\};*;\vec{a};w)$ is a
  manifold-with-boundary.  By a similar argument to the proof of
  Lemma~\ref{lem:bd-ev-proper} (the key step in which is
  Lemma~\ref{lem:bdy-deg-ind-2}), the space
  $\cN(\{J_t\};*;\vec{a};w)$ is compact. The evaluation
  map extends continuously to
  $\cN(\{J_t\};*;\vec{a};w)$. Hence, its degree
  restricted to $t=1$ coincides with its degree restricted to $t=2$;
  i.e.,
  \[
    \#\cN(J_1;[\alpha_i^a];\vec{a};w)=
    \#\cN(J_2;[\alpha_i^a];\vec{a};w),
  \]
  as desired.
\end{proof}

Proposition~\ref{prop:BoundaryDegenerationsIndepOfJ} implies that the
moduli spaces are $U$-equivariant, in the sense of
Formula~\eqref{eq:Noduli-U-equivariant}, since the $U$-powers in the
algebra elements only affect the moduli spaces through which complex
structures are considered. So, for the rest of the section, it will
suffice to consider moduli spaces
$\#\cN(J;[\alpha_i^a];\rho^1,\rho^2,\dots,\rho^n;w)$ where the
asymptotics are Reeb chords.

\subsection{The genus-one case}
\label{sec:GenusOne}

In~\cite{LOT:torus-alg}, we gave several definitions of a weighted
$\Ainf$-algebra $\MAlg(T^2)$; we recall the one which is most directly
connected with the pseudo-holomorphic curves.
For this section, it
suffices to consider only split complex structures on
$\Sigma\times\HHH$, as in Section~\ref{sec:bdy-degen}, not the more
general admissible almost complex structures of
Definition~\ref{def:ac-HHH}. (See also
Proposition~\ref{prop:IdentifyAlgebraGenusOne} below.) 

As usual, equip the torus $T^2$
with two embedded curves $\alpha_1^a$ and $\alpha_2^a$ that intersect in a
single, transverse intersection point $p$,  i.e., the result of
collapsing the dotted circle in the left or middle of Figure~\ref{fig:curves-on-surfs}.
The operations in the algebra count maps $f\co \Delta\to T^2$ of
the disk to the torus  with the
following properties:
\begin{itemize}
  \item The boundary of the disk is mapped to $\alpha_1^a\cup\alpha_2^a$.
  \item The map $f$ is an immersion except for possible branching
    at points of $\bdy \Delta$ that map to $p$.
\end{itemize}
In the operations, the weight $w$ counts the number of preimages of
$p$ in the interior of $\Delta$, and the sequence of algebra elements
corresponds to the corners around $p$, in the cyclic order they appear
along $\partial \Delta$. Equivalently, the operations
count tilings of the disk by squares marked as on the left of
Figure~\ref{fig:curves-on-surfs}, where the tiling is arranged so that
exactly four tiles meet at each corner in the interior of $\Delta$.

Counts of such maps can be given an explicit, combinatorial
description in terms of planar graphs, as in our previous paper~\cite{LOT:torus-alg}. (See
especially~\cite[Section~\ref{TA:subsec:Immersions}]{LOT:torus-alg} for the
correspondence between the two points of view.)

\begin{proposition}
  \label{prop:IdentifyAlgebraGenusOne}
  When $g=1$, any choice of complex structure $j$ on $\Sigma$ is regular, and the operations
  $\{\mu^w_n\co \Alg^{\otimes n}\to \Alg[U]\}$ coincide with the weighted
  operations for $\MAlg$. In particular, the operations $\mu^w_n$ make
  $\Alg[U]$ into a weighted algebra.
\end{proposition}

\begin{proof}
  Consider a boundary degeneration
  $u\co \DegSource \to \Sigma\times \HHH$. Since $g=1$,
  $\pi_\HHH\circ u$ is one-to-one, so $\DegSource$ is planar.  The
  condition that $\ind(u)=2$ ensures that the branching of
  $f=\pi_\Sigma\circ u$ occurs at the punctures of $S$. Specifically,
  similar to the proof of Proposition~\ref{prop:emb-ind}, the
  Riemann-Hurwitz formula gives
  $1-w=\chi(S)=e(B)+|\vec{\rho}|/2-\br(u)$ where $B$ is the domain
  $\br(u)$ is the total ramification of $\pi_\Sigma\circ u$. If
  $B=k[\Sigma]$, this reduces to $1-w=-k+|\vec{\rho}|/2-\br(u)$. On
  the other hand, the index formula from
  Proposition~\eqref{eq:index-source-bdy-degen} gives
  $2=1-\chi(S)-2k+|\vec{\rho}|+w$; since $\chi(S)=1-w$, this reduces
  to $2k+2=2w+|\vec{\rho}|$. Combining these two gives $\br(u)=0$.

  The complex structure on $\DegSource$ is determined by the complex
  structure on $\Sigma$ and the branched cover $f$.  This gives the
  one-to-one correspondence between the boundary degenerations counted
  in the $\Ainf$-algebra operations associated to the Heegaard surface
  with genus one (as in
  Definition~\ref{def:AinfAlgebraOfHeegaardSurface}) and the branched
  covers counted in the $\Ainf$-algebra operations in the torus
  algebra~\cite{LOT:torus-alg}.

  Regularity of (any) $j$ follows, for instance, from the fact that
  the source curves are punctured disks and Hofer-Lizan-Sikorav's
  automatic transversality argument~\cite{HLS97:GenericityHoloCurves}
  (as in~\cite[Proposition~\ref{LOT1:prop:east_transversality}]{LOT1}).
\end{proof}

\subsection{Stabilization invariance of boundary degeneration counts}
\label{subsec:StabilizationInvarianceBDeg}

The algebra associated to $\Sigma$ and $J$ was shown to be independent
of $J$ in Proposition~\ref{prop:BoundaryDegenerationsIndepOfJ}.  In
this section, we show that it is also independent of the Heegaard
surface $\Sigma$. This is done by a familiar neck-stretching argument,
analogous to the one from the cylindrical proof of stabilization invariance~\cite[Section~12]{Lipshitz06:CylindricalHF}.
As in the forthcoming book on Heegaard Floer homology~\cite{OSSz},
we stabilize at two points rather than one, in the
interest of technical simplicity; see Figure~\ref{fig:stab-alg}.

We will continue to work in the setting of admissible but not
necessarily split almost complex structures, as in
Definition~\ref{def:ac-HHH}. We require the almost complex structure
to be split near the points where the stabilization occurs, in the
following sense:
\begin{definition}
  \label{def:ProductLike}
  Fix a finite set of points $\{p_i\}_{i=1}^m\subset \Sigma$.  An
  admissible almost complex structure $J$ on $\Sigma\times\HHH$ is said to
  be {\em product-like near $\{p_i\}_{i=1}^m$} if there is a
  neighborhood $D_1\amalg \dots\amalg D_m\subset \Sigma$ with
  $p_i\in D_i$, so that the restriction $J|_{(D_1\amalg\dots\amalg
    D_m)\times\HHH}$ is split, as the product of a complex structure
  on $D_i$ and the standard complex structure on $\HHH$.
\end{definition}

\begin{definition}\label{def:attach-cyl}
  Fix a partial Heegaard diagram $(\Sigma,\alphas)$ of genus $g$,
  points $z_R$ and $w_R$ in the interior of $\Sigma$, and an
  admissible almost complex structure which is product-like near $z_R$
  and $w_R$. Fix also disjoint disk neighborhoods of $z_R$ and $w_R$
  so that $J$ is split over these neighborhoods and a real parameter
  $T$.  There is an associated almost complex structure $J(T)$ on the
  surface $\Sigma'$ of genus $g+1$, obtained by attaching a cylinder
  with length $T$ to $\Sigma$ at the points $z_R$ and $w_R$, using the
  chosen disk neighborhoods. That is, if $D_z$ and $D_w$ are the
  chosen disks around $z_R$ and $w_R$, with some fixed identifications
  with the unit disk $D^2\subset \CC$ (identifying $z_R$ and $w_R$
  with the origin), then we build $(\Sigma', J(T))$ by deleting the
  open disks of radius $1/2$ from $D_z$ and $D_w$ and gluing the
  complex cylinder $([0,\pi]\times [0,T/2])/(0,t)\sim(\pi,t)$ to the
  result by the canonical identifications.  View $\Sigma'$ as a
  partial Heegaard diagram via the $g+1$ curves from $\Sigma$,
  $\alpha_1^a,\alpha_2^a,\alpha_1^c,\dots,\alpha_{g-1}^c$, and an
  additional circle $\alpha_g^c$ which contains one arc that runs
  through the attached cylinder as in Figure~\ref{fig:stab-alg}.
\end{definition}
Note that the point $z_R$ in Definition~\ref{def:attach-cyl} is
different from the point $z$ in (or near) $\bdy\Sigma$ in the
definition of a Heegaard diagram.

\begin{figure}
  \centering
  \includegraphics{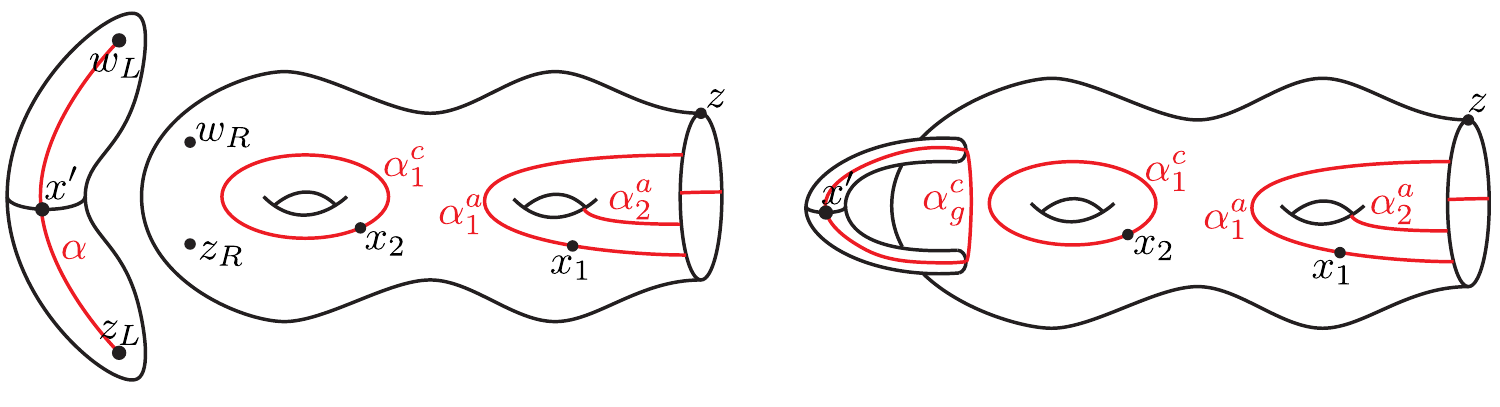}
  \caption[Stabilizing boundary degenerations]{\textbf{Stabilizing at two points.} Left: the sphere
    $\Sphere$ and the surface $\Sigma$. Right: the stabilized surface
    $\Sigma'$, obtained by gluing $\Sigma$ and $\Sphere$ at two
    points.}
  \label{fig:stab-alg}
\end{figure}

Fix a point $x'\in\alpha_g$ which is at the center of the attached cylinder.
Given $\x\in (\alpha^a_1\cup\alpha^a_2)\times \alpha_1\times\dots\times
\alpha_{g-1}$, there is an associated
$\x'\in(\alpha^a_1\cup\alpha^a_2)\times \alpha_1\times\dots\times
\alpha_{g-1}\times \alpha_{g}$ that agrees with $\x$ on the first
$g$ factors, and whose component on $\alpha_g$ is $x'$.

Let $\Sphere$ denote the 2-sphere, and fix points $w_L,z_L\in \Sphere$
and an arc $\alpha\subset\Sphere$ of a great circle connecting $w_L$
and $z_L$.

\begin{definition}\label{def:stabilizing-degen}
  Fix a point $x'\in\alpha$.
A {\em stabilizing boundary degeneration at $x'$} is a 
connected Riemann surface $\DegSource$ and a map
$u\co (\DegSource,\bdy\DegSource) \to ((\Sphere\setminus\{z_L,w_L\})\times \HHH,\alpha\times\bdy\HHH)$ with the following properties:
\begin{itemize}
  \item $\DegSource$ is equipped with a boundary puncture $q$ with the property that
    $u$ is asymptotic to $(x',\infty)$ at $q$.
  \item $\DegSource$ is equipped with two more sets of interior
    punctures $\{w_1,\dots,w_k\}$ and $\{z_1,\dots,z_k\}$ so that
    $\pi_\Sphere\circ u$ is asymptotic to the (simple) Reeb orbit
    around $z_L$ (respectively $w_L$) at $z_i$ (respectively $w_i$).
  \item The map $u$ is holomorphic, with respect to the split complex
    structure on $\Sphere\times\HHH$ and some complex structure on
    $\DegSource$.
  \item The projection to $\HHH$, $\pi_\HHH\circ u$, has degree $1$. (In particular, $S$
    is planar.)
\end{itemize}
\end{definition}

Let $\cN(x';k)$ denote the moduli space of stabilizing boundary degenerations
asymptotic to $x'\in\alpha$.

The diagonal
$\Delta\subset \Sym^k(\HHH)$ consists of all elements $\x$ with a
repeated entry.  More generally, given a Cartesian product of
symmetric products $(\Sym^k(\HHH))^{\times m}$, let $\Delta\subset
(\Sym^k(\HHH))^{\times m}$ denote the diagonal consisting of those
elements $\x_1\times\dots\times \x_m$ for which either some $\x_i$ has
a repeated entry, or $\x_i\cap \x_j$ is nonempty, for some $i\neq j$.
So, $\Delta$ is the preimage of the diagonal under the quotient map
$(\Sym^k(\HHH))^{\times m}\to \Sym^{mk}(\HHH)$.

There is an evaluation map 
$\ev\co \cN(x';k)\to
\Sym^k(\HHH)\times \Sym^k(\HHH)$ defined by
\[
  u\mapsto \Bigl(\bigl\{\pi_{\HHH}(u(z_1)),\dots,\pi_{\HHH}(u(z_k))\bigr\},
  \bigl\{\pi_{\HHH} (u(w_1)),\dots,\pi_\HHH(u(w_k))\bigr\}\Bigr).
\]

\begin{lemma}\label{lem:rat-maps}
  The evaluation map induces a proper map of degree $1$
  \[
    \ev\co \cN(x';k)\to \left(\Sym^k(\HHH)\times \Sym^k(\HHH)\right)\setminus\Delta.
  \]
  In fact, the evaluation map is a diffeomorphism.
\end{lemma}

\begin{proof}
  We think of $\Sphere$ as the one-point compactification of $\CC$, so
  that $\alpha$ corresponds to $\RR^{+}\cup \{\infty\}$, and $z_L$ and
  $w_L$ correspond to $0$ and $\infty$. (This uses the fact that
  $\alpha$ was an arc of a great circle.) Under this parameterization,
  the point $x'$ corresponds to some real number $r>0$.

  For notational simplicity, we also reparameterize the domain $\HHH$
  by $\HHH^+=\{x+iy\mid y\geq 0\}$, so that the point at infinity in
  $\HHH$ corresponds to the origin $0\in \HHH^+$.

  Curves in $\cN(k,x')$ correspond to
  holomorphic maps $f\colon \HHH^+\to \CC$ satisfying:
  \begin{itemize}
  \item $f(\RR)\subset \RR^+$; in fact, $f(\RR)$ is a closed interval
    in $\RR^+$.
  \item $f(0)=r$.
  \item The poles and zeroes of $f$ do not occur on $\partial \HHH^+$.
  \item All of the poles and zeroes of $f$ are simple.
  \end{itemize}
  (The last condition corresponds to the condition that curves 
  in $\cN(k,x')$ are required to be asymptotic to simple Reeb orbits.)
  After a Schwartz reflection, these in turn correspond to meromorphic
  (and hence rational) 
  functions $F\colon \CC\cup\{\infty\} \to \CC\cup\{\infty\}$
  with:
  \begin{itemize}
  \item $F(\RR)\subset \RR^+$
  \item $F(0)=r$.
  \item There are no zeroes or poles of $F$ on $\RR$.
  \item All of the zeros and poles of $F$ are simple.
  \end{itemize}
  Such functions are specified by their set of zeroes $Q\subset \HHH^+$ 
  and poles $P\subset \HHH^+$ (all of which come in conjugate pairs, since
  they came from a Schwartz reflection),
  and a scaling parameter $\kappa$:
  \[ 
    F(\zeta)=\kappa
    \left(\frac{\prod_{q\in Q}(\zeta-q)(\zeta-{\overline q})}
      {\prod_{p\in P}(\zeta-p)(\zeta-{\overline p})}\right);
  \]
  where $\kappa$ is expressed in terms of $r$ (i.e. $x'$) by
  \[ \kappa \cdot \left(\frac{\prod_{q\in Q} |q|}{\prod_{p\in P}|p|}\right)^2=r.\]
  The image of the map $f$ under the evaluation map $\ev$ is the pair $(Q,P)$.
  Thus, $\ev$ induces a one-to-one correspondence. Since it is a
  continuous map between manifolds, it follows that $\ev$ is a
  homeomorphism.

To see that $\ev$ is a diffeomorphism, we argue as follows.  
The smooth structure on $\cN(k,x')$ agrees with the smooth structure
induced on the rational functions in $\zeta\in\CC$.  Consider the map that
associates to a rational function in $\zeta$ its (unordered) set of roots
$Q$ and poles $P$. If none of those roots or poles is repeated, the
Jacobian of the map is non-zero.
\end{proof}

Boundary degenerations in $\Sigma'$ will be obtained by gluing
boundary degenerations in $\Sigma$ to stabilizing boundary
degenerations. To formulate this, we choose also two points $z_R$ and
$w_R$ on $\Sigma$.

\begin{proposition}
  \label{prop:StabilizationInvariance}
  Fix an energy $E_0$.
  If $J$ is regular to energy $E_0$ and product-like near points $z'_R,w'_R\in\Sigma$
  and $z_R$ and $w_R$ are chosen generically in the neighborhoods of
  $z'_R$ and $w'_R$ where $J$ is split then for $T$ sufficiently large
  the associated complex structure $J(T)$ over $\Sigma'$ is also
  regular to energy $E_0$.  Moreover, for any $\vec{\rho}$ and $m$
  with $E(\vec{\rho},m)\leq E_0$ and $\indsq(\vec{\rho},m)=2$,
  \[
    \#\cN(J;[\alpha_i];\rho^1,\dots,\rho^n;m)
    =\#\cN(J(T);[\alpha_i];\rho^1,\dots,\rho^n;m).
  \]
\end{proposition}
\begin{proof}
  Since there are finitely many pairs $(\vec{\rho},m)$ with a given
  energy, it suffices to prove the result one tuple $(\vec{\rho},m)$ at
  a time. 
  Fix a tuple of points
  $\x\in (\alpha_1^a\cup\alpha_2^a)\times
  \alpha_1^c\times\dots\times\alpha_{g-1}^c$.  Let $\{u_T\}$ be a
  sequence of curves in $\cN(J(T);\x\times x_0;\rho^1,\dots,\rho^n;m)$
  with $T\to \infty$, with domain $k[\Sigma]$.
  By a compactness argument (similar to the proof of~\cite[Proposition
  5.24]{LOT1} or~\cite[Proposition 9.6]{LOT1}), for a generic choice
  of $z_R$ and $w_R$, there is a subsequence
  of the $u_T$ which converges to a pair $(u_L,v_R)$ where:
  \begin{itemize}
  \item $u_L$ is a stabilizing boundary degeneration, in $\cN(x';k)$.
  \item $u_R\in \cN(J;\x;\rho^1,\dots,\rho^n;m)$.
  \item The $k$-tuple of points
    $(\pi_\HHH\circ u_R)\bigl((\pi_\Sigma\circ u_R)^{-1}(z_R)\bigr)$
    agree with $\{(\pi_\HHH\circ u_L)(\{w_1,\dots,w_k\})$ and similarly
    the $k$-tuple of points
    $(\pi_\HHH\circ u_R)\bigl((\pi_\Sigma\circ u_R)^{-1}(w_R)\bigr)$
    agree with $\{(\pi_\HHH\circ u_L)(\{z_1,\dots,z_k\})$.
  \end{itemize}
  In particular, the condition that $u_L$ has only simple Reeb orbits at
  $z_L$ is equivalent to the condition that $z_R$ is a regular value for
  $\pi_{\Sigma}\circ u_R$; by Sard's theorem, generic $z_R$ and
  $w_R$ satisfy this condition. The fact that $u_L$ and $u_R$ meet
  only along orbits, not also Reeb chords, follows from elementary
  complex analysis: if they met along a Reeb chord, a maximum modulus
  theorem argument would imply that $\pi_\HHH\circ u_R$ is constant.

  We claim that 
  \[
    \ind(u)=\ind(u_T).
  \]
  Indeed, if we use a subscript $T$ to denote the terms in
  Formula~\eqref{eq:index-source-bdy-degen} for $u_T$ then $g_T=g+1$
  and $\chi_T=\chi+1-4k$, while the other terms are unchanged, giving
  the result. (Alternatively, the equality of indices also follows
  from Proposition~\ref{prop:emb-ind-bdy-degen}.)
  
  So, it follows from Lemma~\ref{lem:bdy-deg-ind-2} that if
  $\indsq(\rho^1,\dots,\rho^n;m)\leq 1$ then the corresponding moduli space must
  be empty.

  Assume next that $\indsq(\rho^1,\dots,\rho^n;m)=2$. 
  It follows from standard gluing arguments (see the proof
  of~\cite[Propositions 5.30 and 5.31]{LOT1} and the citations there) that
  for sufficiently large $T$, $\cN(J(T),\x\times x_0;\rho^1,\dots,\rho^n;m)$
  is in the range of a gluing map from
  \[
    \cN(J;\x;\rho^1,\dots,\rho^n;m)
    \times_{(\Sym^k(\HHH)\times\Sym^k(\HHH))\setminus\Delta}
    \cN(x';k).
  \]
  This uses the fact that the map
  $\cN(J;\x;\rho^1,\dots,\rho^n;m)\times
  \cN(x';k)\to\Sym^m(\HHH)\times\Sym^m(\HHH)$ in the
  definition of the fibered product is a submersion, which follows
  from Lemma~\ref{lem:rat-maps}.  Further, since the map from the
  second factor is, in fact, a diffeomorphism, this fibered product is
  identified with its first factor. Thus, for $T$ large enough,
  \[
    \cN(J;\x;\rho^1,\dots,\rho^n;m)\cong \cN(J(T),\x\times x_0;\rho^1,\dots,\rho^n;m),
  \]
  as desired.
\end{proof}

\begin{proof}[Proof of Theorem~\ref{thm:AlgOfSurface}]
  This is an induction on the genus of $\Sigma$; the base case is
  Proposition~\ref{prop:IdentifyAlgebraGenusOne}, and the inductive
  step is provided by Proposition~\ref{prop:StabilizationInvariance}.
\end{proof}

Finally, we turn briefly to composite boundary, observing that the
same argument shows that the counts of composite boundary
degenerations are the same as the counts of simple ones
(Proposition~\ref{prop:extended-is-composite}):

\begin{proof}[Proof of Proposition~\ref{prop:extended-is-composite}]
  In the genus $1$ case, this follows from the description of boundary
  degenerations in Section~\ref{sec:GenusOne}: there is an obvious
  bijection between the two moduli spaces. (In the language
  of~\cite{LOT:torus-alg}, both correspond to the same tiling
  patterns.) The higher genus case then follows by the same
  degeneration argument that was used to prove
  Proposition~\ref{prop:StabilizationInvariance}.
\end{proof}


\section{The module \textalt{$\CFAmb$}{CFAnu}}\label{sec:CFA}
In this section, we construct a version of the main invariant of
3-manifolds with torus boundary, the weighted $\Ainf$-module
$\CFAmb(\HD)$, and prove it is independent of the choice of Heegaard
diagram $\HD$. The module $\CFAmb(\HD)$ has a mild algebraic defect: the
actions on it are not $U$-equivariant. Remedying this defect requires
one additional complication, which we discuss in
Section~\ref{sec:GroundingU}. 
So, the ground ring for $\CFAmb(\HD)$ is
$\Field$, not $\Field[U]$. On the other hand, if one is interested
in studying $\HFm(Y)/(U^n)$ for some $n$, the slightly simpler
construction in this section suffices (for an appropriate choice of
almost complex structures): one can arrange that the actions are
$U$-equivariant up to any fixed order in $U$
(Remark~\ref{rem:UtoTheN}). The grading on $\CFAmb$ is deferred to
Section~\ref{sec:CFD}.

\subsection{The definition of \textalt{$\CFAmb$}{CFAnu} as a weighted module}
Fix a provincially admissible bordered Heegaard diagram $\HD$.

The weighted $\Ainf$-module $\CFAmb(\HD)$ is the left
$\Field[U]$-module generated by $\Gen(\HD)$, the set of generators for
$\HD$. The idempotents of
$\Alg(\bdy\HD)$ act on $\CFAmb(\HD)$ by $m_2^0(\x,\iota_i)=\x$
if $\x\cap \alpha_{i+1}^a\neq\emptyset$ and $m_2^0(\x,\iota_i)=0$
otherwise. (Recall that, for our conventions, $\iota_0$ corresponds to
$\alpha_1^a$ and $\iota_1$ corresponds to $\alpha_2^a$.)
We also define $m^w_{1+n}(\x,a_1,\dots,a_n)=0$ for $n+2w>1$ if either
$\iota_i\in \{a_1,\dots,a_n\}$.

To define the rest of the weighted $\Ainf$-module structure, it
suffices to specify the operation
$m_{1+n}^w(\x,a_1,\dots,a_n)$ for each $n$, $w$, and
basic algebra sequence  $(a_1,\dots,a_n)$,
i.e., where each $a_i$ is of the form 
$U^{r_i}$ with $r_i\geq 1$ or $U^{r_i}\rho^i$ with $r_i\geq 0$ and
$\rho^i$ some Reeb chord.
Fix a tailored family of almost complex structures.
Define 
\begin{equation}\label{eq:CFA-m-def}
  m_{1+n}^w(\x,a_1,\dots,a_n)=\sum_{\y\in\Gen(\HD)}\sum_{\substack{B\in\pi_2(\x,\y)\\\ind(B;\rho_1,\dots,\rho_n;w)=1}}\#\cM^B(\x,\y,a_1,\dots,a_n;w)
  U^{n_z(B)+\sum_{i=1}^n r_i} \y.
\end{equation}

For this to make sense, we need the following:
\begin{lemma}\label{lem:admis}
  Fix generators $\x,\y\in\Gen(\HD)$, basic algebra elements $a_1,\dots,a_n$,
  and $w\in\ZZ$.  Provincial admissibility of the bordered Heegaard diagram $\HD$
  implies that there are finitely many homology
  classes $B\in\pi_2(\x,\y)$ so that $\cM^B(\x,\y;a_1,\dots,a_n;w)$ is
  non-empty.
  Further, given a tailored family of almost complex structures, if
  $\ind(B;a_1,\dots,a_n;w)=1$ then $\cM^B(\x,\y;a_1,\dots,a_n;w)$
  consists of finitely many points.
\end{lemma}
\begin{proof}
  This is a simple adaptation of the proof in the closed
  case~\cite[Lemma 4.13]{OS04:HolomorphicDisks} (see
  also~\cite[Proposition 4.28]{LOT1}). If the moduli space is
  non-empty then the domain $B$ must have only non-negative
  coefficients. The sequence of chords $\rho^i$ corresponding to the
  $a_i$ and the integer $w$
  determine the multiplicities of $B$ near $\partial
  \Sigma$. Consequently, any two such domains $B$ differ by a
  provincial periodic domain.
  Provincial admissibility implies that there is an area form on $\Sigma$ so
  that all provincial periodic domains have area $0$~\cite[Proposition
  4.28]{LOT1}. In particular, all domains $B$ so that the moduli space
  is non-empty have the same area. There are only finitely many
  positive domains with a given area, so the result follows.

  The finiteness statement is Lemma~\ref{lem:0d-compactness}.
\end{proof}

\begin{remark}
  \label{rem:HFAmFree}
  Since the basepoint $z$ is adjacent to the boundary, it follows that
  $m^0_1(\x)$ is in the $\Field$-span of ${\mathfrak S}(\HD)$; 
  in particular, $H_*(\CFAmb(\HD),m^0_1)$ is a free (left) $\Field[U]$-module.  
\end{remark}

As mentioned in the introduction, since $\pi_2(\x,\y)=\emptyset$
unless $\x$ and $\y$ represent the same $\SpinC$-structure (see
Section~\ref{sec:background}), the module $\CFAmb(\HD)$ decomposes as a
direct sum
\[
  \CFAmb(\HD)=\bigoplus_{\spinc\in\Spinc(Y)}\CFAmb(\HD,\spinc).
\]

\subsection{Proof of the structure equation}
Recall that Theorem~\ref{thm:CFAmb-is} states that $\CFAmb$ satisfies
the weighted $\Ainf$-relations.
\begin{proof}[Proof of Theorem~\ref{thm:CFAmb-is}]
  As usual, this follows by considering the ends of the
  $1$-dimensional moduli spaces, i.e., from Theorem~\ref{thm:master}.
  Consider the coefficient of $U^m\y$ in the weight $w$ $\Ainf$-relation with inputs
  $(\x,a_1,\cdots,a_n)$. Since $\CFAmb$ is defined
  to be strictly unital, if any of the $a_i$ is an idempotent then
  verifying the weighted $\Ainf$-relation is straightforward, and is
  left to the reader. So, assume that all the $a_i$ are basic algebra
  elements.  Then, this relation has the following kinds of terms:
  \begin{enumerate}
  \item Terms of the form $m_{1+n-j}^{w-w'}\bigl(m_{1+j}^{w'}(\x,
    a_1,\dots,a_j),a_{j+1},\dots,a_n\bigr)$.
    These are in bijection with terms of
    Type~\ref{end:TwoStoryBuilding} in Theorem~\ref{thm:master}.
  \item Terms of the form
    $m^w_n\bigl(\x,a_1,\dots,a_{j-1},\mu_2^0(a_j,a_{j+1}),a_{j+2},\dots,a_n\bigr)$. These
    are in bijection with terms of Type~\ref{end:Collision} in
    Theorem~\ref{thm:master}.
  \item Terms of the form
    $m_{2+n}^w\bigl(\x,a_1,\dots,a_j,\mu_0^1,a_{j+1},\dots,a_n\bigr)$.
    These are in bijection with terms of Type~\ref{end:EscapingOrbit}
    in Theorem~\ref{thm:master}.
  \item Terms of the form
    $m_{2+n-j}^{w-w'}\Bigl(\x,a_1,\dots,a_{i},\mu_j^{w'}\bigl(a_{i+1},\dots,a_{i+j}\bigr),a_{i+j+1},\dots,a_n\Bigr)$ where
    $j+2w'>2$ and $j\geq 1$. There are two sub-cases:
    \begin{enumerate}
    \item If $w-w'=n-j=0$, these correspond to terms of
      Type~\ref{end:SimpleBoundaryDegeneration} in
      Theorem~\ref{thm:master}.
    \item Otherwise, these correspond to terms of
      Type~\ref{end:CompositeBoundaryDegeneration} in Theorem~\ref{thm:master}.
    \end{enumerate}
  \end{enumerate}
  So, by Theorem~\ref{thm:master}, the coefficient of $U^m\y$ vanishes
  (modulo 2).
\end{proof}

\subsection{Invariance under changes of almost complex structures and
  isotopies}\label{sec:CFA-iso-invariance}
In this section, we note invariance of $\CFAmb$ under isotopies and
changes in the almost complex structure.

\begin{proposition}\label{prop:J-inv}
  Fix a provincially admissible bordered Heegaard diagram
  $\HD=(\Sigma,\alphas,\betas)$.  If $J$ and $J'$ are tailored
  families of almost
  complex structures then the modules $\CFAmb(\HD;J)$ and
  $\CFAmb(\HD;J')$ computed with respect to $J$ and $J'$ are
  $\Ainf$-homotopy equivalent.
\end{proposition}
\begin{proof}
  The proof is similar to the proof of the analogous result for
  $\CFAa$~\cite[Section
  7.3.1]{LOT1} or ordinary Heegaard Floer homology~\cite[Proposition
  7.1]{OS04:HolomorphicDisks}, but with some superficial complications
  because we are considering families of almost complex structures. 

  Recall from Definition~\ref{def:polygons} that a bimodule component
  has two distinguished boundary punctures, labeled $\pm\infty$, and
  some other interior and boundary punctures labeled by energies. The
  $\pm\infty$ punctures divide the rest of the boundary into
  $\{0\}\times\RR$ and $\{1\}\times\RR$. A \emph{bimodule map component}
  is defined similarly to a bimodule component, except with a third
  distinguished puncture, which we call $(1,0)$, along the edge
  $\{1\}\times\RR$. There is a compactified moduli space $\BModPol$ of
  bimodule map polygons, whose boundary points are trees with one
  vertex a bimodule map component and the other vertices bimodule
  components or algebra-type components. Given a point $P$ in the boundary of
  $\BModPol$, each bimodule component of $P$ (if there are
  any) either comes above or below the bimodule map
  component. Call bimodule components \emph{top} or \emph{bottom}
  bimodule components in these two cases; see Figure~\ref{fig:BModPol}.

  \begin{figure}
    \centering
    \includegraphics{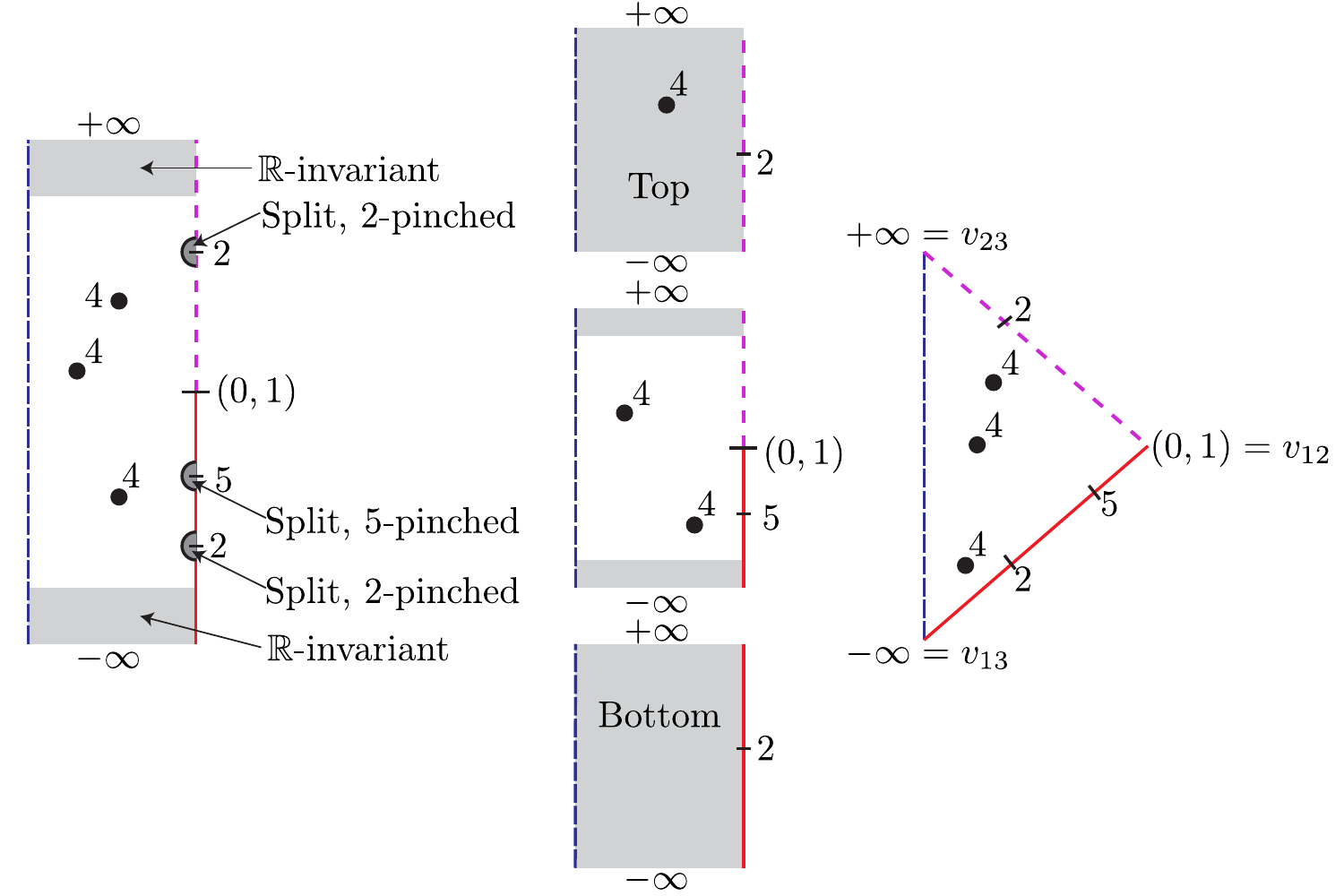}
    \caption[Bimodule map polygons]{\textbf{Bimodule map polygons.} Left: a bimodule map component,
      drawn as $[0,1]\times\RR$ with marked points, with some
      conditions on the almost complex structure marked. Center: a
      point in the boundary of $\BModPol$, with a top bimodule
      component and a bottom bimodule component. Right: in
      Section~\ref{sec:CFA-hs-invariance}, we will think of bimodule
      map polygons as triangles.}
    \label{fig:BModPol}
  \end{figure}
  
  An \emph{interpolating tailored family of almost complex structures}
  is a continuous family of almost complex structures $\wt{J}$ on
  $\Sigma\times P$ parameterized by $P\in \BModPol$ so that
  \begin{itemize}
  \item Each $J(P)$ satisfies
    Conditions~\ref{item:J-piD},~\ref{item:J-s-t},~\ref{item:J-split},~\ref{item:J-const-at-punct},
    and~\ref{PinchedOverVertex} of the definition of an
    $\eta$-admissible almost complex structure
    (Definition~\ref{def:ac}), for some sufficient pinching function
    $\eta$, and preserves $\{x\}\times TP$ for each $x\in \Sigma$,
  \item $\wt{J}$ agrees with the $\RR$-invariant almost complex
    structure $J'_\infty$ specified by $J'$ near $+\infty$ and with the
    $\RR$-invariant almost complex structure $J_\infty$ specified by $J$ near $-\infty$,
  \item $\wt{J}$ agrees with $J'$ over every top bimodule component,
    and with $J$ over every bottom bimodule component, and
  \item $\wt{J}$ satisfies the obvious analogue of the coherence
    condition~\ref{item:admis-Pprime} from Definition~\ref{def:AdmissibleJs} over the bimodule
    map component, and
  \item $\wt{J}$ satisfies the obvious analogues of
    Conditions~\ref{item:tail-pinched} (being sufficiently pinched),
    and~\ref{item:tail-main-comp} (the moduli spaces of main
    components being transversely cut out, up to index $1$), of a
    tailored family (Definition~\ref{def:tailored}).
  \end{itemize}
  The proof that interpolating tailored families exist is similar to
  the proof of Corollary~\ref{cor:TailoredExist}.

  Given an interpolating tailored family $\wt{J}$,
  define moduli spaces
  $\cM^B(\wt{J};\x,\y;\vec{a};w)$ as in
  Definition~\ref{def:moduli-embedded}, but using $\wt{J}$. Note
  that, because of the extra puncture, these moduli spaces are one
  dimension larger than the moduli spaces
  $\cM^B(\x,\y;\vec{a};w)$.
  Define $f^{w}_{1+n}\co \CFAmb(\HD;J)\otimes\MAlg^{\otimes n}\to
  \CFAmb(\HD;J')$ by
  \[
    f_{1+n}^w(\x,a_1,\dots,a_n)=\sum_{\y\in\Gen(\HD)}\sum_{\substack{B\in\pi_2(\x,\y)\\\ind(B;\vec{a};w)=0}}\#\cM^B(\wt{J};\x,\y;\vec{a};w)
    U^{n_z(B)+m} \y,
  \]
  where $m$ is the sum of the $U$-powers in the elements of $\vec{a}$.
  The index-1 moduli spaces
  $\cM^B(\wt{J};\x,\y;\vec{a};w)$ have ends corresponding
  to two-story buildings, collisions of levels, orbit curve
  degenerations, and composite boundary degenerations. (Simple
  boundary degenerations do not occur in index $1$ moduli spaces.)  In
  the $(n+1)$-input, weight $w$ $\Ainf$-relation for
  $F=\{f_{1+m}^v\}$, these correspond to terms of the form
  \begin{align*}
  f_{1+m}^v&(m_{1+n-m}^{w-v}(\x,a_1,\dots,a_{n-m}),a_{n-m+1},\dots,a_{n})\quad
  \text{and}\\
  m_{1+m}^v&(f_{1+n-m}^
             {w-v}(\x,a_1,\dots,a_{n-m}),a_{n-m+1},\dots,a_{n});\\
    f_{n}^w&(\x,a_1,\dots,a_ia_{i+1},\dots,a_{n});\\
  f_{n+2}^{w-1}&(\x,a_1,\dots,a_i,\mu_0^1,a_{i+1},\dots,a_n); \quad
  \text{and}\\
  f_{1+m}^v&(\x,a_1,\dots,a_{i-1},\mu_{n-m+1}^{w-v}(a_{i},\dots,a_{i+n-m}),
    a_{i+n-m+1},\dots,a_{n}) \quad\text{with $n-m+1>2$},
  \end{align*}
  respectively. Compactness and gluing results for these index-1
  moduli spaces follow as in Section~\ref{sec:compactness}.  So, the
  number of ends is even, and the $f_{1+n}^w$ form a weighted
  $\Ainf$-module homomorphism.
  
  Define a homomorphism $G=\{g_{1+n}^w\}$ from $\CFAmb(\HD;J')$ to
  $\CFAmb(\HD;J)$ similarly, with the roles of $J$ and $J'$ exchanged,
  using an interpolating tailored family of almost complex structures
  $\overline{J}$. Finally, to show that $G\circ F$ and $F\circ G$ are
  homotopic to the identity maps, consider a moduli space $\HModPol$
  \emph{bimodule homotopy polygons} which have four distinguished
  marked points, two of which are in $\{1\}\times\RR$, and
  corresponding families of almost complex structures on
  $\Sigma\times P$. One kind of boundary point of $\HModPol$ occurs
  when the distance between the two distinguished points on
  $\{1\}\times \RR$ goes to infinity; these ends are identified with
  pairs of bimodule map polygons, and correspond to $G\circ F$ (or
  $F\circ G$). Another end is where the two distinguished points on
  $\{1\}\times\RR$ collide; identify these polygons with bimodule
  polygons by forgetting the extra pair of marked points. These ends
  then correspond to the identity map. Other ends of the space
  correspond to homotopy terms.
\end{proof}

\begin{proposition}\label{prop:iso-inv}
  If provincially admissible bordered Heegaard diagrams
  $\HD=(\Sigma,\alphas,\betas)$ and $\HD'=(\Sigma,\alphas',\betas')$
  differ by an isotopy of the $\alpha$- and $\beta$-curves then the
  modules $\CFAmb(\HD)$ and $\CFAmb(\HD')$ are $\Ainf$-homotopy
  equivalent.
\end{proposition}
\begin{proof}
  Again, the proof is similar to the case of $\HFa$, as well as to the
  proof of Proposition~\ref{prop:J-inv}. It suffices to consider the
  case of a small Hamiltonian isotopy, where one defines chain maps by
  counting holomorphic curves as in the proof of
  Proposition~\ref{prop:J-inv}, except now with varying boundary
  conditions instead of varying almost complex structures (compare,
  for example~\cite[Section 9]{Lipshitz06:CylindricalHF}).
\end{proof}

\subsection{Handleslide invariance}\label{sec:CFA-hs-invariance}
The goal of this section is to prove:
\begin{proposition}\label{prop:handleslide-inv}
  Let $\HD$ be a bordered Heegaard diagram, and $\HD'$ the result of
  performing a handleslide on $\HD$, of one $\beta$-circle over
  another, one $\alpha$-circle over another, or an $\alpha$-arc over
  an $\alpha$-circle. Assume that both $\HD$ and $\HD'$ are
  provincially admissible. Then $\CFAmb(\HD)\simeq \CFAmb(\HD')$.
\end{proposition}

The three cases  in the proposition are similar, but
there are some complications for $\alpha$-handleslides that do not
occur for $\beta$-handleslides, and for $\alpha$-handleslides the case
of sliding an arc over a circle is slightly more complicated than the
case of sliding a circle. So, we will focus on that most complicated
case and leave the others to the reader.

\begin{figure}
  \centering
  \includegraphics{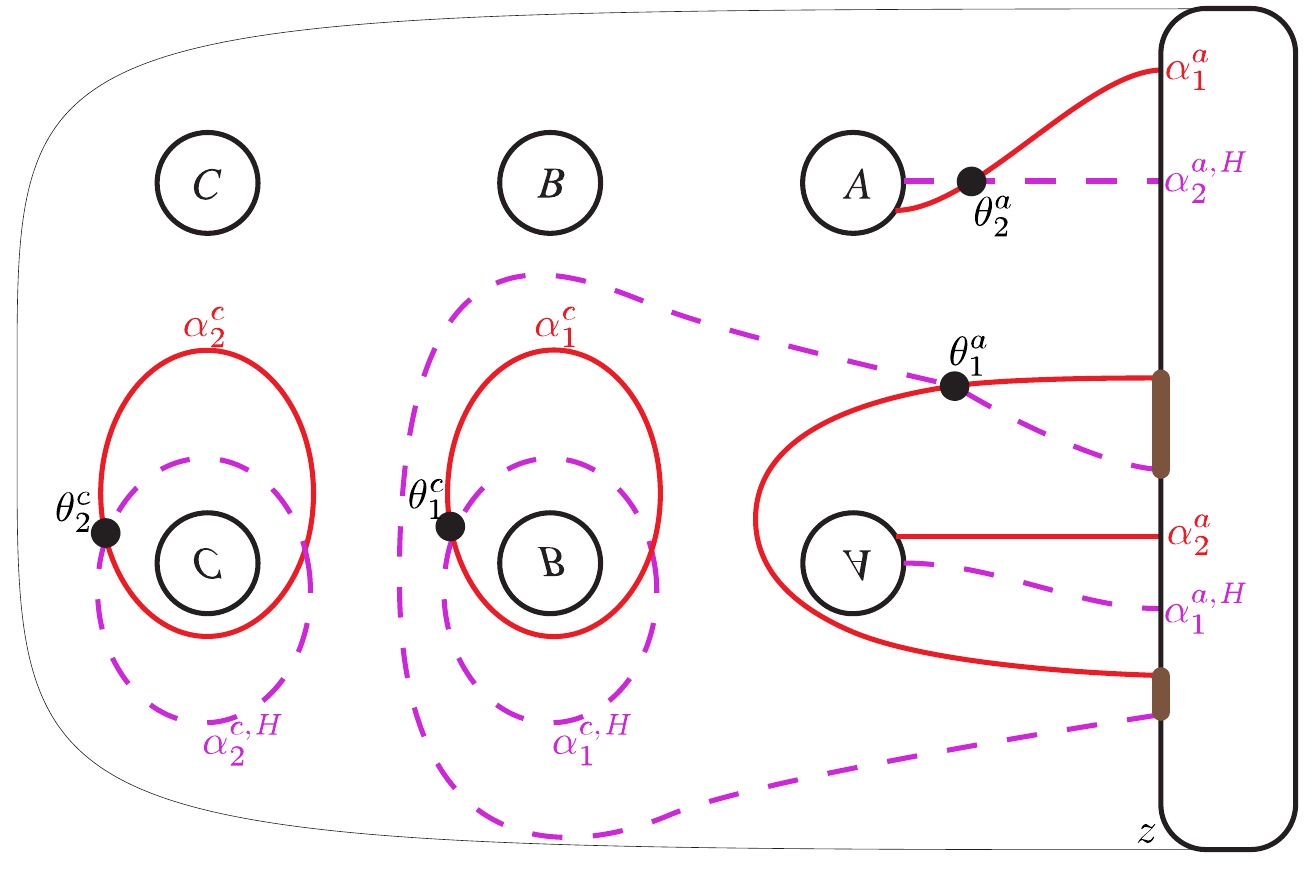}
  \caption[Handleslide curves]{\textbf{Handleslide curves.} The two
    short Reeb chords connecting $\alpha_1^a$ and $\alpha_1^{a,H}$ are
    indicated by \textcolor{darkbrown}{thick} segments. This figure is
    adapted from~\cite[Figure 6.3]{LOT1}.}
  \label{fig:hs-curves}
\end{figure}

Write $\HD=(\Sigma,\alphas,\betas)$ and
$\HD'=(\Sigma,\alphas^H,\betas)$, where $\alphas^H$ is obtained from
$\alphas$ by a handleslide. Like the case of $\HFa$, a quasi-isomorphism between $\CFAmb(\HD)$ and $\CFAmb(\HD')$ is constructed by counting pseudo-holomorphic triangles.
We will consider the case of sliding
$\alpha_1^a$ over $\alpha_1^c$; the case of sliding $\alpha_2^a$ over
an $\alpha$-circle is identical. Choose the curves $\alphas^H$ so that
each circle $\alpha_i^{c,H}$ intersects $\alpha_i^c$ in two points and
each $\alpha_i^{a,H}$ intersects $\alpha_i^a$ in one point, and all
the $\alpha$-curves are disjoint other than this. Further require that
there is a small triangle with boundary on $\bdy\Sigma$, $\alpha_1^a$,
and $\alpha_1^{a,H}$ (in that counter-clockwise order), and two small
triangles with boundary on $\bdy\Sigma$, $\alpha_2^a$, and
$\alpha_2^{a,H}$. See Figure~\ref{fig:hs-curves}. For each
$\alpha$-circle, there are two bigons with boundary on
$(\alpha_i^c,\alpha_i^{c,H})$ from one intersection point $\theta_i^c$
to the other point $\eta_i^c$. Let $\Theta_1$ (respectively $\Theta_2$)
be the $g$-tuple of points in $\alphas\cap \alphas^H$ consisting of
all the points $\theta_i^c$ and the unique point $\theta_i^a$ in
$\alpha_1^a\cap \alpha_1^{a,H}$ (respectively
$\alpha_2^a\cap \alpha_2^{a,H}$). (Roughly, the condition about the
triangles with boundary on $\alpha_i^a$, $\alpha_i^{a,H}$, and
$\bdy \Sigma$ means that $\theta_i^a$ corresponds to the top-graded generator
in a closed Heegaard diagram.)

We start by describing the families of almost complex structures we
will use. Let $\BModPol$ be the compactified moduli space of bimodule
map polygons from the proof of Proposition~\ref{prop:J-inv}.  Let
$\Delta$ be a disk with three boundary punctures. Label the edges of
$\Delta$ as $e_1,e_2,e_3$ counter-clockwise, and let $v_{ij}$ be the
puncture between $e_i$ and $e_j$.  Given a bimodule map component, there
is a unique biholomorphic identification of
$P\setminus\{\infty,-\infty,(0,1)\}$ with $\Delta$.
Fix tailored families of almost complex
structures for $\HD$ and $\HD'$. A \emph{tailored triangular family of
almost complex structures} is a family $J(P)$ of almost complex
structures on $\Sigma\times P$ for $P\in\BModPol$ such that:
\begin{enumerate}
\item Each $J(P)$ agrees with the $\RR$-invariant almost complex structure
  for $\HD$ in a neighborhood of $v_{13}$ and with the $\RR$-invariant
  almost complex structure for $\HD'$ in a neighborhood of $v_{23}$.
\item The projection $\pi_\Delta$ is $J(P)$-holomorphic, and $J(P)$
  preserves $\{x\}\times T\Delta$ for each $x\in \Sigma$.
\item $J$ is split near the puncture of $\Sigma$.
\item\label{item:tri-J-pinched} $J(P)$ is sufficiently pinched in the
  following sense. For each algebra puncture along $e_{1}$
  (respectively $e_{2}$), each $J(P)$ is constant and sufficiently
  pinched for the diagram $\HD$ (respectively $\HD'$) in an
  $\epsilon$-neighborhood of the puncture. (Here, the
  $\epsilon$-neighborhoods are with respect to a metric coming from
  identifying $\Delta$ with an unbounded, closed region in $\CC$,
  e.g., $[-1,1]\times\RR\cup \RR_{>0}\times[-1,1]$.)
\item On the compactification of $\BModPol$, for bimodule components
  at the bottom, $J(P)$ agrees with the chosen family for $\HD$, and
  for bimodule components at the top $J(P)$ agrees with the chosen
  family for $\HD'$.
\item $\wt{J}$ satisfies the obvious analogue of the coherence
  condition from Definition~\ref{def:AdmissibleJs} over the bimodule
  map component.
\item The moduli spaces considered below, including the boundary
  degenerations which occur in the proof of
  Proposition~\ref{prop:handleslide-inv}, are transversely cut out.
\end{enumerate}
The proof that tailored triangular families exist is similar to the
proof of Corollary~\ref{cor:TailoredExist}. Note that the pinching
conditions along $e_1$ and $e_2$ are different and incompatible; see
Figure~\ref{fig:handleslide-pinch}.

\begin{figure}
  \centering
  \includegraphics{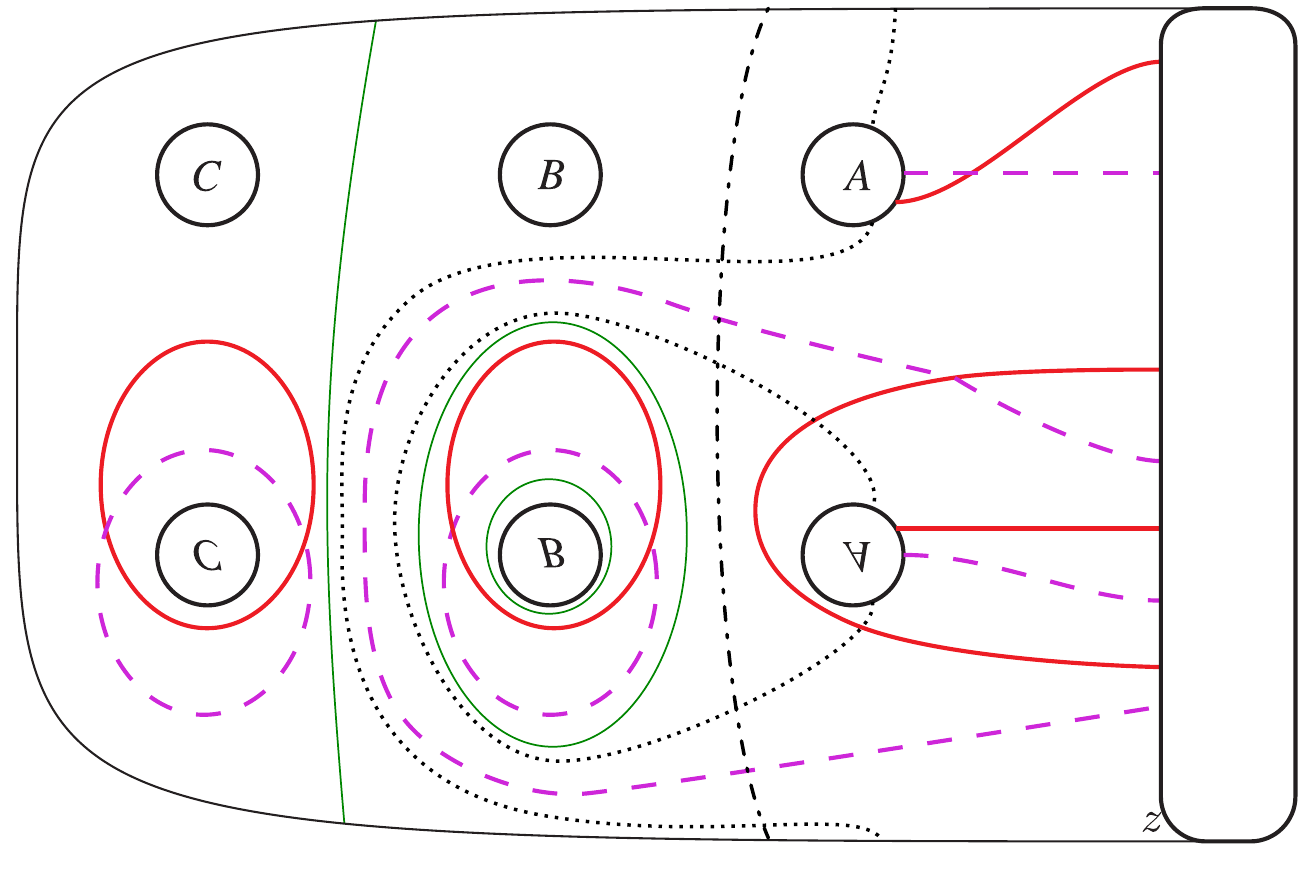}
  \caption[Pinching the almost complex structure for
  handleslides]{\textbf{Pinching the almost complex structure for
      handleslides.} The condition of being sufficiently
    $\epsilon$-pinched for $\alphas$ requires pinching along the
    dot-dashed circle, near chords on $e_1$. The condition of being
    sufficiently pinched for $\alphas^H$ requires
    pinching along the dotted circle, near chords on $e_2$. Note
    that these two pinchings cannot both be done at once, forcing
    the almost complex structure to be non-constant. For the proof
    of handleslide invariance we additionally pinch along the \textcolor{darkgreen}{thin}
    circles (denoted $C_1$, $C_2$, and $C_3$ in the proof of
    Lemma~\ref{lem:CancelInPairs}).}
  \label{fig:handleslide-pinch}
\end{figure}

The chain homotopy equivalence
\begin{equation}
  \label{eq:HandleSlideMap}
  f_{1+n}^w\co \CFAmb(\HD)\otimes \MAlg^{\otimes n}\to\CFAmb(\HD')
\end{equation}
is defined by counting index-0 maps to $\Sigma\times\Delta$ with
boundary on
$(\alphas\times e_1)\cup (\alphas^H\times e_2)\cup(\betas\times e_3)$,
and asymptotic to $\Theta_1$ or $\Theta_2$ at $v_{12}$. (The analogue
of $f$ was denoted $f^{\alphas,\alphas^H,\betas}$ in~\cite{LOT1}.) Let
$\cM^B(\x,\y,\Theta;\vec{a};w)$ denote the moduli spaces of such
$J$-holomorphic curves, asymptotic to the sequence of basic algebra elements
$\vec{a}$ along the edges $e_1\cup e_2$, and to $w$ orbits. Then
\[
  f_{1+n}^w(\x,a_1,\dots,a_n) = \sum_{\y,B}\#\cM^B(\x,\y,\Theta;a_1,\dots,a_n;w)U^{n_z(B)+m} \y
\]
where the sum is over homology classes $B$ of triangles so that the
moduli space is $0$-dimensional, and $m$ is the number of factors of
$U$ appearing in $a_1,\dots,a_n$. This makes sense by analogues of
Lemma~\ref{lem:0d-compactness} (so that every individual moduli space
is compact) and Lemma~\ref{lem:admis} (so that the sum is finite).

In effect, we are using $\Theta_i$ as generators feeding into a
triangle map. As in the closed case, to see that the induced map is a
chain map (and, in our case, an $\Ainf$-homomorphism), one must check
that $\Theta_i$ is a cycle in a suitable sense. In the closed case,
this was the statement that finite energy flow lines out of $\Theta$
cancel in pairs.  We formulate the analogous statement for the
bordered case as follows.

\begin{definition}
  \label{def:GenGen}
  A {\em generalized generator} for the diagram $(\Sigma,\alphas^H,\alphas)$
  consists of a $g$-tuple $\{x_i\}_{i=1}^{g}$,
  with $x_i\in \alpha^c_i\cap \alpha^{c,H}_i$ for $i=1,\dots,g-1$ and $x_g$
  either
  the intersection point between $\alpha^{a}_j$ and $\alpha^{a,H}_j$ 
  for $j\in 1,2$ or a Reeb chord with initial point on $\alpha^{a,H}_j$
  and terminal point at $\alpha^{a}_k$.
\end{definition}

Given generalized generators $\x$ and $\y$ for
$(\Sigma,\alphas^H,\alphas)$, we will be interested in moduli spaces
of embedded curves that connect then, labeled
$\cM^B(\x,\y;\vec{a},\vec{b};w)$.  Here, $\vec{a}$ specifies the basic
algebra elements for the $\alpha^{a,H}_i$
encountered over $\{1\}\times \RR$; $\vec{b}$ specifies the
basic algebra elements for
$\alpha^{a}_j$ encountered over $\{0\}\times\RR$; and $w$
specifies the number of Reeb orbits (all of which are required to be
simple). To see that these moduli spaces are well behaved, recall that
we require the almost complex structure to agree with the standard,
split complex structure near the puncture in $\Sigma$. So, if we let
$\overline{\Sigma}$ denote the result of filling in the puncture, by
the removable singularities theorem any holomorphic curve in
$\Sigma\times \Delta$ extends to a proper map to
$\overline{\Sigma}\times \Delta$, where the generalized generators
correspond to ordinary Heegaard Floer generators, and the usual
Fredholm theory and compactness results for symplectic field theory
apply. (If we do not fill in the puncture, the product of a
neighborhood of the puncture with a neighborhood of $v_{23}$ is not a
cylindrical end, but rather a product of two cylindrical ends and, as
far as we know, the analytic underpinnings of pseudoholomorphic curves
have not been developed in this setting.)

\begin{lemma}
  \label{lem:CancelInPairs}
  Fix $\ell\in\{1,2\}$.  There exist tailored families of almost
  complex structures $J$ so that for any generalized generator $\y$,
  if $\cM^B_{J}(\Theta_\ell,\y,\vec{a},\vec{b})$ is a non-empty, rigid
  moduli space then $n_z(B)=0$.
\end{lemma}

The complex structures used in the lemma are indicated in
Figure~\ref{fig:handleslide-pinch}.  We postpone the proof of
Lemma~\ref{lem:CancelInPairs} to Section~\ref{subsec:cancel-in-pairs}.
Before that, we use it to deduce
Proposition~\ref{prop:handleslide-inv}, and the following intermediate
result:

\begin{lemma}
  \label{lem:HandleSlideMap}
  For a family of almost complex structures as in Lemma~\ref{lem:CancelInPairs},
  the map $f$ from 
  Equation~\eqref{eq:HandleSlideMap} is a homomorphism of
  weighted $A_{\infty}$ modules.
\end{lemma}

\begin{proof}
  A similar argument to the proof of Theorem~\ref{thm:master}
  shows that the codimension-1 boundary
  of the space of holomorphic triangles has the following pieces:
  \begin{enumerate}[label=(T-\arabic*)]
  \item \label{e:fm} Limits where the point in $\BModPol$ consists of
    a pair of polygons, one a bimodule component containing $v_{13}$ and
    the other a bimodule map polygon component $v_{12}$ and
    $v_{23}$. These correspond to terms of the form
    \[ f_{1+n}^v(m_{1+m}^w(\x,a_1,\dots,a_m),a_{m+1},\dots,a_{m+n}).\] (This
    includes the case that the component containing $v_{12}$ is a
    bigon.)
  \item \label{e:mf} A pair of polygons, one a module map component containing $v_{12}$ and
    $v_{13}$ and the other a component containing $v_{23}$. These correspond to
    terms of the form
    \[ m_{1+n}^v(f_{1+m}^w(\x,a_1,\dots,a_m),a_{m+1},\dots,a_{m+n}).\] (This
    includes the case that the component containing $v_{13}$ is a
    bigon.)
  \item \label{e:fmu} Limits where the point in $\BModPol$ consists of
    a pair, or possibly a triple, of polygons, one a module map component containing all of
    $v_{12}$, $v_{23}$, and $v_{13}$ and a second one containing
    punctures and marked points on one of $e_2$ or $e_3$. In cases
    where the second component corresponds to a map to $e\infty$, there
    can be a third component that meets it, as with a composite boundary
    degeneration.  These ends correspond to terms of the form
    \[ f_{1+n}^v(\x,a_1,\dots,a_{i-1},\mu_{1+m}^w(a_i,\dots,a_{i+m}),\dots,a_{m+n}).\]
    (This includes the case $m=-1$, which corresponds to an orbit
    curve end.)
  \item \label{e:bdeg} Limits consisting of a module map component and a
    disk bubble with no interior or boundary punctures or marked
    points. There are algebraically zero ends of this form, because
    the count of disks with boundary on a single torus vanishes
    (Lemma~\ref{lem:OS-no-disks}) and standard gluing arguments.
  \item \label{e:pair} Limits consisting of a module map component containing $v_{13}$
    and $v_{23}$ and a polygon containing $v_{12}$.
  \end{enumerate}
  (See Figure~\ref{fig:TriangleDegens}.)

\begin{figure}
  \centering
  \includegraphics{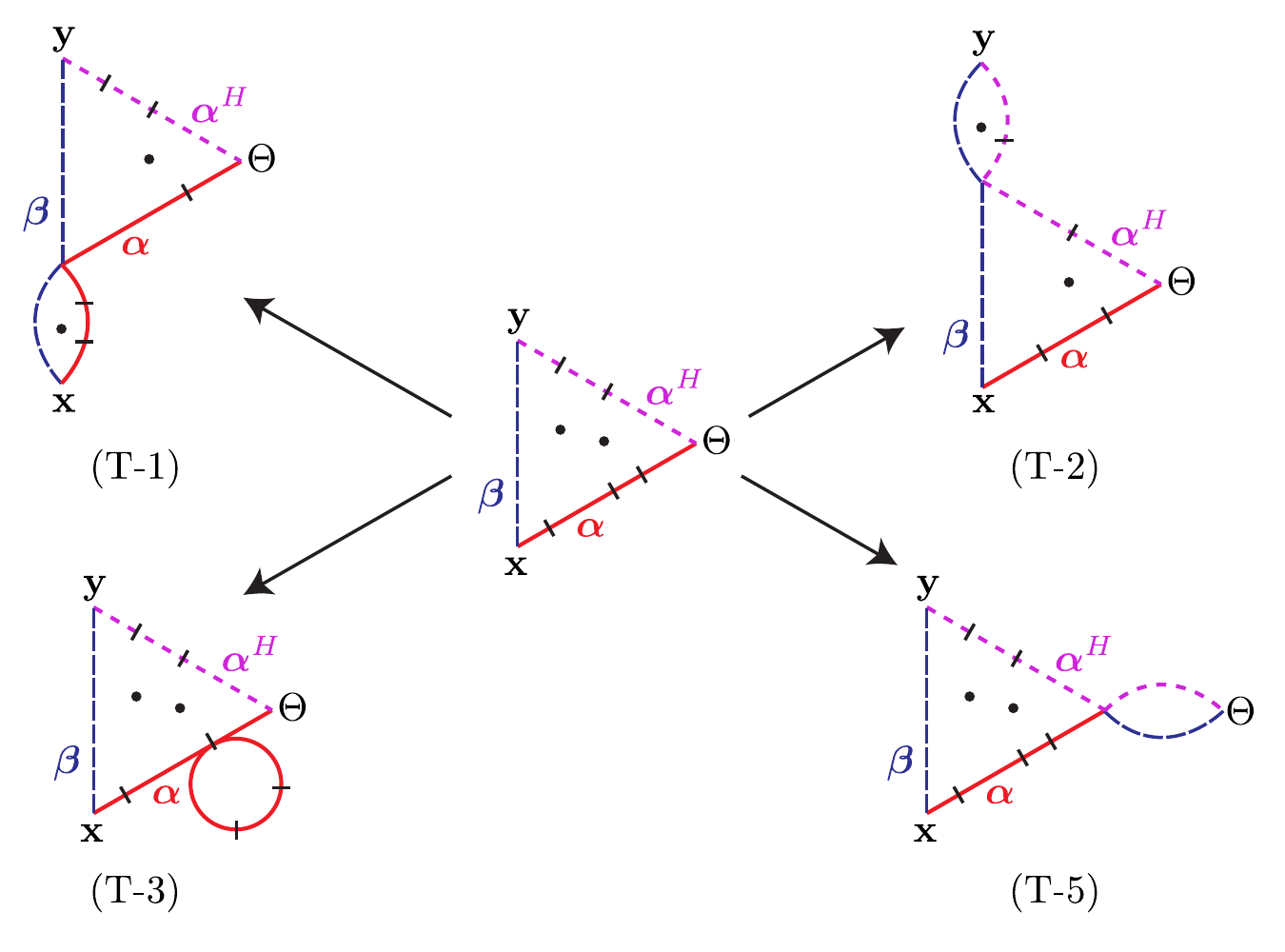}
  \caption[Triangle degenerations]{\label{fig:TriangleDegens}\textbf{Triangle degenerations.}  We
    have illustrated here four of the possible degenerations
    occurring for triangles in codimension one.}
\end{figure}

In more detail, the codimension one ends of the form~\ref{e:fm}
and~\ref{e:mf} are analogous to the two-story buildings from
Theorem~\ref{thm:master}. The fact that they each contribute with
multiplicity $1$ follows from a straightforward adaptation of
Proposition~\ref{prop:glue-2-story}. Thus, the algebraic contributions
of these ends to the one-dimensional moduli spaces are indeed
\begin{align*}
  f_{1+n}^v(m_{1+m}^w&(\x,a_1,\dots,a_m),a_{m+1},\dots,a_{m+n})
&{\text{and}}\\
  m_{1+n}^v(f_{1+m}^w&(\x,a_1,\dots,a_m),a_{m+1},\dots,a_{m+n}),
\end{align*}
as stated.

Type~\ref{e:fmu} ends can be classified: they can be collision ends (like
Case~\ref{end:Collision} from Theorem~\ref{thm:master}), where one of
the polygons is constant; they can correspond to an orbit curve degeneration
(like Case~\ref{end:EscapingOrbit}), where the degenerating polygon is
a 1-gon containing an orbit; or they can contain a configuration like
Case~\ref{end:CompositeBoundaryDegeneration} consisting of a
holomorphic triangle with a composite boundary degeneration
at one Reeb chord along $e_1$ or $e_2$.

The fact that only these kinds of Type~\ref{e:fmu} ends
contribute follows from index calculations and the fact the complex
structure is sufficiently pinched, as in the proof of
Theorem~\ref{thm:master}; and the
fact that each of these ends counts with multiplicity $1$ uses a
straightforward adaptation of the gluing methods used there
(cf.~Propositions~\ref{prop:glue-2-story},~\ref{prop:glue-split},~\ref{prop:glue-orbit},~\ref{prop:glue-pseudo-split},~\ref{prop:glue-degen},
and~\ref{prop:glue-join}). Thus, the contribution of these ends of
Type~\ref{e:fmu} is indeed
\[ f_{1+n}^v(\x,a_1,\dots,a_{i-1},\mu_{1+m}^w(a_i,\dots,a_{i+m}),\dots,a_{m+n}).\]
(Collision ends correspond to the case of $\mu_2^0$, orbit curve ends
to $\mu_0^1$, and the remaining ends to $\mu_m^w$ for $m+2w>0$.)

As mentioned above, ends of type~\ref{e:bdeg}
do not contribute to the count. It remains to consider ends of
Type~\ref{e:pair}, to prove that all non-trivial such ends cancel in
pairs.

\begin{figure}
  \centering
  \includegraphics{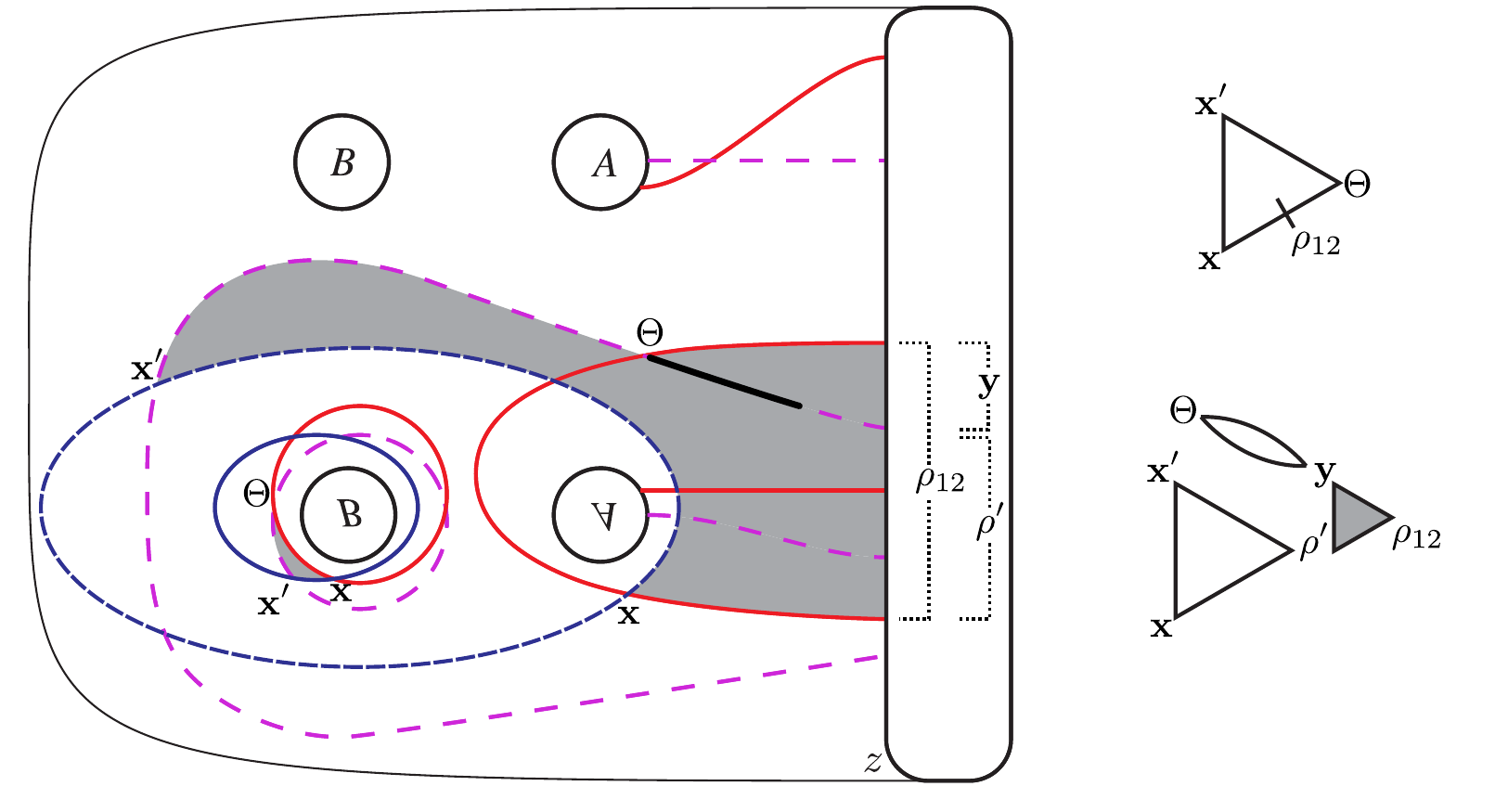}
  \caption[An end of type~\ref{e:pair}]{\textbf{An end of
      type~\ref{e:pair}.} On the left is a family of holomorphic
    triangles in the 1-dimensional moduli space
    $\cM(\x,\x',\Theta;\rho_{12};)$, whose domain is shaded. As the
    thick cut approaches the boundary of $\Sigma$, a degeneration of
    Type~\ref{e:pair} occurs. The components of the limit are
    indicated on the right; the shaded triangle is at $e\infty$.}
  \label{fig:pair-end-eg}
\end{figure}

The usual index considerations show that if the limit contains some
component containing $v_{12}$, but not $v_{13}$ and $v_{23}$, then in
fact that building consists of a rigid polygon containing $v_{13}$,
$v_{23}$, and some chords and orbits, a rigid polygon containing
$v_{12}$, in a $\cM^B_{J_\epsilon}(\Theta_\ell,\y,\vec{a},\vec{b})$,
for some sequence of chords $\vec{a}$ on $\alphas$ and $\vec{b}$ on
$\alphas^H$ and some generalized generator $\y$, and perhaps a
triangle at $e\infty$. (See Figure~\ref{fig:pair-end-eg} for an
example.)
These moduli spaces could {\em a priori} contain quite complicated
curves; for example, curves that cover the basepoint $z$.  However,
Lemma~\ref{lem:CancelInPairs} ensures $n_z(B)=0$. Thus, these ends
cancel in pairs exactly as they did in the $\HFa$ case~\cite{LOT1}
(see for example~\cite[Proposition~6.39]{LOT1}).
\end{proof}

\begin{proof}[Proof of Proposition~\ref{prop:handleslide-inv}]
  The proof now follows using the same logic as in the case of
  bordered $\HFa$;~\cite[Proposition~6.45]{LOT1} (which in turn is
  similar to the proof in the closed case~\cite{OS04:HolomorphicDisks}),
  so we will be brief.  Using
  $(\Sigma,\alphas^H,\alphas,\betas)$, we constructed a map by
  counting triangles $f\co \CFAmb(\Hab)\to\CFAmb(\HaHb)$, which we
  showed in Lemma~\ref{lem:HandleSlideMap} is a weighted
  $A_{\infty}$ homomorphism. Reverse the roles of $\alphas^H$ and
  $\alphas$ and displace the latter curves slightly to get a new
  set of curves $\alphas'$.
  Construct an analogous map $f'\co
  \CFAmb(\HaHb)\to\CFAmb(\HD_{\alphas',\betas})$, which is a weighted homomorphism by the
  same reasoning.
  Counting quadrilaterals (compare~\cite[Proposition~6.44]{LOT1})
  now gives a weighted $A_{\infty}$-homotopy between the 
  the composite
  \[ f'\circ f\co \CFAmb(\Hab)\to \CFAmb(\Hapb) \]
  and another triangle-counting map 
  \[ \phi\co \CFAa(\Hab)\to \CFAa(\Hapb),\] associated to the
  triple $(\Sigma,\alphas,\alphas',\betas)$.  (The fact that
  quadrilateral-counting furnishes a weighted $A_\infty$ chain
  homotopy uses an analogue of Lemma~\ref{lem:CancelInPairs}, which keeps algebra
  elements from escaping in $\alphas^H$-$\alphas$ and
  $\alphas^H$-$\alphas'$ bigons.)  We claim that the latter
  triangle-counting map $\phi$ is an isomorphism of weighted
  $A_\infty$ modules, which follows immediately once we show that
  $\phi^0_1\co \CFAmb(\Hab)\to\CFAmb(\Hapb)$ is an isomorphism of
  $\Field[U]$-modules.

  Setting $U=0$ in $\phi^0_1$, we obtain the isomorphism of
  $A_\infty$ modules $f^{\alphas,\alphas',\betas}\co \CFAa(\Hab)\to
  \CFAa(\Hapb)$ from the $\HFa$ case~\cite[Proof of
  Proposition~7.25]{LOT1}, which we now 
  denote $\widehat{\phi}^0_1$.
  Recall that our basepoints are on the boundary, so there is an
  isomorphism of chain complexes $\CFAmb(\Hab)\cong
  \CFAa(\Hab)\otimes_{\Field}\Field[U]$ and the map
  \[ \phi^0_1\co \CFAmb(\Hab)\to \CFAmb(\Hapb) \]
  is induced from $\widehat{\phi}^0_1$ by a change of coefficients (i.e., 
  tensoring with $\Field[U]$).
  Since $\widehat{\phi}^0_1$ is an isomorphism of vector spaces~\cite{LOT1},
  it follows that $\phi^0_1$ is an isomorphism of $\Field[U]$-modules.
  Hence, $f$ is a homotopy equivalence of weighted $\Ainf$-modules, as needed.
\end{proof}

\subsubsection{Cancellation of flows out of \textalt{$\Theta_\ell$}{the top graded generator}}
\label{subsec:cancel-in-pairs}

The aim of this section is to prove Lemma~\ref{lem:CancelInPairs}.
To do this, 
we will find it convenient to consider first a diagram
$\HD^0$, whose Heegaard surface has genus one and one boundary
component, equipped with two pairs of arcs 
$\alphas'=\{\alpha'_1,\alpha'_2\}$
and
$\alphas=\{\alpha_1,\alpha_2\}$.
(We suppress the ``a'' superscript as there are no $\alpha$-circles.)
Here, for $i=1,2$,  $\alpha'_i$ is a small isotopic
translate of $\alpha_i$, translated against the boundary 
orientation, as shown in 
Figure~\ref{fig:SmallDiagramH}. 
\begin{figure}
  \centering
  \includegraphics{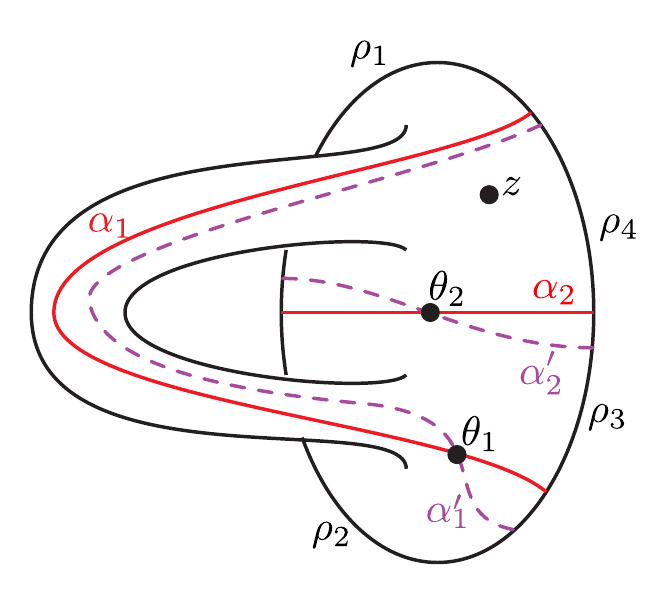}
  \caption[Isotopic curves]{\textbf{The genus one diagram $\HD^0$.} }
  \label{fig:SmallDiagramH}
\end{figure}
The Reeb chords $\rho_i$ connecting the endpoints in
$\alpha_1\cup\alpha_2$ have analogous nearby copies, which we also
denote $\rho_i$. For each $i=1,2$ there are two {\em distinguished short
  chords} from $\alpha'_i$ to $\alpha_i$ which do not contain
any points of $\partial \alpha'_j$ or $\partial \alpha^{a}_j$ in
their interior, and two {\em distinguished small bigons} $R^a_{i,1}$
and $R^a_{i,2}$ from $\theta_i^a$ to these four chords.

The natural analogues of generalized generators
(Definition~\ref{def:GenGen}) are the intersection point
$\theta_1^a$ of $\alpha_1$ with $\alpha_1'$, the intersection
point $\theta_2^a$ of $\alpha_2$ with $\alpha_2'$, and any Reeb
chord connecting some endpoint of $\alpha_i$ with $\alpha'_j$.

We will be concerned with moduli spaces
$\cM^B(x,y;\vec{a},\vec{b};w)$, where $x\in\{\theta_1^a,\theta_2^a\}$,
$y$ is a
generalized generator, $\vec{a}$ is a sequence of algebra
elements whose Reeb
chords connect endpoints of $\alpha_1\cup\alpha_2$, and
$\vec{b}$ is a sequence of algebra elements whose Reeb chords connect endpoints of
$\alpha_1'\cup\alpha_2'$. (The Reeb sequence $\vec{a}$ is
listed in order of increasing $t$ parameter in $[0,1]\times \RR$,
while the sequence $\vec{b}$ is listed in decreasing $t$
parameter, as suggested by the boundary orientation on $[0,1]\times
\RR$.)  Note that our curves have a degree $g=1$ projection to the
bigon; in particular, the source curves all have genus $0$.

\begin{example}
  Consider the moduli space
    $\cM^B(\theta_1^a,\rho_1^*; (\rho_3,\rho_{23},\rho_2), 
    (\rho_{41},\rho_4);0)$, where $\rho_1^*$ 
    is the generalized generator represented by the chord connecting $\alpha'_1$ to $\alpha_2$ approximating $\rho_1$, and $B$ is the domain illustrated in
  Figure~\ref{fig:SampleCurve}.  This is a one-dimensional moduli
  space of disks (after dividing out by the usual $\RR$ action)
  parameterized by a cut parameter at $\theta_1^a$.
\end{example}

\begin{figure}
  \centering
  \includegraphics{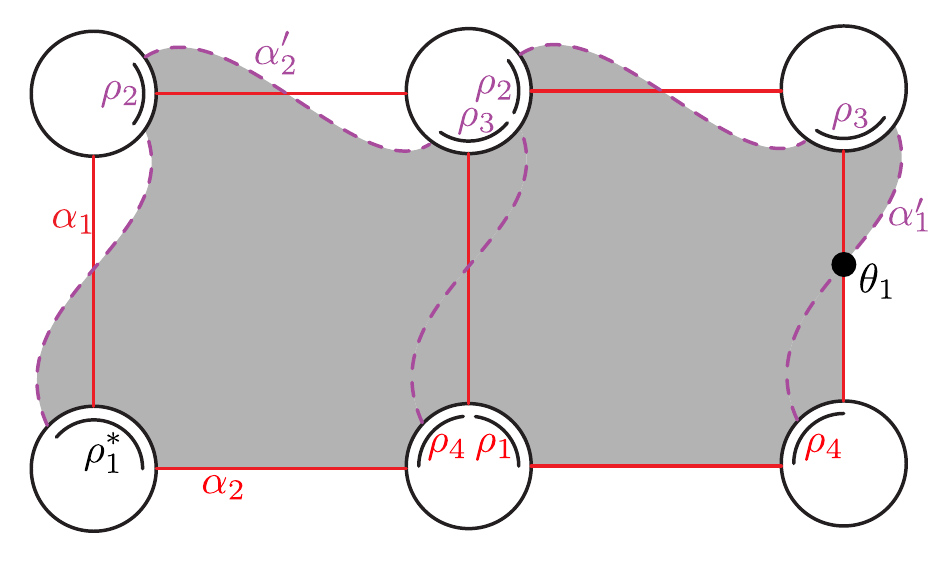}
  \caption[Disks out of $\theta_1^a$]{\textbf{Disks out of $\theta_1^a$.} 
    The shaded domain $B$ can be represented by a one-dimensional (index two) moduli space
    of pseudo-holomorphic disks. Asymptotics are specified by
    $\cM^B(\theta_1^a,\rho_1^*; \{\rho_3,\rho_{23},\rho_2\}, 
    \{\rho_{41},\rho_4\})$.
  \label{fig:SampleCurve} }
\end{figure}

\begin{lemma}
  \label{lem:IndexBranching}
  Fix $\ell\in \{1,2\}$.  Let $u$ be a non-constant holomorphic curve
  in some moduli space $\cM^B(\theta_\ell^a,y,\vec{a},\vec{b};w)$, and
  let $\br$ denote the number of branched points of the projection
  from $S$ to $[0,1]\times\RR$. Suppose $m$ of the elements of
  $\vec{a}\cup\vec{b}$ are not Reeby (i.e., have the form
  $U^i\iota_j$). Then, $u$ lies in a moduli space with expected
  dimension given by
  \begin{equation}
    \label{eq:IndexBranching}
    \ind=1+2\br+m.
  \end{equation}
\end{lemma}

\begin{proof}
  The result is, effectively, about disks in the plane, and hence is
  standard~\cite{SilvaRobbinSalamon}; nonetheless, we give a proof in
  the spirit of this present work, along the lines 
  of the proof of
  Lemma~\ref{lem:bdy-deg-ind-2}.  There are two slightly different
  cases, according to whether $y$ is an intersection point or a chord.

  Each non-Reeby element adds one unconstrained marked point, hence
  increases the expected dimension by $1$. So, it suffices to prove
  the result when all of the chords are Reeby.

  Assume first that it $y\in\{\theta_1^a,\theta_2^a\}$.
  Analogous to the proof of Lemma~\ref{lem:bdy-deg-ind-2},
  Proposition~\ref{prop:index-source} gives
  \[
    \ind(B,\Source,\vec{a},\vec{b},w,0) = 1-\chi(\Source)+2e(B)+|\vec{a}|+|\vec{b}|+w,
  \]
  as in this case $g=1$; also, we are only considering the case of simple
  orbits, i.e., $r=0$, and drop $r$ from the notation.
  Let $k=n_z(B)$, so $e(B)=-k$.
  Since $\pi_\CDisk\circ u$ is a $1$-fold cover, $\Source$ has genus $0$, so
  \begin{equation}
    \label{eq:EulerChar}
    \chi(\Source)=1-w; 
  \end{equation}
  and we find that 
  \begin{equation}
    \label{eq:IndexInPlane}
    \ind(B,\Source,\vec{a},\vec{b},w,0)=1-(1-w)-2k+|\vec{a}|+|\vec{b}|+w = 2w-2k+|\vec{a}|+|\vec{b}|.
  \end{equation}

  We can compute $e(\Source)$ either by projecting to the Heegaard surface,
  to find that 
  \begin{equation}
    \label{eq:EulerMeasureFromH}
    e(\Source)=-k-\br;
  \end{equation} or by comparing $e(S)$ with $\chi(S)$,
  via the formula 
  \begin{equation}
    \label{eq:EulerMeasureChar}
    e(S)=\chi(S)-|\vec{a}|/2-|\vec{b}|/2-1/2=1/2-w-|\vec{a}|/2-|\vec{b}|/2.
  \end{equation}
  (The term $1/2$ comes from the two generators,
  each of which contributes $1/4$ to the Euler measure.)
  Equations~\eqref{eq:EulerMeasureFromH} and~\eqref{eq:EulerMeasureChar} give
  \begin{equation}
    \label{eq:kis-2}
    -2k = 2\br+1-2w-|\vec{a}|-|\vec{b}|.
  \end{equation}
  Combining Equations~\eqref{eq:kis-2} and~\eqref{eq:IndexInPlane}
  gives the result, Equation~\eqref{eq:IndexBranching}, in the case
  where $y$ is not a chord.

  Next, consider the case that $y$ is a chord.
  In this case, $e(B)=-k+1/4$.
  The analogue of Proposition~\ref{prop:index-source} is
  \begin{equation}\label{eq:hs-index-source}
    \ind(B,\theta_1^a,y,\Source,\vec{a},\vec{b};w,0)= \frac{3}{2}-\chi(\Source)+2e(B)+  |\vec{a}|+|\vec{b}|+w.
  \end{equation}
  (Note that in the
  reduction to the closed case, we must now glue one half-bigon, whose
  Euler measure is 1/4; i.e., following the proof of
  Proposition~\ref{prop:index-source}, we now glue a region $B'$ with
  $e(B')=e(B)+|\vec{a}|/2 + |\vec{b}|/2 + 1/4$.)
  Equation~\eqref{eq:EulerMeasureFromH} is replaced by
  \[     e(\Source)=-k+1/4-\br, \]
  while Equation~\eqref{eq:EulerMeasureChar} remains unchanged.
  Equation~\eqref{eq:IndexBranching} follows.
\end{proof}

\begin{proposition}
  \label{prop:IndexOneFlowsIsotopyDiag}
  Let $\ell\in\{1,2\}$.
  If $u$ is a pseudo-holomorphic curve in some rigid moduli space
  $\cM^B(\theta_\ell^a,y,\vec{a},\vec{b})$, then $u$ is one of the two small bigons from $\theta_\ell^a$ to the
  boundary. In particular the sequences $\vec{a}$ and
  $\vec{b}$ are empty, and $y$ is one of the two short Reeb
  chords connecting a point on the boundary of $\alpha_1$ to the
  corresponding point on the boundary of $\alpha'_1$.
\end{proposition}

\begin{proof}
  By Lemma~\ref{lem:IndexBranching}, our curve and all other curves in
  its moduli space have no branch points.  Further, near the initial
  puncture, the curve is asymptotic to an arc connecting one of the
  two quadrants out of $\theta_\ell^a$ (since otherwise there is a curve in
  the same moduli space with a branch point).  That quadrant is
  contained in one of the two distinguished small bigons $R^a_{1,j}$ out of
  $\theta_1^a$. In the rest of
  the proof we will show that the entire curve consists of that
  bigon. (This is clear, but we give a detailed argument.)

  Let $U$ be a collar neighborhood of the boundary of the Heegaard
  surface, large enough to contain $R^a_{1,j}$. Equivalently, we can
  think of $U$ as a neighborhood of $p\in {\overline \Sigma}$.
  Foliate $U$ by concentric circles $\{C_t\}_{t\in (0,1]}$ around $p$,
  restricting to a foliation of $R^a_{1,j}$ by arcs. See
  Figure~\ref{fig:Foliation}.

  \begin{figure}
    \centering
    \includegraphics{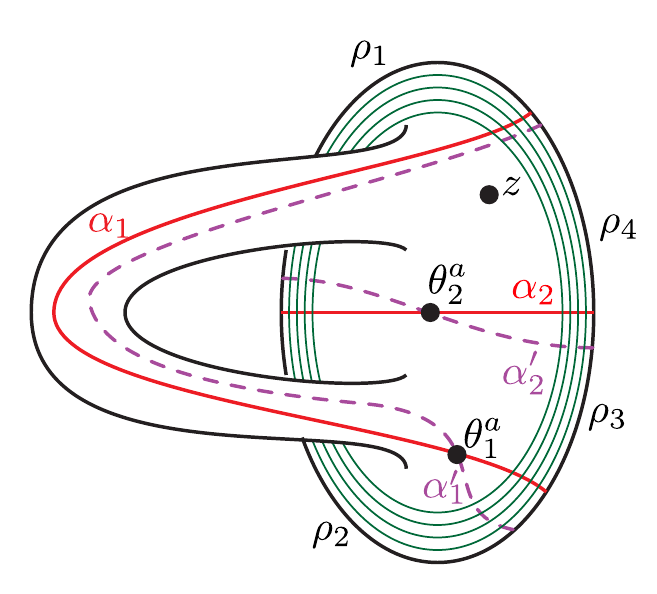}
    \caption[Foliation by circles]{\textbf{Foliation by circles.} 
      \label{fig:Foliation} The foliation is indicated by the
      \textcolor{darkgreen}{thin} ellipses. The foliation continues
      until it covers $\theta_1^a$.}
  \end{figure}

  Take the preimage of this foliation to obtain a partial foliation
  on the source curve $\Source$. (This foliation is
  defined only over $(\pi_\Sigma\circ u)^{-1}(U)$; as such, the
  foliation might not be defined near the $+\infty$ puncture of
  $\Source$, provided that the latter is not labeled by a Reeb
  chord; but it is defined near all other punctures.)  The asymptotics
  around each puncture ensure that the induced partial foliation
  is indeed a foliation around it by shrinking arcs. (See Figure~\ref{fig:FoliateSource}.) The condition
  of no branching ensures that the foliation by concentric arcs around
  the initial vertex extends until the point where one of the
  concentric arcs hits another boundary puncture.
  \begin{figure}
    \centering
    \includegraphics{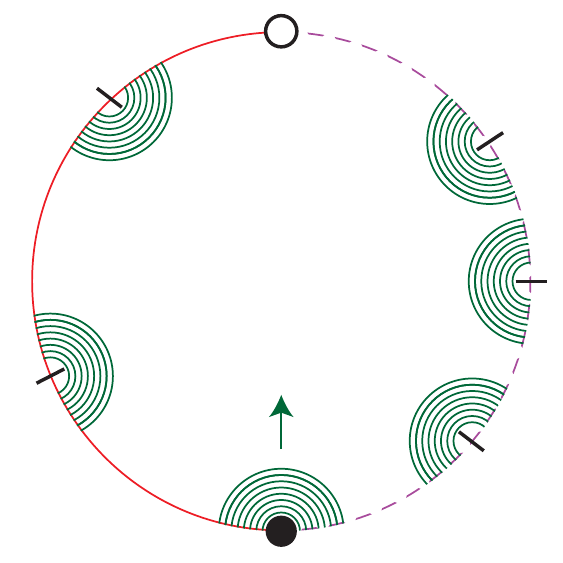}
    \caption[Induced foliation on the source]{\textbf{Induced foliation on the source.} 
      \label{fig:FoliateSource} The initial puncture $(-\infty)$ is indicated
      by the black dot. The induced partial foliation is indicated by the
      \textcolor{darkgreen}{thin} arcs.}
  \end{figure}
  But these conditions exclude the existence of any boundary
  punctures other than the two at $\pm \infty$, which in turn
  ensures that $B=R^a_{1,j}$.
\end{proof}

With these preliminaries in place, we turn our attention to
Lemma~\ref{lem:CancelInPairs}, which is proved by
destabilizing to the genus $1$ case.

\begin{proof}[Proof of Lemma~\ref{lem:CancelInPairs}]
  Fix a separating curve $C_1\subset \HD$ so that the $\alpha_i^a$ and
  $\alpha_i^{a,H}$, $i=1,2$, and $\alpha_1^c$, and $\alpha_1^{c,H}$
  are on one side of $C_1$, and the $\alpha_j^c$ and $\alpha_j^{c,H}$,
  $j=2,\dots,g-1$, are on the other side. Choose two more curves
  $C_2$, $C_3$ that together separate the $\alpha_i^a$ and
  $\alpha_i^{a,H}$ from $\alpha_1^c$ and $\alpha_1^{c,H}$. See
  Figure~\ref{fig:handleslide-pinch}.

  Let $J_\epsilon$ be a one-parameter family of tailored complex
  structures which are fixed outside a neighborhood of
  $C_1\cup C_2\cup C_3$ and so that a neighborhood of each $C_i$ is
  biholomorphic to $\{z\in\CC\mid 1<|z|<1/\epsilon\}$.
  Taking Gromov limits of curves in
  $\cM^B_{J_\epsilon}(\Theta_\ell,\y,\vec{a},\vec{b})$ as $\epsilon\to 0$
  gives rise to a curve in
  $\HD_0$. We will deduce topological properties of $B$ from the
  existence of these limits. For example, we will prove that if
  $\cM^B_{J_\epsilon}(\Theta_1,\y,\vec{a},\vec{b})$ is a rigid moduli space
  which is non-empty for all sufficiently small $J_\epsilon$,  then
  \begin{equation}
    \label{eq:ZeroMultAtW}
    n_z(B)=0.
  \end{equation}

  \begin{figure}
    \centering
    \includegraphics{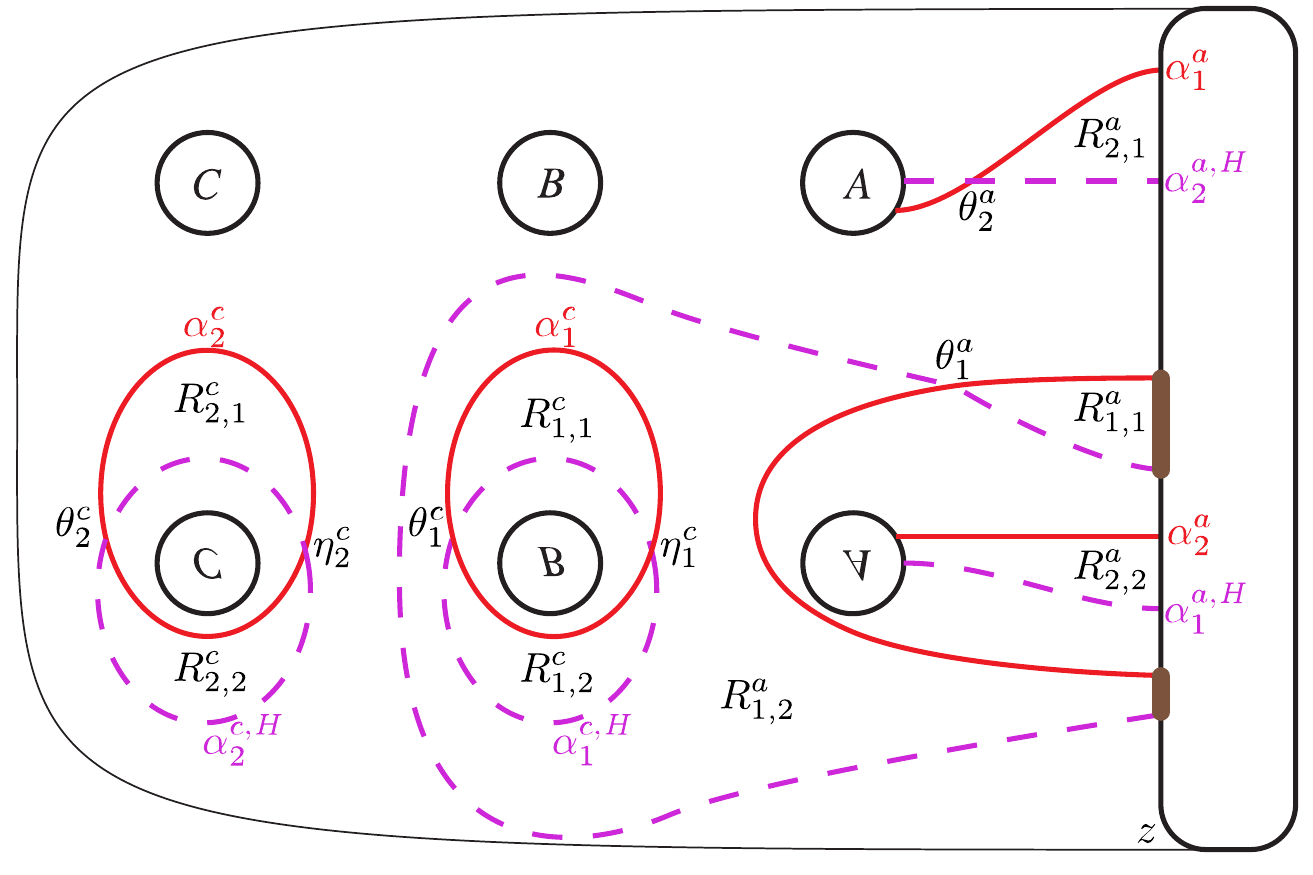}
    \caption[Regions in the handleslide diagram]{\textbf{Regions in
        the handleslide diagram.} The regions $R_{i,j}^c$ and
      $R_{i,j}^1$ are labeled.  figure is also adapted
      from~\cite[Figure 6.3]{LOT1}.}
    \label{fig:hs-curves-regions}
  \end{figure}

  In comparing the Heegaard diagram $\HD$ with $\HD_0$, note that
  $\HD$ has $2g-2$ elementary domains not present in $\HD_0$,
  corresponding to bigons connecting $\theta_i^c$ to $\eta_i^c$ with boundary
  on $\alpha_i^c\cup\alpha_i^{c,H}$. We label these domains
  $\{R^c_{i,1},R^c_{i,2}\}_{i=1,\dots,g-1}$; see
  Figure~\ref{fig:hs-curves-regions}.  Let $r^c_{i,j}$ denote
  the local multiplicity of $B$ at $R^c_{i,j}$ and let $k_{i,1}$ and
  $k_{i,2}$ denote the local multiplicities of $B$ at the other two
  regions adjacent to $\theta_i^c$.  (Note that when $i\neq 1$,
  $k_{i,1}=k_{i,2}$, since the two regions are not separated by any
  $\alpha$ or $\alpha^H$ curve.)  Let $J$ be the set of
  $i\in\{1,\dots,g-1\}$ for which $y_i\neq \theta^c_i$.  For each
  $i\in\{1,\dots,g-1\}$, we have the relation
  \[     
    r^c_{i,1}+r^c_{i,2}=k_{i,1}+k_{i,2}+\begin{cases}1 &{\text{if $i\in J$}} \\
      0 &{\text{otherwise.}}
    \end{cases}
  \]
  It follows that, if $K=\sum_{i=1}^{g-1}(k_{i,1}+k_{i,2})$
  then
  \begin{equation}\label{eq:ind-terms-destab}
    \begin{split}
      e(B)&=e(B_1)-K+\OneHalf|J|\\
      n_{\Theta_1}(B)&=n_{\theta_1^a}(B_1)+\frac{K}{2}+\frac{J}{4} \\
      n_{\y}(B)&=n_{y}(B_1)+\frac{K}{2}+\frac{J}{4}.
    \end{split}
  \end{equation}
  The proof of the embedded index formula, Formula~\eqref{eq:emb-ind},
  using the tautological correspondence
  (Section~\ref{sec:tautological}) and Sarkar's formula for the Maslov index of polygons in the
  symmetric product, applies equally well in this setting.
  So, by Formula~\eqref{eq:emb-ind},
  \begin{equation}\label{eq:ind-B-ind-B1}
    \ind(B,\Theta_1,\y,\vec{a},\vec{b},w,r)=
    \ind(B_1,\theta_1^a,y,\vec{a},\vec{b},w,r)+|J|.
  \end{equation}
  
  Applying compactness to the sequence
  $u_T\in\ModFlow^B_{J_\epsilon}(\Theta_1,\y,\vec{a},\vec{b})$ as
  $\epsilon\to 0$ gives rise to a holomorphic curve
  $u\in \ModFlow^{B_1}(\theta_1^a,y,\vec{a},\vec{b})$. (See also
  the proof of Lemma~\ref{lem:bdy-deg-ind-2}.)  
  Since we assumed that $\ind(B,\Theta_1,\y,\vec{a},\vec{b},w,r)=1$,
  we can conclude from Lemma~\ref{lem:IndexBranching} that
  $|J|\leq 1$. If $|J|=1$, then $u$ is a constant curve, so
  $n_z(u)=0$. If $|J|=0$, then $u$ has index $1$, so
  Proposition~\ref{prop:IndexOneFlowsIsotopyDiag} ensures that
  $n_z(B_1)=0$.
\end{proof}

\subsection{Stabilization invariance}\label{sec:CFA-stab-invariance}
Recall that by \emph{stabilizing} a bordered Heegaard diagram
$(\Sigma,\alphas,\betas)$ we mean taking the connected sum with the
standard genus-one Heegaard diagram for $S^3$ at some point in the
interior of $\Sigma\setminus (\alphas\cup\betas)$. (See~\cite[Section
4.1]{LOT1} for more discussion.)

\begin{proposition}\label{prop:stab-inv}
  Let $\HD=(\Sigma,\alphas,\betas)$ be a bordered Heegaard diagram and
  $\HD'=(\Sigma',\alphas',\betas')$ the result of stabilizing
  $\HD$. Then for appropriate choices of tailored almost complex
  structures for $\HD$ and $\HD'$ there is an isomorphism of weighted
  $\Ainf$-modules $\CFAmb(\HD)\cong \CFAmb(\HD')$.
\end{proposition}
\begin{proof}
  This follows from the same argument as in
  Section~\ref{subsec:StabilizationInvarianceBDeg} (which again is an
  extension of the usual proof of stabilization invariance for
  Heegaard Floer homology). The strategy is to show that for an
  appropriate choice of almost complex structures, the moduli spaces
  of holomorphic curves for $\HD$ and $\HD'$ are identified. The
  almost complex structures for $\HD'$ are obtained by stretching the
  neck along a pair of curves. The amount of stretching required
  depends, at least in principle, on the moduli space being
  considered. So, the proof proceeds one energy level at a time,
  producing a complex structure $J_E$ for $\HD$ and $J'_E$ for $\HD'$
  for which the rigid moduli spaces of energy $E$ holomorphic curves are
  identified, so that the $J_E$ and $J'_E$ form tailored families.

  Let $T\subset \Sigma'$ be the new
  punctured torus, so $\Sigma'\setminus T$ is identified with
  $\Sigma\setminus D^2$. Choose a tailored family of almost complex structures
  $J'$ for $\HD'$ which is split over $T$.  Let
  $\alpha_{g}^c\subset T$ be the new $\alpha$-circle and $\beta_{g+1}$
  the new $\beta$-circle. Let $Z=Z_1\amalg Z_2\subset T$ be the boundary of a
  tubular neighborhood of $\alpha_g^c$, chosen so that each $Z_i$
  intersects $\beta_{g+1}$ in a single point.  We will
  stretch the neck along $Z_1$ and $Z_2$, after with $\Sigma'$
  degenerates into a copy of $\Sigma$ and a sphere $\Sphere$. Let
  $z_L$, $w_L$ (respectively $z_R$, $w_R$) be the points in $\Sphere$
  (respectively $\Sigma$) corresponding to $Z_1$ and $Z_2$. Identify
  $\Sphere$ with $\CC\cup\{\infty\}$ so that $z_L$ corresponds to $0$
  and $w_L$ corresponds to $\infty$. Using isotopy invariance, we may
  assume that $\alpha_{g}^c\cap \Sphere$ is identified with the
  non-negative real axis $\RR_{\geq 0}$ and $\beta_{g+1}$ is identified
  with the unit circle $S^1$. Also because of isotopy invariance, we
  will not distinguish between the copy of $\Sigma$ obtained by
  stretching the neck and the original copy in $\HD$.

  For a generic, sufficiently pinched choice of $J'$, the family of
  almost complex structures $J$ in the limit is still tailored. (This
  uses the fact that the circle $Z$ is disjoint from the circle being
  pinched to make $J'$ or $J$ sufficiently pinched.)

  The proof is now inductive on the energy $E$. Assume that for all
  $\vec{a}$ and $w$ with $E(\vec{a},w)<E_0$ we have found a length $T$ so
  that the rigid moduli spaces with respect to $J_T$ and $J$ with
  asymptotics $(\vec{a},w)$ are identified. We will find a perhaps
  larger $T$ with this property for pairs $(\vec{a},w)$ with
  $E(\vec{a},w)=E_0$, and use this more stretched $J_T$ for energy
  $E_0$ moduli spaces (in our tailored family). There are finitely many 
  $(\vec{a},w)$ with $E(\vec{a},w)=E_0$, and finitely many positive domains $B$
  compatible with each $(\vec{a},w)$, so we may work one triple
  $(\vec{a},w,B)$, with $\ind(B,\vec{a},w)=1$, at a time.
  As we stretch the neck,
  a compactness argument (similar to the proof of~\cite[Proposition
  5.24]{LOT1} or~\cite[Proposition 9.6]{LOT1}) implies that any
  sequence $\{u_i\}$ of $J_T$-holomorphic curves in the homology class
  $B$ with asymptotics $(\vec{a},w)$ has a subsequence which
  converges to a pair $(u_L,u_R)$ where
  \begin{align*}
    u_L&\co (S_L,\bdy S_L)\to \bigl((\CC\cup\{\infty\})\times[0,1]\times\RR, \RR_{\geq 0}\cup S^1\bigr)\\
    u_R&\co (S_R,\bdy S_R)\to
         \bigl(\Sigma\times[0,1]\times\RR,(\alphas\times\{1\}\times\RR)\cup (\betas\times\{0\}\times\RR)\bigr).
  \end{align*}
  (The curve $u_L$ is analogous to the stabilizing boundary
  degenerations of Definition~\ref{def:stabilizing-degen}.) Elementary
  complex analysis, considering $\pi_\bD\circ u_R$, shows that in fact
  no component of $\bdy S_R$ maps to $\alpha_g^c\cap \Sigma$. The pair
  $(u_L,u_R)$ also satisfies a \emph{matching condition}: the multisets
  \[
    A=(\pi_\bD\circ u_R)\bigl((\pi_\Sigma\circ u_R)^{-1}(z_R)\bigr)\qquad\text{and}\qquad
    (\pi_\bD\circ u_L)\bigl((\pi_\Sphere\circ u_L)^{-1}(z_L)\bigr)
  \]
  agree, as do the sets
  \[
    B=(\pi_\bD\circ u_R)\bigl((\pi_\Sigma\circ u_R)^{-1}(w_R)\bigr)\qquad\text{and}\qquad
    (\pi_\bD\circ u_L)\bigl((\pi_\Sphere\circ u_L)^{-1}(w_L)\bigr).
  \]
  In particular, the degree of $\pi_\Sphere\circ u_L$ is the same as
  the multiplicity of the domain $B$ at $Z_1$ and $Z_2$.
  
  It follows from the index formula (Proposition~\ref{prop:emb-ind})
  that the index of the sequence converging to $(u_L,u_R)$ is the same
  as the index at $u_R$. In particular, since we are only interested
  in index-1 moduli spaces, $u_R$ is rigid. So, for a generic choice
  of almost complex structure, none of the points in $A$ or $B$ has
  multiplicity bigger than one, and $A\cap B=\emptyset$.

  We claim that for each choice of $A$ and $B$, there is a unique
  curve $u_L$ satisfying the matching condition.
  The projection $\pi_{\bD}\circ u_L$ is a degree-one branched cover,
  so an analytic isomorphism. So, the moduli space of maps $u_L$ is
  identified with the moduli space of maps from the strip to
  $\Sphere$. Identify $[0,1]\times\RR$ with the upper half-disk
  $\bD^+=\{x+iy\in\CC\mid x^2+y^2\leq 1,\ y\geq 0\}$ so that
  $\{1\}\times\RR$ is identified with the real axis. This moduli
  space is the same as the space of holomorphic maps $\bD^+\to \CC$
  sending the real axis to the positive real axis and the top
  semi-circle to $S^1$. By Schwartz reflecting (twice), this space is
  the same as the space of maps $f\co \CC P^1\to \CC P^1$ with
  $f(S^1)\subset S^1$ and $f(\RR)\subset \RR_{>0}$. Up to an scaling
  by a complex number, such a function is determined by its zeroes and
  poles: given two entire functions $f,g$ with the same zeroes and poles, $f/g$
  is an entire function with no zeroes (or poles), and hence is
  constant. The zeroes and poles satisfy the restrictions that zeroes
  (respectively poles) come in conjugate pairs (since $\RR$ maps to
  $\RR$) and zeroes correspond to poles under the map $z\mapsto 1/z$
  (since $S^1$ maps to $S^1$). Thus, up to scaling, such a function is
  determined by its set of zeroes and poles inside $\bD^+$. Again up
  to scaling, the condition that $f(\RR)\subset \RR_{>0}$ (rather than
  just $\RR$) is equivalent to $f$ having no zeroes or poles on
  $\RR$. Finally, given a finite set $C\subset \bD^+$, let $g_C$ be
  the Blaschke product
  \[
    f_C(z)=\prod_{c\in C} \left(\frac{z-c}{1-\overline{c}z}\right).
  \]
  Then given a set $Z$ of zeroes and $P$ of poles in the interior of
  $\bD^+$ (with $Z\cap P=\emptyset$) the function $f_Z(z)/f_P(z)$ has
  the specified zeroes and poles. To summarize, given disjoint sets of
  zeroes and poles in the interior of $\bD^+$, there is a unique
  function with these zeroes and poles.

  Thus, for each curve $u_R$ there is a unique curve $u_L$ satisfying
  the matching condition. So, by a gluing argument, the space of
  holomorphic curves with respect to $J_T$ for $T$ sufficiently large
  is in bijection with the space of holomorphic curves $u_R$. Since
  the latter is exactly the space used to define the operations on
  $\CFAmb(\HD)$, and the former the operations on $\CFAmb(\HD')$, the
  result follows.
\end{proof}

\subsection{The invariance theorem}
We assemble these pieces to quickly deduce that $\CFAmb$ is well-defined;
this is the analogue of Theorem~\ref{thm:CFAm-invt} for $\CFAmb$:

\begin{theorem}\label{thm:CFAmb-invt}
  Up homotopy equivalence (of weighted $\Ainf$-modules),
  $\CFAmb(\HD,\spinc)$ depends only on the bordered 3-manifold
  $(Y,\phi)$ specified by $\HD$, and its underlying $\SpinC$ structure $\spinc$.
\end{theorem}

\begin{proof}[Proof of Theorem~\ref{thm:CFAm-invt}]
  Independence of the almost complex structure is
  Proposition~\ref{prop:J-inv}. Next, any two bordered Heegaard
  diagrams representing the same bordered 3-manifold can be connected
  by a sequence of isotopies, handleslides over circles, and
  stabilizations~\cite[Proposition 4.10]{LOT1}, so this is immediate
  from Propositions,~\ref{prop:iso-inv},~\ref{prop:handleslide-inv},
  and~\ref{prop:stab-inv}.
\end{proof}


\section{Grounding \texorpdfstring{$U$}{U}}
\label{sec:GroundingU}

The aim of the present section is to modify the construction of
$\CFAmb$ above to define a  unital, weighted module
$\CFAm(\HD)$ whose operations are $U$-equivariant (i.e., which is
defined over the ground ring $\Ground[U]=\Ground\otimes\Field[U]$).

We construct $\CFAm$ in three steps:
\begin{enumerate}[label=(S-\arabic*),ref=(S-\arabic*)]
\item 
  \label{s:GroundingAlgebra}
    First, we define a larger weighted $\Ainf$-algebra
  $\MAlgg$, called the {\em grounding algebra}, equipped with a
  subalgebra $B\subset \MAlgg$ which is isomorphic, as a weighted
  $\Ainf$-algebra, to $\MAlg$.  (This isomorphism is not
  $U$-equivariant.)  Thus, we can view $\CFAmb$ as a weighted
  $\Ainf$-module over $B$.
\item 
  \label{s:ExtendToGroundingAlgebra}
  Next, we extend the actions on $\CFAmb$ 
  in a $U$-equivariant way to all of
  $\MAlgg$, to
  construct a weighted module $\CFAmg$ over $\MAlgg$.
\item 
  \label{s:RestrictFromGroundingAlgebra}
  There is a $U$-equivariant quasi-isomorphism
  $\MAlgg\simeq \MAlg$. Restriction of scalars now allows us to turn
  any $\MAlgg$-module into a (quasi-isomorphic)
  $\MAlg$-module. Applying this process to $\CFAmg$ gives $\CFAm$.
\end{enumerate}

\subsection{The grounding algebra}\label{sec:grounding-alg}

Informally, $\MAlgg$ from Step~\ref{s:GroundingAlgebra}
is obtained from $\MAlg$ by adding elements $X$ and $e$, satisfying
\[
d(X)=1+e
\qquad{\text{and}}\qquad
X^2=0,
\]
and replacing the algebra operations that output $U^n$ (respectively $U^n
\cdot \rho$) by operations that output $(U e)^n$ (respectively $(U e)^n\cdot
\rho$).  

More formally, the underlying algebra of $\MAlgg$ 
is obtained from the algebra underlying $\UnDefAlg$
by adjoining two commuting 
elements $e$ and $X$, subject to the relation $X^2=0$; i.e.,
\[
  A^g=\frac{\Field[U,e,X]}{(X^2)}\otimes _\Field\AsUnDefAlg.
\]
Equip $A^g$ with the differential $\mu_1^0\co A^g\to
A^g$ that vanishes on the subalgebra generated by
$\AsUnDefAlg$, $U$, and $e$, and that satisfies $\mu_1^0(X)=1+e$.  Let $B\subset \MAlgg$ be the
subalgebra generated by $U\cdot e$, $\rho_1$, $\rho_2$, $\rho_3$,
and $\rho_4$. 

There is an isomorphism of algebras $\phi\co
\MAlg\to B$ characterized by the property that $\phi$ is the
identity on $\AsUnDefAlg$ and $\phi(U)=U\cdot e$.  We can use this
isomorphism to define weighted $\Ainf$ operations on $B$;
explicitly, given a sequence $a_1,\dots,a_n\in \MAlg$, let
\[ \mu^w_n(\phi(a_1),\dots,\phi(a_n))=\phi(\mu^w_n(a_1,\dots,a_n)).\]
We extend these operations to all of $\MAlgg$ as follows.
The algebra $\MAlgg$ is generated by the subalgebra $B$, together
with the elements $U$, $e$, and $X$.  
Extend the operations on $B$ to maps
$\{\mu^w_n\co (A^g)^{\otimes n}\to A^g\}$
by the convention that all maps $\mu^w_n$ are equivariant under
multiplication by $U$, $e$, and $X$.  
Since $\mu_1^0X=1+e$ is central, the resulting operations satisfies the weighted
$\Ainf$ relations, giving $A^g$
the structure of weighted algebra $\MAlgg$, as desired.
By construction $B\subset \MAlgg$ is a weighted $\Ainf$-subalgebra.

\begin{example}
  In $B\subset \MAlgg$, we have the operation
  \[
    \mu^0_4(\rho_4,\rho_3,\rho_2,\rho_{12})=\rho_2 U e.
  \]
  This gives rise to further actions in $\MAlgg$, such as
  \begin{align*}
    \mu^0_4(\rho_4,\rho_3,X\cdot \rho_2,e\cdot\rho_{12})&=X \rho_2 U e^2 \\
    \mu^0_4(\rho_4,X\cdot \rho_3,X\cdot \rho_2,e\cdot\rho_{12})&=0.
  \end{align*}
\end{example}
The isomorphism $\phi\co \MAlg\to B$ allows us to view the module $\CFAmb$
as a weighted $\Ainf$-module over the $\Ainf$-subalgebra $B$,
which we denote  $\CFAmg(\HD)|_B$. More formally:

\begin{definition}
  \label{def:CFAmgB}
  Given a provincially admissible bordered Heegaard diagram $\HD$, the
  \emph{restricted module} $\CFAmg(\HD)|_B$ is defined as follows.  As
  a left $\Field[U]$-module it is freely generated by $\Gen(\HD)$.
  The weight $w$ algebra operation on $\CFAmg(\HD)|_B$ is specified on
  a generator $\x$ and a sequence of algebra elements
  $b_1,\dots,b_{n-1}\in B$ as
  $m^w_n(\x,\phi^{-1}(b_1),\dots,\phi^{-1}(b_{n-1}))$, where $m^w_n$
  is the action on $\CFAmb$ defined in Section~\ref{sec:CFA}.  That
  is, the action by $U^n e^n$ (respectively $U^n e^n\rho$) is
  identified with the action of $U^n$ (respectively $U^n\rho$) on
  $\CFAmb(\HD)$.
\end{definition}

This completes Step~\ref{s:GroundingAlgebra} in our construction of $\CFAm$.
Our next goal is to extend the restricted module
$\CFAmg(\HD)|_B$, defined over the weighted subalgebra $B\subset
\MAlgg$, to one defined over all of $\MAlgg$. We do this in Section~\ref{sec:grounding-CFA}, after setting 
up some preliminaries about almost complex structures and moduli spaces.

\begin{remark}
  Unlike the operations on $\MAlgg$ itself, the action of $\MAlgg$ on
  $\CFAmg$ constructed below is not $X$- or $e$-equivariant.
\end{remark}
    
\subsection{Families of complex structures}

In Section~\ref{sec:main-part}, we introduced the notion of
coherent families of $\eta$-admissible almost complex structures
(Definition~\ref{def:AdmissibleJs}), which associate a complex
structure on $P\times\Sigma$ to each conformal polygon $P$ with two
special vertices labeled $\pm \infty$ and all other vertices (the
``algebra vertices'') marked with integers, called their ``energies''.
Coherent families of almost complex structures satisfy
constraints imposed by the energies and
compatibility conditions as the underlying polygon
degenerates. In Section~\ref{sec:transversality}, we introduced the
notion of tailored families of almost complex structures
(Definition~\ref{def:tailored}), which are coherent families of
$\eta$-admissible almost complex structures satisfying various
further transversality conditions. The $\Ainf$ operations on the bordered
modules were then
defined by counting $J$-holomorphic curves over {\em
  bimodule polygons} in the sense of Definition~\ref{def:polygons},
where the almost complex structures over the product of the bimodule
polygon $P$ with the surface $\Sigma$ is endowed with a tailored
family of almost complex structures.  The energies of the punctures
of the bimodule
polygon, in turn, are specified by the pinching function associated to
the sequence of algebra elements.  

For the generalization to $\MAlgg$, the sequence of basic algebra
elements and polygon $P$ no longer specifies a single almost complex
structure (on $\Sigma\times P$): each copy of the new element $X$
corresponds to a 1-parameter family, as we explain next.

\begin{definition}\label{def:smear-bimod-poly}
  A {\em smeared bimodule polygon} is a bimodule polygon $P$ as in
  Definition~\ref{def:polygons}, except that now the sequence
  $\vec{E}=(E_1,\dots E_\ell)$, rather than being a sequence of
  integers, is a sequence of pairs of integers $E_i=(a_i,b_i)$ with
  $0\leq a_i\leq b_i$, called an {\em energy range}.
  Given a smeared bimodule polygon $P$, there is an associated
  hypercube $\mathbb{E}$, which is the product of intervals
  \[
    \mathbb{E}=\prod_{i=1}^k [a_i,b_i].
  \]
  Let $\ModPol^X_{k,\ell,\vec{E}',\vec{E},w}$ denote the
  moduli space of smeared bimodule polygons.
\end{definition}

Smeared bimodule polygons where all the $a_i=b_i$ (i.e. the hypercube
$\mathbb{E}$ is a point) and for which none of the $a_i=0$
correspond to  bimodule polygons as in Definition~\ref{def:polygons}.

Let $\ModPol^X$ denote the disjoint union of all the
$\ModPol^X_{k,\ell,\vec{E}',\vec{E},w}$.  For fixed $\vec{E}$, let
$\ModPol^X_{\vec{E}}\subset \ModPol^X$ denote the subset with given
$\vec{E}$ (i.e. where $k$, $\vec{E}'$, and $w$ are allowed to vary).

We adapt Definition~\ref{def:ac}. For our present purposes, a
\emph{pinching function} is a monotone-decreasing function $\eta\co
\RR_{>0}\to \RR_{>0}$.

\begin{definition}
  \label{def:acAgain}
  Fix $P\in \ModPol^X$.  An almost complex structure $J$ on the
  $\Sigma$-fibers of $\Sigma\times P\times \mathbb{E}$ is called
  {\em{$\eta$-admissible}} if for each $\mathbf{e}\in \mathbb{E}$, the complex
  structure $J|_{\Sigma\times P\times\{\mathbf{e}\}}$ satisfies
  Properties~\ref{item:J-piD}-\ref{NearlyConstant} from
  Definition~\ref{def:ac} and the following:
  \begin{enumerate}[label=(J-\arabic*${\thinspace}^\prime$)]
    \setcounter{enumi}{4}
  \item If $(a_i,b_i)$ is the energy range over the $i\th$ marked
    point of $P$, then $j_i$ is $\eta(a_i)$-pinched.
    (Note that this vacuously holds when $a_i=0$.)
  \end{enumerate}
\end{definition}

When the dimension of the hypercube $\mathbb{E}$ is zero, and
none of the components of $\mathbb{E}$ is marked by $(0,0)$,
Definition~\ref{def:acAgain} reduces to Definition~\ref{def:ac}.

The space of $\eta$-admissible almost complex structures is a bundle over
$\ModPol^X\times\mathbb{E}$.  Following Definition~\ref{def:ac}, a
coherent family of $\eta$-admissible almost complex structures is a
continuous section of this bundle.

Fix a sequence of pairs of integers $\vec{E}=((a_1,b_1),\dots,(a_k,b_k))$,
and let $\mathbb{E}$ be its associated hypercube.
If $b_j>a_j$, 
let $\vec{E}_{j;0}=((a_1',b_1'),\dots,(a_k',b_k'))$ 
be the sequence specified by 
$(a_i',b_i')=(a_i,b_i)$ if $i\neq j$, and $(a_j',b_j')=(a_j,a_j)$;
and $\vec{E}_{j;1}=((a_1'',b_1''),\dots,(a_k'',b_k''))$ specified similarly, except now 
$(a_j'',b_j'')=(b_j,b_j)$.
Let $\mathbb{E}_{j,0}$ and $\mathbb{E}_{j,1}$ denote their corresponding hypercubes.
Clearly, there are inclusion maps $\mathbb{E}_{j,0}\subset \mathbb{E}$ and 
$\mathbb{E}_{j,1}\subset \mathbb{E}$; and indeed
\begin{equation}
  \label{eq:BoundaryOfE}
  \partial\thinspace \mathbb{E}=\bigcup_{\{j\mid b_j>a_j\}} \left(\mathbb{E}_{j,0}
    \cup \mathbb{E}_{j,1}\right).
\end{equation}

\begin{definition}
  \label{def:AdmissibleJsAgain}
  Fix a pinching function $\eta$.  A \emph{coherent family of
    $\eta$-admissible almost complex structures} associates to each 
  sequence of pairs of integers $\vec{E}=(E_1',\dots,E_k')$
  a continuous section $J_{\vec{E}}$ of the bundle of $\eta$-admissible almost complex structures
  over $\ModPol_{\vec{E}}\times \mathbb{E}$, satisfying the following
  compatibility conditions:
  \begin{itemize}
    \item If $P$ is in the boundary of
      $\ModPol$ and $P'$ is a bimodule component of $P$, then the
      restriction of $J(P)$ to $P'$ agrees with $J(P')$.  
    \item For $\delta\in\{0,1\}$ and $j$ so that $b_j>a_j$, the restrictions of 
      $J$ to $\ModPol_{\vec{E}_{j,0}}\times \mathbb{E}_{j,0}\subset \ModPol_{\vec{E}}\times \mathbb{E}$
      and 
      $\ModPol_{\vec{E}_{j,1}}\times \mathbb{E}_{j,1}\subset \ModPol_{\vec{E}}\times \mathbb{E}$
      agree with $J_{\vec{E}_{j,0}}$ and $J_{\vec{E}_{j,1}}$.
  \end{itemize}
\end{definition}

We would like to establish the analogue of
Lemma~\ref{lem:admis-J-exists}, showing that admissible families $J$ exist.

For the purpose of the existence proof, it will be useful
 to slightly expand the notion of
bimodule polygons (in the unsmeared case,
Definition~\ref{def:polygons}), to include marked points with energy
constraint $0$. Specifically, such bimodule polygons
are equipped with  a function from the marked points (other than the two
distinguished boundary points) to $\ZZ_{\geq 0}$. Points with energy
$0$ are called {\em inert}. There is a
corresponding generalization of the notion of $\eta$-admissible
complex structure (compare Definition~\ref{def:ac}), which is a
complex structure over $\Sigma\times P$ satisfying the conditions from
that definition, with the understanding that
Condition~\ref{NearlyConstant} applies only over the sequence
$(1,t_1),\dots,(1,t_n)$ for which the corresponding energies are
non-zero. Thus, over the inert points there is no additional
requirement on the almost-complex structure.  Letting
$\ModPol_{\geq 0}$
denote the space of such module polygons, there is a continuous
forgetful map $\ModPol_{\geq 0}\to \ModPol$ which drops the inert points. Pulling back via this map
induces a map from the $\eta$-admissible complex structures over
$\ModPol$ to $\eta$-admissible complex structures over $\ModPol_{\geq 0}$.

Idempotents have energy zero; so, holomorphic curves over
$\ModPol_{\geq 0}$ appear when considering algebra sequences
$(a_1,\dots,a_j)$ where some of the $a_i$ are allowed to be
idempotents. If $a_i$ is marked by an idempotent, the space
$\tcM^B(\x,\y;a_1,\dots,a_n;w)$ is rarely rigid.  For example,
suppose that $i>1$. There is a map
\[ f\co \tcM^B(\x,\y;a_1,\dots,a_n;w)\to \RR_{>0}\]
defined by $u\mapsto t \circ u(p_i)-t\circ u(p_{i-1})$.
By moving $p_i$ around, we see that $f$ has open image.

\begin{lemma}\label{lem:admis-J-exists-again}
  For any pinching function $\eta$, a coherent family $J$ of
  $\eta$-admissible almost complex structures exist.
\end{lemma}
\begin{proof}
  We will construct the $\eta$-admissible almost complex structures by
  induction on the dimension of the hypercube $\mathbb{E}$.  When the
  dimension is zero, and $(0,0)$ is not an energy range for any of the points,
  Definition~\ref{def:acAgain} reduces to 
  Definition~\ref{def:ac}, so this base case follows from
  Lemma~\ref{lem:admis-J-exists}.
  When $(0,0)$ does appear as the energy range of some of the points in $\vec{E}$,
  the moduli space $\ModPol_{\vec{E}}$, is identified with a component of $\ModPol_{\geq 0}$, over which we can construct the complex structure
  by pulling back via the forgetful map $\ModPol_{\geq 0}\to \ModPol$
  that forgets inert points.
  
  For the inductive step, our goal is to construct an almost-complex
  structure over $\ModPol_{E}\times \mathbb{E}$.  The complex
  structure over $\ModPol_{\vec{E}}\times \partial \mathbb{E}$ is
  constructed by the induction hypothesis and the observation that
  $\ModPol_{\vec{E}}\times (\partial \mathbb{E})$ can be expressed as
  a union of $\ModPol_{\widetilde {E}}\times \widetilde{\mathbb{E}}$
  for the faces $\widetilde{\mathbb{E}}\subset \mathbb{E}$,
  (Equation~\eqref{eq:BoundaryOfE}).  Extend the complex structure
  over $(\partial \ModPol_{\vec{E}})\times \mathbb{E}$ by an internal
  induction on the dimension of $\ModPol_{\vec{E}}$, as in the proof
  of Lemma~\ref{lem:admis-J-exists}. This defines the family of almost
  complex structures over
  $\partial(\ModPol_{\vec{E}}\times \mathbb{E})$. Extend the family to
  a neighborhood of the boundary as in the proof of
  Lemma~\ref{lem:admis-J-exists}, via compatible choices of gluings
  (or strip-like ends).  Finally, extend to all of
  $\ModPol_{\vec{E}}\times\mathbb{E}$ using contractibility of the
  space of $\eta$-admissible almost-complex structures.
\end{proof}

\subsection{Moduli spaces}
We are now ready to define the moduli spaces that we will use to
construct $\CFAmg$.

The following is a generalization of Definition~\ref{def:Basic} to $\MAlgg$.
\begin{definition}
  \label{def:BasicX}
  The \emph{basic algebra
  elements} in $\MAlgg$ are elements of the form $e^a X^b\iota_i$
  with
$a\in \ZZ_{\geq 0}$, $b\in\{0,1\}$, and $\min(a,b)=1$;
  or $e^a X^b \rho$,
where $\rho$ is a Reeb chord and
$a\in \ZZ_{\geq 0}$, $b\in\{0,1\}$.
\end{definition}
Note that we do not consider $U^b \iota_i$ a basic algebra element for
$\MAlgg$, although it was for $\MAlg$: the action we will define is
$U$-equivariant and  unital, so the action by sequences
containing $U^b \iota_i$ is already determined (to be either the
identity or $0$).

In Section~\ref{sec:moduli}, we defined an energy function for basic algebra elements
in $\MAlg$. The analogous definition for basic algebra elements in $\MAlgg$ is
\begin{equation}
  \label{eq:NewE}
  E(e^a X^b \rho) =4a+|\rho|
\end{equation}
for $a\geq 0$ and $b\in\{0,1\}$.
This determines a grading on $\MAlgg$, with the understanding that multiplication by $U$ leaves the grading unchanged.
The grading extends to weighted sequences
exactly as in Equation~\eqref{eq:ExtendEnergy}.
(The energy from Section~\ref{sec:moduli}
changes by $4$ with multiplication by $U$; this is consistent
with the present conventions since the element $U\in \MAlg$ corresponds
to the element $Ue$ in $B\subset \MAlgg$.)

In Section~\ref{sec:moduli}, the energies on the module polygons were
specified by $E$ of the sequence.  In the present extension, the
function $E$ provides only a constraint on the energies.  Specifically,
if $(a_1,\dots,a_m)$ is a sequence of basic algebra elements in
$\MAlgg$ and $w$ is a non-negative integer $w$, there is a
corresponding component of $\ModPol^X$, denoted
$\ModPol^X(a_1,\dots,a_m;w)$, which is the closure of the stratum with:
\begin{itemize}
\item $w$ interior marked points of weight $4$,
\item no marked points on $\{0\}\times \RR$,
\item $m$ boundary marked points on $\{1\}\times \RR$, corresponding
  to $a_1,\dots,a_m$, and
\item the energy range associated to the $i\th$ boundary marked point
  is $(E(a_i),E(a_i)+1)$ if $a_i$ is divisible by $X$ and
  $(E(a_i),E(a_i))$ otherwise.
\end{itemize}

Given the sequence $(a_1,\dots,a_m)$, 
there is an associated hypercube $\mathbb{E}$, whose dimension coincides with the number of $a_i$ that
are divisible by $X$.
Given an $\eta$-admissible family $J$ of almost complex
structures, let ${\mathcal J}(a_1,\dots,a_m;w)$ denote the restriction
of $J$ to $\ModPol^X(a_1,\dots,a_m;w)\times \mathbb{E}$; in this way $J$ associates 
 to $\mathbf{e}\in \mathbb{E}$ and $P\in
\ModPol^X(a_1,\dots,a_m;w)$  an almost complex structure $J_{\mathbf e}$ over
$\Sigma\times P$.

Adapting Definition~\ref{def:decorated-source}, our source curves
$\Source$ have boundary punctures, one of which is labeled $+\infty$,
another is labeled $-\infty$, and all others are labeled by 
basic algebra elements, now for $\MAlgg$. Adapting
Definition~\ref{def:respectful}, the sequence of basic algebra
elements $(a_1,\dots,a_m)$ specifies the asymptotics of the curve,
and a smooth map
\begin{equation}
  \label{eq:u-source-targ-again}
  u\co (S,\bdy S) \to (\Sigma\times[0,1]\times\RR,\alphas\times\{1\}\times\RR\cup\betas\times\{0\}\times\RR)
\end{equation}
respects the source if the basic algebra elements (and $\x$ and $\y$)
specify the asymptotics of the source at the punctures. (Note that
these notions are independent of the factors of $e$, $X$, and $U$ that
can appear in each basic algebra element; these factors only influence
the moduli spaces through the complex structures which are
considered.)

Definition~\ref{def:moduli-fixed-source} can be adapted as follows.
Fix an $\eta$-admissible family $J$ of almost complex structures.
Given a pair of generators $\x,\y\in\Gen(\HD)$, homology class
$B\in\pi_2(\x,\y)$, and a decorated source $\Source$, let
$\tcM^B(\x,\y;a_1,\dots,a_\ell;w;\Source)$ be the set of choices
$\mathbf{e}\in \mathbb{E}$ and maps as in
Formula~\ref{eq:u-source-targ-again} which are $(j,J_{\mathbf{e}})$-holomorphic for some complex structure $j$ on the source, with
specified algebra sequence $\vec{a}$ and weight $w$.  Recall that
$J_{\mathbf{e}}$ is the complex structure on the product $\Sigma\times
P$ specified by the point $\mathbf{e}\in \mathbb{E}$ and the
$\eta$-admissible family of almost-complex structures.  It is
sometimes useful to think of $\mathbf{e}$ as determining a continuous
function
\[ \mathbf{e}\co \tcM^B(\x,\y;\vec{a};w;\Source)\to \mathbb{E}.\]

Let $d$ denote the number of algebra elements in the
sequence $(a_1,\dots a_m)$ that are divisible by $X$; i.e., $d=\dim\mathbb{E}$.
Following Proposition~\ref{prop:index-source}, the expected dimension of $\tcM^B(\x,\y;\vec{a};w;\Source)$ is given by
\begin{equation}\label{eq:index-source-again}
  \ind(B,\Source;\vec{a};w,r)\coloneqq  g-\chi(\Source)+2e(B)+m+w+d-r,
\end{equation}
where $\vec{a}=\vec{a}(\Source)$, $m=|\vec{a}|$ denotes the number
of algebra punctures of $S$, $w$ denotes the number of interior punctures of
$S$, $r$ is the ramification index at the Reeb orbits, and $d$
is as above.

We call an $\eta$-admissible family of almost complex structures $J$
\emph{tailored for $\MAlgg$} if it is as in Definition~\ref{def:tailored}, bearing
in mind that now the index takes into account the number of smeared
boundary punctures, and the moduli spaces
$\tcM^B(\x,\y;\vec{a};w;\Source)$ now have source curves marked by
basic algebra elements in the larger algebra $\MAlgg$. (In particular,
this makes use of a sufficient pinching function as in
Definition~\ref{def:sufficient-pinching}.)

Proposition~\ref{prop:transv-exists} has the following adaptation:
\begin{proposition}\label{prop:transv-exists-again}
  For any pinching function $\eta$, the set of coherent families of
  $\eta$-admissible almost complex structures $J$ so that the moduli
  spaces $\tcM^B(\x,\y;\vec{a};w;\Source)$ are all
  transversely cut out is co-meager.
\end{proposition}
\begin{proof}
  This follows from the same standard transversality arguments as
  Proposition~\ref{prop:transv-exists}, and we refer the reader to the
  references there.
\end{proof}

We have the following analogue of Corollary~\ref{cor:TailoredExist}:
\begin{corollary}
  \label{cor:TailoredExist-again}
  Tailored families of almost complex structures for $\MAlgg$ exist.
\end{corollary}

\begin{proof}
  The proof is similar to the proof of
  Corollary~\ref{cor:TailoredExist}. In particular, sphere bubbles
  carry index $2$, hence bubbling one off drops the expected dimension
  of the moduli space by $2$, so the only way these can occur is when
  they are glued to a trivial strip, in which case
  Corollary~\ref{cor:no-spheres} applies (as long as the single
  $\mathbb{R}$-invariant almost complex structure we chose near
  $\pm\infty$ is close to being
  split). Proposition~\ref{prop:transv-exists} (for disk bubbles with
  boundary on the $\beta$-circles) and
  Proposition~\ref{prop:transv-exists-again} (for the main component)
  imply that for generic families, there are no disk bubbles and the
  main components are transversely cut out. The proof that moduli
  spaces of boundary degenerations are transversely cut out for
  generic choices of almost complex structures
  (Condition~\ref{item:tail-bdy}) follows from the same argument as in
  Corollary~\ref{cor:TailoredExist}.
\end{proof}

From now on, we will abuse terminology and refer to families which are
tailored for $\MAlgg$ simply as \emph{tailored}.

We have the following analogue of Theorem~\ref{thm:master}.

\begin{theorem}\label{thm:masterB}
  Fix a provincially admissible bordered Heegaard diagram $\HD$ and a
  tailored family of almost complex structures $J$.  Given a homology
  class $B$ and
  sequence $\vec{a}=(a_1,\dots,a_k)$ of basic algebra elements for
  $\MAlgg$ so that $\ind(B;a_1,\dots,a_k;w)=2$, the sum of the
  following kinds of ends is $0\pmod{2}$: ends as described in
  Theorem~\ref{thm:master}
  (Types~\ref{end:TwoStoryBuilding}-\ref{end:CompositeBoundaryDegeneration}),
  and new kinds of ends corresponding to the differential of $X$;
  i.e., ends of the form
  \[
    \sum_{a_i= X\cdot b_i}\#\cM^B(\x,\y;a_1,\dots,a_{i-1},b_i,a_{i+1},\dots,a_n;w)
  \]
  and
  \[
    \sum_{a_i=X\cdot b_i}\#\cM^B(\x,\y;a_1,\dots,a_{i-1},e\cdot b_i,a_{i+1},\dots,a_n;w),
  \]
  provided that $a_i= b_i\cdot X$. Here, in the special case that
  $a_i=X\iota_j$, we define
  \[
    \#\cM^B(\x,\y;a_1,\dots,a_{i-1},\iota_j,a_{i+1},\dots,a_n;w)=
    \begin{cases}
      0 & n>1\text{ or }w>0\\
      1 & (n,w)=(1,0).
    \end{cases}
  \]
\end{theorem}

\begin{proof}
  This follows mostly as in Theorem~\ref{thm:master}, with two
  additional remarks, one regarding ends of Types~\ref{end:Collision},
  and another regarding the new types of ends, when the map
  \[
    \mathbf{e}\co
    \cM^B(\x,\y;a_1,\dots,a_{i-1},a_i,a_{i+1},\dots,a_n;w)\to
    \mathbb{E}
  \]
  hits $\partial \mathbb{E}\subset \mathbb{E}$.

  For collision levels (Type~\ref{end:Collision}), note that if both
  $a_i$ and $a_{i+1}$ are divisible by $X$, then the product
  $a_i a_{i+1}$ vanishes, so the statement of the theorem implies that
  such collision levels do not contribute to the count. This is true
  because before the collision, each of $a_i$ and $a_{i+1}$
  contributes a dimension to $\mathbb{E}$, while after the collision
  the new vertex on the main component contributes only one dimension
  to $\mathbb{E}$. The dimension of the space $\ModPol$ also drops by
  one with the collision, so these kinds of collisions, in all, have
  codimension $2$.

  The two new kinds of ends occur when the value of $\mathbf{e}$
  approaches a point in $\partial \mathbb{E}\subset\mathbb{E}$;
  i.e., when the $i^{th}$ component of ${\mathbf e}$ corresponding to
  some algebra element $a_i=X\cdot b_i$ approaches $E(b_i)$ or
  $E(b_i)+1$. Finally, there are the special cases when
  $b_i=\iota_j$, i.e., $E(b_i)=0$. Assuming $w>0$, $n>1$, or
  $B\neq 0$, the corresponding end of the moduli space is never rigid:
  there is no constraint at the corner corresponding to $b_i$. If
  $w=0$, $n=1$, and $B=0$ the moduli space consists of a single
  trivial strip.
\end{proof}

\subsection{Grounding \texorpdfstring{$\CFAm$}{CFA-minus}}\label{sec:grounding-CFA}

We now define $\CFAmg$ over $\MAlgg$.  The operations on $\CFAmg(\HD)$
are defined to be  unital and $U$-equivariant; so, to specify these operations, it remains to
define $m_{n+1}^w(\x,a_1,\dots,a_n)$ for each $n$, $w$, and sequence
$(a_1,\dots,a_n)$ of basic algebra elements (Definition~\ref{def:BasicX}). Define
\begin{equation}\label{eq:CFAg-m-def}
  m_{n+1}^w(\x,a_1,\dots,a_n)=\sum_{\y\in\Gen(\HD)}\sum_{\substack{B\in\pi_2(\x,\y)\\\ind(B;a_1,\dots,a_n;w)=1}}\#\cM^B(\x,\y;a_1,\dots,a_n;w)
  U^{n_z(B)} \y.
\end{equation}
The sums make sense, according to the following extension of Lemma~\ref{lem:admis}

\begin{lemma}\label{lem:admisG}
  Fix generators $\x,\y\in\Gen(\HD)$, basic algebra elements $a_1,\dots,a_n$,
  and $w\in\ZZ_{\geq 0}$.  If the bordered Heegaard diagram $\HD$ is
  provincially admissible then there are finitely many homology
  classes $B\in\pi_2(\x,\y)$ so that $\ind(B;a_1,\dots,a_n;w)=1$
  and $\cM^B(\x,\y;a_1,\dots,a_n;w)$ is non-empty.
\end{lemma}

\begin{proof}
  This follows from the same proof as Lemma~\ref{lem:admis}.
\end{proof}

\begin{theorem}
  \label{thm:CFAmg-is}
  For any provincially admissible bordered Heegaard diagram and
  tailored family of almost complex structures, the operations
  $m_{1+k}^w$ satisfy the weighted $\Ainf$ relations, thus making $\CFAmg$ 
  into a weighted $\Ainf$-module over $\MAlgg$. Moreover, the restriction 
  of $\CFAmg$ to $B$ coincides with the restricted module  from 
  Definition~\ref{def:CFAmgB}.
\end{theorem}
\begin{proof}
  This is an immediate consequence
  of Theorem~\ref{thm:masterB}; compare the proof of
  Theorem~\ref{thm:CFAmb-is}.
\end{proof}

Theorem~\ref{thm:CFAmg-is} completes Step~\ref{s:ExtendToGroundingAlgebra} in
the construction of $\CFAm$.  To establish
Step~\ref{s:RestrictFromGroundingAlgebra}, 
we will use homological perturbation theory to induce a weighted
$\Ainf$-algebra structure on $\AsUnDefAlg[U]\subset\MAlgg$, and verify
that in fact this induced structure is exactly $\MAlg$. We then use
restriction of scalars to turn $\CFAmg(\HD)$ into a weighted
$\Ainf$-module over $\MAlg$; the result is $\CFAm(\HD)$. We start by
formulating the Homological Perturbation Lemma in this context:
\begin{lemma}
  \label{lem:HomologicalPertAinfAlg}
  Let $A$ be a  unital weighted $\Ainf$-algebra and $B$ be a
  chain complex, both over a ground ring $\Ring$.  Consider
  $\Ring$-module maps
  \[ 
\begin{tikzpicture}
  \node at (0,0) (a) {$A$};
  \node at (2,0) (b) {$B$};
  \draw[->] (a)[bend left=15] to node[above]{\lab{\Pi}} (b);
  \draw[->] (b)[bend left=15] to node[below]{\lab{i}} (a);
  \draw[->] (a)  [loop left] to node[left]{\lab{h}} (a);
  \end{tikzpicture},
\]
and suppose that the following properties hold:
\begin{itemize}
\item Both the unit $\One\in A$ and the elements $\mu_0^w\in A$, for
  all $w>0$, are in the image of the map $i$.
\item The map $h$ satisfies $h^2=0$, $\Pi\circ h=0$, and $h\circ i=0$.
\item The maps $\Pi$ and $i$ are chain maps.
\item The maps satisfy $\Pi\circ i=\Id_B$ and
  $i\circ \Pi=\Id_A+\partial h + h \partial$.
\end{itemize}
(Using the terminology  from Section~\ref{def:CFDm},
 this data forms a \emph{strong deformation
   retraction} from $A$ to $B$.)
 Then, $B$ can be given the structure of a  unital weighted $\Ainf$-algebra
 with $\mu_1^0$ the given differential on $B$ and the map $i$ can be
 extended to a homomorphism of weighted $\Ainf$-algebras
 $\{i^w_n\co B^{\otimes n}\to A \}_{w\geq 0,n\geq 1}$ so that
 $i^0_1=i$.

 Further, if $h(\mu_0^w)=0$ for all $w$ then the curvature terms
 $\mu_0^w$ on $B$ are equal to the projections $\Pi(\mu_0^w)$ of the
 curvature terms on $A$.
\end{lemma}

\begin{proof}
  The unweighted case of this lemma is well known (see for
  example~\cite[Theorem 1]{Kadeisvili80:hpt},~\cite[Section 6.4]{KontsevichSoibelman01:HMS},
  or~\cite[Theorem~2.3]{KellerPertTheory}).
  We outline the modifications needed for the weighted case.

  Operations with weight $w\geq 0$ are specified by trees (which we refer
  to as ``operation trees'') whose input
  edges are labeled by the inclusion map $i$, output edge is
  labeled by the projection map $\Pi$, edges between internal vertices are labeled
  by the homotopy operator $h$, and vertices are labeled by integer
  weights, so that the sum of these weights is $w$. The portion of the
  map $i^w_n$ with $(w,n)\neq (0,1)$ is specified by trees labeled
  similarly, except that now the output edge is labeled by $h$ rather
  than $\Pi$.

  It is straightforward to verify that the resulting operations on $B$
  satisfy the weighted $\Ainf$ relation. Strict unitality of $B$
  (using the element whose image under $i$ is the unit in $A$) uses strict unitality of $A$ and the hypothesis that $h^2=0$.
  Specifically, strict unitality of $A$ implies that, on $B$, the $\mu^w_n$ operations with $(w,n)\neq
  (0,2)$ and where an input is $1$ vanish except in the special
  cases where the input $1$ is channeled into a
  $\mu^0_2$ operation.  For
  these special trees, the contribution vanishes by the hypothesis
  that $h^2=\Pi\circ h=h\circ i=0$.

  Finally, in the case that $h(\mu^w_0)=0$, if an operation tree has
  some node with no inputs, then that tree consists of simply that
  node and an output edge labeled by $\Pi$.
\end{proof}

The algebra $\MAlgg$ inherits a grading from $\MAlg$
(Sections~\ref{sec:build-alg-combinatorially} and~\ref{sec:gradings})
by declaring that $X$ has grading $\lambda^1$ and $e$ has grading
$\lambda^0$.

In the next lemma, we will consider the inclusion of vector spaces
$i\co \AsUnDefAlg[U]\to \MAlgg$ which surjects onto the subspace of
$\MAlgg$ spanned by elements of the form $U^n\iota_j$ and $U^n\rho$
(with no factors of $e$ or $X$).

\begin{lemma}
  \label{lem:InducedOnMAlg}
  The inclusion map $i\co \AsUnDefAlg[U]\subset \MAlgg$ of chain
  complexes can be completed to a diagram as in
  Lemma~\ref{lem:HomologicalPertAinfAlg} (with $B=\AsUnDefAlg[U]$ and $A=\MAlgg$). The induced weighted
  $\Ainf$-algebra structure on $\AsUnDefAlg[U]$ is $\MAlg$.
\end{lemma}
\begin{proof}
  Let $\Pi\co \MAlgg\to \AsUnDefAlg[U]$ be the projection to
  $\AsUnDefAlg[U]$ induced by the properties that $\Pi(e^r)=1$,
  $\Pi(X e^r)=0$, for all $r\geq 0$, and $\Pi$ respects the
  multiplications $\mu_2^0$.
  
  Consider the $\AsUnDefAlg[U]$-equivariant map $h\co \MAlgg\to
  \MAlgg$ characterized by the property that $h(X e^r)=0$ and
  $h(e^n)=X(1+\dots+e^{n-1})$. It is easy to see that $\Pi\circ
  i=\Id$, $h^2=0$, $\Pi\circ h=0$, and
  \[
    \Id+\partial\circ h + h\circ \partial=i\circ \Pi.
  \]
  It follows that $i$ is a chain homotopy equivalence, so
  Lemma~\ref{lem:HomologicalPertAinfAlg} gives $\AsUnDefAlg[U]$ the
  structure of a weighted $\Ainf$-algebra. Note that in this
  application of the lemma, the operations are rather simple: the
  image of $h$ is divisible by $X$, all the operations on $\MAlgg$ are
  $X$-equivarant, $X^2=0$, and $\Pi$ vanishes on multiples of $X$.  It
  follows that the morphism constructed by the lemma is given simply by
  \[ i^w_n(a_1,\dots,a_n)=h\circ \mu^w_n(i(a_1),\dots,i(a_n)),\]
  and the induced operations $\{\underline{\mu}^w_n\}$ on $B$ are given by
  \[ \mathbf{\mu}^w_n(a_1,\dots,a_n)=\Pi\circ \mu^w_n(a_1,\dots,a_n).\]
  So, this induced structure on $B$ coincides with $\MAlg$.
\end{proof}

\begin{example}\label{eg:induced-mu4}
  For example, observe that $\mu^0_4(\rho_4,\rho_3,\rho_2,\rho_1)=Ue\iota_1$ in $\MAlgg$.
  Since $\Pi(e)=1$, we conclude that $\mu^0_4(\rho_4,\rho_3,\rho_2,\rho_1)=U\iota_1$
  for the weighted algebra structure on $\AsUnDefAlg[U]$ induced by the inclusion $i$.
  Note also that $i^0_4(\rho_4,\rho_3,\rho_2,\rho_1)=UX\iota_1$.
\end{example}

Next, we would like to use restriction of scalars to pull
$\CFAmg(\HD)$ back to a module $\CFAm(\HD)$ over $\MAlg$. To do this,
we first need to define a restriction of scalars functor
in this weighted context.

Given a weighted $\Ainf$-algebra ${\mathcal A}$ over $\Ground$, we consider its bar complex
$\Tensor^*(A)=\bigoplus_{n\geq 0} A^{\otimes n}$ (where tensor
products are taken over $\Ground$), equipped with maps
$\{\partial^w\co \Tensor^*(A)\to \Tensor^*(A)\}$, $w\geq 0$, defined by
\[
  \partial^0(a_1\otimes\dots \otimes a_n)=\sum_{s=1}^n\sum_{r=1}^{n+1-s}
 a_1\otimes\dots\otimes a_{r-1} \otimes \mu^0_s(a_{r}\otimes\dots\otimes a_{r+s-1})\otimes a_{r+s}\otimes\dots\otimes a_n;
\]
or, when $w>0$,
\[
  \partial^w(a_1\otimes\dots \otimes a_n)=
\sum_{s=0}^n \sum_{r=1}^{n+1-s} a_1\otimes\dots\otimes a_{r-1} \otimes \mu^w_s(a_{r}\otimes\dots\otimes a_{r+s-1})\otimes a_{r+s}\otimes\dots\otimes a_n.
\]
(Note that for $w>0$, the terms in $\partial^w$ include the operations $\mu^w_0$.)
The structure equations for a weighted algebra can be summarized by requiring
that $\partial=\sum_{w\geq 0}t^w\partial^w$ is a differential.

A homomorphism $f\co \mathcal{A}\to \mathcal{B}$ of weighted $\Ainf$
algebras is a set of maps
$\{f^w_n\co A^{\otimes n}\to B\}_{\{n,w\in \ZZ^{\geq 0}\mid (n,w)\neq
  (0,0)\}}$, satisfying a suitable structure
relation~\cite[Definition~\ref{Abs:def:wAlg-homo}]{LOT:abstract}.  We
say that $f$ is {\em  unital} if $f^w_n(a_1,\dots,a_n)=0$ if
$a_i=\One_A$, unless $(n,w)=(1,0)$ in which case
$f^0_1(\One_A)=\One_B$; c.f. Section~\ref{subsec:Unitality}.

Given a homomorphism $f\co \Alg \to \Blg$ of weighted $\Ainf$-algebras over
$\Ground$, with components
$\{f^w_n\co A^{\otimes n}\to B\}_{n,w\in\ZZ\mid (n,w)\neq (0,0)}$
there are induced maps
$\{F^w\co \Tensor^*(A)\to \Tensor^*(B) \}_{w\geq 0}$ specified by
\begin{multline*}
  F^w(a_1\otimes\dots\otimes a_n)\\
  =\sum_{\substack
    {n_1+\dots+n_k=n \\ w_1+\dots+w_n=w \\(n_i,w_i)\neq (0,0)}}
  f^{w_1}_{n_1}(a_1,\dots,a_{n_1})\otimes
  f_{n_2}^{w_2}(a_{n_1+1},\dots,a_{n_1+n_2})\otimes \dots\otimes
  f_{n_{k}}^{w_k}(a_{n-n_k+1},\dots,a_n).
\end{multline*}

Fix  a weighted $\Ainf$-module 
$\fModule=(M,\{m^w_{n}\co M\otimes B^{\otimes n-1}\to M\}_{n\geq 1,w\geq 0})$,
and abbreviate $m^w =\sum m^w_n$, thought of as a map
$m^w\co M\otimes \Tensor^*(B)\to M$. 
We can endow $M$ with operations 
$\{{\underline m}^w_{n}\co M\otimes A^{\otimes (n-1)}\to
M\}_{\text{\tiny$\begin{subarray}{l}\thinspace \thinspace w,n\geq 0 \\ (w,n)\neq (0,0)\end{subarray}$}}$,
specified by 
\begin{equation}
  \label{eq:PullbackOperations}
  \underline{m}^w_n(x\otimes a_1\otimes\dots\otimes a_{n-1})
  = \sum_{w_1+w_2=w} m^{w_1}(x,F^{w_2}(a_1\otimes\dots\otimes a_{n-1})).
\end{equation}

\begin{lemma}
  \label{lem:RestrictionFunctor}
  Fix weighted $A_\infty$ algebras $\Alg$ and $\Blg$, a weighted
  $A_\infty$ module $\fModule$ over $\Blg$, a morphism of weighted $A_\infty$ algebras
  $f\co \Alg\to \Blg$. The operations
  $\{{\underline m}^w_{n}\co M\otimes A^{\otimes (n-1)}\to M\}_{w\geq
    0, n\geq 1}$
  in Equation~\eqref{eq:PullbackOperations}
  endow $\fModule$ with the structure of a weighted
  $\Ainf$-module over $\Alg$, which we denote $f^*(\fModule)$.  Moreover, if
  $\phi\co \fModule\to \fNodule$ is a weighted $\Ainf$-homomorphism with
  components
  $\{\phi^w\co M\otimes \Tensor^*(B)\to N\}_{w\geq 0}$,
  then there is an induced weighted $\Ainf$-homomorphism
  $f^*\phi\co f^*(\fModule)\to f^*(\fNodule)$ with components
  $(f^*\phi)^w=\sum_{w_1+w_2=w} \phi^{w_1}\circ (\Id_M\otimes
  F^{w_2})$.   If $\phi,\psi\co \fModule\to \fNodule$ are chain homotopic
  weighted $A_\infty$-module homomorphisms, then so are $(f^*\phi)$
  and $(f^*\psi)$.

  The above statements also hold in the unital case:
  when $f$ and $\fModule$ are  unital, so is $f^*(\fModule)$;
  if $\phi\co \fModule\to \fNodule$ is also  unital, so is $f^*(\phi)$;
  and if $\phi$ and $\psi$ are chain homotopic as  unital morphisms,
  then so are $f^*(\phi)$ and $f^*(\psi)$.
\end{lemma}

\begin{proof}
  The fact that the operations $\underline{m}_n^w$ make $f^*(\fModule)$ into
  a weighted $\Ainf$-module is straightforward from the
  definitions. For the statements about functoriality, recall that the
  chain complex of weighted $\Ainf$-module morphisms from $\fModule$ to
  $\fNodule$, denoted $\Mor_{\Blg}(\fModule,\fNodule)$, is given by
  \[
    \Mor_{\Blg}(\fModule,\fNodule)=\prod_{w\geq 0} \Hom(M\otimes \Tensor^*(B),N);
  \]
  i.e., a morphism $\phi$ has components $\{\phi^w\co M\otimes
  \Tensor^*(B) \to N\}_{w\geq 0}$~\cite[Section~4.5]{LOT:abstract}. The differential of a
  morphism is the morphism with components $\{(d
  \phi)^w\co M\otimes\Tensor^*(B)\to N\}_{w\geq 0}$ given
  by
  \[
    (d \phi)^w=\sum_{w_1+w_2=w} m^{w_1}_{N}\circ (\phi^{w_2}\otimes\Id_{\Tensor^*(A)})\circ (\Id_M\otimes \Delta)+ 
    \phi^{w_1}\circ (m^{w_2}_M\otimes \Id_{\Tensor^*(A)})\circ (\Id_M\otimes \Delta) 
    + 
    \phi^{w_1}\circ (\Id_M\otimes \partial^{w_2}),
  \]
  where $\Delta\co \Tensor^*(A)\to \Tensor^*(A)\otimes\Tensor^*(A)$ is
  the sum of all ways of breaking the tensor product into two pieces (one of
  which is allowed to be empty).
  A weighted $\Ainf$-homomorphism is a cycle in the
  morphism space and two weighted $\Ainf$-homomorphisms are homotopic
  if and only if the corresponding cycles are homologous.
  So, the last two statements in the lemma follow from the easily
  established fact that the map
  \[
    \Mor_\Blg(\fModule,\fNodule)\to \Mor_A(f^*(M),f^*(N)),
  \]
  defined by $\phi\mapsto f^*\phi$ (as in the statement of the lemma)
  is a chain map.

  The statements about strict unitality are immediate from the construction.
\end{proof}

\begin{definition}
  \label{def:CFAm}
  Let $\CFAm(\HD)$ be the weighted $\Ainf$-module over $\MAlg$
  defined by $i^*(\CFAmg(\HD))$.
\end{definition}

\begin{theorem}\label{thm:CFAm-is}
  Assuming that the bordered Heegaard diagram $\HD$ is provincially
  admissible and the family of almost complex structures $J$ is tailored
  the operations $m_{1+k}^w$ give $\CFAm(\HD)$ the structure
  of a  unital, weighted $\Ainf$-module over $\MAlg$.
\end{theorem}

\begin{proof}
  This follows immediately from Theorem~\ref{thm:CFAmg-is}
  and Lemma~\ref{lem:RestrictionFunctor}.
\end{proof}

\begin{example}
  For example, for the induced module structure on $\CFAm(\HD)$, the
  operation ${\underline m}_5^0(\x,\rho_4,\rho_3,\rho_2,\rho_1)$ is
  computed as a sum of two terms:
  $m_5^0(\x,\rho_4,\rho_3,\rho_2,\rho_1)+U m_2^0(\x,X\iota_1)$
  (cf.~Example~\ref{eg:induced-mu4}).  
\end{example}

The invariance theorem for $\CFAmb$ can be extended to $\CFAmg$, as follows:

\begin{theorem}\label{thm:CFAmg-invt}
  Up homotopy equivalence (of weighted $\Ainf$-modules),
  $\CFAmg(\HD,\spinc)$ depends only on the bordered 3-manifold
  $(Y,\phi)$ specified by $\HD$, and its underlying $\SpinC$ structure $\spinc$.
\end{theorem}
\begin{proof}
  The proof is similar to the proof of invariance for $\CFAmb$ given
  in Section~\ref{sec:CFA}
  (Propositions~\ref{prop:J-inv},~\ref{prop:iso-inv},~\ref{prop:handleslide-inv},
  and~\ref{prop:stab-inv}) and is left to the reader.
\end{proof}

The results above immediately imply topological invariance of $\CFAm$:

\begin{proof}[Proof of Theorem~\ref{thm:CFAm-invt}]
  This follows from Theorem~\ref{thm:CFAmg-invt} and Lemma~\ref{lem:RestrictionFunctor}.
\end{proof}

Finally, we note that $\CFAm$ and $\CFAmb$ are homotopy equivalent, non-$U$-equivariantly:
\begin{proposition}\label{prop:CFA-is-CFA}
  For any provincially admissible Heegaard diagram $\HD$, there is a
  homotopy equivalence of weighted $\Ainf$-modules
  $\CFAm(\HD)\simeq\CFAmb(\HD)$ over $\MAlg$, thought of as a weighted
  algebra over $\Field$ (rather than $\Field[U]$).
\end{proposition}
\begin{proof}
  There are two inclusions of $\MAlg$ into $\MAlgg$: the inclusion
  $\phi$ as the sub-algebra $B$ from Section~\ref{sec:grounding-alg}
  and the inclusion as $\AsUnDefAlg[U]$ from
  Section~\ref{sec:grounding-CFA}. As in the proof of
  Lemma~\ref{lem:InducedOnMAlg}, denote this second inclusion map by $i$.
  There is also a projection map $\Pi\co \MAlgg\to\MAlg$ from the
  proof of Lemma~\ref{lem:InducedOnMAlg}; 
  note that while there we
  only use the fact that $\Pi$ is a chain map, it is true that the
  sequence of maps $\{\Pi^w_n\}$ specified by
  \[
    \Pi^w_k=
    \begin{cases}
      \Pi & \text{if $(w,k)=(0,1)$} \\
      0 & \text{otherwise}
    \end{cases}
  \]
  is a homomorphism of weighted algebras. The maps $i$ and $\Pi$ are
  $U$-equivariant while $\phi$ is not. Also, $\Pi\circ i$ and
  $\Pi\circ\phi$ are both the identity map of $\MAlg$.

  By construction, for corresponding choices of tailored families of
  almost complex structures, $\CFAmb(\HD)\cong \phi^*\CFAmg(\HD)$, and
  by definition, $\CFAm(\HD)=i^*\CFAmg(\HD)$. We will show that for
  any weighted module $M$ over $\MAlgg$,
  $M\simeq (i\circ\Pi)^*M$. Assuming this, we have
  \[
    \CFAmb(\HD)\cong \phi^*\CFAmg(\HD)\simeq \phi^*(i\circ\Pi)^*\CFAmg(\HD)
    \cong (\Pi\circ\phi)^*i^*\CFAmg(\HD)=i^*\CFAmg(\HD),
  \]
  since $\Pi\circ\phi=\Id$, completing the proof.

  To verify that $M\simeq (i\circ\Pi)^*M$, observe that since
  $i_n^w=0$ and $\Pi_n^w=0$ for $(n,w)\neq(1,0)$, the operations
  $\underline{m}_n^w$ on $(i\circ\Pi)^*M$ are given by
  \[
    \underline{m}_{1+n}^w(x,a_1,\dots,a_n)=m_{1+n}^w(x,i(\Pi(a_1)),\dots,i(\Pi(a_n))).
  \]
  Also, since $i$ is the obvious inclusion, we will drop it from the
  notation.

  Recall the homotopy $h$ from the proof of
  Lemma~\ref{lem:InducedOnMAlg}. Define a map
  $f\co M\to (i\circ\Pi)^*M$ by $f_1^0(x)=x$, $f_1^w=0$ for $w>0$, and
  \[
    f_{1+n}^w(x,a_1,\dots,a_n)=\sum_{i=0}^n m_{1+n}^w(x,a_1,\dots,a_{i-1},h(a_i),\Pi(a_{i+1}),\dots,\Pi(a_n))
  \]
  for $n\geq 1$. Using the fact that the operations on $\MAlgg$ are
  equivariant with respect to multiplication by $e$ and $X$, it is
  easy to verify that
  \begin{align*}
    h(\mu_n^w(a_1,\dots,a_n))&=\sum_{i=1}^n\mu_n^w(a_1,\dots,a_{i-1},h(a_i),\Pi(a_{i+1}),\dots,\Pi(a_{n}))\\
    &=\sum_{i=1}^n\mu_n^w(\Pi(a_1),\dots,\Pi(a_{i-1}),h(a_i),a_{i+1},\dots,a_{n})
  \end{align*}
  from the first of which it in turn follows easily that $f$ is a
  weighted module homomorphism.

  Define $g\co (i\circ \Pi)^*M\to M$ by $g_1^0(x)=x$, $g_1^w=0$ for
  $w>0$, and 
  \[
    g_{1+n}^w(x,a_1,\dots,a_n)=\sum_{i=0}^n m_{1+n}^w(x,\Pi(a_1),\dots,\Pi(a_{i-1}),h(a_i),a_{i+1},\dots,a_n).
  \]
  The same argument as above shows that $g$ is a weighted module
  homomorphism. Finally, there is a homotopy $k$ from $g\circ f$ to
  the identity map given by
  \begin{multline*}
    k_{1+n}^w(x,a_1,\dots,a_n)\\
    =\sum_{1\leq i\leq j\leq n}m_{1+n}^w(x,a_1,\dots,a_{i-1},h(a_i),\Pi(a_{i+1}),\dots,\Pi(a_{j-1}),h(a_j),a_{j+1},\dots,a_n)
  \end{multline*}
  and a homotopy $\ell$ from $f\circ g$ to the identity map given by
  \begin{multline*}
    \ell_{1+n}^w(x,a_1,\dots,a_n)\\
    =\sum_{1\leq i\leq j\leq n}m_{1+n}^w(x,\Pi(a_1),\dots,\Pi(a_{i-1}),h(a_i),a_{i+1},\dots,a_{j-1},h(a_j),\Pi(a_{j+1}),\dots,\Pi(a_n)).
  \end{multline*}
  Again, verifying that these are homotopies is straightforward from
  the relation above. This concludes the proof.
\end{proof}

\begin{remark}
  \label{rem:UtoTheN}
  A key complication in our construction of $\CFAm$ is the fact that
  the amount of pinching required to achieve regularity depends on the
  energy of the algebra sequence; this traces back to the fact that,
  in the statement of Lemma~\ref{lem:bdy-deg-ind-2}, the choice of
  $\epsilon$ comes after the sequence of algebra elements. So, to
  define $\CFAmg$, we are forced to use families of almost complex
  structures parameterized by the moduli space of polygons
  (Definition~\ref{def:acAgain}). If we are only interested in the
  quotient module $\CFAm(Y,\spinc)/U^n$ for some fixed $n$, we could
  bypass $\MAlgg$ entirely and work with a single
  $\mathbb{R}$-invariant almost complex structure on
  $\Sigma\times[0,1]\times\RR$ which is $\epsilon$-pinched for some
  $\epsilon$ sufficiently small that Lemma~\ref{lem:bdy-deg-ind-2}
  holds for all moduli spaces crossing $z$ at most $n$ times.
\end{remark}


\section{The module \texorpdfstring{$\CFDm$}{CFD}}\label{sec:CFD}
As mentioned in the introduction, we define $\CFDm$ in terms of
$\CFAm$ and the dualizing bimodule $\CFDDm(\Id)$. The basic properties
of $\CFDm$
then follow from corresponding properties of $\CFAm$. We spell out
this quick definition in Section~\ref{def:CFDm}, where we also discuss
how $\CFDm$ extends $\CFDa$. Before that, in
Section~\ref{sec:DD-Id} we verify that $\CFDDm$ is, in fact, a type
\DD\ bimodule, and in Section~\ref{sec:boundedness} we discuss some
boundedness conditions needed for the homological algebra.

\subsection{The Koszul-dualizing module \texorpdfstring{$\CFDDm(\Id)$}{CFDD(Id)}}\label{sec:DD-Id}
We begin by verifying that the type \DD\ $(\MAlg^{U=1},\MAlg)$-bimodule
$\CFDDm(\Id)$ from Section~\ref{sec:intro-DD-Id} actually is a type \DD\
bimodule, for any choice of weighted algebra diagonal.
The proof makes use of the grading on $\CFDDm(\Id)$, and it turns out
to be simpler to prove the result for a variant $\uId$ of
$\CFDDm(\Id)$ and then deduce it for $\CFDDm(\Id)$. Specifically, the
type \DD\ bimodule $\uId$ is a bimodule over $(\MAlg,\MAlg)$ (rather
than $(\MAlg^{U=1},\MAlg)$). Denote the variable $U$ for the first
copy of $\MAlg$ by $U_1$ and the variable for the second by $U_2$. The
variables $\Yvar_1$ and $\Yvar_2$ from the definition of a weighted algebra
diagonal (Definition~\ref{def:algebra-diagonal}) act by $U_1$ and
$U_2$, respectively. The bimodule $\uId$ has the same generators
$\iota_0\otimes\iota_0$ and $\iota_1\otimes\iota_1$ as $\CFDDm(\Id)$
and differential given by the same formula
\begin{equation}
\begin{split}
  \delta^1(\iota_0\otimes\iota_0)&=(\rho_1\otimes\rho_3+\rho_3\otimes\rho_1+\rho_{123}\otimes\rho_{123}+\rho_{341}\otimes\rho_{341})\otimes(\iota_1\otimes\iota_1)\\
  \delta^1(\iota_1\otimes\iota_1)&=(\rho_2\otimes\rho_2+\rho_4\otimes\rho_4+\rho_{234}\otimes\rho_{412}+\rho_{412}\otimes\rho_{234})\otimes(\iota_0\otimes\iota_0).
\end{split}
\end{equation}
So, $\CFDDm(\Id)$ is obtained from $\uId$ by setting $U_1=1$ and
$U_2=U$. In particular, the structure equation for $\uId$ implies
the structure equation for $\CFDDm(\Id)$.

Before giving the proof of the structure equation, we discuss the
gradings on $\uId$ and $\CFDDm(\Id)$.

Let $\bigGroup$ be the big grading group for $\MAlg$, from
Section~\ref{sec:big-gr-group}
(or~\cite[Section~\ref{TA:sec:grading}]{LOT:torus-alg}).  Let $r'$ be
the anti-automorphism of $\bigGroup$ defined by
\[
  r'(m; a,b,c,d) = (m;-c,-b,-a,-d).
\]
Since
$\gr'(U)=(-1;1,1,1,1)$, the algebra $\MAlg^{U=1}$ is graded by
$\bigGroup/(-1;1,1,1,1)$.

The bimodule $\uId$ is graded by $\bigGroup$ as a $\bigGroup
\times_\ZZ \bigGroup$-set, with action
\[
  (g,h) \cdot \langle x \rangle = \langle h  x r'(g) \rangle.
\]
(Here, we are using angle brackets to distinguish elements of
$\bigGroup$ viewed as a $\bigGroup
\times_\ZZ \bigGroup$-set from elements of $\bigGroup$ viewed as a group.)
The generators $\iota_i \otimes \iota_i$ both have grading
$\langle 0;0,0,0,0\rangle$. In particular,
$\grb(U_1 \otimes \iota) = \langle -1;-1,-1,-1,-1\rangle$ while $\grb(\iota
\otimes U_2) = \grb(U) = \langle -1;1,1,1,1\rangle$.
Recall from Section~\ref{sec:gradings-abstract} that each weighted
algebra has a weight grading
$\lambda_w$. For the first copy of $\MAlg$, the weight grading is
$\lambda_{w,1} = (1;-1,-1,-1,-1)$, while for the second it is
$\lambda_{w,2}=(1;1,1,1,1)$.

\begin{lemma}
  \label{lem:IdIsDD}
  The bimodules $\uId$ and $\CFDDm(\Id)$ satisfy the type \DD\ structure
  equation.
\end{lemma}

In particular, the proof will show that the type \DD\ structure
equation involves only finitely many terms. (A different proof of this
fact is in Section~\ref{sec:boundedness}.)

\begin{proof}
  Abbreviate an element $(\xi\otimes \eta)\otimes(\iota_i\otimes\iota_i)$
  of (the modulification of) $\uId$ as just $\xi\otimes \eta$,
  so $\grb(\xi \otimes \eta) = \langle\grb(\eta)
  r'(\grb(\xi))\rangle$.

  Every term in $\delta^1(\iota_i\otimes\iota_i)$ has grading
  $\langle\lambda^{-1}\rangle = \langle -1;0,0,0,0\rangle$ while every term in
  the type \DD\ structure relation has grading $\langle\lambda^{-2}\rangle$.
  To see what the possible elements $\xi \otimes \eta$ of grading
  $\langle \lambda^{-2}\rangle$ are, write $\xi =
  U_1^{j_1} \rho_a$ and $\eta = U_2^{j_2} \rho_b$, and write $\rho_a =
  \rho'_a \rho''_a$ and $\rho_b = \rho'_b \rho''_b$, where $\rho''_a$
  and $\rho''_b$ have length $4k_1$ and $4k_2$, respectively, and
  $\rho'_a$ and $\rho'_b$ have length $\le 3$.

  The component of $\grb(\eta) r'(\grb(\xi))$ in $H_1(Z,\CircPts)$ must
  be $0$, which means that
  \begin{equation}\label{eq:CFDD-Id-spinc-comps}
    (j_1 + k_1) [Z] + [\rho'_a] = (j_2 + k_2)[Z] + [\rho'_b],
  \end{equation}
  where $[\rho]\in H_1(Z,\CircPts)=\ZZ^4$ is the $\SpinC$ component of
  the grading (the last four entries) and $[Z]=(1,1,1,1)$ is the
  homology class represented by the circle $Z$ itself. 
  In particular, $\rho'_a$ and $\rho'_b$ are either both trivial or
  both nontrivial and, moreover, $\grb(\rho'_a)$ and $\grb(\rho'_b)$
  commute. So,
  the homological (first) component of $\grb(\eta) r'(\grb(\xi))$ is
  \[
    m = -j_1 - k_1 - j_2 - k_2 + m'_a + m'_b
  \]
  where $m'_a$ and $m'_b$ are the homological components of $\rho'_a$
  and $\rho'_b$ (which are $0$ for trivial chords and $-1/2$
  otherwise). Thus, $m$ is even if and only if $\rho'_a$ and
  $\rho'_b$ are both trivial. So, for $m$ to be $-2$,
  $j_1+k_1+j_2+k_2=2$ and thus, from
  Equation~\eqref{eq:CFDD-Id-spinc-comps},
  $j_1+k_1=j_2+k_2=-1$. Hence, we have the following possible
  terms, and the result of cyclically permuting the indices in them:
  \[
    \rho_{1234} \otimes U_2, \quad
    U_1 \otimes \rho_{1234}, \quad
    \rho_{1234} \otimes \rho_{1234}, \quad
    \rho_{1234} \otimes \rho_{3412}, \quad
    U_1 \otimes U_2.
  \]
  (The cyclic permutation rotates the indices in
  opposite directions on the left and right.)
  
  The only ways these terms arise are:
  \begin{itemize}
  \item $\rho_{1234}\otimes U_2$. Since any operation $\mu_n^w$ except
    $\mu_2^0$ and $\mu_0^1$ lands in $U\MAlg$, the output
    $\rho_{1234}$ must come from $\mu_2^0$-operations or a single
    $\mu_0^1$-operation. So, the only contributions of this form are
    obtained by applying a binary tree of $\mu_2^0$s on the left and a
    single $\mu_4^0$ on the right, to the type $DD$ module's outputs
    $(\rho_1\otimes\rho_3),(\rho_2\otimes\rho_2),(\rho_3\otimes\rho_1),(\rho_4\otimes\rho_4)$,
    or the operation $\mu^1_0\otimes U_2$.
  \item $U_1\otimes \rho_{1234}$. This is symmetric to the previous case.
  \item $\rho_{1234}\otimes \rho_{1234}$. By a similar argument to the
    previous cases, these come from the operations $\mu_2^0$ on the two
    sides applied to
    $(\rho_{123}\otimes\rho_{123}),(\rho_4\otimes\rho_4)$ and from the
    operation $\mu_0^1\otimes\mu_0^1$. The terms $\rho_{2341}\otimes
    \rho_{4123}$, $\rho_{3412}\otimes\rho_{3412}$, and
    $\rho_{4123}\otimes\rho_{2341}$ are similar.
  \item $\rho_{1234}\otimes\rho_{3412}$. These come from the
    operations $\mu_2^0$ on the two sides applied to 
    $(\rho_1\otimes\rho_3),(\rho_{234}\otimes \rho_{412})$ and from
    the operation $\mu_0^1\otimes\mu_0^1$. The terms
    $\rho_{2341}\otimes\rho_{2341}$, $\rho_{3412}\otimes\rho_{1234}$,
    and $\rho_{4123}\otimes\rho_{4123}$ are similar.
  \item $U_1 \otimes U_2$. It is straightforward to see that
    such a term cannot arise from actions  on the type $DD$  outputs.
  \end{itemize}
  So, in each case, the terms cancel in pairs.

  As noted earlier, the result for $\CFDDm(\Id)$ follows by setting
  $U_1 = 1$.
\end{proof}

While $\uId$ is graded by $\bigGroup$, setting $U_1=1$ forces the
grading to be by the smaller set $S'_{\DD} = \bigGroup/(r'(\grb(U))) =
\bigGroup/(1;1,1,1,1)$. Conveniently, this quotient is compatible with
the projection map $\pi \co \bigGroup\to\interGroup$ recalled in
Section~\ref{sec:int-gr}, so we may also take $\CFDDm(\Id)$ to be
graded by $\interGroup$. See Section~\ref{sec:gradings} for further discussion.

\subsection{Boundedness}\label{sec:boundedness}
We digress to discuss boundedness hypotheses for $\MAlg$, $\CFAm$, and
$\CFDDm(\Id)$: we need these in order for the box products involved in the
definition of $\CFDm$ and the pairing theorem to make sense (and be
well-behaved).

There is a filtration---in fact, grading---on $\MAlgc$ by the length
of an element, where each $\rho_i$ has length $1$ and $U$ has length
$4$. The operations $\mu_n^w$ respect this filtration, in the sense
that 
\begin{equation}
  \mu_n^w\co \Filt^{m_1}A\otimes\cdots\otimes\Filt^{m_n}A\to \Filt^{m_1+\cdots+m_n+4w}A.
\end{equation}
We will take this property as our definition of a filtered weighted
$\Ainf$-algebra, except that the $4$ may be replaced by any positive
number.

Let $\MAlgc$ denote the completion of $\MAlg$ with respect to the length
filtration, and $\CFAc$ the completion of $\CFAm$ with respect to the $U$-power
filtration, i.e., $\CFAc=\CFAm\otimes_{\Field[U]}\Field[[U]]$. In
particular, $\MAlgc$ is an algebra over $\Field[[U]]$.

Recall that a weighted $\Ainf$-algebra $\Alg=(A,\{\mu_n^w\})$ is \emph{bonsai} if for
any stably weighted tree $T$ with $n$ inputs and dimension
$\dim(T)$ sufficiently large, the corresponding operation $\mu(T)\co A^{\otimes n}\to A$
vanishes~\cite[Definition~\ref{Abs:def:w-bonsai}]{LOT:abstract}. A weighted
algebra $\Alg$ with a descending filtration
$\Filt^0A\supset \Filt^1A\supset\cdots$ is \emph{filtered bonsai} if $A$ is
complete with respect to this filtration and each quotient $A/\Filt^mA$ is
bonsai~\cite[Section~\ref{Abs:sec:boundedness}]{LOT:abstract}. In particular,
$\MAlgc$ is filtered bonsai
(compare~\cite[Proposition~\ref{TA:lem:Bonsai}]{LOT:torus-alg}).

To make sense of the structure relation for a type $D$ (or \DD)
bimodule, it suffices to know that the algebra is filtered bonsai, as
is the case for~$\MAlgc$.

The two versions of the pairing theorem involve
$\bigl(\CFAm_{U=1}(Y_1)\otimes\CFAm(Y_2)\bigr)\DT\CFDDm(\Id)$ and its
reparenthesization $\CFAm(Y_2)\DT\bigl(\CFAm_{U=1}(Y_1)\DT\CFDDm(\Id)\bigr)$. To make
sense of these, we need some further conditions and results.

In our previous paper, we gave a boundedness condition for type
$D$
structures~\cite[Definition~\ref{Abs:def:wD-bounded}]{LOT:abstract},
but this condition only makes sense when the charge (which for the
present paper is $1$) is nontrivial. An alternate condition that is
relevant for $\CFDDm(\Id)$ is the following. Fix weighted algebras
$\wAlg$ and $\wBlg$ over $\Field[\Yvar_1]$ and $\Field[\Yvar_2]$ and a
filtration $\Filt$ on $\wBlg$ with respect to which $\wBlg$ is
complete. Call a type \DD\ bimodule
$(P,\delta^1)$ over $\wAlg$ and $\wBlg$
\emph{right-filtered short} if the variable $\Yvar_2$ lies in
$\Filt^1B$ and the image of $\delta^1$ lies in
$(A\otimes\Filt^1B)\otimes P$. Similarly, call a type $D$ structure
$(P,\delta^1)$ over $\wBlg$ \emph{filtered short} if the image of
$\delta^1$ lies in $(\Ground\oplus \Filt^1B)\otimes P$.

Recall that $\CFDDm(\Id)$ (Section~\ref{sec:intro-DD-Id}) is a type \DD\
$(\MAlg^{U=1},\MAlg)$-bimodule, i.e., a type $D$ structure over
$(\MAlg^{U=1})\otimes\MAlg$.
Abusing notation, let $(\MAlg^{U=1})\otimes\MAlgc$ denote the
completion of $(\MAlg^{U=1})\otimes\MAlg$ with respect to the length
filtration on the second copy of $\MAlg$. Then,
$(\MAlg^{U=1})\otimes\MAlgc$ is filtered bonsai and $\CFDDm(\Id)$ is
right-filtered short.

The definitions of \emph{bonsai} and \emph{filtered bonsai} for a weighted
$\Ainf$-module are exactly analogous to the definitions for a weighted
algebra (see, e.g.,~\cite[Definition~\ref{Abs:def:filtered-bonsai}]{LOT:abstract}). Recall that a Heegaard diagrams is \emph{admissible} (respectively
\emph{provincially admissible}) if every periodic domain (respectively
provincial periodic domain) has both positive and negative
coefficients. We have:
\begin{lemma}\label{lem:admis-bounded}
  If $\HD$ is admissible then $\CFAc(\HD)$ is a well-defined module over
  $\MAlgc$ and, in fact, is filtered bonsai.
\end{lemma}
\begin{proof}
  To verify that $\CFAc(\HD)$ is both well-defined and filtered bonsai, it
  suffices to show that for any integer $n$ there is a bound on
  $w+\sum |\rho^i|$ so that there is a non-zero operation on $\CFAm(\HD)$ of the
  form
  \begin{equation}\label{eq:admis-contra}
    m_{1+k}^w(\x,\rho^1,\dots,\rho^k)=U^n\y.
  \end{equation}
  Fix an area form on $\Sigma$ so that any element of $\pi_2(\x,\x)$
  has area $A(P)=n_z(P)$; this is possible by admissibility of $\HD$~\cite[Lemma 4.26]{LOT1}.
  Suppose there are infinitely
  many non-zero operations as in Formula~\eqref{eq:admis-contra}. Since there
  are finitely many intersection points, infinitely many of these operations
  have the same starting generator $\x$ and ending generator $\y$. The domains
  of these operation differ by periodic domains (because each
  domain has $n_z=n$), and hence have the same area; but there are only finitely
  many positive domains with any given area.
\end{proof}

\begin{proposition}\label{prop:l-b-DT}
  Let $\Alg$ and $\Blg$ be weighted algebras over $\Field[\Yvar_1]$ and
  $\Field[\Yvar_2]$ so that $\Blg$ is filtered and complete with
  respect to the filtration and $\Yvar_2\in\Filt^1B$.
  Let $\fModule$ be a filtered $\Alg$-module, $\fNodule$ a
  $\Blg$-module, and $P$ a type \DD\ $(\Alg,\Blg)$ bimodule. Assume
  that $\fNodule$ is
  filtered bonsai and $P$ is right-filtered short. Fix any weighted module
  diagonal $\TrMDiag$ and weighted module diagonal primitive $\TrMPrim$. Then:
  \begin{enumerate}
  \item The box product $(\fModule\otimes_{\TrMDiag}\fNodule)\DT P$ is
    a well-defined chain complex.
  \item If $\TrMDiag'$ is another module diagonal and $\fModule'$ and
    $\fNodule'$ are
    modules with $\fNodule'$ filtered bonsai, $\fNodule'\simeq
    \fNodule$ by a filtered bonsai
    homotopy equivalence, and $\fModule'\simeq \fModule$, then
    $(\fModule\otimes_{\TrMDiag}\fNodule)\DT P\simeq
    (\fModule'\otimes_{\TrMDiag'}\fNodule')\DT P$.
  \item The box product $\fModule\DT_{\TrMPrim}P$ is a well-defined
    type $D$ structure, and is filtered short.
  \item If $\TrMPrim'$ is another module diagonal primitive and $M'$ is a 
    modules with $\fModule'\simeq \fModule$ then
    $\fModule\DT_{\TrMPrim}P\simeq \fModule'\DT_{\TrMPrim'}P$.
  \item There is a homotopy equivalence of chain complexes
    $(\fModule\otimes_{\TrMDiag}\fNodule)\DT P\simeq
    \fNodule\DT(\fModule\DT_{\TrMPrim}P)$.
\end{enumerate}
\end{proposition}
\begin{proof}
  The proof involves checking that the boundedness hypotheses in the statement
  of the proposition are sufficient for each expression in the proofs
  of~\cite[Theorem~\ref{Abs:thm:wTripleTensorProduct}, Proposition~\ref{Abs:prop:one-sided-DT-works} and Corollary~\ref{Abs:cor:w-DT-preserve-hequiv}]{LOT:abstract} to make sense. This is left to the reader.
\end{proof}

\begin{corollary}\label{cor:bounded-DTs}
  Fix bordered Heegaard diagrams $\HD_1$ and $\HD_2$ representing $Y_1$ and
  $Y_2$, respectively, so that $\HD_1$ is provincially admissible and $\HD_2$ is
  admissible. Then, the tensor products
  $\bigl(\CFAm_{U=1}(\HD_1)\otimes \CFAc(\HD_2)\bigr)\DT\CFDDm(\Id)$ and
  $\CFAm_{U=1}(\HD_1)\DT\CFDDm(\Id)$ are well-defined and independent of the choices
  of Heegaard diagrams $\HD_i$, almost complex structures, module diagonal, and
  module diagonal primitive. Further,
  \[
    \bigl(\CFAm_{U=1}(\HD_1)\otimes \CFAc(\HD_2)\bigr)\DT\CFDDm(\Id)\simeq \CFAc(\HD_2)\DT\bigl(\CFAm_{U=1}(\HD_1)\DT\CFDDm(\Id)\bigr).
  \]
\end{corollary}
\begin{proof}
  This is immediate from Lemma~\ref{lem:admis-bounded},
  Proposition~\ref{prop:l-b-DT}, and invariance of $\CFAm(\HD)$
  (Theorem~\ref{thm:CFAm-invt}).
\end{proof}

\begin{remark}
  Recall that our definition of admissibility for bordered Heegaard
  diagrams corresponds to weak admissibility for all
  $\SpinC$-structures in the closed case. In particular, if $\HD_1$ is
  provincially admissible and $\HD_2$ is admissible then the closed
  Heegaard diagram $\HD_1\cup\HD_2$ is weakly admissible for all
  $\SpinC$-structures. So, the $U$-completed complex
  $\CFmm(\HD_1\cup\HD_2)$ is defined, but the uncompleted complex
  $\CFm(\HD_1\cup\HD_2)$ typically is not.
\end{remark}

\subsection{Definition and first properties of \texorpdfstring{$\CFD^-$}{CFD}}
\begin{definition}\label{def:CFDm}
  Given a provincially admissible bordered Heegaard diagram $\HD$ with torus
  boundary define $\CFDm(\HD)$ to be the type $D$ structure
  \[
    \CFDm(\HD)=\CFAm_{U=1}(\HD)\DT \CFDDm(\Id)
  \]
  over $\MAlgc$. Given a bordered $3$-manifold $Y$ with torus boundary, let
  $\HD$ be any provincially admissible bordered Heegaard diagram representing
  $Y$ and define $\CFDm(Y)=\CFDm(\HD)$.
\end{definition}

\begin{proof}[Proof of Theorem~\ref{thm:intro-CFD}]
  This is a restatement of part of Corollary~\ref{cor:bounded-DTs}.
\end{proof}

\begin{remark}\label{rem:alg-diag-dependent}
  We have not shown that $\CFDm(\HD)$ is independent of the choice of
  algebra diagonal. Doing so requires some further homological
  algebra---a relationship between module diagonal primitives with
  respect to different algebra diagonals---which we leave to future
  work.
\end{remark}

In practice, one can often deduce $\CFDm$ from $\CFDa$, via the
following principle, which is illustrated in Section~\ref{sec:knot}:
\begin{proposition}\label{prop:extend-CFDa}
  Fix a bordered manifold $Y$ with torus boundary. Let $P$ be a type
  $D$ structure over $\Alg(T^2)$ which is homotopy equivalent to
  $\CFDa(Y)$. Assume that $P$ is reduced, that is, that the image of
  the differential $\delta^1$ lies in the augmentation ideal (has no
  terms of the form $1\otimes x$).
  Then, $P$ can be extended to a model for $\CFDm(Y)$ in
  the following sense: there are operations $D_\rho\co P\to P$,
  associated to the Reeb chords $\rho$ in $\MAlg$,
  so that:
  \begin{itemize}
  \item The differential on $P$ is given by
    \[
      \delta^1=
      \rho_1\otimes D_{\rho_1}+
      \rho_2\otimes D_{\rho_2}+
      \rho_3\otimes D_{\rho_3}+
      \rho_{12}\otimes D_{\rho_{12}}+
      \rho_{23}\otimes D_{\rho_{23}}+
      \rho_{123}\otimes D_{\rho_{123}}
    \]
  \item If we define $Q$ to be the type $D$ structure over $\MAlg$
    with the same underlying vector space as $P$ and differential given by
    \[
      \delta^1=
      \sum_{\text{chords $\rho$ in $\MAlg$}}\rho\otimes D_{\rho}
    \]
    then $P$ is homotopy equivalent to $\CFDm(Y)$.
  \end{itemize}
\end{proposition}

One step in the proof of Proposition~\ref{prop:extend-CFDa} is a
version of the Homological Perturbation Lemma for weighted
$\Ainf$-modules:
\begin{lemma}\label{lem:hpt}
  Let $\fModule=(M,\{m_n^w\})$ be a weighted module over a weighted
  algebra $\Alg$ and $N$ a chain complex homotopy equivalent to
  $(M,m_1^0)$. Then, the differential on $N$ can be extended to a
  weighted module structure $\fNodule$ on $N$ which is homotopy
  equivalent to $\fModule$ as a weighted module over $\Alg$.
\end{lemma}
\begin{proof}
  The proof is obtained from the unweighted case
  by summing over the weights.

  Given (possibly ungraded) chain complexes $(C,\bdy_C)$ and
  $(D,\bdy_D)$, a \emph{deformation retraction} from $C$ to $D$
  consists of chain maps $f\co C\to D$ and $g\co D\to C$ and a chain
  homotopy $T\co C\to C$ so that $f\circ g = \Id_D$ and $\Id_C+g\circ
  f= \partial_C\circ T + T\circ\partial_C$. A \emph{strong
    deformation retraction} is a deformation retraction so that
  $T\circ g=0$, $f\circ T=0$, and $T\circ T=0$. Lambe and
  Stasheff~\cite[Section 2.1]{LambeStasheff87:perturb}
  give magic formulas for turning any deformation retraction into a
  strong deformation retraction: replace $T$ by $T''$ where
  \begin{align*}
    T' &= (\Id_C+g\circ f)\circ T\circ (\Id+g\circ f)\\
    T'' &= T'\circ \partial_C\circ T'.
  \end{align*}
  (In verifying these formulas, it is helpful to notice that
  $(g\circ f+\Id_C)^2=(g\circ f+\Id_C)$.)

  Any chain homotopy equivalence can be turned into a zig-zag of
  deformation retractions, by taking mapping cylinders. Specifically,
  given a homotopy equivalence consisting of chain maps $f\co C\to D$
  and $g\co D\to C$ and homotopies $T\co C\to C$ and $U\co D\to D$,
  the mapping cylinder of $f$ is $C\oplus C[1]\oplus D$ with
  differential
  \[
    \begin{bmatrix}
      \bdy_C & \Id & 0\\
      0 & \bdy_C & 0\\
      0 & f & \bdy_D
    \end{bmatrix}.
  \]
  This cylinder retracts to $D$ via the obvious inclusion, the projection
  $ 
  \begin{bmatrix}
    f & 0 & \Id
  \end{bmatrix}
  $
  and the homotopy
  \[
    \begin{bmatrix}
      0 & T & g\\
      \Id & 0 & 0\\
      0 & 0 & 0
    \end{bmatrix}.
  \]
  The cylinder retracts to $C$ via the inclusion as the first summand,
  the projection $
  \begin{bmatrix}
    \Id & T & g
  \end{bmatrix}
  $ and the homotopy
  \[
    \begin{bmatrix}
      0 & 0 & 0\\
      0 & T & g\\
      fT+Uf & 0 & U
    \end{bmatrix}.
  \]

  So, to prove the lemma, it suffices to prove two special cases:
  \begin{enumerate}
  \item The case that $(M,m_1^0)$ is a strong deformation retraction
    of $(N,\bdy)$.
  \item The case that $(N,\bdy)$ is a strong deformation retraction of
    $(M,m_1^0)$.
  \end{enumerate}

  In the first case, define the operations $n_{1+k}^w$ on $N$ by
  \[
    n_{1+k}^w(x,a_1,\dots, a_k)=f(m_{1+k}^w(g(x),a_1,\dots,a_k)).
  \]
  It is immediate that these satisfy the weighted
  $\Ainf$-relations. Promote $f$ and $g$ to maps of weighted modules
  by setting $f_1^0=f$, $f_{k}^w=0$ for $k+2w>1$, $g_1^0=g$, and
  $g_k^w=0$ for $k+2w>1$. Again, the weighted $\Ainf$ relations for
  these maps are immediate. Finally, promote the homotopy $U$ from
  $f\circ g$ to the identity to a weighted homotopy in the same way.

  The second case is more interesting. Given maps $\alpha\co W\otimes
  V^{\otimes a}\to X$ and $\beta\co X\otimes V^{\otimes b}\to Y$ let
  \[
    \beta\circ_1\alpha=\beta\circ (\alpha\otimes\Id_{V^{\otimes b}})\co
    W\otimes V^{a+b}\to Y.
  \]
  Let $f\co M\to N$, $g\co N\to M$, and $T\co M\to M$ denote the
  strong deformation retract from $M$ to $N$. Define the operations on
  $N$ by
  \[
    n_{1+k}^w=
    \sum_{\substack{i_1+\cdots+i_\ell=k\\ w_1+\cdots+w_\ell=w}}
    f\circ_1 m_{i_\ell}^{w_\ell}\circ_1 T\circ_1 m_{i_{\ell-1}}^{w_{\ell-1}}\circ_1\cdots\circ_1 
    T\circ_1 m_{i_1}^{w_1} \circ_1 g
  \]
  for $k+w>0$ and declaring that $n_1^0$ is the given
  differential. In this sum (and later ones), each $m_{i_j}^{w_j}$ has
  $i_j+2w_j\geq 2$. (Compare, for instance,~\cite[Lemma 9.6]{LOT4}.) The
  special case $\ell=1$ is allowed in this sum; that term is
  $f\circ m_{1+k}^w\circ g$.

  Promote $f$ to a weighted module map by declaring that
  $f_1^0$ is the given map $f$ and for $k+w>0$,
  \[
    f_{1+k}^w=
    \sum_{\substack{i_1+\cdots+i_\ell=k\\ w_1+\cdots+w_\ell=w}}
    f\circ_1 m_{i_{\ell}}^{w_{\ell}}\circ_1 T\circ_1\cdots\circ_1 
    m_{i_1}^{w_1}\circ_1 T.
  \]
  Promote $g$ to a weighted module map similarly, via the formula
  \[
    g_{1+k}^w=
    \sum_{\substack{i_1+\cdots+i_\ell=k\\ w_1+\cdots+w_\ell=w}}
    T\circ_1 m_{i_{\ell}}^{w_{\ell}}\circ_1\cdots\circ_1 T\circ_1
    m_{i_1}^{w_1} \circ_1 g.
  \]
  Promote $T$ to a homotopy of weighted module maps similarly, via the
  formula
  \[
    T_{1+k}^w=
    \sum_{\substack{i_1+\cdots+i_\ell=k\\ w_1+\cdots+w_\ell=w}}
    T\circ_1 m_{i_{\ell}}^{w_{\ell}}\circ_1\cdots\circ_1 T\circ_1
    m_{i_1}^{w_1} \circ_1 T.
  \]

  Verifying that the $f_k^w$, $g_k^w$, and $T_k^w$ satisfy the required structure
  relations is left to the reader; note that the fact that $T^2=0$ and
  $f\circ g=\Id$ imply that $\{f_k^w\}\circ\{g_k^w\}=\Id$.
\end{proof}

Another step in the proof of Proposition~\ref{prop:extend-CFDa} is to
show that we can push weighted $\Ainf$-module structures through
certain isomorphisms (which is obvious until one writes a proof):
\begin{lemma}\label{lem:extend-iso}
  Fix a weighted $\Ainf$-module $\fModule$ over $\MAlg$, and let
  $\widehat{M}=M/UM$, viewed as an $\Ainf$-module over $\Alg(T^2)$
  (the $\HFa$ version of the bordered algebra, from~\cite{LOT1}). Let
  $\widehat{N}$ be another $\Ainf$-module over $\Alg(T^2)$, and
  $\widehat{F}=\{\widehat{f}_{1+n}\}\co \widehat{M}\to\widehat{N}$ and $\widehat{G}=\{\widehat{g}_{1+n}\}\co
  \widehat{N}\to \widehat{M}$ be inverse isomorphisms of $\Ainf$-modules over
  $\Alg(T^2)$. Let $N=\widehat{N}\otimes_{\Field}\Field[U]$, with
  its induced module structure over $\Alg(T^2)$. Then, the operations
  on $N$ extend to a weighted $\Ainf$-module structure $\fNodule$ over
  $\MAlg$ isomorphic to $M$.
\end{lemma}
\begin{proof}
  Extend the isomorphisms $\widehat{F}$ and $\widehat{G}$ to weighted maps
  $F\co \fModule\to \fNodule$ and $G\co \fNodule\to \fModule$ defined by $f_{1+n}^w=0$ and $g_{1+n}^w=0$
  for $w>0$, linearity over $\FF_2[U]$, and the formulas
  \begin{align*}
    f_{1+n}^0(x,a_1,\dots,a_n)&=\widehat{f}_{1+n}(x,\Pi(a_1),\dots,\Pi(a_n))\\
      g_{1+n}^0(x,a_1,\dots,a_n)&=\widehat{g}_{1+n}(x,\Pi(a_1),\dots,\Pi(a_n))
  \end{align*}
  where $\Pi\co \MAlg\to\Alg(T^2)\otimes\Field[U]$ is the quotient map by the ideal
  generated by $\rho_4$. (Note that this is not an $\Ainf$-ideal,
  since $\mu_4^0(\rho_4,\rho_3,\rho_2,\rho_1)=U\iota_1$.)

  The operations on $\fNodule$ are defined by
  \begin{multline}\label{eq:extend-iso-how}
    m_{1+n}^w(x,a_1,\dots,a_n)=\sum_{p+q+r=n} f_{1+p}^0(m_{1+q}^w(g_{1+r}^0(x,a_1,\dots,a_r),a_{r+1},\dots,a_{r+q}),a_{r+q+1},\dots,a_{n})\\
    +\sum_{\substack{p+q+r=n\\q+1\leq i\leq n-r}} f_{1+p}^0(g_{1+q}^0(x,a_1,\dots,a_q),a_{q+1},\dots,a_{i-1},\mu_{r+1}^w(a_i,\dots,a_{i+r}),a_{i+r+1},\dots,a_n),
  \end{multline}
  where the $m_{1+q}^w$ on the right is the operation on $\fModule$.
  In the graphical notation of our previous paper~\cite{LOT:abstract},
  these operations are
  \[
    \mathcenter{    \begin{tikzpicture}
      \node at (0,0) (tc) {};
      \node at (0,-1) (g) {$g^0_\bullet$};
      \node at (0,-2) (m) {$m^\bullet_\bullet$};
      \node at (0,-3) (f) {$f^0_\bullet$};
      \node at (0,-4) (bc) {};
      \node at (1.5,0) (tr) {};
      \node at (1.5,-1) (Delta) {$\Delta$};
      \draw[taa] (tr) to (Delta);
      \draw[taa] (Delta) to (g);
      \draw[taa] (Delta) to (m);
      \draw[taa] (Delta) to (f);
      \draw[moda] (tc) to (g);
      \draw[moda] (g) to (m);
      \draw[moda] (m) to (f);
      \draw[moda] (f) to (bc);
    \end{tikzpicture}
  }
  +
  \mathcenter{
    \begin{tikzpicture}
      \node at (0,0) (tc) {};
      \node at (2,0) (tr) {};
      \node at (0,-2) (g) {$g^0_\bullet$};
      \node at (0,-3) (f) {$f^0_\bullet$};
      \node at (0,-4) (bc) {};
      \node at (2,-1) (Delta) {$\Delta$};
      \node at (1,-2) (mu) {$\mu_\bullet^\bullet$};
      \draw[moda] (tc) to (g);
      \draw[moda] (g) to (f);
      \draw[moda] (f) to (bc);
      \draw[alga] (mu) to (f);
      \draw[taa] (tr) to (Delta);
      \draw[taa] (Delta) to (mu);
      \draw[taa, bend left=15] (Delta) to (f);
      \draw[taa, bend right=15] (Delta) to (f);
      \draw[taa] (Delta) to (g);
    \end{tikzpicture}.
  }
  \]

  The fact that $\widehat{F}$ and $\widehat{G}$ are
  inverse isomorphisms (between $\widehat{M}$ and $\widehat{N}$) implies
  that the second sum in Equation~\eqref{eq:extend-iso-how} is equal
  to
  \[
    \sum_{\substack{p+q+r=n\\1\leq i\leq q}} f^0_{1+p}(g^0_{1+q}(x,a_1,\dots,a_{i-1},\mu_{r+1}^w(a_i,\dots,a_{i+r}),a_{i+r+1},\dots,a_{q+r}),a_{q+r+1},\dots,a_n),
    \]
  where the sum over $r$ starts from $r=-1$;
  or, in graphical notation,
  \[
    \mathcenter{
      \begin{tikzpicture}
        \node at (0,0) (tc) {};
        \node at (2,0) (tr) {};
        \node at (0,-2) (g) {$g^0_\bullet$};
        \node at (0,-3) (f) {$f^0_\bullet$};
        \node at (0,-4) (bc) {};
        \node at (2,-1) (Delta) {$\Delta$};
        \node at (1,-2) (mu) {$\mu_\bullet^\bullet$};
        \draw[moda] (tc) to (g);
        \draw[moda] (g) to (f);
        \draw[moda] (f) to (bc);
        \draw[alga] (mu) to (f);
        \draw[taa] (tr) to (Delta);
        \draw[taa] (Delta) to (mu);
        \draw[taa, bend left=15] (Delta) to (f);
        \draw[taa, bend right=15] (Delta) to (f);
        \draw[taa] (Delta) to (g);
      \end{tikzpicture}
    }
    +
    \mathcenter{
      \begin{tikzpicture}
        \node at (0,0) (tc) {};
        \node at (2,0) (tr) {};
        \node at (0,-2) (g) {$g^0_\bullet$};
        \node at (0,-3) (f) {$f^0_\bullet$};
        \node at (0,-4) (bc) {};
        \node at (2,-1) (Delta) {$\Delta$};
        \node at (1,-1.5) (mu) {$\mu_\bullet^\bullet$};
        \draw[moda] (tc) to (g);
        \draw[moda] (g) to (f);
        \draw[moda] (f) to (bc);
        \draw[alga] (mu) to (g);
        \draw[taa] (tr) to (Delta);
        \draw[taa] (Delta) to (mu);
        \draw[taa, bend left=15] (Delta) to (g);
        \draw[taa, bend right=15] (Delta) to (g);
        \draw[taa] (Delta) to (f);
      \end{tikzpicture}
    }=0.
  \]
  Using this, verifying the structure relations is an exercise.
\end{proof}

\begin{proof}[Proof of Proposition~\ref{prop:extend-CFDa}]
  By Lemma~\ref{lem:hpt}, there is a weighted module $\fModule$ so
  that the
  image of $m_1^0$ lies in $U\cdot M$ and a homotopy equivalence
  $\CFAm(Y)\simeq \fModule$. Consequently
  (by~\cite[Lemma~\ref{Abs:lem:w-DT-bifunc}]{LOT:abstract})
  $\CFDm(Y)\simeq \fModule\DT\CFDDm(\Id)$.
  Let $\widehat{M}=M/UM$ viewed as a module over $\Alg(T^2)$ as in
  Lemma~\ref{lem:extend-iso}. The homotopy equivalence $\CFAm(Y)\simeq
  \fModule$ induces a homotopy equivalence $\CFAa(Y)\simeq
  \widehat{M}$, and hence $\CFDa(Y)\simeq \widehat{M}\DT\CFDDa(\Id)$.

  Hedden-Levine show that, given a reduced type $D$ structure $P$ over
  $\Alg(T^2)$ there is an $\Ainf$-module $\widehat{N}$ over
  $\Alg(T^2)$ so that $P\cong \widehat{N}\DT\CFDDa(\Id)$, and the
  differential $m_1$ on $\widehat{N}$ vanishes. 
  Taking $P=\widehat{M}\DT\CFDDa(\Id)$, since tensoring with $\CFDDa(\Id)$ is an
  equivalence of categories, $\widehat{N}\simeq \CFAa(Y)$. Thus,
  $\widehat{N}\simeq\widehat{M}$, but since the differential on both
  vanishes, in fact $\widehat{N}\cong\widehat{M}$, as $\Ainf$-modules
  over $\Alg(T^2)$. Thus, by Lemma~\ref{lem:extend-iso}, $\widehat{N}$
  can be extended to a weighted $\Ainf$-module $\fNodule$ over $\MAlg$
  isomorphic to $\fModule$.

  We claim that $\fNodule\DT\CFDDm(\Id)$ is the desired type $D$
  structure. Since $\fNodule$ extends $\widehat{N}$ and
  $\widehat{N}\DT\CFDDa(\Id)\cong P$, $N\DT\CFDDm(\Id)$ extends
  $P$. Further, $\fNodule\DT\CFDDm(\Id)\cong M\DT\CFDDm(\Id)\simeq
  \CFAm(Y)\DT\CFDDm(\Id)=\CFDm(Y)$, as claimed.
\end{proof}

Given a knot $K$ in $S^3$, we showed previously that
$\CFDa(S^3\setminus\nbd(K))$ is determined by the knot Floer complex
of $K$~\cite[Chapter 11]{LOT1}. A key step in the proof was a
structure relation for the \emph{generalized coefficient
  maps}~\cite[Proposition 11.30]{LOT1}. Specifically:
\[
  D\circ D_{0123}+D_3\circ D_{012} + D_{23}\circ D_{01} + D_{123} \circ D_0
  + D_{0123}\circ D =\Id
\]
and similarly for any cyclic permutation of the indices $\{0,1,2,3\}$.
We note that the structure equation for $\CFDm$ implies these relations:
\begin{proposition}\cite[Proposition 11.30]{LOT1}
  The generalized coefficient maps $D_I$ satisfy the structure
  relations above.
\end{proposition}
\begin{proof}
  The map $D_I$ is the coefficient of $\rho_I$ in the differential
  $\delta^1$ on $\CFDm(Y)$, so the equation shown is the coefficient
  of $\rho_{0123}$ in the weighted type $D$ structure equation (and
  the three equations obtained by cyclically permuting the indices are
  the coefficients of $\rho_{1230}$, $\rho_{2301}$, and
  $\rho_{3012}$). The identity map on the right comes from the
  operation $\mu_0^1$ that appears in the structure relation.
\end{proof}


\section{Gradings}\label{sec:gradings}
In this section we discuss the gradings on the bordered algebras. This
is a straightforward extension of the $\HFa$ case~\cite[Chapter
10]{LOT1}, as the grading on the algebra was~\cite[Section
4]{LOT:torus-alg}. Generalities on gradings were already discussed in
Section~\ref{sec:gradings-abstract}. Sections~\ref{sec:big-gr},~\ref{sec:int-gr},
and~\ref{sec:small-gr} recall the big, intermediate, and small grading
groups from our previous paper~\cite{LOT:torus-alg} and construct the
gradings on $\CFAm$ and $\CFDm$ by them. Section~\ref{sec:gr-pairing}
states a graded version of the pairing theorem.

\subsection{Gradings by the big grading group}\label{sec:big-gr}

Fix a bordered Heegaard diagram $\HD$. We will define a
$\bigGroup$-set $S'_A(\HD)$ and a grading on $\CFAm(\HD)$ by
$S'_A(\HD)$, where $\bigGroup$ is the grading group from
Section~\ref{sec:big-gr-group}.
The grading set $S'_A(\HD)$ is a disjoint union over $\SpinC$-structures,
$S'_A(\HD)=\coprod_{\spinc\in\Spinc(Y)}S'_A(\HD,\spinc)$, and we will
work one $\SpinC$-structure at a time. Fix a generator~$\x_0$ with
$\spinc(\x)=\spinc$. (If no such generator exists, the module
$\CFAm(\HD)$ is trivial, so there is nothing to grade, but if one
wants to construct the grading set anyway perform an isotopy to create such a
generator.) Associated to each domain $B\in\pi_2(\x,\y)$ are the
multiplicities
$\bdy^\bdy(B)\in H_1(Z,\alphas\cap Z)=\ZZ^4$. Let
\[
  g'(B)=(-e(B)-n_\x(B)-n_\y(B);\bdy^\bdy(B))
\]
(compare~\cite[Formula~(10.2)]{LOT1} and
Formula~\eqref{eq:emb-ind}). Let $[\Sigma]$ be the domain
with multiplicity $1$ everywhere.
\begin{lemma}
  For any domain $B$, $g'(B)\in \bigGroup$. Further,
  $g'(B+[\Sigma])=\grb(U)\cdot g'(B)$.
\end{lemma}
\begin{proof}
  The first statement follows from the second and the corresponding
  fact in the $\HFa$ case~\cite[Lemma 10.3]{LOT1}. For the second
  statement, it
  is obvious that the $\SpinC$-components agree.
  For the Maslov
  components, observe that $e(B+[\Sigma])=e(B)-2g+1$,
  $n_\x(B+[\Sigma])=n_\x(B)+g$, and $n_\y(B+[\Sigma])=n_\y(B)+g$.
\end{proof}

Now, define the type~$A$ grading set to be the quotient of the grading
group by the gradings of the periodic domains:
\[
  S'_A(\HD,\spinc)=
  P'(\x_0) \backslash\bigGroup \quad\text{where}\quad
  P'(\x_0) = \langle g'(B)\mid B\in\pi_2(\x_0,\x_0),\ n_z(B)=0\rangle,
\]
Given a generator $\x$ with $\spinc(\x)=\spinc$, fix
$B_x\in\pi_2(\x_0,\x)$ with $n_z(B_x)=0$ and define
\[
  \gr'(\x)=P'(\x_0)g'(B_x)
\]
(compare~\cite[Definition 10.11]{LOT1}).
\begin{proposition}
  The function $\gr'$ is independent of the choice of domains $B_x$
  and defines a grading on $\CFAm(\HD,\spinc)$ by $S'_A(\HD,\spinc)$.
\end{proposition}
\begin{proof}
  From~\cite[Lemma 10.4]{LOT1} (and an easy extension to cover cases with
  $n_z \ne 0$), given $B_1\in\pi_2(\w,\x)$ and $B_2\in\pi_2(\x,\y)$,
  \[
    g'(B_1*B_2)=g'(B_1)*g'(B_2).
  \]
  So,
  given another domain $B'_x\in\pi_2(\x_0,\x)$ with
  $n_z(B'_x) = 0$,
  $B_x*(B'_x)^{-1}\in\pi_2(\x_0,\x_0)$ and hence
  \[
    g'(B_x)=g'(B_x*(B'_x)^{-1})g'(B'_x)
  \]
  so these descend to the same element of $S'_A(\HD,\spinc)$.

  It remains to verify that if $U^\ell\y$ occurs as a term in
  $m_{1+n}^k(\x,\rho^1,\dots,\rho^n)$ then
  \[
    \grb(U)^\ell\gr'(\y)=\grb(\x)\lambda_d^{n-1}\lambda_w^k\gr'(\rho^1)\cdots\gr'(\rho^n).
  \]
  (Recall that $\lambda_d=(1;0,0,0,0)$ and $\lambda_2=(1;1,1,1,1)$.)
  By construction, there is a domain $B\in\pi_2(\x,\y)$ so that
  $\ind(B,\rho^1,\dots,\rho^n,w)=1$. From the first part of the proof,
  we may assume that $B_y*\ell[\Sigma]=B_x*B$ (since $n_z(B)=\ell$).
  Then, $\grb(U)^\ell\gr'(\y)=\gr'(\x)g'(B)$ so it
  suffices to show that
  $g'(B)=\lambda_d^{n-1}\lambda_w^k\gr'(\rho^1)\cdots\gr'(\rho^n)$. It
  is immediate from the definitions that the $\SpinC$-components of
  the two sides agree.
  Considering the Maslov components, by
  Proposition~\ref{prop:emb-ind},
  \[
    1=e(B)+n_\x(B)+n_\y(B)+n+\iota(\vec{\rho})+k
  \]
  so the Maslov component of $g'(B)$ is given by
  \[
    -e(B)-n_\x(B)-n_\y(B)=n-1+\iota(\vec{\rho})+k.
  \]
  By Lemma~\ref{lem:iota-is-gr}, the Maslov component of
  $\gr'(\rho^1)\cdots\gr'(\rho^n)$ is $\iota(\vec{\rho})$. The Maslov
  component of $\lambda_w^k$ is $k$ and the Maslov component of
  $\lambda_d^{n-1}$ is $n-1$, as desired.
\end{proof}

\begin{lemma}
  Let $\x_0$ and $\x_0'$ be base generators representing $\spinc$, and
  let $S'_A(\HD,\spinc,\x_0)$ and $S'_A(\HD,\spinc,\x'_0)$ be the
  corresponding $\bigGroup$-sets. Then there is an isomorphism
  $S'_A(\HD,\spinc,\x_0)\to S'_A(\HD,\spinc,\x'_0)$ of $\bigGroup$-sets so
  that the diagram
  \[
    \begin{tikzcd}
      & \Gen(\HD)\arrow[dl,"\gr'_{\x_0}"  above]\arrow[dr,"\gr'_{\x'_0}" above]& \\
      S'_A(\HD,\spinc,\x_0)\arrow[rr] & & S'_A(\HD,\spinc,\x'_0)
    \end{tikzcd}
  \]
  commutes.
\end{lemma}
\begin{proof}
  This follows from the same proof as in the case of
  $\HFa$~\cite[Theorem 10.18]{LOT1}, and is left to the reader.
\end{proof}

We turn next to the grading on $\CFDm(\HD)=\CFAm_{U=1}(\HD)\DT\CFDDm(\Id)$.
Recall that $\CFDDm(\Id)$ is a bimodule over $\MAlg^{U=1}$ and
$\MAlg$. The algebra $\MAlg^{U=1}$ inherits a grading by
$\bigGroup/(-1;1,1,1,1)$, with weight grading
$\lambda'_w=(1;1,1,1,1)=(\lambda'_d)^2$.  The type \DD\ structure
$\CFDDm(\Id)$ has a grading by the left
$\bigl(\bigGroup/(-1;1,1,1,1)\bigr)\times_\ZZ\bigGroup$-set
\[
  S'_{\DD}=\bigGroup/r'(\grb(U))=\bigGroup/(-1;-1,-1,-1,-1)
\]
as at the end of Section~\ref{sec:DD-Id}. Here, the variables $\Yvar_1$ and
$\Yvar_2$ in the definition of a weighted diagonal act by $1$ and~$U$,
respectively, which have gradings
$(0;0,0,0,0)=(\lambda'_d)^{-2}\lambda'_w$ and
$(-1;1,1,1,1)=\lambda_d^{-2}\lambda_w$, as required (see
Section~\ref{sec:gradings-abstract}). Thus, $\CFDm(\HD)$ inherits a
grading by
\begin{align*}
  S'_D(\HD,\spinc)&=\bigl(S'_A(\HD,\spinc)/(-1;1,1,1,1)\bigr)\times_{\bigGroup/(-1;1,1,1,1)}S'_{\DD}\\
  &\cong \bigGroup/\bigl\langle(1;1,1,1,1),\ r'(g'(B))\mid B\in\pi_2(\x_0,\x_0),\ n_z(B)=0\bigr\rangle.
\end{align*}

Note that in $S'_\DD$ the element $1\otimes U$ has grading
\[
  (-1;1,1,1,1)=(-2;0,0,0,0)
\]
so the action of $U$ on $\CFDm(\HD)$ shifts the grading by
$\lambda^{-2}$; this is analogous to the fact that $U$ has Maslov
grading $-2$ in $\CF^-$.

\subsection{Projection to the intermediate grading group}\label{sec:int-gr}

We discuss two ways to reduce the $\bigGroup$-gradings to gradings by
smaller groups. The first is analogous to the construction used
in~\cite[Chapter 11]{LOT1}. Let
\[
  \interGroup=\biggl\{(m;a,b)\in (\OneHalf\ZZ)^3\biggm\vert a+b\in\ZZ, m+\frac{(2a+1)(a+b+1)+1}{2}\in\ZZ\biggr\}
\]
(cf. \cite[Section 4.2]{LOT:torus-alg}).
We call $G$ the \emph{intermediate grading group}. There is a
homomorphism $\pi\co \bigGroup\to \interGroup$ defined by
\[
  \pi(j;a,b,c,d)=\left(j-d;\frac{a+b-c-d}{2},\frac{-a+b+c-d}{2}\right).
\]
Given a homogeneous algebra element $a$, let
$\gr(a)=\pi(\grb(a))$. Since $\pi$ is a homomorphism and $\grb$
defines a
grading on $\MAlg$, $\gr$ also defines a grading on $\MAlg$. For this
grading, the central elements are $\lambda_d=(-1;0,0)$ and
$\lambda_w=\pi(1;1,1,1,1)=(0;0,0)$.  As usual, we will often
abbreviate $\lambda_d$ as just $\lambda$.

Because we will use them in Section~\ref{sec:knot}, we record the
intermediate gradings of certain algebra elements explicitly; from
these, it is easy to obtain the grading of any algebra element:
\begin{gather}\label{eq:gr-comps}
\begin{aligned}
  \gr(\rho_1)&=(-\OneHalf;\OneHalf,-\OneHalf) &
    \gr(\rho_{12})&=(-\OneHalf;1,0) &
      \gr(\rho_{123})&=(-\OneHalf;\OneHalf,\OneHalf) \\
  \gr(\rho_2)&=(-\OneHalf;\OneHalf,\OneHalf)  &
    \gr(\rho_{23})&=(-\OneHalf;0,1) &
      \gr(\rho_{234})&=(-\textstyle\frac{3}{2};-\OneHalf,\OneHalf) \\
  \gr(\rho_3)&=(-\OneHalf;-\OneHalf,\OneHalf) &
    \gr(\rho_{34})&=(-\textstyle\frac{3}{2};-1,0)&
      \gr(\rho_{341})&=(-\textstyle\frac{3}{2};-\OneHalf,-\OneHalf) \\
  \gr(\rho_4)&=(-\textstyle\frac{3}{2};-\OneHalf,-\OneHalf)\quad &
    \gr(\rho_{41})&=(-\textstyle\frac{3}{2};0,-1)\quad &
      \gr(\rho_{412})&=(-\textstyle\frac{3}{2};\OneHalf,-\OneHalf)
  \end{aligned}\\
  \gr(\rho_{1234})=\gr(\rho_{2341})=\gr(\rho_{3412})=\gr(\rho_{4123})=\gr(U)=(-2;0,0).\nonumber
\end{gather}
(A version of this table is in~\cite[Section 4.2]{LOT:torus-alg}.)
The $\bigGroup$-set gradings on $\CFAm$ and $\CFDm$ induce
$\interGroup$-set gradings. Specifically, the homomorphism $\pi$ makes
$\interGroup$ into a left $\bigGroup$-set (via $g'\cdot g=\pi(g')g$
for $g'\in\bigGroup$ and $g\in\interGroup$). So, we can set
\begin{equation}\label{eq:inter-gr-set-1}
  S_A(\HD,\spinc)=S'_A(\HD,\spinc)\times_{\bigGroup}\interGroup,
\end{equation}
as a right $\interGroup$-set. More explicitly,
\begin{equation}\label{eq:inter-gr-set-2}
  S_A(\HD,\spinc) = \langle \pi(g'(B))\mid B\in\pi_2(\x_0,\x_0),\ n_z(B)=0\rangle \backslash \interGroup.
\end{equation}
With respect to the identification~\eqref{eq:inter-gr-set-1}, the
grading of an element $\x\in\CFAa(\HD,\spinc)$ is given by
$\gr(x)=\grb(\x)\times (0;0,0)$; with respect to the more concrete
identification~\eqref{eq:inter-gr-set-2}, the grading of $\x$ is
simply $\gr(x)=\pi(\grb(\x))$.

Similarly, to define $S_D(\HD,\spinc)$ we can view $\interGroup$ as
having a right action by $\bigGroup$ via $\pi$ and let
\[
  S_A(\HD,\spinc)=\interGroup\times_{\bigGroup}S'_A(\HD,\spinc)
\]
as a left $\interGroup$-set. The grading on $\CFDm$ is then defined
similarly to the grading on $\CFAm$.

\subsection{Refinement to the small grading group}\label{sec:small-gr}
Recall that the small grading group $\smallTGroup$ is the central extension
\[
  0\to \ZZ\to \smallTGroup\to H_1(T^2)\to 0
\]
with commutation relation $gh=\lambda^{2[g]\cdot[h]}hg$ where $\cdot$
denotes the intersection pairing on $H_1(T^2)$. Explicitly, we can
describe $\smallTGroup$ as elements
$(m;a,b)\in \OneHalf\ZZ\times \ZZ^2$ with $m+\frac{a+b}{2}\in\ZZ$ and
multiplication as above.

The group $\smallTGroup$ is identified with the subgroup of
$\bigGroup$ consisting of tuples $(m;a,b,c,d)$ where $b=a+c$ and
$d=0$; the subquotient of $\bigGroup$ consisting of elements with
$a+c=b+d$ modulo $\langle(1;1,1,1,1)\rangle$; and the subgroup of
$\interGroup$ of tuples $(m;a,b)$ where
$a,b\in\ZZ$~\cite[Lemma~\ref{TA:lem:G-is-subgroup}]{LOT:torus-alg}. We
will mainly use the first of these identifications. The map from
$\smallTGroup$ to this subgroup is $(m;a,b)\mapsto (m;a,a+b,b,0)$.

As in the $\HFa$ case~\cite{LOT1}, to define a grading on the algebra
and modules by $\smallTGroup$ we use \emph{grading refinement
  data}. Specifically, choose a base idempotent, which we take to
be~$\iota_0$. Choose
also an element $\psi_1=(m_1;a_1,b_1,c_1,d_1)\in \bigGroup$ so that
$a_1-b_1+c_1-d_1=1$. (That is, $\psi_1$ connects $\iota_0$ to
$\iota_1$ in a suitable sense.) For example,
$\psi_1=\grb(\rho_1)=(-1/2;1,0,0,0)$ works. Define $\psi_0$ to be the
identity element in $\bigGroup$.

Let $\ol{\pi}\co \bigGroup\to \bigGroup/\langle(1;1,1,1,1)\rangle$ be
the quotient map. (It can be convenient to identify
$\bigGroup/\langle(1;1,1,1,1)\rangle$ with the subgroup of $\bigGroup$
where $d=0$.) Then, the $\smallTGroup$-grading on $\MAlg$ is given by
\[
  \gr_\psi(a)=\ol{\pi}(\psi_i \grb(a)\psi_j^{-1})\qquad\qquad\text{if }\iota_ia\iota_j=a.
\]
If we let $\ol{\psi}_i=\ol{\pi}(\psi_i)$, then we can rewrite this as
\[
  \gr_\psi(a)=\ol{\psi}_i\ol{\pi}(\grb(a))\ol{\psi}_j^{-1}. 
\]
The distinguished central elements are $\lambda_w=0$ and
$\lambda_d=\lambda$.
We record the refined gradings of some elements of the
algebra (see also~\cite[Section~\ref{TA:sec:gr-refine}]{LOT:torus-alg}):
\begin{align*}
  \gr_\psi(\rho_1) &= (0;0,0) &
    \gr_\psi(\rho_{123}) &= (1/2;0,1) &
      \gr_\psi(U)&=(-2;0,0). \\
  \gr_\psi(\rho_2) &= (-1/2;1,0) &
    \gr_\psi(\rho_{234}) &= (-2;0,0)\\
  \gr_\psi(\rho_3)&=(0;-1,1) &
    \gr_\psi(\rho_{341}) &= (-3/2;-1,0) \\
  \gr_\psi(\rho_4) &= (-5/2;0,-1) &
    \gr_\psi(\rho_{412}) &= (-2;1,-1)
\end{align*}

The refined grading on $\CFAm(\HD)$ is now defined just as in the
$\HFa$ case. Given $B\in\pi_2(\x,\y)$ where $\x\iota_i=\x$ and
$\y\iota_j=\y$ define
\[
  g(B)=\ol{\pi}(\psi_i g'(B)\psi_j^{-1}).
\]
Fix a base generator $\x_0$ representing the $\SpinC$-structure
$\spinc$ and let
\[
  \widehat{S}_A(\HD,\spinc)= \langle g(B)\mid B\in\pi_2(\x_0,\x_0),\
  n_z(B)=0\rangle\backslash\smallTGroup.
\]
Define $\gr(\x)=g(B)\in \widehat{S}_A(\HD,\spinc)$ for any
$B\in\pi_2(\x_0,\x)$.

\begin{lemma}
  The $\smallTGroup$-set $\widehat{S}_A(\HD,\spinc)$ and function $\gr$
  define a grading on $\CFAm(\HD)$.
\end{lemma}
\begin{proof}
  This is routine, and is left to the reader.
\end{proof}

The grading on $\CFDm(\HD)$ can be obtained from the grading on
$\CFAm(\HD)$ using the same strategy as at the end of
Section~\ref{sec:big-gr}. To make this work, for the copy of
$\smallTGroup$ corresponding to $\MAlg^{U=1}$ we use different grading
refinement data and view it as a subquotient of $\bigGroup$ in a
different way. Specifically, if $\psi_i=(m;a,b,c,d)$ then let
$\phi_i=(-m;c,b,a,d) = r'(\psi_i^{-1})$.  So, for our particular choice,
$\phi_0=(0;0,0,0,0)$ and $\phi_1=(1/2;0,0,1,0)$. Use $\phi$ as grading
refinement data and view $\smallTGroup$ as the subgroup
$\{(m;a,a+b,b,0)\}$ of $\bigGroup/\langle(-1;1,1,1,1)$. This gives
gradings
\begin{align*}
  \gr_\phi(\rho_1) &= (-1;1,-1) &
    \gr_\phi(\rho_{123}) &=(-3/2;1,0) &
      \gr_\phi(U)&=(0;0,0). \\
  \gr_\phi(\rho_2) &=(-1/2;0,1) &
    \gr_\phi(\rho_{234}) &=(1;-1,1) \\
  \gr_\phi(\rho_3)&= (-1;0,0) &
    \gr_\phi(\rho_{341}) &= (1/2;0,-1) \\
  \gr_\phi(\rho_4) &= (3/2;-1,0) &
    \gr_\phi(\rho_{412}) &= (1;0,0)
\end{align*}

There is an anti-automorphism $r$ of $\smallTGroup$ defined by
$r(m;a,b)=(m;-b,-a)$. View $\smallTGroup$ as a set with two left
actions by $\smallTGroup$ where the first action is the obvious one
and the second is given by $g\cdot x = x r(g)$, i.e., is the obvious
right action re-interpreted as a left action via the anti-automorphism
$r$. The bimodule $\CFDDm(\Id)$ is graded by this set, where the
grading group of the action by $\MAlg$ uses the grading refinement
data $\psi_i$ and quotient map $\pi$ (corresponding to quotienting by
$(1;1,1,1,1)$) and the grading group of the action by $\MAlg^{U=1}$
uses the grading refinement data $\phi_i$ and the projection
corresponding to quotienting by $(-1;1,1,1,1)$. Give both generators
of $\CFDDm(\Id)$ grading $(0;0,0)$. Then, every term in $\delta^1$ on
$\CFDDm(\Id)$ has grading $(-1;0,0)$.  (For the specific grading
refinement data above, this is a simple, direct check, while in
general, assuming $\phi_i=r'(\psi_i^{-1})$, it is easy to see that
the grading refinement data drops out and the formula reduces to
$\grb(\rho)r'(\grb(\sigma))=(-1;0,0,0,0)$ whenever $\rho\otimes\sigma$
occurs in the differential on $\CFDDm(\Id)$.) So, if
we grade $\CFAm_{U=1}$ by 
$\smallTGroup$ using the grading refinement data $\phi_i$ and the
quotient by $(-1;1,1,1,1)$ then the box tensor product induces a
grading on $\CFDm(\Id)$, as desired.

Let $\widehat{S}_D(\HD,\spinc)$ denote the refined grading set for
$\CFDm(\HD,\spinc)$. Let $\widehat{S}_A(\HD)=\coprod_{\spinc}\widehat{S}_A(\HD,\spinc)$
and $\widehat{S}_D(\HD)=\coprod_{\spinc}\widehat{S}_D(\HD,\spinc)$.

The refined grading group is the same as the (genus $1$ case of the)
small grading group for the $\HFa$-case~\cite{LOT1}. Further, the
refined gradings we have constructed agree with the gradings in the
$\HFa$-case:
\begin{lemma}\label{lem:gr-same-HFa}
  The refined grading sets $\widehat{S}_A(\HD)$ and $\widehat{S}_D(\HD)$ agree
  with the refined grading sets in the $\HFa$-case~\cite[Section
  10.5]{LOT1}, and the induced gradings on the $U=0$-specializations
  of $\CFAm(\HD)$ and $\CFDm(\HD)$ agree with the gradings on
  $\CFAa(\HD)$ and $\CFDa(\HD)$.
\end{lemma}
\begin{proof}
  For $\CFA$, this is immediate from the definitions; for $\CFD$ it
  follows from the graded pairing theorem for bimodules in the
  $\HFa$-case~\cite[Section 7.1.1]{LOT2}.
\end{proof}

\subsection{Graded pairing theorems}\label{sec:gr-pairing}
Suppose we have bordered $3$-manifolds $Y_1$ and $Y_2$ with boundary
$T^2$.  Given grading refinement data as in
Section~\ref{sec:small-gr}, there are associated $\smallTGroup$-sets
$\widehat{S}_A(Y_1)$ and $\widehat{S}_D(Y_2)$. There is also a grading $\ZZ$-set
$S(Y)$ for the closed $3$-manifold $Y=Y_1\cup_{T^2}Y_2$. The
$\ZZ$-orbits in $S(Y)$ correspond to
$\SpinC$-structures on $Y$. The $\ZZ$-orbit corresponding to $\spinc(\x)$ is
$\ZZ/\divis(c_1(\spinc))$, though the identification depends on a
choice (e.g., of a base generator).

By Lemma~\ref{lem:gr-same-HFa}, the grading sets $\widehat{S}_A(Y)$ and
$\widehat{S}_D(Y)$ agree with the grading sets in the $\HFa$-case. The
graded pairing theorem in that case~\cite[Theorem 10.43]{LOT1} says
that there is a (canonical, and easy to construct) isomorphism of
$\ZZ$-sets
\[
  S(Y)\cong \widehat{S}_A(Y_1)\times_{\smallTGroup}\widehat{S}_D(Y_2)
\]
which is covered by a chain homotopy equivalence
\[
  \CFa(Y)\simeq \CFAa(Y_1)\DT\CFDa(Y_2).
\]
In the sequel we will prove the analogous result in the $\HFm$-case:
\begin{theorem}\label{thm:pairing-gradings}\cite{LOT:torus-pairing}
  For any choice of grading refinement data $\psi$ for $\MAlg$
  (Section~\ref{sec:small-gr}), the homotopy equivalence in the
  pairing theorem, Theorem~\ref{thm:pairing2}, respects the grading on
  the two sides by
  $S(Y)\cong \widehat{S}_A(Y_1)\times_{\smallTGroup}\widehat{S}_D(Y_2)$.
\end{theorem}

One can also extract the gradings on the tensor product using the
intermediate or big grading groups, as we did (for the large grading
group) in a previous paper~\cite{LOT4}. (See also~\cite[Remark
10.44]{LOT1}.) Before giving the theorem we discuss the general
setting.

Suppose we are given a (grading) group~$G$, and two subgroups $P_1$ and
$P_2$, determining left and right $G$-sets $S_1 = P_1 \backslash G$
and $S_2 = G / P_2$, respectively.  Given a larger
group $H$ containing $G$ as a normal subgroup, the double coset spaces
$P_1 \backslash G / P_2 = S_1 \times_G S_2$ and
$P_1 \backslash H / P_2$ are related as follows:

\begin{proposition}
  For $H$ a group, $G \subset H$ a normal subgroup, and $P_1, P_2
  \subset G$ a pair of subgroups of~$G$, the double coset space
  \[
    P_1 \backslash H / P_2
  \]
  has a natural map to the quotient group $H / G$, so that the
  subset of $P_1 \backslash H/P_2$ mapping to the identity in $H/G$ is
  in bijection with $P_1 \backslash G/P_2$. The bijection preserves the action by
  the center of~$G$ on both sets.
\end{proposition}
\begin{proof}
  For $h \in H$, we map $P_1 h P_2 \in P_1 \backslash H / P_2$ to
  $hG \in H/G$; this is well-defined since $P_1,P_2 \subset G$
  and $G$ is normal (so $P_1h \subset Gh = hG$).

  Those double cosets $P_1 h P_2$ mapping to the identity in $H/G$
  are those with $h \in G$. (This does not depend on the double coset
  representative $h$.) In this case, $P_1 h P_2 \subset G \subset H$,
  i.e., the double coset in $H$ is the same set as the double coset
  in~$G$. The statement about the action of the center is immediate.
\end{proof}

Here is the other version of the graded pairing theorem.
\begin{theorem} \cite{LOT:torus-pairing}\label{thm:pairing3}
  Given bordered $3$-manifolds $Y_1$ and $Y_2$ with torus boundary,
  for any generator $\x_1\otimes\x_2\in \CFAm(Y_1)\DT\CFDm(Y_2)$ the
  grading of $\x_1\otimes \x_2$ in
  $S'_A(Y_1)\times_{\bigGroup}S'_D(Y_2)$ lies over the identity in
  $\bigGroup/\smallTGroup$, and the chain homotopy equivalence
  \[
    \CFAm(Y_1)\DT\CFDm(Y_2)\simeq\CFmm(Y_1\cup_{T^2}Y_2)
  \]
  respects the induced grading by
  $\widehat{S}_A(Y_1)\times_{\smallTGroup}\widehat{S}_D(Y_2)\cong
  S(Y_1\cup_{T^2}Y_2)$.
\end{theorem}

In practice, one uses Theorem~\ref{thm:pairing3} as follows. One
computes the $\bigGroup$-set gradings on $\CFAm(Y_1)$ and
$\CFDm(Y_2)$. Then, for any pair of generators $\x_1\otimes\x_2$ and
$\y_1\otimes\y_2$ for $\CFmm(Y_1\cup_{T^2}Y_2)$, $\x_1\otimes\x_2$ and
$\y_1\otimes\y_2$ represent the same $\SpinC$-structure if and only if
$\grb(\x_1)\times\grb(\x_2)$ and $\grb(\y_1)\times\grb(\y_2)$ lie in the
same $\ZZ$-set, in which case their difference in $\ZZ$ is their
relative grading. (If the $\SpinC$-structure is not torsion, the
relative grading is only well-defined up to the divisibility of the
first Chern class, as usual.) That is, practical computations work
exactly as in the $\HFa$-case.

\begin{remark}
  The extra double cosets in $P_1\backslash H/P_2$ (corresponding to
  non-identity elements of $H/G$) might have different behavior than
  the original ones. Consider the inclusion
  $\smallTGroup \subset \interGroup$ and
  $P_1 = P_2 = P = \langle(0;1,0)\rangle$, with the double coset
  spaces as $\ZZ$-sets from the central action of $\lambda$. (This is
  relevant for computing $\HFa(S^1\times S^2)$ by pairing two solid
  tori.) Then the $\ZZ$ orbit of $P(0;0,n)P$ is isomorphic to
  $\ZZ/2n$. (This is easily verified by conjugating $(0;0,n)$ by
  $(0;1,0)$.) On the other hand, in $P \backslash G / P$, the $\ZZ$
  orbit of $P(\OneHalf;\OneHalf,n+\OneHalf)P$ is isomorphic to
  $\ZZ/(2n+1)$.
\end{remark}


\definecolor{ycolor}{rgb}{0,0.3,0}
\definecolor{xcolor}{rgb}{0,0,0}
\section{Knot complements in \textalt{$S^3$}{the three sphere}}\label{sec:knot}
Recall that, given a knot in $S^3$, there is a formula computing
$\CFDa(S^3\setminus\nbd(K))$ (for any integer framing) from
$\CFKm(S^3,K)$~\cite[Theorem 11.27]{LOT1}. In this section,
we show that, in many cases, one can deduce
$\CFDm(S^3\setminus\nbd(K))$ from $\CFDa(S^3\setminus\nbd(K))$. We
then use these computations to give an example of using $\CFDm$ to
compute $\HFm$ of a Dehn surgeries and $\CFKm$ of a cable,
assuming the pairing theorem (deferred to~\cite{LOT:torus-pairing}).

\subsection{Constructing \textalt{$\CFDm$}{CFD minus} from \textalt{$\CFK$}{CFK}: examples}
\label{sec:CFD-eg}
We start with two of the examples studied in the $\HFa$ case~\cite[Figure 11.5]{LOT1}. 
First, consider the left-handed $(3,4)$ torus knot,
$T_{-3,4}$. Its knot Floer
complex has five generators $x_1,\dots,x_5$ with Alexander gradings
\[
A(x_1)=3 \qquad A(x_2)=2 \qquad A(x_3)=0 \qquad A(x_4)=-2 \qquad A(x_5)=-3
\]
and differentials
\begin{align*}
  \bdy(x_1)&=x_2 & \bdy(x_3)&=U^2x_2 + x_4\\
  \bdy(x_5) &=U x_4 & \bdy(x_2)&=\bdy(x_4)=0.
\end{align*}
As is customary, we encode this graphically as
\[
\mathcenter{\begin{tikzpicture}
  \node at (0,0) (a1) {$\bullet$};
  \node at (0,-1) (a2) {$\bullet$};
  \node at (2,-1) (a3) {$\bullet$};
  \node at (2,-3) (a4) {$\bullet$};
  \node at (3,-3) (a5) {$\bullet$};
  \draw[->] (a1) to (a2);
  \draw[->] (a3) to (a2);
  \draw[->] (a3) to (a4);
  \draw[->] (a5) to (a4);
\end{tikzpicture}}
\qquad\text{ or }\qquad
\mathcenter{\begin{tikzpicture}
  \node at (0,0) (a1) {$x_1$};
  \node at (0,-1) (a2) {$x_2$};
  \node at (2,-1) (a3) {$x_3$};
  \node at (2,-3) (a4) {$x_4$};
  \node at (3,-3) (a5) {$x_5$};
  \draw[->] (a1) to (a2);
  \draw[->] (a3) to (a2);
  \draw[->] (a3) to (a4);
  \draw[->] (a5) to (a4);
\end{tikzpicture}}
\]
where a length $n$ vertical arrow from $x$ to $y$ indicates that $y$
is a term in $\bdy(x)$ and $A(x)-A(y)=n$, and a length $n$ horizontal
arrow from $x$ to $y$ indicates that $U^ny$ is a term in
$\bdy(x)$ and $A(x)-A(y)=-n$. All arrows in the diagram have length $1$ or $2$.

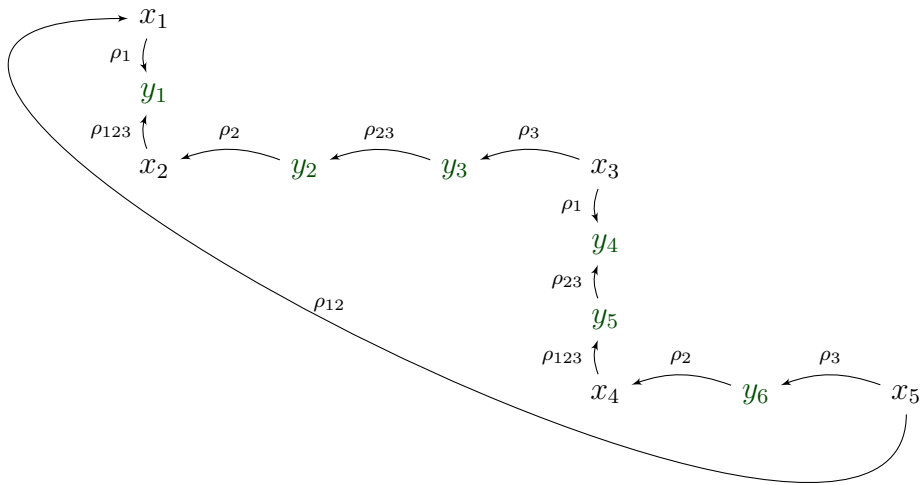
\begin{figure}
  \centering
  \begin{tikzpicture}
    \useasboundingbox (-.5,-6.5) rectangle (8.5,0.5);
    \node[text=xcolor] at (0,0) (a1) {$x_1$};
    \node[text=xcolor] at (0,-2) (a2) {$x_2$};
    \node[text=xcolor] at (6,-2) (a3) {$x_3$};
    \node[text=xcolor] at (6,-5) (a4) {$x_4$};
    \node[text=xcolor] at (10,-5) (a5) {$x_5$};
    \node[text=ycolor] at (0,-1) (b1) {$y_1$};
    \node[text=ycolor] at (2,-2) (b2) {$y_2$};
    \node[text=ycolor] at (4,-2) (b3) {$y_3$};
    \node[text=ycolor] at (6,-3) (b4) {$y_4$};
    \node[text=ycolor] at (6,-4) (b5) {$y_5$};
    \node[text=ycolor] at (8,-5) (b6) {$y_6$};
    \draw[->, bend right=20] (a1) to node[left]{\lab{\rho_1}} (b1);
    \draw[->, bend left=20] (a2) to node[left]{\lab{\rho_{123}}} (b1);
    \draw[->, bend right=20] (a3) to node[above]{\lab{\rho_3}} (b3);  
    \draw[->, bend right=20] (b3) to node[above]{\lab{\rho_{23}}} (b2);
    \draw[->, bend right=20](b2) to node[above]{\lab{\rho_{2}}} (a2);
    \draw[->, bend right=20] (a3) to node[left]{\lab{\rho_1}} (b4);
    \draw[->, bend left=20] (b5) to node[left]{\lab{\rho_{23}}} (b4);
    \draw[->, bend left=20] (a4) to node[left]{\lab{\rho_{123}}} (b5);
    \draw[->, bend right=20] (a5) to node[above]{\lab{\rho_3}} (b6);
    \draw[->, bend right=20] (b6) to node[above]{\lab{\rho_2}} (a4);
    \draw[->, bend left=60] (a5) .. controls (10,-9) and (-7,0) .. node[above]{\lab{\rho_{12}}} (a1);
  \end{tikzpicture}
  \caption[$\CFDa(S^3\setminus T_{-3,4},-6)$.]{\textbf{The
      complex $\CFDa(S^3\setminus T_{-3,4},-6)$}. The generators $x_i$
    satisfy $\iota_0x_i=x_i$, the generators $y_i$ satisfy
    $\iota_1y_i=y_i$, and an arrow from $x$ to $y$ labeled by an
    algebra element $a$ indicates that $a\otimes y$ is a term in
    $\delta^1(x)$.  This computation is from~\cite{LOT1}.
    \label{fig:CFDa-T34}}
\end{figure}

The type $D$ module associated to the
complement of this knot, with framing $-6$, $\CFDa(S^3\setminus
T_{-3,4},-6)$, is homotopy equivalent to the complex shown in
Figure~\ref{fig:CFDa-T34} \cite[Figure 11.5]{LOT1}. 
By Proposition~\ref{prop:extend-CFDa}, $\CFDm(S^3\setminus
T_{-3,4},-6)$ can be obtained from this complex by adding terms to the
differential corresponding to algebra elements $U^n\rho$ where either
$n>0$ or where $\rho$ contains $\rho_4$ (i.e.,
$\rho=\rho'\rho_4\rho''$ for some $\rho'$ and $\rho''$), or both. The
result must be a type $D$ structure over the weighted algebra
$\MAlg$. In particular, for any generator $x$, there is a term
\[
  \mu^1_0\otimes
  x=(\rho_{1234}+\rho_{2341}+\rho_{3412}+\rho_{4123})\otimes x
\]
in the structure equation. (For any given $x$ only two of
these terms are non-zero, because of the idempotents.) To cancel these
terms in the structure relation, many other terms are forced in
$\delta^1$. For instance, the only way to cancel the term
$\rho_{1234}$ from $x_1$ to itself is for there to be a differential
from $y_1$ to $x_1$ with algebra element $\rho_{234}$. The only way to
cancel the term $\rho_{3412}$ from $x_1$ to itself is for there to be
a differential from $x_1$ to $x_5$ with algebra element
$\rho_{34}$. Continuing with this logic gives the complex
$\CFDm(S^3\setminus T_{-3,4},-6)$ shown in Figure~\ref{fig:CFDm-T34}.
\begin{figure}
  \centering
    \begin{tikzpicture}
        \useasboundingbox (-.5,-6.5) rectangle (9.5,0.5);
        \node[text=xcolor] at (0,0) (a1) {$x_1$};
        \node[text=xcolor] at (0,-2) (a2) {$x_2$};
        \node[text=xcolor] at (6,-2) (a3) {$x_3$};
        \node[text=xcolor] at (6,-5) (a4) {$x_4$};
        \node[text=xcolor] at (10,-5) (a5) {$x_5$};
        \node[text=ycolor] at (0,-1) (b1) {$y_1$};
        \node[text=ycolor] at (2,-2) (b2) {$y_2$};
        \node[text=ycolor] at (4,-2) (b3) {$y_3$};
        \node[text=ycolor] at (6,-3) (b4) {$y_4$};
        \node[text=ycolor] at (6,-4) (b5) {$y_5$};
        \node[text=ycolor] at (8,-5) (b6) {$y_6$};
        \draw[->, bend right=20] (a1) to node[left]{\lab{\rho_1}} (b1);
        \draw[->, bend right=20] (b1) to node[right]{\lab{\rho_{234}}} (a1); 
        \draw[->, bend left=20] (a2) to node[left]{\lab{\rho_{123}}} (b1);
        \draw[->, bend left=20] (b1) to node[right]{\lab{\rho_{4}}} (a2); 
        \draw[->, bend right=20] (a3) to node[above]{\lab{\rho_3}} (b3);
        \draw[->, bend right=20] (b3) to node[below]{\lab{\rho_{412}}} (a3); 
        \draw[->, bend right=20] (b3) to node[above]{\lab{\rho_{23}}} (b2);
        \draw[->, bend right=20] (b2) to node[below]{\lab{\rho_{41}}} (b3); 
        \draw[->, bend right=20] (b2) to node[above]{\lab{\rho_{2}}} (a2);
        \draw[->, bend right=20] (a2) to node[below]{\lab{\rho_{341}}} (b2); 
        \draw[->, bend right=20] (a3) to node[left]{\lab{\rho_1}} (b4);
        \draw[->, bend right=20] (b4) to node[right]{\lab{\rho_{234}}} (a3); 
        \draw[->, bend left=20] (b5) to node[left]{\lab{\rho_{23}}} (b4);
        \draw[->, bend left=20] (b4) to node[right]{\lab{\rho_{41}}} (b5); 
        \draw[->, bend left=20] (a4) to node[left]{\lab{\rho_{123}}} (b5);
        \draw[->, bend left=20] (b5) to node[right]{\lab{\rho_{4}}} (a4); 
        \draw[->, bend right=20] (a5) to node[above]{\lab{\rho_3}} (b6);
        \draw[->, bend right=20] (b6) to node[below]{\lab{\rho_{412}}} (a5); 
        \draw[->, bend right=20] (b6) to node[above]{\lab{\rho_2}} (a4);
        \draw[->, bend right=20] (a4) to node[below]{\lab{\rho_{341}}} (b6); 
        \draw[->, bend left=60] (a5) .. controls (10,-9) and (-7,0) .. node[above]{\lab{\rho_{12}}} (a1);
        \draw[->, bend left=30] (a1) .. controls (6,-1) .. node[above]{\lab{\rho_{34}}}  (a5); 
      \end{tikzpicture}
      \qquad
      \begin{tikzpicture}[every node/.style={inner sep=0pt}, xscale=1.5, yscale=1.5]
      \draw[thin, gray] (-1,0) grid (1,6);
    \node[text=xcolor] at (0,0) (x1) {$x_1$};
    \node[text=ycolor] at (-1/2,1/2) (y1) {$y_1$};
    \node[text=xcolor] at (0,1) (x2) {$x_2$};
    \node[text=ycolor] at (1/2,3/2) (y2) {$y_2$};
    \node[text=ycolor] at (1/2,5/2) (y3) {$y_3$};
    \node[text=xcolor] at (0,3) (x3) {$x_3$};
    \node[text=ycolor] at (-1/2,7/2) (y4) {$y_4$};
    \node[text=ycolor] at (-1/2,9/2) (y5) {$y_5$};
    \node[text=xcolor] at (0,5) (x4) {$x_4$};
    \node[text=ycolor] at (1/2,11/2) (y6) {$y_6$};
    \node[text=xcolor] at (0,6) (x5) {$x_5$};
    \node[text=xcolor] at (-1,6) (x1p) {$x_1$};
    \draw[->] (x1) to node[below,sloped]{\lab{\rho_1}} (y1);
    \draw[->] (y1) to node[below,sloped]{\lab{\rho_4}} (x2);
    \draw[->] (x2) to node[below,sloped]{\lab{\rho_{341}}} (y2);
    \draw[->] (y2) to node[left]{\lab{\rho_{41}}} (y3);
    \draw[->] (y3) to node[below,sloped]{\lab{\rho_{412}}} (x3);
    \draw[->] (x3) to node[below,sloped]{\lab{\rho_{1}}} (y4);
    \draw[->] (y4) to node[right]{\lab{\rho_{41}}} (y5);
    \draw[->] (y5) to node[below,sloped]{\lab{\rho_{4}}} (x4);
    \draw[->] (x4) to node[below,sloped]{\lab{\rho_{341}}} (y6);
    \draw[->] (y6) to node[below,sloped]{\lab{\rho_{412}}} (x5);
    \draw[->] (x5) to node[below]{\lab{\rho_{12}}} (x1p);
     \draw[dashed] (x1p) .. controls (-1,2) and (-1,0) .. node[right]{\lab{=}} (x1);
  \end{tikzpicture}
  \caption[$\CFDm(S^3\setminus T_{-3,4},-6)$ and its
  relative gradings]{\textbf{The complex
      $\CFDm(S^3\setminus T_{-3,4},-6)$ (left) and the relative
      gradings on it (right).} On the right, the relative
    $\SpinC$-gradings $(a,b)$ are shown, with $a$ horizontal and $b$
    vertical. The gray squares are $1\times 1$, and the first and last dots in the
chain both correspond to the same generator.}
  \label{fig:CFDm-T34}
\end{figure}
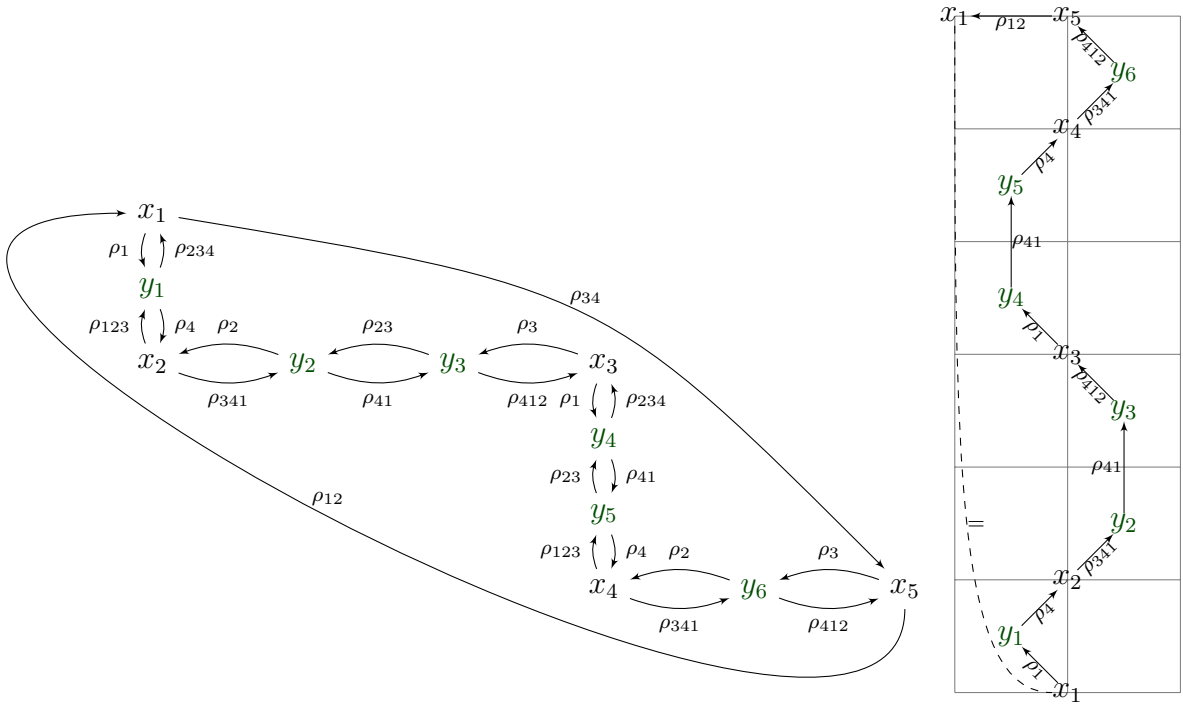

It remains to see there are no other terms in $\delta^1$. For this, we
use the relative gradings. We will use the intermediate grading (from
Section~\ref{sec:int-gr}).  The intermediate grading of an algebra
element is a triple $(m;a,b)$. We first compute the relative $\SpinC$
gradings (the pairs $(a,b)$), by starting at the top-left generator
and following the arrows.
So, for instance, if $\delta^1(x)=\rho_1\otimes y$ then
$\gr(y)=(-1/2;-1/2,1/2)\gr(x)$.  The resulting relative intermediate
$\SpinC$ gradings $(a,b)$ are also shown in Figure~\ref{fig:CFDm-T34}.
The grading ambiguity (periodic domain) corresponds to the loop, so
has $\SpinC$-grading $(-1,6)$.

Since every element of the form $U^m\rho_I$ has intermediate grading
$(m;a,b)$ with $|a|+|b|\in\{0,1\}$ (see Formula~\eqref{eq:gr-comps}),
any additional differential must
go between two generators which lie at points $(x,y)$ and $(z,w)$
satisfying $(x,y)-(z,w)=n\cdot(-1,6)+(\epsilon_1,\epsilon_2)$ where
$|\epsilon_1|+|\epsilon_2|\leq 1$. So, inspecting the grading diagram in
Figure~\ref{fig:CFDm-T34}, the only possible additional terms are
from a generator to itself, or between generators where there is
already an arrow in the complex in Figure~\ref{fig:CFDm-T34}. Additional
differentials $\delta^1(x)=U^m\rho_I\otimes y$ where the total length $|U^m\rho_I|>3$
are now ruled out by considering the Maslov component of the
gradings. Additional differentials with algebra element in $\Alg(T^2)$
were already ruled out, by Proposition~\ref{prop:extend-CFDa}. This
leaves chords with length $\leq 3$ containing $\rho_4$. In each case,
the grading would have to be the same as the grading of the
term which actually does occur; so, for instance, if the occurring term is $\rho_{4}$ then the
new term would have to have grading $(-3/2;-1/2,-1/2)$ hence be either
$\rho_4$ (excluded, by the structure equation) or $\rho_{341}$
(excluded by idempotents). In fact, inspecting the gradings in
Section~\ref{sec:int-gr}, the only pairs of chords of length $\leq 3$ with the
same gradings have different left (and right) idempotents, so if one
occurs the other does not. This completes the proof that
Figure~\ref{fig:CFDm-T34} gives the correct complex.

A similar strategy works for more complicated knots. For instance,
consider the $(-2,1)$ cable of the left handed trefoil, which
Hedden~\cite{HeddenThesis} showed has knot Floer complex
\begin{equation}\label{eq:Hedden-comp}
\mathcenter{\begin{tikzpicture}
  \node at (0,0) (a1) {$\bullet$};
  \node at (1,0) (a2) {$\bullet$};
  \node at (2,0) (a3) {$\bullet$};
  \node at (0,-1) (a4) {$\bullet$};
  \node at (2,-1) (a5) {$\bullet$};
  \node at (1,-2) (a6) {$\bullet$};
  \node at (2,-2) (a7) {$\bullet$};
  \draw[->] (a1) to (a4);
  \draw[->] (a2) to (a4);
  \draw[->] (a2) to (a6);
  \draw[->] (a3) to (a2);
  \draw[->] (a3) to (a5);
  \draw[->] (a5) to (a4);
  \draw[->] (a5) to (a6);
  \draw[->] (a7) to (a6);
\end{tikzpicture}.}
\end{equation}
(Again, all arrows have length $1$ or $2$. We verify this computation
using techniques from this paper in Section~\ref{sec:surgery}.)
The complex $\CFDa(S^3\setminus K)$
with framing $-4$ is shown in Figure~\ref{fig:CFDa-Hedden} \cite[Figure 11.5]{LOT1}.

\begin{figure}
  \centering
  \begin{tikzpicture}
    \useasboundingbox (-0.5,-8.5) rectangle (8.5,0.5);
    \node[text=xcolor] at (0,0) (a1) {$x_1$};
    \node[text=xcolor] at (4,0) (a2) {$x_5$};
    \node[text=xcolor] at (8,0) (a3) {$x_4$};
    \node[text=xcolor] at (0,-4) (a4) {$x_2$};
    \node[text=xcolor] at (8,-4) (a5) {$x_3$};
    \node[text=xcolor] at (4,-8) (a6) {$x_6$};
    \node[text=xcolor] at (8,-8) (a7) {$x_7$};
    \node[text=ycolor] at (6,0) (b1) {$y_5$};
    \node[text=ycolor] at (0,-2) (b2) {$y_1$};
    \node[text=ycolor] at (4,-2) (b3) {$y_6$};
    \node[text=ycolor] at (8,-2) (b4) {$y_4$};
    \node[text=ycolor] at (2,-4) (b5) {$y_2$};
    \node[text=ycolor] at (6,-4) (b6) {$y_3$};
    \node[text=ycolor] at (4,-6) (b7) {$y_7$};
    \node[text=ycolor] at (6,-8) (b8) {$y_8$};  
    \draw[->, bend right=20] (a1) to node[left]{\lab{\rho_1}} (b2);
    \draw[->, bend right=20] (a2) to node[left]{\lab{\rho_1}} (b3);
    \draw[->, bend right=20] (b1) to node[above]{\lab{\rho_2}} (a2);
    \draw[->, bend right=20] (a3) to node[above]{\lab{\rho_3}} (b1);
    \draw[->, bend right=20] (a3) to node[left]{\lab{\rho_1}} (b4);
    \draw[->, bend left=20] (a4) to node[left]{\lab{\rho_{123}}} (b2);
    \draw[->, bend right=20] (b5) to node[above]{\lab{\rho_{2}}} (a4);
    \draw[->, bend right=20] (b6) to node[above]{\lab{\rho_{23}}} (b5);
    \draw[->, bend right=20] (a5) to node[above]{\lab{\rho_{3}}} (b6);
    \draw[->, bend left=20] (a5) to node[left]{\lab{\rho_{123}}} (b4);
    \draw[->, bend left=20] (b7) to node[left]{\lab{\rho_{23}}} (b3);
    \draw[->, bend left=20] (a6) to node[left]{\lab{\rho_{123}}} (b7);
    \draw[->, bend right=20] (b8) to node[above]{\lab{\rho_2}} (a6);
    \draw[->, bend right=20] (a7) to node[above]{\lab{\rho_3}} (b8);
    \draw[->] (a7) .. controls (5,-12) and (-7,0).. node[below]{\lab{\rho_{12}}} (a1);
  \end{tikzpicture}  
  \caption[$\CFDa$ of the complement of a cable of a
  trefoil]{\textbf{The complex $\CFDa(S^3\setminus K,-4)$ where $K$ is
      the (-2,1)-cable of the left-handed trefoil.} Conventions are as
    in Figure~\ref{fig:CFDa-T34}. This computation is from~\cite{LOT1}.}
  \label{fig:CFDa-Hedden}
\end{figure}
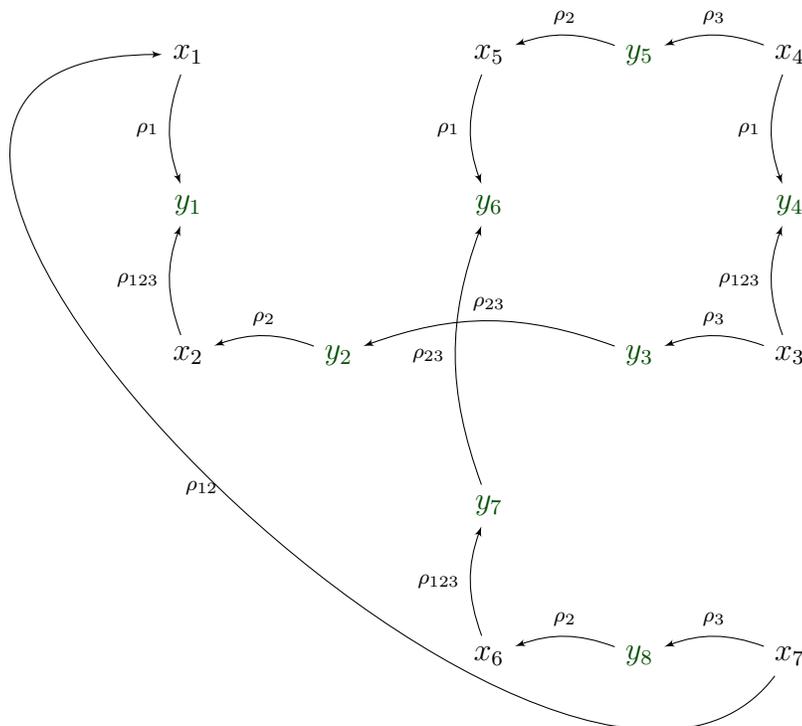

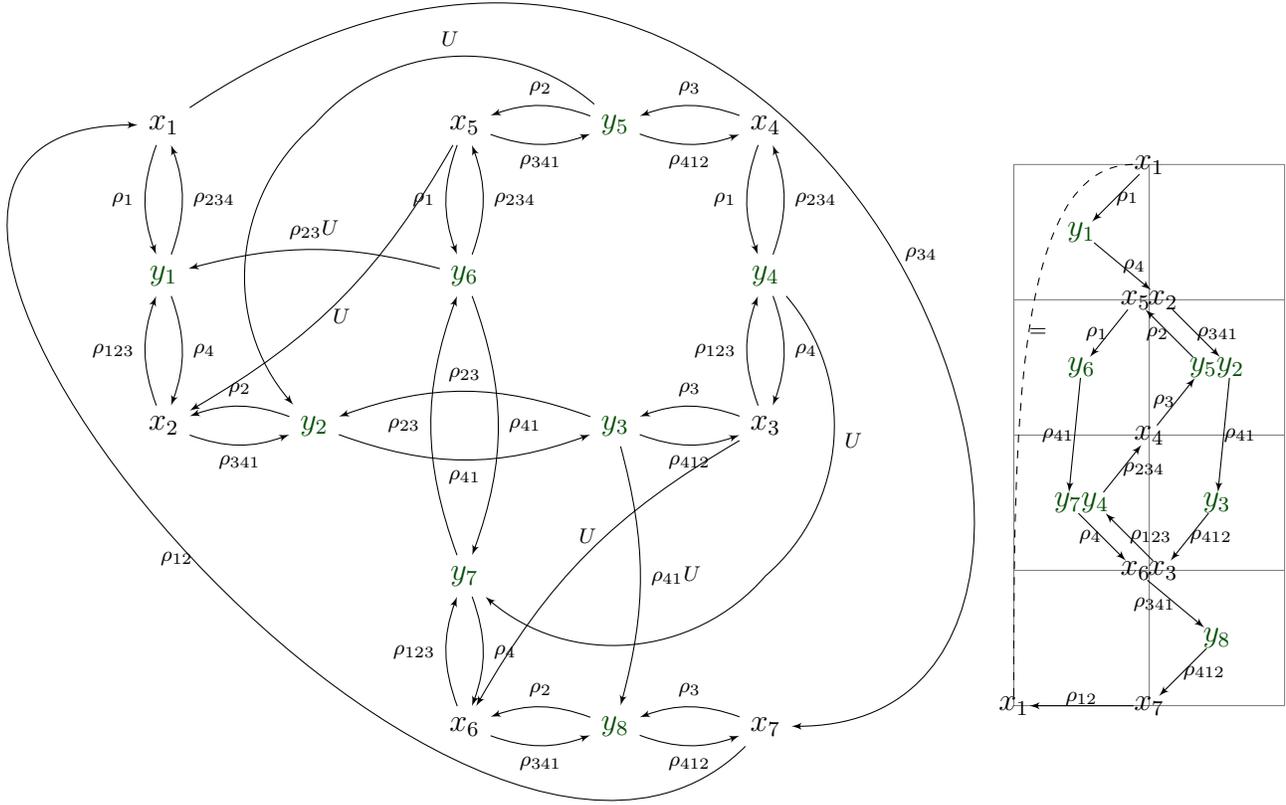
\begin{figure}
  \centering
  \begin{tikzpicture}
    \useasboundingbox (-1.5,-9.5) rectangle (10,3.5);
    \node[text=xcolor] at (0,0) (a1) {$x_1$};
    \node[text=xcolor] at (4,0) (a2) {$x_5$};
    \node[text=xcolor] at (8,0) (a3) {$x_4$};
    \node[text=xcolor] at (0,-4) (a4) {$x_2$};
    \node[text=xcolor] at (8,-4) (a5) {$x_3$};
    \node[text=xcolor] at (4,-8) (a6) {$x_6$};
    \node[text=xcolor] at (8,-8) (a7) {$x_7$};
    \node[text=ycolor] at (6,0) (b1) {$y_5$};
    \node[text=ycolor] at (0,-2) (b2) {$y_1$};
    \node[text=ycolor] at (4,-2) (b3) {$y_6$};
    \node[text=ycolor] at (8,-2) (b4) {$y_4$};
    \node[text=ycolor] at (2,-4) (b5) {$y_2$};
    \node[text=ycolor] at (6,-4) (b6) {$y_3$};
    \node[text=ycolor] at (4,-6) (b7) {$y_7$};
    \node[text=ycolor] at (6,-8) (b8) {$y_8$};  
    \draw[->, bend right=20] (a1) to node[left]{\lab{\rho_1}} (b2);
    \draw[->, bend right=20] (b2) to node[right]{\lab{\rho_{234}}} (a1); 
    \draw[->, bend right=20] (a2) to node[left]{\lab{\rho_1}} (b3); 
    \draw[->, bend right=20] (b3) to node[right]{\lab{\rho_{234}}} (a2); 
    \draw[->, bend right=20] (b1) to node[above]{\lab{\rho_2}} (a2);
    \draw[->, bend right=20] (a2) to node[below]{\lab{\rho_{341}}} (b1); 
    \draw[->, bend right=20] (a3) to node[above]{\lab{\rho_3}} (b1);
    \draw[->, bend right=20] (b1) to node[below]{\lab{\rho_{412}}} (a3); 
    \draw[->, bend right=20] (a3) to node[left]{\lab{\rho_1}} (b4);
    \draw[->, bend right=20] (b4) to node[right]{\lab{\rho_{234}}} (a3); 
    \draw[->, bend left=20] (a4) to node[left]{\lab{\rho_{123}}} (b2);
    \draw[->, bend left=20] (b2) to node[right]{\lab{\rho_{4}}} (a4); 
    \draw[->, bend right=20] (b5) to node[above]{\lab{\rho_{2}}} (a4);
    \draw[->, bend right=20] (a4) to node[below]{\lab{\rho_{341}}} (b5); 
    \draw[->, bend right=20] (b6) to node[above]{\lab{\rho_{23}}} (b5);
    \draw[->, bend right=20] (b5) to node[below]{\lab{\rho_{41}}} (b6); 
    \draw[->, bend right=20] (a5) to node[above]{\lab{\rho_{3}}} (b6);
    \draw[->, bend right=20] (b6) to node[below]{\lab{\rho_{412}}} (a5); 
    \draw[->, bend left=20] (a5) to node[left]{\lab{\rho_{123}}} (b4);
    \draw[->, bend left=20] (b4) to node[right]{\lab{\rho_{4}}} (a5); 
    \draw[->, bend left=20] (b7) to node[left]{\lab{\rho_{23}}} (b3);
    \draw[->, bend left=20] (b3) to node[right]{\lab{\rho_{41}}} (b7); 
    \draw[->, bend left=20] (a6) to node[left]{\lab{\rho_{123}}} (b7);
    \draw[->, bend left=20] (b7) to node[right]{\lab{\rho_{4}}} (a6); 
    \draw[->, bend right=20] (b8) to node[above]{\lab{\rho_2}} (a6);
    \draw[->, bend right=20] (a6) to node[below]{\lab{\rho_{341}}} (b8); 
    \draw[->, bend right=20] (a7) to node[above]{\lab{\rho_3}} (b8);
    \draw[->, bend right=20] (b8) to node[below]{\lab{\rho_{412}}} (a7); 
    \draw[->, bend right=45] (b1) to node[above]{\lab{U}} (2,0) to (b5); 
    \draw[->, bend left=15] (a2) to node[below]{\lab{U}} (a4); 
    \draw[->, bend right=15] (b3) to node[above]{\lab{\rho_{23}U}} (b2);
    \draw[->, bend right=15] (a5) to node[above]{\lab{U}} (a6);
    \draw[->, bend left=15] (b6) to node[right]{\lab{\rho_{41}U}} (b8);
    \draw[->, bend left=45] (b4) to node[right]{\lab{U}} (8,-6) to (b7);
    \draw[->] (a7) .. controls (4,-12) and (-6,0).. node[below]{\lab{\rho_{12}}} (a1);
    \draw[->] (a1) .. controls (9,6) and (14,-8)..  node[right]{\lab{\rho_{34}}} (a7); 
  \end{tikzpicture}
  \qquad
  \raisebox{4em}{
  \begin{tikzpicture}[every node/.style={inner sep=0pt}, xscale=1.8, yscale=-1.8]
    \draw[thin, gray] (-1,0) grid (1,4);
    \node[text=xcolor] at (0,0) (x1) {$x_1$};
    \node[text=ycolor] at (-.5,.5) (y1) {$y_1$};
    \node[text=xcolor] at (.1,1) (x2) {$x_2$};
    \node[text=ycolor] at (.6,1.5) (y2) {$y_2$};
    \node[text=ycolor] at (.5,2.5) (y3) {$y_3$};
    \node[text=xcolor] at (.1,3) (x3) {$x_3$};
    \node[text=ycolor] at (-.4,2.5) (y4) {$y_4$};
    \node[text=xcolor] at (0,2) (x4) {$x_4$};
    \node[text=ycolor] at (.4,1.5) (y5) {$y_5$};
    \node[text=xcolor] at (-.1,1) (x5) {$x_5$};
    \node[text=ycolor] at (-.5,1.5) (y6) {$y_6$};
    \node[text=ycolor] at (-.6,2.5) (y7) {$y_7$};
    \node[text=xcolor] at (-.1,3) (x6) {$x_6$};
    \node[text=ycolor] at (.5,3.5) (y8) {$y_8$};
    \node[text=xcolor] at (0,4) (x7) {$x_7$};
    \node[text=xcolor] at (-1,4) (x1p) {$x_1$};
    \draw[->] (x1) to node[right]{\lab{\rho_1}} (y1);
    \draw[->] (y1) to node[right]{\lab{\rho_4}} (x2);
    \draw[->] (x2) to node[right]{\lab{\rho_{341}}} (y2);
    \draw[->] (y2) to node[right]{\lab{\rho_{41}}} (y3);
    \draw[->] (y3) to node[right]{\lab{\rho_{412}}} (x3);
    \draw[->] (x3) to node[right]{\lab{\rho_{123}}} (y4);
    \draw[->] (y4) to node[right]{\lab{\rho_{234}}} (x4);
    \draw[->] (x4) to node[left]{\lab{\rho_{3}}} (y5);
    \draw[->] (y5) to node[left]{\lab{\rho_{2}}} (x5);
    \draw[->] (x5) to node[left]{\lab{\rho_{1}}} (y6);
    \draw[->] (y6) to node[left]{\lab{\rho_{41}}} (y7);
    \draw[->] (y7) to node[left]{\lab{\rho_{4}}} (x6);
    \draw[->] (x6) to node[left]{\lab{\rho_{341}}} (y8);
    \draw[->] (y8) to node[right]{\lab{\rho_{412}}} (x7);
    \draw[->] (x7) to node[above]{\lab{\rho_{12}}} (x1p);
    \draw[dashed] (x1p) .. controls (-1,2) and (-1,0) .. node[right]{\lab{=}} (x1);
  \end{tikzpicture}}
  \caption[$\CFDm$ of the complement of a cable of the trefoil]{\textbf{The complex $\CFDm(S^3\setminus K,-4)$ where $K$ is
      the (-2,1)-cable of the left-handed trefoil, and relative
      gradings on it.} Conventions are as in Figure~\ref{fig:CFDm-T34}.}
  \label{fig:CFDm-Hedden}
\end{figure}
The complex $\CFDm(S^3\setminus K)$ with framing $-4$ is shown in
Figure~\ref{fig:CFDm-Hedden}.
The proof is similar to the previous example. The type $D$ structure
relation and Proposition~\ref{prop:extend-CFDa} force the indicated
terms to occur. For instance, the term $U\otimes y_2$ in
$\delta^1(y_5)$ is needed to cancel the sequence of differentials
\[
  y_5\stackrel{\rho_{412}}{\lra}x_4\stackrel{\rho_1}{\lra}y_4\stackrel{\rho_4}{\lra}x_3\stackrel{\rho_3}{\lra}y_3
\]
combined with the $\Ainf$ operation
$\mu_4^0(\rho_{412},\rho_1,\rho_4,\rho_3)=U\rho_{41}$.

To see that there are no other terms, we again appeal to the relative
gradings. If we declare that $x_1$ has intermediate grading $(0;0,0)$
then the gradings of the other generators are given by
\begin{center}
  \begin{tabular}{ccccccc}
    \toprule
    $x_1$&$x_2$&$x_3$&$x_4$&$x_5$&$x_6$&$x_7$\\
    $(0;0,0)$ & $(1/2;0,1)$ & $(1/2;0,3)$ & $(0;0,2)$ & $(-1/2;0,1)$ & $(3/2;0,3)$ & $(2;0,4)$\\
  \end{tabular}
  \begin{tabular}{cccc}
    $y_1$&$y_2$&$y_3$&$y_4$\\ 
    $(-1/2;-1/2,1/2)$ & $(3/2;1/2,3/2)$ & $(3/2;1/2,5/2)$ & $(-3/2;-1/2,5/2)$\\
    $y_5$&$y_6$&$y_7$&$y_8$\\
     $(1/2;1/2,3/2)$ & $(-3/2,-1/2,3/2)$ & $(-1/2;-1/2,5/2)$ & $(7/2;1/2,7/2)$\\
    \bottomrule
  \end{tabular}
\end{center}
The $\SpinC$-gradings are shown in Figure~\ref{fig:CFDm-Hedden}.  
The grading ambiguity is $(-5/2;-1,4)$.

The relative Maslov grading implies that there are no other
differentials between generators already connected by a differential,
except perhaps a term like $\delta^1(y_5)=\rho_{4123}y_2$ between
generators connected by a differential with coefficient $U$ (or the
analogue for $\rho_{41}U$ and $\rho_{23}U$). The relative
$\SpinC$-grading implies that the only possible additional
differentials are between terms like $x_1$ and $x_2$ or $y_5$ and
$y_6$ or $x_4$ and $y_3$ whose $\SpinC$-grading difference has length
$1$ in the Manhattan metric. Algebra elements with $\SpinC$-grading
$(0,1)$ have left and
right idempotent $\iota_2$, so cannot be coefficients of differentials
from $x_i$ to $x_j$. Similarly, algebra elements with $\SpinC$-grading
$(1,0)$ cannot be coefficients of differentials from $y_i$ to
$y_j$. Computing the $\SpinC$-gradings in the remaining cases and
comparing them with the gradings of the algebra elements from
Section~\ref{sec:int-gr} (and keeping in mind the idempotents), the
other possibilities are:
\begin{itemize}
\item A term $\rho_2\otimes x_4$ in $\delta^1(y_3)$ or a term
  $\rho_ {123}\otimes y_6$ in $\delta^1(x_4)$. These are prohibited by
  Proposition~\ref{prop:extend-CFDa}.
\item A term $\rho_4\otimes x_4$ in $\delta^1(y_6)$. This violates the
  structure equation (from $y_6$ to $\rho_{41}\otimes y_4$).
\item A term $\rho_{341}\otimes y_3$ in $\delta^1(x_4)$.
\item The terms mentioned above obtained by replacing a copy of
  $U$ by a length-4 chord.
\end{itemize}
(This analysis can easily be---and in fact has been---checked by
computer.) These last two cases are not independent. By the structure equation,
the term $\rho_{341}\otimes y_3$ occurs in $\delta^1(x_4)$ if and only
if $\rho_{4123}\otimes y_2$ occurs in $\delta^1(y_5)$. These terms can
be gauged away by the change of basis replacing $x_4$ by
$x_4+\rho_3\otimes y_2$. Similarly, for the remaining terms:
\begin{center}
  \begin{tabular}{ll}
    \toprule
    Terms that occur together& Gauge away via\\
    \midrule
    $x_5\mapsto \rho_{3412}\otimes x_2$, $y_5\mapsto \rho_{2341}\otimes y_2$  & $y'_5=y_5+\rho_2\otimes x_2$\\
    $x_5\mapsto \rho_{1234}\otimes x_2$, $y_6\mapsto \rho_{234123}\otimes y_1$ & $y'_6=y_6+\rho_{234}\otimes x_2$\\
    $x_3\mapsto \rho_{3412}\otimes x_6$, $y_3\mapsto \rho_{412341}\otimes y_8$ & $y_3'=y_3+\rho_{412}\otimes x_6$\\
    $x_3\mapsto \rho_{1234}\otimes x_6$, $y_4\mapsto \rho_{4123}\otimes y_7$ & $y'_4=y_4+\rho_4\otimes x_6$\\
    \bottomrule
  \end{tabular}
\end{center}
Finally, the term $\rho_{2341}\otimes y_7$ in $\delta^1(y_4)$ violates the structure equation, so does not occur. This completes the verification of Figure~\ref{fig:CFDm-Hedden}.


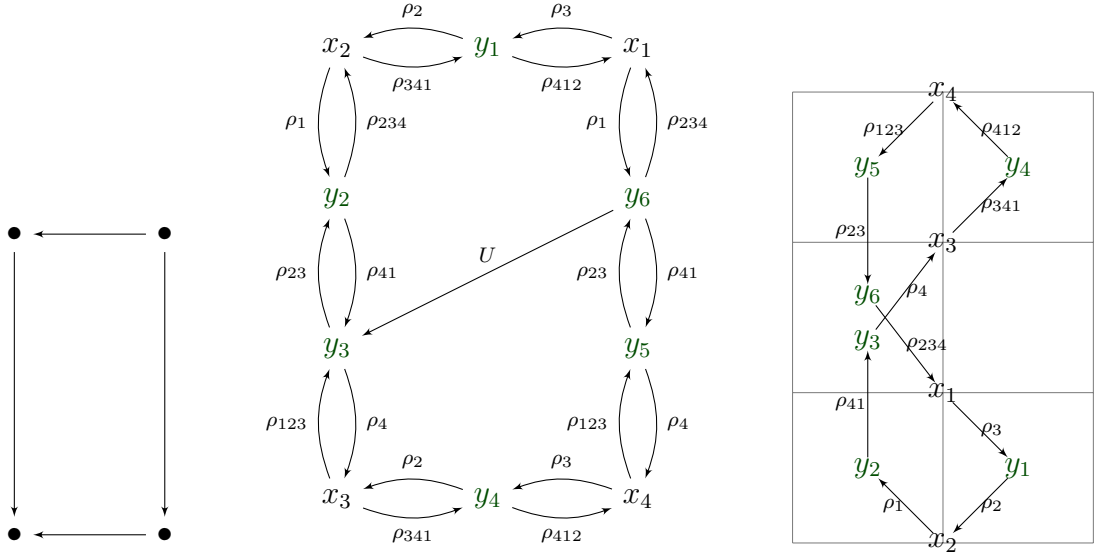
\begin{figure}
  \centering
  \begin{tikzpicture}
    \node at (0,0) (a1) {$\bullet$};
    \node at (2,0) (a2) {$\bullet$};
    \node at (0,-4) (a3) {$\bullet$};
    \node at (2,-4) (a4) {$\bullet$};
    \draw[->] (a2) to (a1);
    \draw[->] (a2) to (a4);
    \draw[->] (a1) to (a3);
    \draw[->] (a4) to (a3);
  \end{tikzpicture}
  \qquad
  \begin{tikzpicture}
    \node[text=xcolor] at (0,0) (a1) {$x_2$};
    \node[text=xcolor] at (4,0) (a2) {$x_1$};
    \node[text=xcolor] at (0,-6) (a3) {$x_3$};
    \node[text=xcolor] at (4,-6) (a4) {$x_4$};
    \node[text=ycolor] at (2,0) (b1) {$y_1$};
    \node[text=ycolor] at (0,-2) (b2) {$y_2$};
    \node[text=ycolor] at (4,-2) (b3) {$y_6$};
    \node[text=ycolor] at (0,-4) (b4) {$y_3$};
    \node[text=ycolor] at (4,-4) (b5) {$y_5$};
    \node[text=ycolor] at (2,-6) (b6) {$y_4$};
    \draw[->, bend right=20] (b1) to node[above]{\lab{\rho_2}} (a1);
    \draw[->, bend right=20] (a1) to node[below]{\lab{\rho_{341}}} (b1);
    \draw[->, bend right=20] (b1) to node[below]{\lab{\rho_{412}}} (a2);
    \draw[->, bend right=20] (a2) to node[above]{\lab{\rho_{3}}} (b1);
    \draw[->, bend right=20] (a1) to node[left]{\lab{\rho_{1}}} (b2);
    \draw[->, bend right=20] (b2) to node[right]{\lab{\rho_{234}}} (a1);
    \draw[->, bend left=20] (b4) to node[left]{\lab{\rho_{23}}} (b2);
    \draw[->, bend left=20] (b2) to node[right]{\lab{\rho_{41}}} (b4);
    \draw[->, bend left=20] (a3) to node[left]{\lab{\rho_{123}}} (b4);
    \draw[->, bend left=20] (b4) to node[right]{\lab{\rho_{4}}} (a3);
    \draw[->, bend right=20] (b6) to node[above]{\lab{\rho_2}} (a3);
    \draw[->, bend right=20] (a3) to node[below]{\lab{\rho_{341}}} (b6);
    \draw[->, bend right=20] (b6) to node[below]{\lab{\rho_{412}}} (a4);
    \draw[->, bend right=20] (a4) to node[above]{\lab{\rho_{3}}} (b6);
    \draw[->, bend right=20] (a2) to node[left]{\lab{\rho_{1}}} (b3);
    \draw[->, bend right=20] (b3) to node[right]{\lab{\rho_{234}}} (a2);
    \draw[->, bend left=20] (b5) to node[left]{\lab{\rho_{23}}} (b3);
    \draw[->, bend left=20] (b3) to node[right]{\lab{\rho_{41}}} (b5);
    \draw[->, bend left=20] (a4) to node[left]{\lab{\rho_{123}}} (b5);
    \draw[->, bend left=20] (b5) to node[right]{\lab{\rho_{4}}} (a4);
    \draw[->] (b3) to node[above]{\lab{U}} (b4);
  \end{tikzpicture}
  \qquad
  \begin{tikzpicture}[every node/.style={inner sep=0pt}, xscale=2, yscale=2]
    \draw[thin, gray] (-1,-1) grid (1,2);
    \node[text=xcolor] at (0,0) (x1) {$x_1$};
    \node[text=xcolor] at (0,-1) (x2) {$x_2$};
    \node[text=xcolor] at (0,1) (x3) {$x_3$};
    \node[text=xcolor] at (0,2) (x4) {$x_4$};
    \node[text=ycolor] at (.5,-.5) (y1) {$y_1$};
    \node[text=ycolor] at (-.5,-.5) (y2) {$y_2$};
    \node[text=ycolor] at (-.5,.35) (y3) {$y_3$};
    \node[text=ycolor] at (.5,1.5) (y4) {$y_4$};
    \node[text=ycolor] at (-.5,1.5) (y5) {$y_5$};
    \node[text=ycolor] at (-.5,.65) (y6) {$y_6$};
    \draw[->] (x1) to node[right]{\lab{\rho_3}} (y1);
    \draw[->] (y1) to node[right]{\lab{\rho_2}} (x2);
    \draw[->] (x2) to node[left]{\lab{\rho_{1}}} (y2);
    \draw[->] (y2) to node[left]{\lab{\rho_{41}}} (y3);
    \draw[->] (y3) to node[right]{\lab{\rho_{4}}} (x3);
    \draw[->] (x3) to node[right]{\lab{\rho_{341}}} (y4);
    \draw[->] (y4) to node[right]{\lab{\rho_{412}}} (x4);
    \draw[->] (x4) to node[left]{\lab{\rho_{123}}} (y5);
    \draw[->] (y5) to node[left]{\lab{\rho_{23}}} (y6);
    \draw[->] (y6) to node[right]{\lab{\rho_{234}}} (x1);
  \end{tikzpicture}
  \caption[A $1\times 2$ rectangle and the associated type $D$
  structure.]{\textbf{A $1\times 2$ rectangle and the associated type $D$
      structure.} Left: a possible summand of a knot Floer complex, a
    $1\times 2$ rectangle. Center: the associated module
    $\CFDm$. Right: the relative $\SpinC$-gradings. Conventions are as
    in Figures~\ref{fig:CFDa-T34} and~\ref{fig:CFDm-T34}.}
  \label{fig:rect-cx}
\end{figure}
Another interesting example comes from a $1\times 2$ rectangle shown
in Figure~\ref{fig:rect-cx}. (This is not the knot Floer complex of a
knot, but is a summand of the knot Floer complex of some knots.)  The
gradings are given by
\begin{center}
  \begin{tabular}{cccc}
    \toprule
    $x_1$&$x_2$&$x_3$&$x_4$\\
    $(0;0,0)$ & $(-1/2;0,-1)$ & $(3/2;0,1)$ & $(2;0,2)$ \\
  \end{tabular}
  \begin{tabular}{ccc}
    $y_1$&$y_2$&$y_3$\\ 
    $(-1/2;1/2,-1/2)$ & $(-1/2;-1/2,-1/2)$ & $(1/2;-1/2,1/2)$ \\
    $y_4$ & $y_5$&$y_6$\\
    $(5/2;1/2,3/2)$ & $(1/2;-1/2,3/2)$ & $(-1/2;-1/2,1/2)$ \\
    \bottomrule
  \end{tabular}
\end{center}
The grading ambiguity here is $(0;0,0)$. The relative $\SpinC$
gradings are drawn in Figure~\ref{fig:rect-cx}.

As usual, the type $D$ structure relation forces the differentials
shown. In particular, the term $U\otimes y_3$ in $\delta^1(y_6)$ is
forced by the higher product
$\mu_4(\rho_3,\rho_2,\rho_1,\rho_{41})=U\rho_1$. The relative gradings
imply that the only other possible differentials are between $x_2$ and
$x_1$, $x_1$ and $x_3$, or $x_3$ and $x_4$, all of which are
prohibited by the fact that there is no algebra element with
idempotent $\iota_1$ and intermediate grading $(*;0,\pm 1)$; between
$y_1$ and $y_6$ or $y_4$ and $y_5$, which are similarly prohibited by
the idempotents; or between $y_2$ and $y_6$, $y_3$ and $y_5$, $x_1$
and $y_2$, or $x_3$ and $y_5$. Taking these in turn, if
$\delta^1(y_2)=a\otimes y_6$ then $\gr(a)=(-1/2;-1,0)$, and no algebra
element has this grading. If $\delta^1(y_6)=a\otimes y_2$ then
$\gr(a)=(-3/2;0,1)$ and again no algebra element has this
grading. Similarly, $\delta^1(y_3)=a\otimes y_5$ implies
$\gr(a)=(-1/2;0,-1)$ and $\delta^1(y_5)=a\otimes y_3$ implies
$\gr(a)=(-3/2;0,1)$, and there are no algebra elements in these
gradings.  If $\delta^1(x_1)=a\otimes y_2$ then
$\gr(a)=(-1/2;1/2,1/2)$ so $a=\rho_2$ or $\rho_{123}$, which are
prohibited by Proposition~\ref{prop:extend-CFDa}.  If
$\delta^1(y_2)=a\otimes x_1$ then $\gr(a)=(-3/2;-1/2,-1/2)$ so
$a=\rho_4$, which is ruled out by the structure equation, or
$\rho_{341}$, which is prohibited by the idempotents.  If
$\delta^1(x_3)=a\otimes y_5$ then $\gr(a)=(-1/2;1/2,-1/2)$ so
$a=\rho_1$ which is prohibited by Proposition~\ref{prop:extend-CFDa}.
If $\delta^1(y_5)=a\otimes x_3$ then $\gr(a)=(-3/2;-1/2,1/2)$ so
$a=\rho_{234}$ which is ruled out by the structure equation.  So, we
have found all the terms.

It is also interesting to consider larger rectangles or pieces of
them; for instance, a length-$2$ horizontal arrow followed by a
length-$2$ vertical arrow gives a complex $\CFDm$ in which a
$\mu_6^0$-operation on the algebra contributes to the structure
relation.

The module $\CFDm$ may be much simpler with some framings than
others. For instance, recall $\CFDm$ for the complement of the $(3,4)$ torus knot with
framing $-6$, from Figure~\ref{fig:CFDm-T34}. Let us consider the
same knot but with framing $-1$. Its module and relative $\SpinC$
grading graph are given in Figure~\ref{fig:T34m1}.
\begin{figure}
  \centering
  \begin{tikzpicture}[xscale=.9]
    \node[text=xcolor] at (0,0) (a1) {$x_1$};
    \node[text=xcolor] at (0,-2) (a2) {$x_2$};
    \node[text=xcolor] at (6,-2) (a3) {$x_3$};
    \node[text=xcolor] at (6,-5) (a4) {$x_4$};
    \node[text=xcolor] at (10,-5) (a5) {$x_5$};
    \node[text=ycolor] at (0,-1) (b1) {$y_1$};
    \node[text=ycolor] at (2,-2) (b2) {$y_2$};
    \node[text=ycolor] at (4,-2) (b3) {$y_3$};
    \node[text=ycolor] at (6,-3) (b4) {$y_4$};
    \node[text=ycolor] at (6,-4) (b5) {$y_5$};
    \node[text=ycolor] at (8,-5) (b6) {$y_6$};
    \node[text=ycolor] at (-2, -2) (c1) {$y_{11}$};
    \node[text=ycolor] at (0,-4) (c2) {$y_{10}$};
    \node[text=ycolor] at (2,-5) (c3) {$y_9$};
    \node[text=ycolor] at (4,-6) (c4) {$y_{8}$};
    \node[text=ycolor] at (6,-7) (c5) {$y_{7}$};
    \draw[->, bend right=20] (a1) to node[left]{\lab{\rho_1}} (b1);
    \draw[->, bend right=20] (b1) to node[right]{\lab{\rho_{234}}} (a1); 
    \draw[->, bend left=20] (a2) to node[left]{\lab{\rho_{123}}} (b1);
    \draw[->, bend left=20] (b1) to node[right]{\lab{\rho_{4}}} (a2); 
    \draw[->, bend right=20] (a3) to node[above]{\lab{\rho_3}} (b3);
    \draw[->, bend right=20] (b3) to node[below]{\lab{\rho_{412}}} (a3); 
    \draw[->, bend right=20] (b3) to node[above]{\lab{\rho_{23}}} (b2);
    \draw[->, bend right=20] (b2) to node[below]{\lab{\rho_{41}}} (b3); 
    \draw[->, bend right=20] (b2) to node[above]{\lab{\rho_{2}}} (a2);
    \draw[->, bend right=20] (a2) to node[below]{\lab{\rho_{341}}} (b2); 
    \draw[->, bend right=20] (a3) to node[left]{\lab{\rho_1}} (b4);
    \draw[->, bend right=20] (b4) to node[right]{\lab{\rho_{234}}} (a3); 
    \draw[->, bend left=20] (b5) to node[left]{\lab{\rho_{23}}} (b4);
    \draw[->, bend left=20] (b4) to node[right]{\lab{\rho_{41}}} (b5); 
    \draw[->, bend left=20] (a4) to node[left]{\lab{\rho_{123}}} (b5);
    \draw[->, bend left=20] (b5) to node[right]{\lab{\rho_{4}}} (a4); 
    \draw[->, bend right=20] (a5) to node[above]{\lab{\rho_3}} (b6);
    \draw[->, bend right=20] (b6) to node[below]{\lab{\rho_{412}}} (a5); 
    \draw[->, bend right=20] (b6) to node[above]{\lab{\rho_2}} (a4);
    \draw[->, bend right=20] (a4) to node[below]{\lab{\rho_{341}}} (b6); 
    \draw[->, bend right=20] (a1) to node[right]{\lab{\rho_{341}}} (c1);
    \draw[->, bend left=70] (c1) to node[above]{\lab{\rho_{2}}} (a1);
    \draw[->, bend left=20] (c1) to node[above]{\lab{\rho_{41}}} (c2);
    \draw[->, bend left=20] (c2) to node[above]{\lab{\rho_{41}}} (c3);
    \draw[->, bend left=20] (c3) to node[above]{\lab{\rho_{41}}} (c4);
    \draw[->, bend left=20] (c4) to node[above]{\lab{\rho_{41}}} (c5);
    \draw[->, bend left=20] (c2) to node[below]{\lab{\rho_{23}}} (c1);
    \draw[->, bend left=20] (c3) to node[below]{\lab{\rho_{23}}} (c2);
    \draw[->, bend left=20] (c4) to node[below]{\lab{\rho_{23}}} (c3);
    \draw[->, bend left=20] (c5) to node[below]{\lab{\rho_{23}}} (c4);
    \draw[->, bend right=20] (c5) to node[above]{\lab{\rho_4}} (a5);
    \draw[->, bend left=90] (a5) to node[below]{\lab{\rho_{123}}} (c5);
    \draw[->] (c2) to node[above,sloped]{\lab{U}} (b2);
    \draw[->] (c2) to node[above,sloped]{\lab{U\rho_4}} (a3);
    \draw[->] (c3) to node[above,sloped]{\lab{U}} (b3);
    \draw[->] (c3) to node[above,sloped]{\lab{U}} (b4);
    \draw[->] (c4) to node[below,sloped]{\lab{U}} (b5);
    \draw[->] (c4) to node[above,sloped]{\lab{U\rho_2}} (a3);
  \end{tikzpicture}
  \quad
  \begin{tikzpicture}[every node/.style={inner sep=0pt}, xscale=1.5, yscale=1.5]
    \draw[thin, gray] (-1,0) grid (1,6);
    \node[text=xcolor] at (0,0) (x1) {$x_1$};
    \node[text=ycolor] at (-1/2,1/2) (y1) {$y_1$};
    \node[text=xcolor] at (0,1) (x2) {$x_2$};
    \node[text=ycolor] at (1/2,3/2) (y2) {$y_2$};
    \node[text=ycolor] at (1/2,5/2) (y3) {$y_3$};
    \node[text=xcolor] at (0,3) (x3) {$x_3$};
    \node[text=ycolor] at (-.4,7/2) (y4) {$y_4$};
    \node[text=ycolor] at (-.4,9/2) (y5) {$y_5$};
    \node[text=xcolor] at (0,5) (x4) {$x_4$};
    \node[text=ycolor] at (1/2,11/2) (y6) {$y_6$};
    \node[text=xcolor] at (0,6) (x5) {$x_5$};
    \node[text=ycolor] at (-1/2,11/2) (y7) {$y_7$};
    \node[text=ycolor] at (-.6,9/2) (y8) {$y_8$};
    \node[text=ycolor] at (-.6,7/2) (y9) {$y_9$};
    \node[text=ycolor] at (-1/2,5/2) (y10) {$y_{10}$};
    \node[text=ycolor] at (-1/2,3/2) (y11) {$y_{11}$};
    \node[text=xcolor] at (-1,1) (x1p) {$x_1$};
    \draw[->] (x1) to node[below,sloped]{\lab{\rho_1}} (y1);
    \draw[->] (y1) to node[below,sloped]{\lab{\rho_4}} (x2);
    \draw[->] (x2) to node[below,sloped]{\lab{\rho_{341}}} (y2);
    \draw[->] (y2) to node[left]{\lab{\rho_{41}}} (y3);
    \draw[->] (y3) to node[below,sloped]{\lab{\rho_{412}}} (x3);
    \draw[->] (x3) to node[below,sloped]{\lab{\rho_{1}}} (y4);
    \draw[->] (y4) to node[right]{\lab{\rho_{41}}} (y5);
    \draw[->] (y5) to node[below,sloped]{\lab{\rho_{4}}} (x4);
    \draw[->] (x4) to node[below,sloped]{\lab{\rho_{341}}} (y6);
    \draw[->] (y6) to node[above,sloped]{\lab{\rho_{412}}} (x5);
    \draw[->] (x5) to node[above,sloped]{\lab{\rho_{123}}} (y7);
    \draw[->] (y7) to node[left]{\lab{\rho_{23}}} (y8);
    \draw[->] (y8) to node[left]{\lab{\rho_{23}}} (y9);
    \draw[->] (y9) to node[left]{\lab{\rho_{23}}} (y10);
    \draw[->] (y10) to node[left]{\lab{\rho_{23}}} (y11);
    \draw[->] (y11) to node[above,sloped]{\lab{\rho_{2}}} (x1p);
    \draw[dashed] (x1p) .. controls (-1,.25) .. node[right]{\lab{=}} (x1);
   \end{tikzpicture}
   \caption[$\CFDm(S^3\setminus T_{-3,4},-1)$ and its relative
   gradings]{\textbf{The complex $\CFDm(S^3\setminus T_{-3,4},-1)$ (left) and the relative
       gradings on it (right).} Conventions are as in Figure~\ref{fig:CFDm-T34}.}
   \label{fig:T34m1}
 \end{figure}
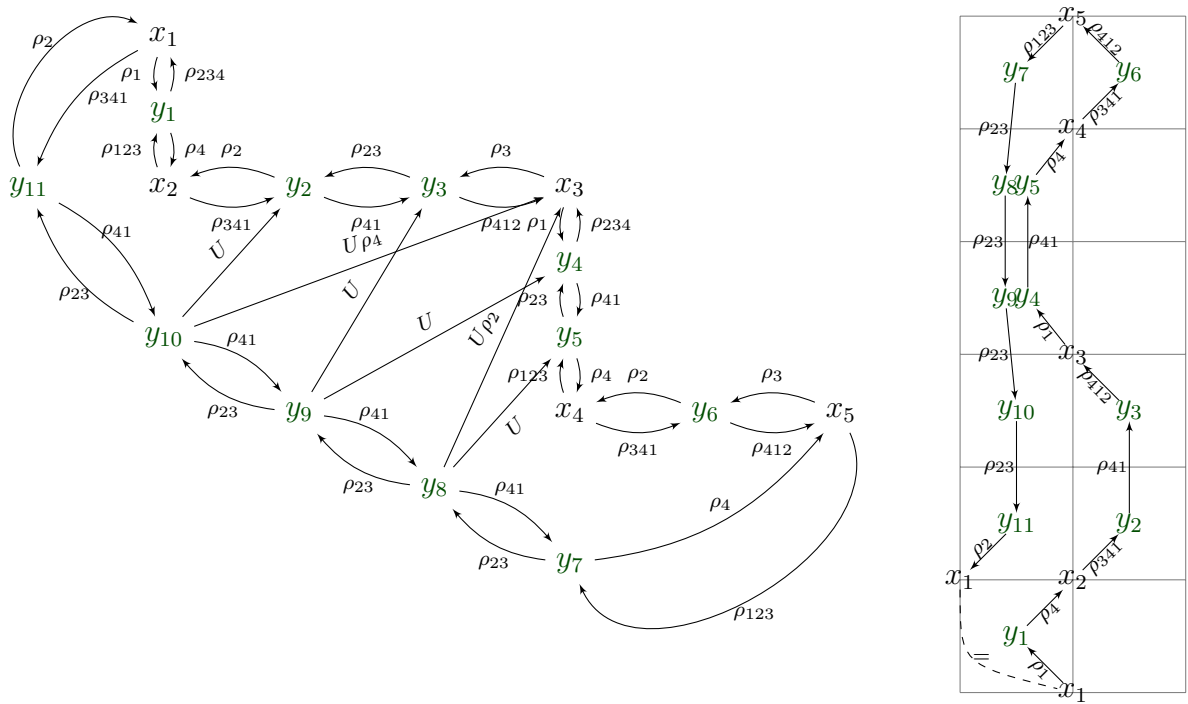

The grading ambiguity is $(6;1,1)$.  To see these are all the
terms, computer computation shows that the only possible differentials
respecting the relative gradings and idempotents are:
\begin{align*}
\delta^1(x_1) &= (\rho_1 + \boldsymbol{U^{3}\rho_3}) \otimes y_{1} + \rho_{341} \otimes y_{11}\\
\delta^1(x_{2}) &= \rho_{123} \otimes y_{1} + \rho_{341} \otimes y_{2}+ \rho_3 \otimes y_{11}\\
\delta^1(x_{3}) &= (\rho_3 + \rho_1) \otimes y_{3} + (\rho_1 + \rho_3) \otimes y_{4}\\ 
\delta^1(x_{4}) &= \rho_{123} \otimes y_{5} + \rho_{341} \otimes y_{6} + \rho_1 \otimes y_{7}\\
\delta^1(x_{5}) &= (\rho_3 + \boldsymbol{U^{3}\rho_1}) \otimes y_{6} + \rho_{123} \otimes y_{7}\\
  \delta^1(y_{1}) &= \rho_{234} \otimes x_1 + \rho_4 \otimes x_{2}\\
  \delta^1(y_{2}) &= \rho_2 \otimes x_{2} + \rho_{41} \otimes y_{3} + \boldsymbol{\rho_{41}} \otimes y_{4}\\
  \delta^1(y_{3}) &= \rho_{23} \otimes y_{2} + (\rho_{412} + \boldsymbol{\rho_{234}}) \otimes x_{3} + \boldsymbol{\rho_{41}} \otimes y_{5}\\
\delta^1(y_{4}) &= \rho_{23} \otimes y_{2} + (\rho_{234} + \rho_{412}) \otimes x_{3} + \rho_{41} \otimes y_{5}\\
\delta^1(y_{5}) &= \rho_{23} \otimes y_{3} + \rho_{23} \otimes y_{4} + \rho_4 \otimes x_{4}\\
\delta^1(y_{6}) &= \rho_2 \otimes x_{4} + \rho_{412} \otimes x_{5}\\ 
\delta^1(y_{7}) &= (\boldsymbol{\rho_{234}} + \boldsymbol{U^{2}\rho_{412}}) \otimes x_{4} + \rho_4 \otimes x_{5} + \rho_{23} \otimes y_{8}\\ 
\delta^1(y_{8}) &= U^{1}\rho_2 \otimes x_{3} + U^{1} \otimes y_{5} + \rho_{41} \otimes y_{7} + \rho_{23} \otimes y_{9}\\ 
\delta^1(y_{9}) &= U^{1} \otimes y_{3} + U^{1} \otimes y_{4} + \rho_{41} \otimes y_{8} + \rho_{23} \otimes y_{10}\\ 
\delta^1(y_{10}) &= U^{1}\otimes y_{2} + U^{1}\rho_4 \otimes x_{3} + \rho_{41} \otimes y_{9} + \rho_{23} \otimes y_{11}\\ 
\delta^1(y_{11}) &= \rho_2 \otimes x_1 + (\boldsymbol{U^{2}\rho_{234}} + \boldsymbol{\rho_{412}}) \otimes x_{2} + \rho_{41} \otimes y_{10}.
\end{align*}
Here, any factor of $U$ can also be replaced by a length-4 chord, so
for instance $U^3\rho_3$ also includes the possibilities of
$U^2\rho_{34123}$, $U\rho_{341234123}$, and
$\rho_{3412341234123}$. Terms which do not appear in
Figure~\ref{fig:T34m1} and which are not excluded by $\CFDa$ and
Proposition~\ref{prop:extend-CFDa} are in bold; these are the terms
we still must rule out. Additionally, any of the terms $U$ appearing
in Figure~\ref{fig:T34m1} could also be replaced by a length-4 chord
(e.g., the term $U\rho_2\otimes x_3$ in $\delta^1(y_8)$ could be
replaced by $\rho_{23412}\otimes x_3$), and we must rule out these
terms as well. 

We work roughly in order of increasing length. If the term
$\rho_{41}\otimes y_4$ occurs in $\delta^1(y_2)$ then the structure
equation implies that the term $\rho_{234}\otimes x_3$ must occur in
$\delta^1(y_3)$ but then $\rho_{2341}\otimes y_4$ occurs in the
structure equation with input $y_3$, a contradiction. So, neither of
these terms occurs. Then, a term $\rho_{41}\otimes
y_5$ in $\delta^1(y_3)$ violates the structure equation (with output
$\rho_{4123}\otimes y_4$). Adding the pair of terms
$\delta^1(y_7)=\rho_{234}\otimes x_4$ and
$\delta^1(y_8)=\rho_{4123}\otimes y_5$ does not violate the structure
equation, but these can be gauged away by replacing $y_7$ by
$y_7+\rho_{23}\otimes y_5$. A term $\delta^1(y_8)=\rho_{2341}\otimes
y_5$ forces a term $\delta^1(y_9)=\rho_{4123}y_4$, and this pair can
be gauged away by replacing $y_8$ by $y_8+\rho_{23}\otimes y_4$. A
term $\rho_{2341}\otimes y_4$ in $\delta^1(y_9)$ forces a term
$\rho_{41234}\otimes x_3$ in $\delta^1(y_{10})$, and this pair can be
gauged away by replacing $y_9$ by $y_9+\rho_{234}\otimes x_3$.

Perhaps it is worth pausing to summarize where we are. The remaining
terms to be ruled out are: $U^3\rho_3\otimes y_1$ and its cousins in $\delta^1(x_1)$;
$(U^2\rho_{234}+\rho_{412})\otimes x_2$ in $\delta^1(y_{11})$;
$(\rho_{2341}+\rho_{4123})\otimes y_2$ in $\delta^1(y_{10})$;
$(\rho_{2341}+\rho_{4123})\otimes y_3$ in $\delta^1(y_{9})$;
$\rho_{23412}\otimes x_3$ in $\delta^1(y_8)$; $U^2\rho_{412}\otimes
x_4$ and cousins in $\delta^1(y_7)$; and $U^3\rho_1\otimes y_6$ and cousins in
$\delta^1(y_7)$. All of these except the ones involving $U^3$, $U^2$,
and their cousins can be gauged away by a similar process to the
previous paragraph starting, say, by observing that a term
$\rho_{412}\otimes x_2$ in $\delta^1(y_{11})$ forces a term
$\rho_{2341}\otimes y_2$ in $\delta^1(y_{10})$, which can be gauged
away by replacing $y_{11}$ with $y_{11}+\rho_{41}\otimes y_2$.

A term $U^3\rho_3\otimes y_1$ in $\delta^1(x_1)$ or a term
$U^3\rho_1\otimes y_6$ in $\delta^1(x_5)$ violates the structure
equation. The pair $U^2\rho_{34123}\otimes y_1$ in $\delta^1(x_1)$ and
$U^2\rho_{234}$ in $\delta^1(y_{11})$ (which must occur together) can
be gauged away by replacing $y_{11}$ by $y_{11}+U^2\rho_{23}\otimes y_1$. The
pair $U\rho_{341234123}\otimes y_1$ in $\delta^1(x_1)$ and
$U\rho_{2341234}$ in $\delta^1(y_{11})$ (which must occur together)
can be gauged away by replacing $y_{11}$ by
$y_{11}+U\rho_{234123}\otimes y_1$. The pair $\rho_{3412341234123}\otimes y_1$
in $\delta^1(x_1)$ and $\rho_{23412341234}$ in $\delta^1(y_{11})$
(which must occur together) can be gauged away by replacing $y_{11}$
by $y_{11}+\rho_{2341234123}\otimes y_1$. The story for the terms in
$\delta^1(y_7)$ and $\delta^1(x_5)$ is similar, and is left to the
reader. This completes the verification of Figure~\ref{fig:T34m1}.

\begin{figure}
  \centering
  \begin{tikzpicture}[yscale=1.5, xscale=2]
    \node[text=xcolor] at (0,0) (a1) {$x_1$};
    \node[text=xcolor] at (0,-2) (a2) {$x_2$};
    \node[text=xcolor] at (2,-2) (a3) {$x_3$};
    \node[text=ycolor] at (0,-1) (b1) {$y_1$};
    \node[text=ycolor] at (1,-2) (b2) {$y_2$};
    \node[text=ycolor] at (1.5,-.5) (b3) {$y_3$};
    \draw[->, bend right=10] (a1) to node[left]{\lab{\rho_1+U\rho_3}} (b1);
    \draw[->, bend right=10] (b1) to node[right]{\lab{\rho_{234}}} (a1);
    \draw[->, bend right=10] (b1) to node[left]{\lab{\rho_4}} (a2);
    \draw[->, bend right=10] (a2) to node[right]{\lab{\rho_{123}}} (b1);
    \draw[->, bend right=10] (a2) to node[below]{\lab{\rho_{341}}} (b2);
    \draw[->, bend right=10] (b2) to node[above]{\lab{\rho_{2}}} (a2);
    \draw[->, bend right=10] (b2) to node[below]{\lab{\rho_{412}}} (a3);
    \draw[->, bend right=10] (a3) to node[above]{\lab{\rho_{3}+U\rho_1}} (b2);
    \draw[->, bend right=10] (a3) to node[right]{\lab{\rho_{123}}} (b3);
    \draw[->, bend right=10] (b3) to node[left]{\lab{\rho_{4}}} (a3);
    \draw[->, bend right=10] (b3) to node[above right]{\lab{\rho_2}} (a1);
    \draw[->, bend right=10] (a1) to node[left]{\lab{\rho_{341}}} (b3);
  \end{tikzpicture}\qquad
    \begin{tikzpicture}[every node/.style={inner sep=0pt}, xscale=2, yscale=2]
    \draw[thin, gray] (-1,0) grid (1,2);
    \node[text=xcolor] at (0,0) (x1) {$x_1$};
    \node[text=xcolor] at (0,1) (x2) {$x_2$};
    \node[text=xcolor] at (0,2) (x3) {$x_3$};
    \node[text=xcolor] at (-1,1) (x1p) {$x_1$};
    \node[text=ycolor] at (-.5,.5) (y1) {$y_1$};
    \node[text=ycolor] at (.5,1.5) (y2) {$y_2$};
    \node[text=ycolor] at (-.5,1.5) (y3) {$y_3$};
    \draw[->] (x1) to node[below left]{\lab{\rho_{1}}} (y1);
    \draw[->] (y1) to node[above left]{\lab{\rho_{4}}} (x2);
    \draw[->] (x2) to node[below right]{\lab{\rho_{341}}} (y2);
    \draw[->] (y2) to node[above right]{\lab{\rho_{412}}} (x3);
    \draw[->] (x3) to node[above left]{\lab{\rho_{123}}} (y3);
    \draw[->] (y3) to node[above left]{\lab{\rho_{2}}} (x1p);
    \draw[dashed] (x1p) .. controls (-1,.35) .. node[right]{\lab{=}} (x1);
    \end{tikzpicture}
    \caption[$\CFDm(S^3\setminus T_{-2,3},-1)$]{\textbf{The complex
        $\CFDm(S^3\setminus T_{-2,3},-1)$.} As usual, the complex is on the left
      and the grading graph is on the right.}
  \label{fig:T23}
\end{figure}
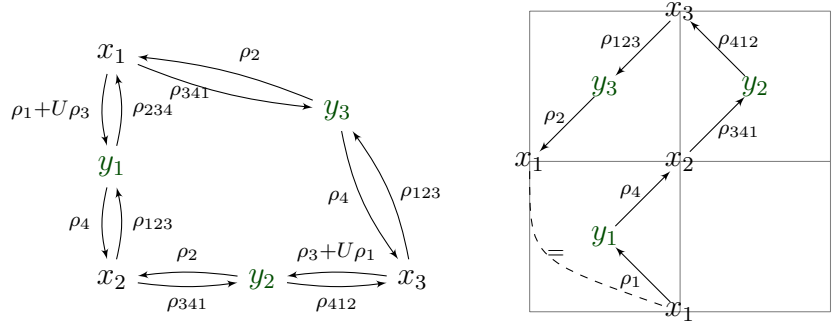

We end the section with a simpler example: the $(-1)$-framed
left-handed trefoil. The module $\CFDm(S^3\setminus T_{-2,3},-1)$ is
shown in Figure~\ref{fig:T23}. The gradings of the generators are
given by
\begin{center}
  \begin{tabular}{ccc}
    \toprule
    $x_1$&$x_2$&$x_3$\\
    $(0;0,0)$ & $(1/2;0,1)$ & $(1;0,2)$ \\
    \midrule
    $y_1$&$y_2$&$y_3$\\
    $(-1/2;-1/2,1/2)$ & $(3/2;1/2,3/2)$ & $(-1/2;-1/2,3/2)$ \\
     \bottomrule
  \end{tabular}
\end{center}
and the grading ambiguity is $(-2;-1,1)$. 

The possible differentials consistent with the gradings are:
\begin{align*}
\delta^1(x_1) &= (\rho_1 + U\rho_3) \otimes y_1 + \rho_{341} \otimes y_3\\ 
\delta^1(y_1) &= \rho_{234} \otimes x_1 + \rho_4 \otimes x_2\\ 
\delta^1(x_2) &= \rho_{123} \otimes y_1 + \rho_{341} \otimes y_2 + \boldsymbol{(\rho_1 + \rho_3)} \otimes y_3\\ 
\delta^1(y_2) &= \rho_2 \otimes x_2 + \rho_{412} \otimes x_3\\ 
\delta^1(x_3) &= (\rho_3 + U\rho_1) \otimes y_2 + \rho_{123} \otimes y_3\\ 
\delta^1(y_3) &= \rho_2 \otimes x_1 + \boldsymbol{(\rho_{234} + \rho_{412})} \otimes x_2 + \rho_4 \otimes x_3.
\end{align*}
All the terms not in bold are forced by the structure equation. (The
terms with coefficients $U\rho_3$ and $U\rho_2$ are forced by the
$\mu_4$ on the algebra, from $y_3$ to $y_1$ and from $y_3$ to $y_2$, respectively.)
As usual, we will rule out the terms in bold. The bold terms with
coefficient $\rho_1$ or $\rho_3$ are ruled out by
Proposition~\ref{prop:extend-CFDa}. The other two bold terms are ruled
out by the structure equation. (In fact, the first two bold terms are
also ruled out by the structure equation.)

\subsection{The type \texorpdfstring{$A$}{A} module for a solid torus}\label{sec:CFA-eg}
For genus $1$ Heegaard diagrams, it is straightforward to compute the
module $\CFAm$ to any finite order in $U$. In this section, we give
the first few operations on $\CFAm$ for two examples, which we will
use in Section~\ref{sec:surgery} for a surgery computation and a
cabling computation.

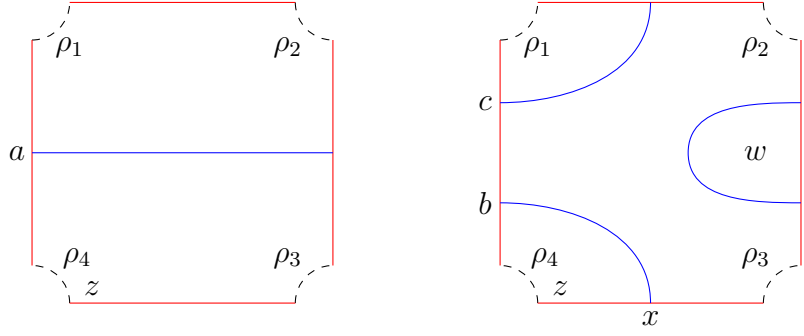
\begin{figure}
  \centering
  \begin{tikzpicture}[scale=2]
    \draw[dashed] (.25,0) arc[radius=.25, start angle = 0, end angle = 90];
    \draw[dashed] (0,1.75) arc[radius=.25, start angle = 270, end angle = 360];
    \draw[dashed] (1.75,0) arc[radius=.25, start angle = 180, end angle = 90];
    \draw[dashed] (1.75,2) arc[radius=.25, start angle = 180, end angle = 270];
    \draw[color = red] (0.25,0) to (1.75,0);
    \draw[color = red] (0,0.25) to (0,1.75);
    \draw[color = red] (0.25,2) to (1.75,2);
    \draw[color = red] (2,0.25) to (2,1.75);
    \node at (.25,1.7) (r1) {$\rho_1$};
    \node at (1.7,1.7) (r2) {$\rho_2$};
    \node at (1.7,.3) (r3) {$\rho_3$};
    \node at (.3,.3) (r4) {$\rho_4$};
    \node at (.4,.1) (z) {$z$};
    \draw[color = blue] (0,1) to (2,1);
    \node at (-.1,1) (a) {$a$};
    \node at (1,-.1) (x) {\phantom{$x$}};
  \end{tikzpicture}
  \qquad\qquad
  \begin{tikzpicture}[scale=2]
    \draw[dashed] (.25,0) arc[radius=.25, start angle = 0, end angle = 90];
    \draw[dashed] (0,1.75) arc[radius=.25, start angle = 270, end angle = 360];
    \draw[dashed] (1.75,0) arc[radius=.25, start angle = 180, end angle = 90];
    \draw[dashed] (1.75,2) arc[radius=.25, start angle = 180, end angle = 270];
    \draw[color = red] (.25,0) to (1.75,0);
    \draw[color = red] (0,0.25) to (0,1.75);
    \draw[color = red] (.25,2) to (1.75,2);
    \draw[color = red] (2,0.25) to (2,1.75);
    \node at (.25,1.7) (r1) {$\rho_1$};
    \node at (1.7,1.7) (r2) {$\rho_2$};
    \node at (1.7,.3) (r3) {$\rho_3$};
    \node at (.3,.3) (r4) {$\rho_4$};
    \node at (.4,.1) (z) {$z$};
    \draw[color = blue] (1,0) to[in=0, out=90] (0,.667);
    \draw[color = blue] (1,2) to[in=0,out=270] (0,1.333);
    \draw[color = blue] (2,.667) to[in=270,out=180] (1.25,1) to[in=180,out=90] (2,1.333);
    \node at (-.1,.667) (b) {$b$};
    \node at (-.1,1.333) (c) {$c$};
    \node at (1,-.1) (x) {$x$};
    \node at (1.7,1) (w) {$w$};
  \end{tikzpicture}  
  \caption[The $0$-framed solid torus and the $(2,1)$-cable]{\textbf{The
      $0$-framed solid torus and the $(2,1)$-cable.} }
  \label{fig:cable-HD}
\end{figure}

A bordered Heegaard diagram for a $0$-framed solid torus $\HD_0$ is shown in
Figure~\ref{fig:cable-HD}. As an $\Field[U]$-module,
$\CFAm(\HD_0)=\Field[U]\langle a\rangle$ is a free, rank one module. Since this
diagram has genus $1$, holomorphic curves correspond to immersions of disks in
the diagram.
Here are the operations $m_n^k$ for $n\leq 6$, as well as some of the
operations $m_8^0$:
\begin{align*}
  m_3^0(a,\rho_2,\rho_1)&=a &
  m_3^0(a,\rho_4,\rho_3)&=Ua\\
  m_4^0(a,\rho_2,\rho_{12},\rho_1)&=a&
  m_4^0(a,\rho_4,\rho_{34},\rho_3)&=U^2a\\
  m_5^0(a,\rho_2,\rho_{12},\rho_{12},\rho_1)&=a&
  m_5^0(a,\rho_4,\rho_{34},\rho_{34},\rho_3)&=U^3a\\
  m_5^0(a,\rho_{23},\rho_2,\rho_1,\rho_{41})&=Ua&
  m_5^0(a,\rho_{41},\rho_4,\rho_3,\rho_{23})&=U^2a\\
  m_6^0(a,\rho_2,\rho_{12},\rho_{12},\rho_{12},\rho_1)&=a&
  m_6^0(a,\rho_4,\rho_{34},\rho_{34},\rho_{34},\rho_3)&=U^4a\\
  m_6^0(a,\rho_{23},\rho_2,\rho_1,\rho_{412},\rho_1)&=Ua&
  m_6^0(a,\rho_2,\rho_{123},\rho_2,\rho_1,\rho_{41})&=Ua\\
  m_6^0(a,\rho_{41},\rho_4,\rho_3,\rho_{234},\rho_3)&=U^3a&
  m_6^0(a,\rho_4,\rho_{341},\rho_4,\rho_3,\rho_{341})&=U^3a\\
  m_6^1(a,\rho_{23},\rho_2,\rho_{12},\rho_1,\rho_{41})&=U^2a&
  m_6^1(a,\rho_{41},\rho_4,\rho_{34},\rho_3,\rho_{23})&=U^4a\\
  m_8^0(a,\rho_{23},\rho_2,\rho_{12},\rho_1,\rho_4,\rho_{3412},\rho_1)&=U^2a&
  m_8^0(a,\rho_{23},\rho_2,\rho_1,\rho_{4123},\rho_2,\rho_1,\rho_{41})&=U^2a\\
  m_8^0(a,\rho_2,\rho_{1234},\rho_3,\rho_2,\rho_{12},\rho_1,\rho_{41})&=U^2a&
  m_8^0(a,\rho_{4},\rho_{3412},\rho_1,\rho_4,\rho_{34},\rho_3,\rho_{23})&=U^4a\\
  m_8^0(a,\rho_{41},\rho_4,\rho_{34},\rho_3,\rho_2,\rho_{1234},\rho_3)&=U^4a&
  m_8^0(a,\rho_{41},\rho_4,\rho_3,\rho_{2341},\rho_4,\rho_3,\rho_{23})&=U^4a.
\end{align*}

It seems these operations can be encoded simply via an extension of a
construction of Hedden-Levine~\cite[Section 2.3]{HeddenLevine16:splicing}. Given
an operation $m_{1+m}^v(x_1,\rho^1,\dots,\rho^m)=U^kx_2$ and an operation
$m_{1+n}^w(x_2,\rho^{m+1},\dots,\rho^{m+n})=U^\ell x_3$, if $\rho^m\rho^{m+1}\neq 0$ there
is an induced operation
\[
  m_{m+n}^{v+w}(x_1,\rho^1,\dots,\rho^m\rho^{m+1},\dots,\rho^{m+n})=U^{k+\ell}x_3.
\]
This operation corresponds to gluing the edge of the first disk between $x_2$
and $\rho^m$ to the edge of the second disk between $x_2$ and
$\rho^{m+1}$. (Recall that we are discussing a genus-1 Heegaard diagram, so
curve counts are combinatorial.) For example, taking both of the input
operations to be the operation $m_3^0(a,\rho_2,\rho_1)=a$ above gives the new
operation $m_4^0(a,\rho_2,\rho_{12},\rho_1)=a$.

For $\CFAm$, there are two more ways to obtain new operations from old. First, given an
operation $m_{1+n}^w(x,\rho^{1},\dots,\rho^{n})=U^ky$ and an integer $j$ so that
$\rho^j\rho^{j+1}=0$, there is a new operation
\[
  m_{3+n}^w(x,\rho^{1},\dots,\rho^{j-1}, \rho^j\rho_{i},\rho_{i-1},\rho_{i-2},\rho_{i-3}\rho^{j+1},\rho^{j+2},\dots, \rho^{n})=U^{k+1}y.
\]
Here, $\rho_i$ is the unique length-1 chord so that $\rho^j\rho_{i}\neq 0$ (and
also $\rho_{i-3}\rho^{j+1}$).  This corresponds to gluing a copy of the whole
torus, cut open along the $\alpha$-arcs, to the disk inducing the operation,
along the edge between $\rho^j$ and $\rho^{j+1}$. For example, the operation
$m_3^0(a,\rho_4,\rho_3)=Ua$ induces a new operation
$m_5^0(a,\rho_{41},\rho_4,\rho_3,\rho_{23})=U^2a$ this way.

Second, given an operation $m_{1+n}^w(x,\rho^1,\dots,\rho^n)=U^ky$ and an
integer $1<j<n$ so that $\rho^j$ has length $4$, there is a new operation
\[
  m_{n-1}^{w+1}(x,\rho^1,\dots,\rho^{j-2},\rho^{j-1}\rho^{j+1},\rho^{j+2},\dots,\rho^n)=U^ky.
\]
This corresponds to gluing the edges on either side of the chord $\rho^j$
together. For example, the operation
\[
  m_8^0(a,\rho_2,\rho_{1234},\rho_3,\rho_2,\rho_{12},\rho_1,\rho_{41})=U^2a
\]
induces an operation
\[
  m_6^1(a,\rho_{23},\rho_2,\rho_{12},\rho_1,\rho_{41})=U^2a.
\]
The same operation can be induced in more than one way: this $m_6^1$
is also induced by
$m_8^0(a,\rho_{23},\rho_2,\rho_{12},\rho_1,\rho_4,\rho_{3412},\rho_1)$ and
$m_8^0(a,\rho_{23},\rho_2,\rho_1,\rho_{4123},\rho_2,\rho_1,\rho_{41})$.

In particular, this suggests encoding $\CFAm(\HD_0)$ with the directed graph
\[
\begin{tikzpicture}
  \node at (0,0) (x) {$\circ$};
  \node at (-1,0) (lphant) {};
  \node at (1,0) (rphant) {};
  \draw[->, bend left=60] (x) to node[left]{\lab{U\rho_4\otimes\rho_3}} (lphant) to (x);
  \draw[->, bend right=60] (x) to node[right]{\lab{\rho_2\otimes\rho_1}} (rphant) to (x);
\end{tikzpicture}.
\]
All the operations we have written down are obtained from this graph by
repeatedly performing the three constructions above.
We do not prove this description of $\CFAm$ here, 
though in this case and the next one, one could likely prove the graph
does, in fact, give $\CFAm$ by elementary combinatorics, similar to
the description of the operations on $\MAlg(T^2)$ via tiling patterns
in our previous paper~\cite[Section 3]{LOT:torus-alg}.

Turning next to the cabling operation, consider the second bordered
Heegaard diagram in Figure~\ref{fig:cable-HD}, which we will denote
$\mathcal{C}_2$. This bordered Heegaard diagram has an extra
basepoint, $w$. We use $U$ to track the number of times a holomorphic
curve crosses $w$ and introduce a new variable $V$ to track the number
of times a curve crosses $z$.
The resulting module $\CFAm_{U,V}(\mathcal{C}_2)$ is then a module not
over $\MAlg$ but rather over
$\MAlg^{U,V}=\MAlg\otimes_{\Field[U]}\Field[U,V]$, where the
operations are obtained from the operations on $\MAlg$ by replacing
every instance of the variable $U$ by the product $UV$.

Again, $\mathcal{C}_2$ has genus $2$, so computing the operations to any finite
order is combinatorial. The operations which output terms $U^iV^jr$ where $i+j\leq 3$ are:
\begin{align*}
  \boldsymbol{m^0_1(b)}&=Uc &
  \boldsymbol{m^0_2(x,\rho_1)}&=c\\
  \boldsymbol{m^0_2(b,\rho_4)}&=Vx &
  m^0_2(b,\rho_{41})&=Vc\\
  \boldsymbol{m^0_3(x,\rho_3,\rho_2)}&=U^2x &
  m^0_4(x,\rho_3,\rho_2,\rho_1)&=Ub\\
  \boldsymbol{m^0_4(c,\rho_4,\rho_3,\rho_2)}&=UVx &
  m^0_5(c,\rho_2,\rho_1,\rho_4,\rho_3)&=Vb\\
  m^0_5(x,\rho_{12},\rho_1,\rho_4,\rho_3)&=Vb &
  m^0_5(c,\rho_2,\rho_1,\rho_4,\rho_{34})&=V^2x\\
  \boldsymbol{m^0_5(c,\rho_2,\rho_1,\rho_4,\rho_{341})}&=V^2c &
  m^0_5(b,\rho_{412},\rho_1,\rho_4,\rho_3)&=V^2b\\
  m^0_5(x,\rho_3,\rho_{23},\rho_2,\rho_1)&=U^3b &
  m^0_5(b,\rho_{412},\rho_1,\rho_4,\rho_{34})&=V^3x\\
  m^0_5(b,\rho_{412},\rho_1,\rho_4,\rho_{341})&=V^3c &
  m^0_5(x,\rho_{12},\rho_1,\rho_4,\rho_{34})&=V^2x\\
  m^0_5(x,\rho_{12},\rho_1,\rho_4,\rho_{341})&=V^2c&
  m^0_6(x,\rho_{34},\rho_3,\rho_2,\rho_{12},\rho_1)&=U^2Vb\\
  m^0_6(x,\rho_3,\rho_{23},\rho_2,\rho_1,\rho_{41})&=U^2Vb &
  m^0_7(c,\rho_{23},\rho_2,\rho_1,\rho_{41},\rho_4,\rho_3)&=UV^2b\\
  m^0_7(c,\rho_2,\rho_{12},\rho_1,\rho_4,\rho_{34},\rho_3)&=UV^2b&
  m^0_7(c,\rho_2,\rho_1,\rho_{41},\rho_4,\rho_3,\rho_{23})&=UV^2b\\
  m^0_7(x,\rho_{123},\rho_2,\rho_1,\rho_{41},\rho_4,\rho_3)&=UV^2b &
  m^0_7(x,\rho_{12},\rho_{12},\rho_1,\rho_4,\rho_{34},\rho_3)&=UV^2b \\
  m^0_7(x,\rho_{12},\rho_1,\rho_{41},\rho_4,\rho_3,\rho_{23})&=UV^2b &
  m^0_8(c,\rho_2,\rho_1,\rho_4,\rho_{3412},\rho_1,\rho_4,\rho_3)&=V^3b\\
  m^0_8(x,\rho_{12},\rho_1,\rho_4,\rho_{3412},\rho_1,\rho_4,\rho_3)&=V^3b &
  m_6^{1}(c,\rho_2,\rho_1,\rho_{41},\rho_4,\rho_3)&=V^3b\\
  m_6^1(x,\rho_{12},\rho_1,\rho_{41},\rho_4,\rho_3)&=V^3b 
\end{align*}
The terms in bold are used in the cable computation in
Section~\ref{sec:surgery}, but have no special meaning here.

In terms of the operation $*$ introduced above, this module appears to
be determined by the following graph:
\[
  \begin{tikzpicture}[scale=2]
    \node at (0,0) (x) {$x$};
    \node at (2,1) (c) {$c$};
    \node at (2,-1) (b) {$b$};
    \draw[->, bend right=15] (x) to node[below]{\lab{\rho_1}} (c);
    \draw[->, bend right=15] (c) to node[above,sloped]{\lab{UV\rho_4\otimes\rho_3\otimes\rho_2}} (x);
    \draw[->, bend left=70] (x) to node[left]{\lab{U^2\rho_3\otimes\rho_2}} (-.5,0) to (x);
    \draw[->, bend right=15] (b) to node[right]{\lab{U}} (c);
    \draw[->, bend right=15] (c) to node[above, sloped]{\lab{V\rho_2\otimes\rho_1\otimes\rho_4\otimes\rho_3}} (b);
    \draw[->, bend right=15] (x) to node[below, sloped]{\lab{U\rho_3\otimes\rho_2\otimes\rho_1}} (b);
    \draw[->, bend right=15] (b) to node[above]{\lab{V\rho_4}} (x);
  \end{tikzpicture}.
\]
Again, we will not prove this graph determines the module $\CFAm$.

\subsection{Surgeries and satellites: examples}\label{sec:surgery}

\subsubsection{A surgery computation}
Assuming the pairing theorem, we can tensor the complex
$\CFDm(S^3\setminus T_{-3,4},-1)$ from Figure~\ref{fig:T34m1} with
$\CFAm(\HD_0)$ from Section~\ref{sec:CFA-eg} to obtain $\CFmm$ of $-1$
surgery on the left-handed $(3,4)$ torus knot.  The result is shown in
Figure~\ref{fig:m1-surg}.

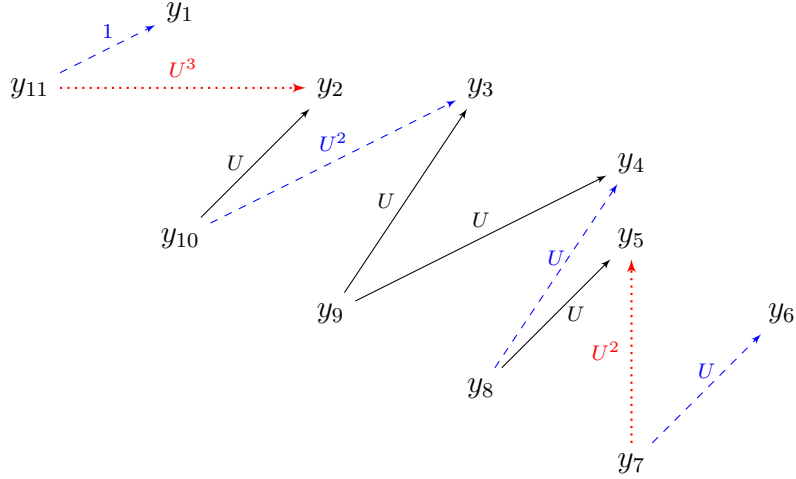
\begin{figure}
  \centering
  \[
\begin{tikzpicture}
  \node at (0,-1) (b1) {$y_{1}$};
  \node at (2,-2) (b2) {$y_2$};
  \node at (4,-2) (b3) {$y_3$};
  \node at (6,-3) (b4) {$y_4$};
  \node at (6,-4) (b5) {$y_5$};
  \node at (8,-5) (b6) {$y_6$};
  \node at (-2, -2) (c1) {$y_{11}$};
  \node at (0,-4) (c2) {$y_{10}$};
  \node at (2,-5) (c3) {$y_9$};
  \node at (4,-6) (c4) {$y_8$};
  \node at (6,-7) (c5) {$y_7$};
  \draw[->] (c2) to node[left]{\lab{U}} (b2);
  \draw[->] (c3) to node[left]{\lab{U}} (b3);
  \draw[->] (c3) to node[above]{\lab{U}} (b4);
  \draw[->] (c4) to node[right]{\lab{U}} (b5);
  \draw[->, blue, dashed] (c5) to node[above]{\lab{U}} (b6);
  \draw[->, blue, dashed] (c1) to node[above]{\lab{1}} (b1);
  \draw[->, blue, dashed] (c2) to node[above]{\lab{U^2}} (b3);
  \draw[->, blue, dashed] (c4) to node[above]{\lab{U}} (b4);
  \draw[->, red, dotted, thick] (c1) to node[above]{\lab{U^3}} (b2);
  \draw[->, red, dotted, thick] (c5) to node[left]{\lab{U^2}} (b5);
\end{tikzpicture}
\]
\caption[Negative 1 surgery on the left-handed (3,4) torus
knot]{\textbf{Negative 1 surgery on the left-handed (3,4) torus knot.} This is
  the tensor product $\CFAc(\HD_0)\DT\CFDm(S^3\setminus T_{-3,4},-1)$ of the
  complexes from Figure~\ref{fig:T34m1} and Section~\ref{sec:CFA-eg}. Generators
  $y_i$ stand for $a\otimes y_i$.}
\label{fig:m1-surg}
\end{figure}
Some of these operations are interesting. The black arrows come from the
differentials in Figure~\ref{fig:T34m1} labeled by $U$. The \textcolor{blue}{dashed} arrows come
from $m_3^0$ operations on $\CFAm(\HD_0)$. The \textcolor{red}{dotted} arrows come from the higher operations
\begin{align*}
  m_4^0(x,\rho_{41},\rho_4,\rho_3,\rho_{23})&=Ux\\
  m_4^0(x,\rho_{23},\rho_2,\rho_1,\rho_{41})&=Ux.
\end{align*}
on $\CFAm(\HD_0)$.

To see there are no other operations we use the relative gradings. By the
pairing theorem, the Maslov component gives a relative $\ZZ$-grading on the
complex. From Figure~\ref{fig:T34m1}, it is easy to compute these relative
Maslov gradings; up to a shift, they are:
\begin{center}
  \begin{tabular}{cccc}
    \toprule
    Generator & Grading & Generator & Grading\\
    \midrule
    $y_1$ & 0 & $y_{11}$ & 1\\
    $y_2$ & 6 & $y_{10}$ & 5\\
    $y_3$ & 8 & $y_{9}$ & 7\\
    $y_4$ & 8 & $y_{8}$ & 7\\
    $y_5$ & 8 & $y_{7}$ & 5\\
    $y_6$ & 6 \\
    \bottomrule
  \end{tabular}
\end{center}
Consequently, the highest $U$-power that can occur in a differential on the
tensor product is $U^4$. Thus, the $\Ainf$ operations on $\CFAm(\HD_0)$ that can
contribute can only use $\rho_4$ at most four times (including weights). Also,
the complex $\CFDm(S^3\setminus T_{-3,4},-1)$ has no pairs of consecutive arrows
labeled $\rho_2,\rho_1$. So, there are only finitely many operations on
$\CFAm(\HD_0)$ which could contribute; it would be straightforward for a
computer to enumerate them and check if they contribute in the tensor
product. (One could avoid the ad hoc discussion of arrows labeled
$\rho_2,\rho_1$ by appealing to the other components of the relative grading.)

\begin{figure}
  \centering
    \begin{tikzpicture}[scale=1]
    \begin{scope}[shift={(0,0)}]
      \draw[dashed] (0,1.75) arc[radius=.25, start angle = 270, end angle = 360];
      \draw[dashed] (1.75,2) arc[radius=.25, start angle = 180, end angle = 270];
      \draw[color = red] (0,1) to (0,1.75);
      \draw[color = red] (0.25,2) to (1.75,2);
      \draw[color = red] (2,1) to (2,1.75);
      \node at (.25,1.7) (r1) {$\rho_1$};
      \node at (1.7,1.7) (r2) {$\rho_2$};
      \draw[color = blue] (0,1) to (2,1);
    \end{scope}
    \begin{scope}[shift={(2,0)}]
      \draw[dashed] (0,1.75) arc[radius=.25, start angle = 270, end angle = 360];
      \draw[dashed] (1.75,2) arc[radius=.25, start angle = 180, end angle = 270];
      \draw[color = red] (0,1) to (0,1.75);
      \draw[color = red] (0.25,2) to (1.75,2);
      \draw[color = red] (2,1) to (2,1.75);
      \node at (.25,1.7) (r1) {$\rho_1$};
      \node at (1.7,1.7) (r2) {$\rho_2$};
      \draw[color = blue] (0,1) to (2,1);
    \end{scope}
    \begin{scope}[shift={(4,0)}]
      \draw[dashed] (0,1.75) arc[radius=.25, start angle = 270, end angle = 360];
      \draw[dashed] (1.75,2) arc[radius=.25, start angle = 180, end angle = 270];
      \draw[color = red] (0,1) to (0,1.75);
      \draw[color = red] (0.25,2) to (1.75,2);
      \draw[color = red] (2,1) to (2,1.75);
      \node at (.25,1.7) (r1) {$\rho_1$};
      \node at (1.7,1.7) (r2) {$\rho_2$};
      \draw[color = blue] (0,1) to (2,1);
    \end{scope}
  \end{tikzpicture}\qquad
  \begin{tikzpicture}[scale=1]
    \begin{scope}[shift={(0,0)}]
      \draw[dashed] (0,1.75) arc[radius=.25, start angle = 270, end angle = 360];
      \draw[dashed] (1.75,2) arc[radius=.25, start angle = 180, end angle = 270];
      \draw[color = red] (0,1) to (0,1.75);
      \draw[color = red] (0.25,2) to (1.75,2);
      \draw[color = red] (2,1) to (2,1.75);
      \node at (.25,1.7) (r1) {$\rho_1$};
      \node at (1.7,1.7) (r2) {$\rho_2$};
      \draw[color = blue] (0,1) to (2,1);
    \end{scope}
    \begin{scope}[shift={(2,0)}]
      \draw[dashed] (0,1.75) arc[radius=.25, start angle = 270, end angle = 360];
      \draw[dashed] (1.75,2) arc[radius=.25, start angle = 180, end angle = 270];
      \draw[color = red] (0,1) to (0,1.75);
      \draw[color = red] (0.25,2) to (1.75,2);
      \draw[color = red] (2,1) to (2,1.75);
      \node at (.25,1.7) (r1) {$\rho_1$};
      \node at (1.7,1.7) (r2) {$\rho_2$};
      \draw[color = blue] (0,1) to (2,1);
    \end{scope}
    \begin{scope}[shift={(2,2)}]
      \draw[dashed] (.25,0) arc[radius=.25, start angle = 0, end angle = 90];
      \draw[dashed] (0,1.75) arc[radius=.25, start angle = 270, end angle = 360];
      \draw[dashed] (1.75,0) arc[radius=.25, start angle = 180, end angle = 90];
      \draw[dashed] (1.75,2) arc[radius=.25, start angle = 180, end angle = 270];
      \draw[color = red] (0.25,0) to (1.75,0);
      \draw[color = red] (0,0.25) to (0,1.75);
      \draw[color = red] (0.25,2) to (1.75,2);
      \draw[color = red] (2,0.25) to (2,1.75);
      \node at (.25,1.7) (r1) {$\rho_1$};
      \node at (1.7,1.7) (r2) {$\rho_2$};
      \node at (1.7,.3) (r3) {$\rho_3$};
      \node at (.3,.3) (r4) {$\rho_4$};
      \draw[color = blue] (0,1) to (2,1);
    \end{scope}
    \begin{scope}[shift={(2,4)}]
      \draw[dashed] (.25,0) arc[radius=.25, start angle = 0, end angle = 90];
      \draw[dashed] (0,1.75) arc[radius=.25, start angle = 270, end angle = 360];
      \draw[dashed] (1.75,0) arc[radius=.25, start angle = 180, end angle = 90];
      \draw[dashed] (1.75,2) arc[radius=.25, start angle = 180, end angle = 270];
      \draw[color = red] (0.25,0) to (1.75,0);
      \draw[color = red] (0,0.25) to (0,1.75);
      \draw[color = red] (0.25,2) to (1.75,2);
      \draw[color = red] (2,0.25) to (2,1.75);
      \node at (.25,1.7) (r1) {$\rho_1$};
      \node at (1.7,1.7) (r2) {$\rho_2$};
      \node at (1.7,.3) (r3) {$\rho_3$};
      \node at (.3,.3) (r4) {$\rho_4$};
      \draw[color = blue] (0,1) to (2,1);
    \end{scope}
        \begin{scope}[shift={(0,4)}]
      \draw[dashed] (.25,0) arc[radius=.25, start angle = 0, end angle = 90];
      \draw[dashed] (0,1.75) arc[radius=.25, start angle = 270, end angle = 360];
      \draw[dashed] (1.75,0) arc[radius=.25, start angle = 180, end angle = 90];
      \draw[dashed] (1.75,2) arc[radius=.25, start angle = 180, end angle = 270];
      \draw[color = red] (0.25,0) to (1.75,0);
      \draw[color = red] (0,0.25) to (0,1.75);
      \draw[color = red] (0.25,2) to (1.75,2);
      \draw[color = red] (2,0.25) to (2,1.75);
      \node at (.25,1.7) (r1) {$\rho_1$};
      \node at (1.7,1.7) (r2) {$\rho_2$};
      \node at (1.7,.3) (r3) {$\rho_3$};
      \node at (.3,.3) (r4) {$\rho_4$};
      \draw[color = blue] (0,1) to (2,1);
    \end{scope}
    \begin{scope}[shift={(0,6)}]
      \draw[dashed] (.25,0) arc[radius=.25, start angle = 0, end angle = 90];
      \draw[dashed] (0,1.75) arc[radius=.25, start angle = 270, end angle = 360];
      \draw[dashed] (1.75,0) arc[radius=.25, start angle = 180, end angle = 90];
      \draw[dashed] (1.75,2) arc[radius=.25, start angle = 180, end angle = 270];
      \draw[color = red] (0.25,0) to (1.75,0);
      \draw[color = red] (0,0.25) to (0,1.75);
      \draw[color = red] (0.25,2) to (1.75,2);
      \draw[color = red] (2,0.25) to (2,1.75);
      \node at (.25,1.7) (r1) {$\rho_1$};
      \node at (1.7,1.7) (r2) {$\rho_2$};
      \node at (1.7,.3) (r3) {$\rho_3$};
      \node at (.3,.3) (r4) {$\rho_4$};
      \draw[color = blue] (0,1) to (2,1);
    \end{scope}
  \end{tikzpicture}
    \qquad
  \begin{tikzpicture}[scale=1]
    \begin{scope}[shift={(0,0)}]
      \draw[dashed] (0,1.75) arc[radius=.25, start angle = 270, end angle = 360];
      \draw[dashed] (1.75,2) arc[radius=.25, start angle = 180, end angle = 270];
      \draw[color = red] (0,1) to (0,1.75);
      \draw[color = red] (0.25,2) to (1.75,2);
      \draw[color = red] (2,1) to (2,1.75);
      \node at (.25,1.7) (r1) {$\rho_1$};
      \node at (1.7,1.7) (r2) {$\rho_2$};
      \draw[color = blue] (0,1) to (2,1);
    \end{scope}
    \begin{scope}[shift={(2,0)}]
      \draw[dashed] (0,1.75) arc[radius=.25, start angle = 270, end angle = 360];
      \draw[dashed] (1.75,2) arc[radius=.25, start angle = 180, end angle = 270];
      \draw[color = red] (0,1) to (0,1.75);
      \draw[color = red] (0.25,2) to (1.75,2);
      \draw[color = red] (2,1) to (2,1.75);
      \node at (.25,1.7) (r1) {$\rho_1$};
      \node at (1.7,1.7) (r2) {$\rho_2$};
      \draw[color = blue] (0,1) to (2,1);
    \end{scope}
    \begin{scope}[shift={(2,2)}]
      \draw[dashed] (.25,0) arc[radius=.25, start angle = 0, end angle = 90];
      \draw[dashed] (0,1.75) arc[radius=.25, start angle = 270, end angle = 360];
      \draw[dashed] (1.75,0) arc[radius=.25, start angle = 180, end angle = 90];
      \draw[dashed] (1.75,2) arc[radius=.25, start angle = 180, end angle = 270];
      \draw[color = red] (0.25,0) to (1.75,0);
      \draw[color = red] (0,0.25) to (0,1.75);
      \draw[color = red] (0.25,2) to (1.75,2);
      \draw[color = red] (2,0.25) to (2,1.75);
      \node at (.25,1.7) (r1) {$\rho_1$};
      \node at (1.7,1.7) (r2) {$\rho_2$};
      \node at (1.7,.3) (r3) {$\rho_3$};
      \node at (.3,.3) (r4) {$\rho_4$};
      \draw[color = blue] (0,1) to (2,1);
    \end{scope}
    \begin{scope}[shift={(2,4)}]
      \draw[dashed] (.25,0) arc[radius=.25, start angle = 0, end angle = 90];
      \draw[dashed] (0,1.75) arc[radius=.25, start angle = 270, end angle = 360];
      \draw[dashed] (1.75,0) arc[radius=.25, start angle = 180, end angle = 90];
      \draw[dashed] (1.75,2) arc[radius=.25, start angle = 180, end angle = 270];
      \draw[color = red] (0.25,0) to (1.75,0);
      \draw[color = red] (0,0.25) to (0,1.75);
      \draw[color = red] (0.25,2) to (1.75,2);
      \draw[color = red] (2,0.25) to (2,1.75);
      \node at (.25,1.7) (r1) {$\rho_1$};
      \node at (1.7,1.7) (r2) {$\rho_2$};
      \node at (1.7,.3) (r3) {$\rho_3$};
      \node at (.3,.3) (r4) {$\rho_4$};
      \draw[color = blue] (0,1) to (2,1);
    \end{scope}
        \begin{scope}[shift={(0,4)}]
      \draw[dashed] (.25,0) arc[radius=.25, start angle = 0, end angle = 90];
      \draw[dashed] (0,1.75) arc[radius=.25, start angle = 270, end angle = 360];
      \draw[dashed] (1.75,0) arc[radius=.25, start angle = 180, end angle = 90];
      \draw[dashed] (1.75,2) arc[radius=.25, start angle = 180, end angle = 270];
      \draw[color = red] (0.25,0) to (1.75,0);
      \draw[color = red] (0,0.25) to (0,1.75);
      \draw[color = red] (0.25,2) to (1.75,2);
      \draw[color = red] (2,0.25) to (2,1.75);
      \node at (.25,1.7) (r1) {$\rho_1$};
      \node at (1.7,1.7) (r2) {$\rho_2$};
      \node at (1.7,.3) (r3) {$\rho_3$};
      \node at (.3,.3) (r4) {$\rho_4$};
      \draw[color = blue] (0,1) to (2,1);
    \end{scope}
    \begin{scope}[shift={(0,2)}]
      \draw[dashed] (.25,0) arc[radius=.25, start angle = 0, end angle = 90];
      \draw[dashed] (0,1.75) arc[radius=.25, start angle = 270, end angle = 360];
      \draw[dashed] (1.75,0) arc[radius=.25, start angle = 180, end angle = 90];
      \draw[dashed] (1.75,2) arc[radius=.25, start angle = 180, end angle = 270];
      \draw[color = red] (0.25,0) to (1.75,0);
      \draw[color = red] (0,0.25) to (0,1.75);
      \draw[color = red] (0.25,2) to (1.75,2);
      \draw[color = red] (2,0.25) to (2,1.75);
      \node at (.25,1.7) (r1) {$\rho_1$};
      \node at (1.7,1.7) (r2) {$\rho_2$};
      \node at (1.7,.3) (r3) {$\rho_3$};
      \node at (.3,.3) (r4) {$\rho_4$};
      \draw[color = blue] (0,1) to (2,1);
    \end{scope}
  \end{tikzpicture}
  \caption[Domains of some holomorphic curves.]{\textbf{Domains of some
      holomorphic curves.} Left: a long rectangle. Center: a domain with chord
    sequence
    $\rho_{23},\rho_{23},\rho_{2},\rho_{123},\rho_2,\rho_1,\rho_{41},\rho_4,\rho_{341},\rho_{412},\rho_1$
    which, in particular, contains $\rho_2,\rho_1$. Right: a domain with no
    consecutive $\rho_2,\rho_1$ pair.}
  \label{fig:tens-eg-domains}
\end{figure}
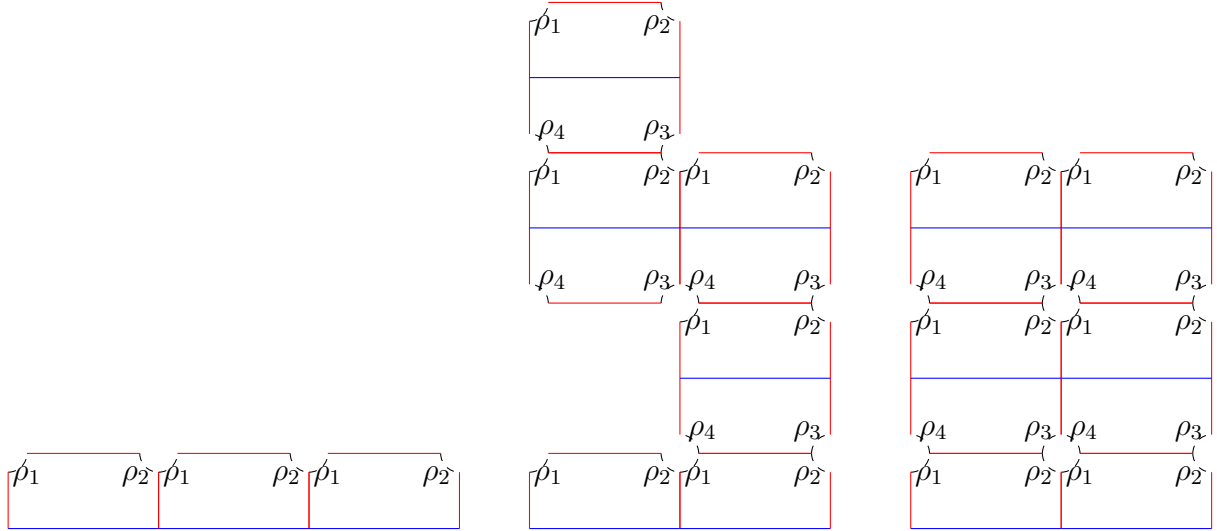

Rather than doing this, we rule out additional terms with coefficient $U^n$,
$n\leq 4$, by a short, ad hoc argument. Consider an operation
$m_{1+n}^k(a,\rho^1,\dots,\rho^n)=U^ma$ on $\CFAm(\HD_0)$ which contributes in
the tensor product. From the discussion above, the source of the holomorphic
curve representing the operation is built from some long horizontal strip,
with chord sequence either $\rho_2,\rho_{12},\dots,\rho_{12},\rho_1$ or
$\rho_4,\rho_{23},\dots,\rho_{23},\rho_4$, by attaching up to four copies of the
rectangle (with chords $\rho_4,\rho_3,\rho_2,\rho_1$). See
Figure~\ref{fig:tens-eg-domains}.

There are no arrows in Figure~\ref{fig:T34m1} labeled by $\rho_{12}$ or
$\rho_{34}$, so the glued-up domain cannot have any such chords. (In particular,
this limits the length of the horizontal strip to at most 9, since it
has at most 4 rectangles glued to it.) Consider the case
that the horizontal strip has a chord sequence of the form
$\rho_2,\rho_{12},\dots,\rho_{12},\rho_1$; the other case is similar but
slightly easier (because copies of $\rho_4$ contribute $U$). The only sequences
of arrows in Figure~\ref{fig:T34m1} containing $\rho_2,\rho_1$ as consecutive terms are
\[
  \cdots,\rho_{23},\rho_2,\rho_1,\rho_4,\rho_{341},\dots
\]
and
\[
\cdots,\rho_{23},U\rho_2,\rho_1,\rho_{41},\rho_4,\rho_{341},\dots.
\]
It is impossible to glue together at most four rectangles and a strip to obtain
a chord sequence $\rho_2,\rho_1,\rho_4,\rho_{341}$ or a chord sequence
$\cdots,\rho_{23},\rho_2,\rho_1$ or $\cdots,\rho_{23},\rho_2,\rho_1,\rho_4$
(that ends at $\rho_1$ or $\rho_4$, respectively). So, the only way for the
first sequence to contribute is $m_{3}^0(a,\rho_2,\rho_1)=a$, which we have
already counted. Similarly, for the second term, if gluing together a strip and
some rectangles gives a chord sequence containing
$\rho_2,\rho_1,\rho_{41},\rho_4$ then it must have two rectangles glued together
vertically, with the top edge of the top rectangle not glued to anything, and
the preceding term in the chord sequence then has length at least $3$, so is not
$\rho_{23}$. Finally, if there is no consecutive pair $\rho_2,\rho_1$
then the domain has chords $\rho_{12}$ on its boundary, which we already excluded.

The fact that this tensor product computes $\CFmm$ uses the fact that the complex
$\CFDm(S^3\setminus T_{-3,4},-1)$ is filtered bounded with respect to a
filtration that takes into account the multiplicity at $\rho_4$, and
$\CFAc(\HD_0)$ is filtered bonsai with respect to this filtration.

It is interesting to compare this answer with the integer surgery formula~\cite{IntSurg}.

\subsubsection{A cabling computation}
We conclude by computing the complex $\CFKm$ for the $(-2,1)$-cable of the
left-handed trefoil, giving another derivation of
Formula~\eqref{eq:Hedden-comp}. Note that the $(-2,1)$-cable of the
$\infty$-framed trefoil is the same as the $(2,1)$-cable of the $-1$-framed
trefoil. So, we can obtain the desired complex by tensoring the module
$\CFDm(S^3\setminus T_{-2,3},-1)$ from Section~\ref{sec:CFD-eg} with the module
$\CFAm_{U,V}(\mathcal{C}_2)$ from Section~\ref{sec:CFA-eg}. More precisely, we
must first extend scalars for $\CFDm(S^3\setminus T_{-2,3},-1)$ from $\MAlg$ to
$\MAlg^{U,V}$; this means that instead of the terms $U\rho_3\otimes y_1$ in
$\delta^1(x_1)$ and $U\rho_1\otimes y_2$ in $\delta^1(x_3)$ we have terms
$UV\rho_3\otimes y_1$ and $UV\rho_1\otimes y_2$. The module
$\CFAc(\mathcal{C}_2)$ is filtered bonsai (with respect to the filtration by
total $U$ plus $V$ power), so the tensor product in question is well-defined
(over $\FF_2[[U,V]]$).

The resulting complex is shown in Figure~\ref{fig:2m1comp}. The
differentials come from the operations on $\CFAc(\mathcal{C}_2)$ in
Section~\ref{sec:CFA-eg} displayed in bold. Note that the differential
from $x\otimes x_3$ to $c\otimes y_2$ has coefficient $UV$ because of
the way we extended scalars on $\CFDm(S^3\setminus T_{-2,3},-1)$ from
$\FF_2[[U]]$ to $\FF_2[[U,V]]$: this came from the differential
$x_3\to V\rho_1\otimes y_2$ on $\CFDm(S^3\setminus T_{-2,3},-1)$.

To see that there are no other differentials, note that the differential
preserves the $U$-grading (i.e., the Maslov index minus twice $n_w$) and the
$V$-grading (the Maslov index minus twice $n_z$). So, any other term would
output $U^iV^j$ with $i+j\leq 3$; but all such operations on
$\CFAc(\mathcal{C}_2)$ are listed in Section~\ref{sec:CFA-eg}.

\begin{figure}
  \centering
  \begin{tikzpicture}[scale=2]
    \node at (0.5,0.5) (x1) {$x x_1$};
    \node at (0,0) (cy1) {$c y_1$};
    \node at (1,0) (by1) {$b y_1$};
    \node at (1,-1) (x2) {$x x_2$};
    \node at (2,0) (cy3) {$c y_3$};
    \node at (3,0) (by3) {$b y_3$};
    \node at (3,-1) (x3) {$x x_3$};
    \node at (2,-2) (cy2) {$c y_2$};
    \node at (3,-2) (by2) {$b y_2$};
    \draw[->] (x1) to (cy1);
    \draw[->] (by1) to node[above]{\lab{U}} (cy1);
    \draw[->] (by1) to node[left]{\lab{V}} (x2);
    \draw[->] (cy3) to node[above left]{\lab{UV}} (x2);
    \draw[->] (by3) to node[above]{\lab{U}} (cy3);
    \draw[->] (by3) to node[left]{\lab{V}} (x3);
    \draw[->] (x3) to node[above,xshift=-2em]{\lab{U^2}} (x2);
    \draw[->] (cy3) to node[left, yshift=-2em]{\lab{V^2}} (cy2);
    \draw[->] (by2) to node[above]{\lab{U}} (cy2);
    \draw[->] (x3) to node[above left]{\lab{UV}} (cy2);
  \end{tikzpicture}
  \caption[Computation of the $(-2,1)$-cable of the left-handed
  trefoil]{\textbf{Computation of the $(-2,1)$-cable of the left-handed
      trefoil.} We have dropped tensor product symbols from the notation so, for
    instance, $xx_1$ means $x\otimes x_1$.}
  \label{fig:2m1comp}
\end{figure}
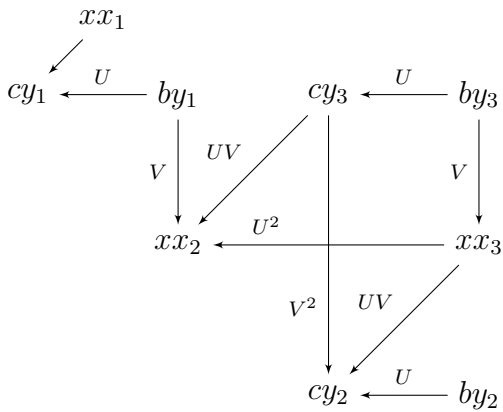


\bibliographystyle{../common/hamsalpha}\bibliography{../common/heegaardfloer}
\end{document}